\documentclass[10pt]{article}
\usepackage{graphicx} 
\usepackage{amsmath,amssymb,amsthm,mathrsfs,dsfont}
\usepackage[margin=3cm]{geometry} 
\usepackage{titlesec,hyperref}
\usepackage{bm}
\usepackage{color}
\usepackage{fancyhdr}
\titleformat{\subsection}{\it}{\thesubsection.\enspace}{1pt}{}

\usepackage{titlesec}
\titleformat{\subsubsection}
{\normalfont\normalsize\itshape} 
{\thesubsubsection}              
{1em}                           
{}                              
[]                              

\newtheorem{theo}{Theorem}[section]
\newtheorem{lemm}[theo]{Lemma}
\newtheorem{defi}{Definition}[section]
\newtheorem{coro}[theo]{Corollary}
\newtheorem{prop}[theo]{Proposition}
\newtheorem{rema}[theo]{Remark}
\numberwithin{equation}{section}

\let\tri=\triangleq

\allowdisplaybreaks

\title{Non‑implosion mechanism of 3D incompessible Euler equations}
\author{
	\vspace{12pt}
	Wenjie Deng\textsuperscript{1}\thanks{E-mail:
		detective2028@qq.com},
	\quad
	Song Jiang\textsuperscript{2}\thanks{E-mail: jiang@iapcm.ac.cn},
	\quad 
	Minling Li\textsuperscript{3}\thanks{E-mail: limling@pku.edu.cn}, \quad
	Zhaonan Luo\textsuperscript{4}\thanks{E-mail: luozhn7@mail.sysu.edu.cn} 
	\\
	\normalsize\textsuperscript{1}{Institute of Applied Physics and Computational Mathematics,
		Beijing, 100088, China}\\
	\normalsize\textsuperscript{2}{LCP, Institute of Applied Physics and Computational Mathematics, Beijing, 100088, China}\\
	\normalsize\textsuperscript{3}{School of Mathematical Sciences, Peking University,
		Beijing, 100871, China}\\
	\normalsize\textsuperscript{4}{School of Science, Shenzhen Campus of Sun Yat-sen University,
		Shenzhen, 518107, China}
}

\date{}

\begin{document}

\maketitle
\begin{abstract}
	This paper studies the non-implosion mechanism for the 3D incompressible Euler equations. We prove that vorticity blows up in finite time, whereas the $L^p_T L^\infty_{loc}$ $(p\in[1,\infty))$ norm of the velocity field remains bounded. Moreover, under an appropriate assumption on the scaling index, the exponent $p$ can be taken to be infinite. The proof is based on the introduction of a refined framework, the new observations for the null structure of transport term, and stability analysis of the self-similar model.
\end{abstract}

\vspace*{10pt}

\tableofcontents

\section{Introduction}
In this paper, we intend to investigate the Euler equations governing ideal fluid flow as:
\begin{align}\label{1eq1}
	\left\{\begin{array}{l}
		\displaystyle \partial_t \bm{u}+ \bm{u}\cdot\nabla  \bm{u}+\nabla P =0,\quad t\in\mathbb{R}^+,~~\bm{x}\in \mathbb{R}^3,
		\\[1ex]
		\displaystyle {\rm div}~ \bm{u}=0,
		\\[1ex]
		\displaystyle  \bm{u}|_{t=0}= \bm{u}_0,
	\end{array}\right.
\end{align}
where $ \bm{u}$ denotes the velocity field of the fluid, and $P$ represents the force of internal pressure. A central quantity in fluid dynamics is the vorticity field $\bm{\omega}\tri \nabla \times  \bm{u}$, which evolves according to
\begin{align}\label{1eq1-1}
	\left\{\begin{array}{l}
		\displaystyle \partial_t \bm{\omega}+ \bm{u}\cdot\nabla \bm{\omega} =\bm{\omega}\cdot\nabla \bm{u},~~~~{\rm div}~\bm{\omega}=0,\quad t\in\mathbb{R}^+,~~\bm{x}\in \mathbb{R}^3,
		\\[1ex]
		\displaystyle \bm{\omega}|_{t=0}=\bm{\omega}_0=\nabla \times  \bm{u}_0,
	\end{array}\right.
\end{align}
The velocity field can be reconstructed from vorticity via the Biot-Savart law:
\begin{align}\label{1eq1-2}
	\bm{u}=(-\Delta)^{-1}(\nabla \times  \bm{\omega}).
\end{align}

Research on the Euler equations \eqref{1eq1} has advanced significantly in the past century, profoundly deepening our understanding of inviscid flow.
Relevant references are summarized below to further elucidate the behavior of inviscid fluids. The early foundational work by Lichtenstein~\cite{L1925} and Gunther~\cite{G1927} established local well-posedness in Hölder spaces with suitable decay. The two-dimensional case saw significant advances with Yudovich \cite{Y63}, who proved the existence and uniqueness of global weak solutions. For the three-dimensional case, a major breakthrough came with the work of Beale et al.~\cite{BKM84}, which provided the BKM criterion for determining the global existence of strong solutions to the Euler equations.
To be specific, a classical solution for equations \eqref{1eq1-1}-\eqref{1eq1-2} cannot be continued in class $\{\bm{u}| \bm{u}\in C([0,T);H^s(\mathbb{R}^3))\cap C^1([0,T);H^{s-1}(\mathbb{R}^3)), s\geq3 \}$ as $t$ goes to $T$ if and only if
\begin{align*}
	\lim\limits_{t\rightarrow T}\int_0^t \sup_{\bm{x}\in \mathbb{R}^3}|\bm{\omega}(s,\bm{x})|ds=+\infty.
\end{align*}
Further contributions include the double exponential growth result for 2D flows by Kiselev and \v{S}ver\'{a}k~\cite{KS14}, local well-posedness with critical regularity by Guo et al.~\cite{GLY19}, ill-posedness in borderline Sobolev spaces by Bourgain and Li~\cite{BL15}, and the resolution of Onsager’s conjecture by Isett~\cite{Is18} and Buckmaster et al.~\cite{BLJV19}. Among many other notable contributions not covered in detail here, these works have collectively advanced our mathematical understanding of ideal fluids.

However, a fundamental gap persists between current mathematical theory and the understanding required to fully characterize solutions of the Euler equations, prompting the following global regularity problem: \\

\noindent
\textbf{Question 1:} 
Let $\bm{u}\in C^\infty([0,T)\times\mathbb{R}^3)$ be a smooth solution to the Euler equations \eqref{1eq1}. Can the derivative of the velocity field blow up in finite time $T$? That is, it is possible that
\begin{align*}
	\lim\limits_{t\rightarrow T}\sup_{\bm{x}\in\mathbb{R}^3}|\nabla \bm{u}(t,\bm{x})|=+\infty ?
\end{align*}

Although this question for smooth solutions remains open, a significant breakthrough was recently achieved by Elgindi \cite{E21}, who established finite-time blow-up for a class of low-regularity solutions to the Euler equations. More precisely, Elgindi constructed self-similar solutions with a finite-time singularity, which satisfy certain symmetries at regularity $C^{1,\alpha}$ with $\alpha>0$ small. Subsequently, Elgindi et al. \cite{EGM21} used a truncation argument to show that these self-similar solutions are stable and persist in the energy space $H^1$ until the blow-up time.
 Nevertheless, the details up to the blow-up time are still poorly understood. For example, it is unclear what quantities remain bounded as opposed to those that diverge, what structure such as point vortices forms, and whether the solution can be continued in a weak sense. A rigorous understanding of these issues is essential both mathematically and physically. 

To further characterize the blow-up dynamics of the Euler equations, we introduce the definition of non-implosion solution as follows:

\begin{defi}[Non-implosion solution]
Let $\bm{u}$ be a solution to the Euler equations \eqref{1eq1} and the derivative of the velocity field blow up in finite time $T$. 
The velocity field $\bm{u}$ is called a non-implosion solution of the Euler equations \eqref{1eq1} if it satisfies the bound
\begin{align*}
	\|\bm{u}(t,\bm{x})\|_{L^\infty_TL^\infty_{loc}}<+\infty.
\end{align*}
\end{defi}

Based on the definition of the non-implosion solution, a more general problem arises: the time integrability of the velocity field up to the blow-up time. That is, we consider the following problem:\\

\noindent
\textbf{Question 2:} 
Let $\bm{u}$ be a solution to the Euler equations \eqref{1eq1} and the derivative of the velocity field blow up in finite time $T$. Dose the velocity remain bounded?
That is, it is possible that for $1\leq p\leq +\infty$,
\begin{align*}
	\|\bm{u}(t,\bm{x})\|_{L^p_TL^\infty_{loc}}<+\infty ?
\end{align*}


In an effort to address this question, or more precisely, to establish a mechanism that captures both the non-implosion phenomenon and the dynamical behavior of the corresponding solutions, this paper develops a refined framework.
This framework serves to construct finite-time singularities with a prescribed structure and to characterize the extent of integrability of the resulting solutions up to the blow-up time.
To be more precise, a scaling-related parameter $\beta$ is introduced, and its effect on the regularity and integrability of solutions is systematically analyzed. The value of $\beta$ is shown to play a critical role in governing the asymptotic behavior of the solution. In particular, a suitable choice of $\beta$ produces solutions with the specific desired time integrability properties. These provide a foundation for systematically constructing and analyzing solutions with finite-time singularities within more specific function spaces.
Specifically, we construct a new velocity field whose derivative blows up in finite time $T$. 
Nevertheless, its $L^p_T L^\infty_{\text{loc}}$-norm remains bounded for any $1 \leq p < +\infty$.
For the endpoint case $p = +\infty$, boundedness holds provided that $\beta$ satisfies a certain condition, thereby excluding the possibility of velocity implosion at any finite point $|x| < \infty$ and time $T$.
In this way, we partially answer \textbf{Question 2}.

\subsection{Statement of main theorem}
Firstly, we give the definition of the odd function as follows:
\begin{defi}
	A velocity vector field $\bm{u}=(u_1,u_2,u_3):\mathbb{R}^3\rightarrow\mathbb{R}^3$ is called \textbf{odd} if $u_i$ is odd in $x_i$ and even in the other two variables for any $i=1,2,3$.
\end{defi}
From now on, we only consider the odd solution $\bm{u}$ of the Euler equations \eqref{1eq1}. 
We define the two-dimensional radial variable $r\tri\sqrt{x_1^2+x_2^2}$. Then the three-dimensional radial variable is given by $|\bm{x}|\tri \sqrt{r^2+x_3^2}$, which expands to $\sqrt{x_1^2+x_2^2+x_3^2}$.
The vorticity components in the $r$ and $x_3$ directions vanish identically, while the $\widetilde{\theta}$ component (i.e., the angle between $x_1$ and $x_2$ axes) is denoted by $\omega$.
To simplify the analysis, we impose an odd symmetry on $\omega$ with respect to the $x_3$-axis. 
Specifically, we seek solutions satisfying the following condition:
\begin{align*}
	\omega(t,r,x_3)=-\omega (t,r,-x_3).
\end{align*}
The axi-symmetric 3D Euler equations with vanishing swirl can be expressed as follows:
\begin{align}\label{1eq2}
	\left\{\begin{array}{l}
		\displaystyle \partial_t \omega+u_r\partial_{r}\omega+u_3\partial_{3}\omega=\frac{u_r}{r}\omega,\quad t\in\mathbb{R}^+,~~(r,x_3)\in D_0,\\[1ex]
		\displaystyle-\partial^2_r(\psi)-\partial^2_3(\psi)-\frac{\partial_r \psi}{r}+\frac{\psi}{r^2}=\omega, \\[1ex]
		\displaystyle u_r=\partial_3 \psi,~~u_3=-\frac{\psi}{r}-\partial_r \psi, \\[1ex]
		\displaystyle \omega |_{t=0}=\omega_0, \\[1ex]
		\displaystyle \omega|_{\partial D_0}=0,~~\psi|_{\partial D_0}=0,
	\end{array}\right.
\end{align}
where the domain $D_0\tri[0,\infty)\times[0,\infty)$.
We introduce some parameters, which will be commonly used in this paper, as follows:
\begin{align}\label{def-lambda}
	\alpha>0,~~~~\beta\in(0,1],~~~~\eta\in(0,1),~~~~\lambda=1+\frac{\alpha}{10\beta}>1.
\end{align}
Then, we define the backward self-similar solution with the scaling index $\beta_0$ as follows:
\begin{defi}
	The function $\omega(t,x)$ is the so-called backward self-similar solution, if it takes the form
	\begin{align*}
		\omega(t,x) =\frac{1}{T_0-t} \widetilde{\Omega} \left(\frac{x}{(T_0-t)^{\beta_0}}\right),
	\end{align*}
	for some positive constants $T_0$ and $\beta_0$. We then assert that $\omega(t,x)$ develops a finite-time singularity at $t=T_0$ with the scaling index $\beta_0$.
\end{defi}
Through a combination of self-similar analysis, nonlinear stability and energy estimates, we prove that classical solutions to the 3D Euler equations \eqref{1eq2} exhibit the finite-time vorticity singularity in this paper.
Now, we can state our main result.

\begin{theo}\label{Theo2}
	For any $\beta\in(0,1]$, there exists a positive constant $\alpha(\beta)$, the initial vorticity $\omega_0 \in C^{0,\frac{\alpha}{20\beta}}(D_0)$ with $|\omega_0|\leq c|\bm{x}|^{\frac{\alpha}{2\beta}}(1+|\bm{x}|)^{-\left({\frac{\alpha}{2\beta}+\frac{\alpha}{2}}\right)}$ for some constant $c>0$, and the positive time $T^\ast$, such that the axi-symmetric 3D Euler equations \eqref{1eq2} admit a solution 
	$\omega\in C^{0,\frac{\alpha}{2\beta}}_{t}([0,T^\ast); C^{0,\frac{\alpha}{20\beta}}_{x}( D_0))$.
	The solution develops a finite-time singularity at $t=T^\ast>0$ with scaling index $\frac{\beta}{\alpha}$.
	Moreover, there exists a parameter $\gamma\in\mathbb{R}^+$ that satisfies $|\frac{\beta}{\gamma}-1|\ll 1$ and $T^\ast= \frac{1}{2}+\frac{\beta}{2\gamma}$, such that the following blow-up results hold:
	\begin{align}\label{est-theo-1}
		\lim\limits_{t\rightarrow T^\ast} \int_0^t \|\omega(s)\|_{L^{\infty}(D_0)} ds = +\infty, \quad 
		\lim\limits_{t\rightarrow T^\ast} \int_0^t \Big\|\frac{u_r}{r} (s)\Big\|_{L^{\infty}(D_0)} ds = +\infty.
	\end{align}
    As for the velocity components $(u_r,u_3)$, we have, for any $1\leq p<+\infty$,
	\begin{align}\label{est-theo-2}
		\int_0^{T^\ast} \|(u_r,u_3) (s)\|^p_{L^{\infty}_{loc}(D_0)}ds  < +\infty.
	\end{align}
    If, moreover, $\gamma<\beta$, then $(u_r,u_3)$ remain uniformly bounded in the $L^\infty$-norm up to the singularity time:
    \begin{align}\label{est-theo-2'}
        \sup_{0\leq s\leq T^\ast} \|(u_r,u_3) (s)\|_{L^{\infty}_{loc}(D_0)}
        <+\infty.
    \end{align}
\end{theo}

\begin{rema}\label{Rema1}
	In Theorem \ref{Theo2}, we investigate the time integrability of the solutions at the blow-up time.
    The estimate \eqref{est-theo-2} shows that for any $1\leq p<+\infty$, the $L^p_{T^\ast}L^\infty_{loc}$-norm of the velocity components $(u_r,u_3)$ remains bounded.
    Whether this boundedness extends to the case $p=+\infty$ is determined by the relationship between the parameters $\beta$ and $\gamma$.
    Specifically, the $L^\infty_{T^\ast}L^\infty_{loc}$-norm remains bounded precisely when $\gamma<\beta$.
\end{rema}

It is straightforward to verify that the Euler equations \eqref{1eq1-1}-\eqref{1eq1-2} admit the solutions that develop the finite-time singularity, as an immediate corollary of Theorem \ref{Theo2}. 
\begin{coro}\label{Coro1}
	For any $\beta\in(0,1]$, there exists a positive constant $\alpha_0(\beta)$, a divergence-free and odd initial velocity $\bm{u}_0$, the corresponding vorticity $\bm{\omega}_0 \in C^{0,\alpha_0}(\mathbb{R}^3)$ with a positive constant $\alpha_0$, and the positive time $T^\ast$, such that the 3D Euler equations \eqref{1eq1-1}-\eqref{1eq1-2} admit a solution $\bm{\omega}\in 
	C^{0,\alpha_0}([0,T^\ast)\times \mathbb{R}^3)$.
    Moreover, there exists a parameter $\gamma\in\mathbb{R}^+$ that satisfies $|\frac{\beta}{\gamma}-1|\ll 1$ and $T^\ast= \frac{1}{2}+\frac{\beta}{2\gamma}$, such that the following blow-up result holds:
	$$
	\lim\limits_{t\rightarrow T^\ast} \int_0^t \|\bm{\omega}(s)\|_{L^{\infty}(\mathbb{R}^3)} ds = +\infty.
	$$
    For any $1\leq p<\infty$, the velocity $\bm{u}$ exhibits the following integrability property:
	\begin{align*}
		\int_0^{T^\ast} \|\bm{u} (s)\|^p_{L^{\infty}_{loc}(\mathbb{R}^3)}ds  < +\infty.
	\end{align*}
    If, moreover, $\gamma<\beta$, then $\bm{u}$ remains uniformly bounded in the $L^\infty$-norm up to the singularity time:
    \begin{align}\label{est-theo-3}
		\sup_{0\leq s\leq T^\ast}\|\bm{u} (s)\|_{L^{\infty}_{loc}(\mathbb{R}^3)}  < +\infty.
	\end{align}
\end{coro}

\begin{rema}
	It should be noted that $\alpha(\beta)$ in Theorem \ref{Theo2} and $\alpha_0(\beta)$ in Corollary \ref{Coro1}, while not identical, are both small parameters.
\end{rema}


\begin{rema}
	In Corollary \ref{Coro1}, we establish the space-time pointwise estimate \eqref{est-theo-3} for the velocity up to the blow-up time, thereby ruling out the velocity implosion at any finite point $|x| < \infty$ and time $T^\ast$ when $\gamma<\beta$.
\end{rema}

\subsection{Reformulation}
We now proceed to reformulate equations \eqref{1eq2} in polar coordinates.
For this purpose, we define:
\begin{itemize}
	\item The three-dimensional radial variable with index $\alpha$ as $R\tri(r^2+x_3^2)^\frac{\alpha}{2}$ with $\alpha>0$;
	\item The angle variable as $\theta\tri\arctan(\frac{x_3}{r})$,  representing the angle between $r$ and $x_3$.
\end{itemize}
We also define the differential operator $D_R\tri R\partial_R$. 
We introduce the new unknown functions as:
$$\Omega(t,R,\theta)\tri\omega(t,r,x_3), \quad R^{\frac{2}{\alpha}}\Psi(t,R,\theta)\tri\psi(t,r,x_3),$$ 
then the evolution equation for $\Omega$ can be written as:
\begin{align}\label{1eq3}
	\left\{\begin{array}{l}
		\partial_t \Omega+U(\Psi)\partial_{\theta}\Omega+V(\Psi)\alpha D_{R}\Omega=\mathcal{R}(\Psi)\Omega,\quad t\in\mathbb{R}^+,~~(R,\theta)\in D,\\[1ex]
		L^{\alpha}_{R}(\Psi)+L_{\theta}(\Psi)=\Omega, \\[1ex]
		\Omega|_{\partial D}=0,~~\Psi|_{\partial D}=0, 
	\end{array}\right.
\end{align}
where the domain $D\tri[0,\infty)\times[0,\frac{\pi}{2}]$, $U(\Psi)$, $V(\Psi)$ and $\mathcal{R}(\Psi)$ are defined by
$$U(\Psi)\tri-3\Psi-\alpha D_R\Psi,
~~
V(\Psi)\tri\partial_\theta \Psi-(\tan\theta)\Psi,
~~
\mathcal{R}(\Psi)\tri V(\Psi)-(\tan\theta) U(\Psi).$$
Moreover, the radial operator $L^{\alpha}_{R}$ and the angular operator $L_{\theta}$ are defined by
$$L^{\alpha}_{R}(\Psi)\tri-\alpha^2D^2_R\Psi-5\alpha D_R\Psi,
~~
L_{\theta}(\Psi)\tri-\partial_\theta V(\Psi)-6\Psi.$$
Next, we introduce two functions that depend solely on angular $\theta$ as:
\begin{align}\label{def_K-theta}
	K(\theta)\tri(\sin\theta)(\cos\theta)^2,
\end{align}
and for any $a>0$, 
$$\Gamma_{\theta,a}(\theta)\tri(K(\theta))^{\frac{\alpha}{3a}}.$$
The function spaces $\mathcal{H}^k$ and $\mathcal{E}^2$ will be defined in Section \ref{sec:notation}.
Considering the system \eqref{1eq3}, we have the following theorem.
\begin{theo}\label{Theo1}
	For any $\beta\in(0,1]$, there exists a positive constant $\alpha(\beta)$, such that the system \eqref{1eq3} admits a solution $\Omega$ that develops a finite-time singularity. 
	Moreover, there exists a function $g\in \mathcal{H}^2\cap\mathcal{E}^2$ and a parameter $\gamma\in \mathbb{R}^{+}$ satisfying $|\frac{\beta}{\gamma}-1|\ll 1$, such that
	$$
	\Omega(t,R,\theta) =\frac{1}{t_{\gamma}} F\left(\frac{R}{t_{\gamma}^{\beta}},\theta\right),~~~F(z,\theta)=F^{\ast}_{\gamma}(z,\theta)+g(z,\theta),~~~t_\gamma=1-\frac{2\gamma}{\gamma+\beta}t, ~~~z=\frac{R}{t_{\gamma}^{\beta}},
	$$
	where $F^{\ast}_{\gamma}$ is defined as
	$$F^{\ast}_{\gamma}(z,\theta)\tri \frac {4\alpha}{3\gamma}\frac{\Gamma_{\theta,\gamma}(\theta)}{\int_0^{\frac{\pi}{2}}K(\theta)\Gamma_{\theta,\gamma}(\theta)d\theta}\frac{z^{\frac{1}{\gamma}}}{(1+z^{\frac{1}{\gamma}})^2}.$$
\end{theo}

\begin{rema}
	Let $\gamma=1$. Elgindi proved that there exists a positive constant $\alpha$, such that the system \eqref{1eq3} admits solutions $\Omega$ that develop singularity in finite time, see \cite{E21}. Specifically, there exists a function $g\in \mathcal{H}^4$ with the different parameter $\lambda=1+\frac{\alpha}{10}$, and a parameter $\beta\in \mathbb{R}^{+}$ with $|\beta-1|\ll 1$, such that
	$$
	\Omega(t,R,\theta) =\frac{1}{t_1} F\left(\frac{R}{t_1^{\beta}},\theta\right),~~~F(z,\theta)=F_1^{\ast}(z,\theta)+g(z,\theta),~~~t_1=1-\frac{2}{\beta+1}t, ~~~ z=\frac{R}{t_1^{\beta}}.$$
	He established that $F^{\ast}_{1}$,  which is smooth in $z$, lies in $\mathcal{W}^{l,\infty}$ for any $l\in\mathbb{N}^+$.
\end{rema}

\begin{rema}
	The fundamental self-similar solutions $F^{\ast}_{\gamma}$,  constructed in Theorem \ref{Theo1}, is well-defined for all $\gamma>0$.
	It should be noted when $\frac{1}{\gamma}\not\in\mathbb{N}^+$, $F^{\ast}_{\gamma}$ fails to be smooth as $z$ goes to $0$.
	We further prove that $F^{\ast}_{\gamma} \in \mathcal{H}^{k,\ast}$, where the space $\mathcal{H}^{k,\ast}$  is introduced in Section \ref{sec:notation}.
	Additionally, the parameter stability of $F^{\ast}_{\gamma}$ is established in Section \ref{sec:toy model}.
\end{rema}

\begin{rema}
	Here, we highlight the critical issue mentioned in Section $8$ of \cite{E21}: 
	``In $\mathcal{H}^2$ the problem becomes critical in a certain sense, and it is not clear whether the non-linear estimates can be closed in an easy way."
	Through a rigorous analysis in Section \ref{sec:Est-T}, we introduce a new second-order space $\mathcal{E}^2$ equipped with a stronger singular weight and establish \textit{a priori} energy estimates in $\mathcal{H}^2\cap\mathcal{E}^2$ by exploiting the null structure of the transport term.
	Finally, we establish the existence of the self-similar solution $g(z,\theta)$, see Theorem \ref{Theo1} for more details.
	Moreover, we point out that $\mathcal{H}^k \hookrightarrow \mathcal{H}^2\cap\mathcal{E}^2$ for any $k\geq 3$.
\end{rema}

\subsection{Motivation and main ideas}

The absence of intrinsic scales in the Euler equations permits the rapid growth of solutions at arbitrarily small spatial scales in a short time, potentially leading to a finite-time singularity. This process is inherently self-similar, thereby rendering a self-similar framework the natural setting for its investigation. 
For self-similar solutions, the non-implosion is tied to a specific scaling exponent, and understanding this mechanism is a fundamental problem in fluid mechanics.
This paper aims to establish a new framework to elucidate the non-implosion mechanism in the self-similar setting. We shall now outline our main ideas in four steps.

\medskip 
\noindent\textbf{Step 1: The backward self-similar system of the Euler equations.}
\medskip 

For any given $\beta\in(0,1]$ and $\mu\in\mathbb{R}$ that would be determined later, we consider the following self-similar variables:
\begin{align}\label{4eq1}
	t_{\mu}\tri 1-(1+\mu)t,~~~z\tri\frac{R}{t_{\mu}^{\beta}}.
\end{align} 
Assuming that solution $(\Omega,\Psi)$ satisfies the self-similar structure as:
\begin{align}\label{4eq2}
	\Omega(t,R,\theta)=\frac{1}{t_{\mu}}F\left(z,\theta\right),~~~
	\Psi(t,R,\theta)=\frac{1}{t_{\mu}}\Phi\left(z,\theta\right),
\end{align}
then system \eqref{1eq3} can be rewritten as the following self-similar form:
\begin{align}\label{1.4eq3}
	\left\{\begin{array}{l} (1+\mu)F+(1+\mu)\beta D_{z}F+T=\mathcal{R}(\Phi)F\quad \text{(The main equation)},\\[1ex]
		L^{\alpha}_{z}(\Phi)+L_{\theta}(\Phi)=F\quad \text{(The potential equation)}, \\[1ex]
		F|_{\partial D}=0,~~\Phi|_{\partial D}=0, 
	\end{array}\right.
\end{align}
where $(z,\theta)\in D$ and $D\tri[0,\infty)\times[0,\frac{\pi}{2}]$,
\begin{align}
	T&\tri U(\Phi)\partial_{\theta}F+V(\Phi)\alpha D_{z}F
	\quad\text{(The transport term)},
	\label{def-T}
	\\
	\mathcal{R}(\Phi)&\tri V(\Phi)-\tan\theta U(\Phi)
	\quad\text{(The stretching term)},
	\label{def-R}
\end{align}
and the differential operators $L^{\alpha}_{z}$ and  $L_{\theta}$ are defined by
\begin{align}\label{def-L}
	L^{\alpha}_{z}(\Phi)=-\alpha^2 D^2_z\Phi-5\alpha D_z\Phi,\quad L_{\theta}(\Phi)=-\partial_\theta V(\Phi)-6\Phi.
\end{align}
To be more specific, $L^{\alpha}_{z}$ represents the radial operator, while $L_{\theta}$ corresponds to the angular operator.
In this new frame of reference, $U(\Phi)$ and $V(\Phi)$ are reduced to
\begin{align}\label{def_U-V}
	U(\Phi)=-3\Phi-\alpha D_z\Phi,~~V(\Phi)=\partial_\theta \Phi-(\tan\theta) \Phi.
\end{align}

\medskip  

\noindent\textbf{Step 2: The fundamental self-similar solution and its parameter stability.}

\medskip

 This step is to construct a family of solutions to the fundamental self-similar model parameterized by $\gamma$, and establish its parameter stability.
Specifically, we begin by analyzing the following fundamental self-similar model, which retains only the leading parts from system \eqref{1.4eq3}:
\begin{align}\label{1.4eq4}
	\left\{\begin{array}{l}
		F^{\ast}_\gamma+\gamma D_zF^{\ast}_\gamma=\frac{3}{2\alpha}L^{-1}_{z,K}(F^{\ast}_\gamma)F^{\ast}_\gamma,
		\quad t\in\mathbb{R}^+,~~(z,\theta)\in D,
		\\[1ex]
		F^{\ast}_\gamma|_{\partial D}=0.
	\end{array}\right.
\end{align}
For the smooth (in $z$) solution $F^{\ast}_1$, Elgindi in \cite{E21} established the energy estimates of $F^{\ast}_1$ in $\mathcal{W}^{l,\infty}$ for any $l\in\mathbb{N}^{+}$.
Moreover, his result can be extended to any $\frac{1}{\gamma}\in \mathbb{N}^{+}$.
It is straightforward to verify that $F^{\ast}_\gamma$ fails to be smooth as $z$ goes to $0$ for $\frac{1}{\gamma}\notin \mathbb{N}^{+}$.
In this paper, we construct the Sobolev space $\mathcal{H}^{k,\ast}$ (defined in Section \ref{sec:notation}), and show that $F^{\ast}_\gamma\in \mathcal{H}^{k,\ast}$ for any $\gamma>0$.
Moreover, we establish the parameter stability of $F^{\ast}_\gamma$, a crucial step in determining a suitable $\gamma(\beta)$ for any given $\beta>0$.
A more detailed analysis is presented in Section \ref{sec:toy model}.
\medskip  

\noindent\textbf{Step 3: The energy estimate for the perturbation at critical regularity.}

\medskip

Based on the fundamental self-similar solution, we now decompose $F$ as $F=F^{\ast}_{\gamma}+g$, then $g$ satisfies the following self-similar system:
\begin{align}\label{1.4eq5}
	\left\{\begin{array}{l} \mathcal{L}_{\Gamma}(g)=-T+R_0+R_1+R_2,\\[1ex]
		L^{\alpha}_{z}(\Phi)+L_{\theta}(\Phi)=F, \\[1ex]
		\mu= \frac{3}{4\alpha}L^{-1}_{z,K}(R_2-T)(0),~~\gamma = \frac{1+\mu}{1-\mu}\beta, \\[1ex]
		L^{-1}_{z,K}(g)(0)=0,~~g|_{\partial D}=0,~~\Phi|_{\partial D}=0, 
	\end{array}\right.
\end{align}
where the linear operator $\mathcal{L}_{\Gamma}(g)$ is defined by
\begin{align}\label{def-LGammag}
	\mathcal{L}_{\Gamma}(g)\tri g+\beta D_{z}g-\frac{3}{2\alpha}L^{-1}_{z,K}(F^{\ast}_{\gamma})g-\frac{3}{2\alpha}L^{-1}_{z,K}(g)F^{\ast}_{\gamma},
\end{align}
and the remaining terms are defined by
\begin{align}
	&R_0\tri 
	-\mu F^{\ast}_{\gamma}+(\gamma-(1+\mu)\beta) D_zF^{\ast}_{\gamma}, 
	\quad
	R_1\tri -\mu g-\mu\beta D_zg,
	\label{def-R1}\\
	&R_2\tri \frac{3}{2\alpha}L^{-1}_{z,K}(g)g-\frac{3
	}{2}(\sin\theta)^2\langle F,K\rangle_{\theta}F+ \mathcal{R}(\Phi-\Phi_{\text{main}})F.
	\label{def-R2-1}
\end{align}
Here, $R_0$ and $R_1$ contain the linear terms of $F^{\ast}_{\gamma}$ and $g$, respectively, whereas $R_2$ contains the nonlinear terms of $F^{\ast}_{\gamma}$ and $g$.
Since $L^{-1}_{z,K}$ is a non-local operator, it is not feasible to directly impose the condition $L^{-1}_{z,K}(g)(0)=0$, see $\eqref{1.4eq5}_4$ (the definition of $L^{-1}_{z,K}$ can be found in Section \ref{sec:notation} below). Thus, we will ensure this boundary condition through selecting $\mu$, as in $\eqref{1.4eq5}_3$.
Another central step of this work is to establish the energy estimates at critical regularity, namely, in $\mathcal{H}^2\cap\mathcal{E}^2$ space.


In order to obtain the coercivity estimate for the linear operator $\mathcal{L}_{\Gamma}$, we select the suitable weights and an appropriate parameter $\eta$, then construct the Sobolev spaces $\mathcal{H}^{-1}$ and $\mathcal{H}^2_{\eta}$ with $0<\beta\leq 1$, to establish coercivity estimate.
More precisely, we begin by decomposing the operator $\mathcal{L}_{\Gamma}$ as:
\begin{align*}
	\mathcal{L}_{\Gamma}=\mathcal{L}^{\beta}_{\Gamma}(f)+\tilde{P}(f),
\end{align*}
where 
\begin{align*}
	\mathcal{L}^{\beta}_{\Gamma}(f)\tri& \mathcal{L}^{\beta}(f)-\frac{3}{2\alpha}L^{-1}_{z,K}(f)F^{\ast}_{\beta}\quad\text{(the main operator)},
	\\
	\mathcal{L}^{\beta}(f) \tri& f + \beta D_z f - L_z^{-1}(\Gamma^{\ast}_{\beta})f \quad \text{( the core operator)},
	\\
	\tilde{P}(f) \tri& L^{-1}_z(\Gamma^{\ast}_{\beta}-\Gamma^{\ast}_{\gamma})f+\frac{3}{2\alpha}L^{-1}_{z,K}(f)(F^{\ast}_{\beta}-F^{\ast}_{\gamma}) \quad\text{(the perturbation term) }.
\end{align*}
We then select an appropriate radial weight $w_z$, such that the core operator $\mathcal{L}^{\beta}$ satisfies the following $L^2$-energy estimate:
\begin{align*}
	\langle\mathcal{L}^{\beta}(g),g(w_z)^2\rangle_z=\left(1-\frac{\beta}{2}\right)\|gw_z\|^2_{L^2_z}.
\end{align*}
Based on the equality above, we choose a suitable angle weight $w^{K}_{\theta}$, such that the main operator $\mathcal{L}^{\beta}_{\Gamma}$ satisfies
\begin{align*}
	\langle\mathcal{L}^{\beta}_{\Gamma}(g), g(w^K)^2\rangle=\|g\|^2_{\mathcal{H}^{-1}},
\end{align*}
where $w^K\tri w_z w^{K}_{\theta}$ and the definition of $\mathcal{H}^{-1}$ is defined in Section \ref{sec:notation}. 
The construction of this new $\mathcal{H}^{-1}$ space is one of the innovations in our paper.
Moreover, we apply the parameter stability of $F^{\ast}_\gamma$ with respect to $\gamma$, to derive the parameter stability of the main operator. 
This ensures that $\tilde{P}(g)$ is indeed the perturbation term. 
More details are shown in Section \ref{sec:toy model} and Section \ref{sec:Coer}.

By standard elliptic estimates, we can obtain that the $\mathcal{H}^{n}$-norm of $(\alpha^2 D_z^2\Phi_g,\partial_\theta^2\Phi_g)$ can be controlled by that of $g$, respectively.
Correspondingly, we also establish that the $\mathcal{H}^{n,\ast}$-norm of $(\alpha^2 D_z^2\Phi_{F^\ast_\gamma},\partial_\theta^2\Phi_{F^\ast_\gamma})$ can be bounded by that of $F^\ast_\gamma$.
Moreover, when establishing the elliptic estimates for $(\alpha^2 D_z^2\Phi_g,\partial_\theta^2\Phi_g)$ in $\mathcal{E}^n$, we need to introduce a new angle weight with the parameter $\lambda>1$. Thus, we derive the Hardy inequality with parameter $\xi>1$. 
For more proof details, the reader can refer to Section \ref{sec:Elli}.

For the transport term $T$, when establishing the energy estimates in $\mathcal{H}^{-1}$, we define the quantity $(1-\eta)^{-1}\|D_{\theta}g w^{\lambda}\|_{L^2}$ in $\mathcal{H}^2_{\eta}$-norm to deal with its linear part.
By selecting $1-\eta$ sufficiently small, we ensure that the $\mathcal{H}^{-1}$ energy estimates remain controlled by the $\mathcal{H}^2_{\eta}$-norm.
When developing energy estimate in $\mathcal{H}^2_{\eta}$, we exploit the null structure present in $T$.
More precisely, we perform the decomposition of $T_g$ as:
\begin{align*}
	T_g =\frac{1}{\mathcal{W}}T_{g\mathcal{W}}-\frac{1}{\mathcal{W}}T_{\mathcal{W}}\cdot g,
\end{align*}
where the notations $\frac{1}{\mathcal{W}}T_{g\mathcal{W}}$ and $\frac{1}{\mathcal{W}}T_{\mathcal{W}}\cdot g$ are defined by
\begin{align*}
	\frac{1}{\mathcal{W}}T_{g\mathcal{W}} 
	\tri \frac{1}{\mathcal{W}}\left(\frac{U(\Phi)}{\sin(2\theta)} D_{\theta}(g\mathcal{W})+V(\Phi)\alpha D_{z}(g\mathcal{W})\right),
	\quad
	\frac{1}{\mathcal{W}}T_{\mathcal{W}}\cdot g 
	\tri \left(\frac{U(\Phi)}{\sin(2\theta)} \frac{D_{\theta}\mathcal{W}}{\mathcal{W}}+\alpha V(\Phi) \frac{D_{z}\mathcal{W}}{\mathcal{W}}\right)g.
\end{align*}
After some direct calculation, we get that
\begin{align}\label{null-1}
	\big\langle T_{g\mathcal{W}},\big(g\mathcal{W}\big)z^{\frac{3}{\alpha}-1}\cos\theta \big\rangle =0.
\end{align}
By selecting a suitable $\mathcal{W}$, the main part of $\frac{1}{\mathcal{W}}T_{\mathcal{W}}$ can be rewritten as: 
\begin{align}\label{null-2}
	\frac{D_{\theta}\mathcal{W}}{\mathcal{W}}+\cos(2\theta)-(\sin\theta)^2
	= (1-2\xi)\cos(2\theta).
\end{align}
This indicates that the homogeneous part or even entire transport term $T$ certainly is the perturbation of $\mathcal{L}^{\beta}_\Gamma$ whenever $\xi$ is approximately equal to $\frac{1}{2}$. \textbf{Accordingly, we refer to the structure \eqref{null-1} and the identity \eqref{null-2} with $\xi=\frac{1}{2}$ as the null structure of transport term $T$.}
This observation is key to establishing the energy estimates at critical regularity.
Thus, our choice of angle weights with the parameters $\eta$ or $\lambda$ allows us to treat additional principal terms as small perturbations relative to the linear part.
Consequently, in $\mathcal{H}^2_\eta$, we can obtain that
\begin{align*}
	T_g=\text{Null~perturbation} + \text{Nonlinear~perturbation}.
\end{align*}
Additionally, when handling $D_z D_\theta T_g$, there is a core difficulty to control the $L^\infty$-norm of $D_z \Big(\frac{U(\tilde{\Phi}_{g})}{\sin(2\theta)}\Big)$, we have to introduce the new norm of $D^2_z g$ with an additional weight, which avoids increasing regularity. That is,
\begin{align*}
	\bigg\|D_z \bigg(\frac{U(\tilde{\Phi}_{g})}{\sin(2\theta)}\bigg)\bigg\|_{L^{\infty}}
	\lesssim_\beta \alpha^{-1}\|g\|_{\mathcal{E}^{2}}.
\end{align*}
Thus, we properly close the nonlinear energy estimates in $\mathcal{H}^{-1}\cap\mathcal{H}^2\cap\mathcal{E}^2$.
We point out that the discovery of the null structure and the introduction of the new space $\mathcal{E}^2$ are key innovations in establishing the energy estimates at critical regularity in our paper.
Further discussion on the transport term $T$ can be found in Section \ref{sec:Est-T}.

We observe that $R_0$ and $R_1$ contain the linear terms, while $R_2$ contains the nonlinear terms.
Through appropriate selection of small $\mu$ and setting $\gamma=\frac{1+\mu}{1-\mu}\beta$, we demonstrate that both $R_0$ and $R_1$ can indeed be regarded as small perturbations.
A complete proof is provided in Section \ref{sec:blow-up}.

\medskip
\noindent\textbf{Step 4: The non-implosion mechanism of the Euler equations.}
\medskip  

From the steps above, we demonstrate that the $L^1([0,t];L^\infty(D_0))$-norm of both $\omega$ and $\frac{u_r}{r}$ exhibits finite-time blow-up as $t$ goes to a finite time $T^\ast$. 
In addition, the time integrability of the velocity is rigorously examined. 
Regarding the choice of weights $w^K$, $w^\eta$, $w^\lambda$ employed above, we only prove that \eqref{est-theo-1} and \eqref{est-theo-2} hold for all $p\in[1,\frac{2}{2-\beta})$.
In the course of the preceding analysis, it is observed that the time integrability of velocity is closely related to the decay at infinity of both the fundamental solution $F^\ast_\gamma$ and the perturbation $g$. 
This observation necessitates the use of new energy estimates with strengthened weights. 
More precisely, the admissible range of the weight parameter, raised to a certain power, falls within an interval related to $\beta$ and $\gamma$.
Accordingly, we can prove \eqref{est-theo-2} with the range $\frac{2}{2-\beta}\leq p<\infty$ and \eqref{est-theo-2'} when $\gamma<\beta$.
The proof, with full details, is presented in Section \ref{sec:non-implosion}.


\medskip
\noindent\textbf{Organization}. 
In Section \ref{sec:notation}, we set up some basic notations, which will be used extensively throughout this paper.
Section \ref{sec:toy model} presents our analysis of the fundamental model, establishing both parameter stability and energy estimates for the fundamental self-similar solutions. 
Section \ref{sec:analysis} is devoted to asymptotic analysis of the self-similar system. 
In Section \ref{sec:Coer}, we establish the coercivity estimate and the parameter stability of the main operator. 
We derive the elliptic estimates of the solutions in Section \ref{sec:Elli}, while we treat the transport term in Section \ref{sec:Est-T}.
Building upon these foundations, we complete the proof of Theorem \ref{Theo1} and show the blow-up phenomenon of the Euler equations in Section \ref{sec:blow-up}. 
In section \ref{sec:non-implosion}, we further study the time integrability of the solutions and complete the proof of Theorem \ref{Theo2}.

\section{Notations}\label{sec:notation}
In this section, we present our notations that will be frequently used in the sequel.
\begin{itemize}
	\item[(1)] We denote the usual Lebesgue $L^p$-norm of $u$ by
	\begin{align*}
		\|u\|_{L^p(\Omega)}\tri\begin{cases}
			& \left(\int_{\Omega} |u|^p\ dx\right)^{\frac{1}{p}},\quad  1\le p<\infty;\\
			& \underset{x\in\Omega}{\operatorname{esssup}}|u(x)|,\quad p=\infty.
		\end{cases}
	\end{align*}
	In addition, we use the $L^p_\xi$ norm to denote the $L^p$ norm with the independent variable $\xi=z~\text{or}~\theta$.
	\item[(2)] We use $\langle \cdot, \cdot \rangle$ to denote the $L^2$ inner product in $\Omega$, $\langle f, g\rangle \tri \int_{\Omega} f(x)g(x)dx$. Moreover, we use $\langle \cdot, \cdot \rangle_\xi$ to denote the $L^2$ inner product with the independent variable $\xi=z~\text{or}~\theta$.
	\item[(3)] We define the radial weights as:
	$$w_z(z)\tri\frac{(1+z^{\frac{1}{\beta}})^2}{z^{\frac{2}{\beta}}},~~~~w^{\ast}_z(z)\tri\frac{1+z^{\frac{1}{\beta}}}{z^{\frac{1}{\beta}+\frac{1}{4}}},$$
	and the angular weights as:
	$$w^K_{\theta}(\theta)\tri (K(\theta))^{\frac{1}{2}(1-\frac{\alpha}{3\beta})},~~~~w^{\eta}_{\theta}(\theta)\tri(\sin(2\theta))^{-\frac{\eta}{2}},~~~~w^{\lambda}_{\theta}(\theta)\tri(\sin(2\theta))^{-\frac{\lambda}{2}}.$$
	Then we take 
	$$w^K(z,\theta)\tri w^K_{\theta}\cdot w_z,~~~~w^{\eta}(z,\theta)\tri w^{\eta}_{\theta}\cdot w_z,~~~~w^{\lambda}(z,\theta)\tri w^{\lambda}_{\theta}\cdot w_z,$$
	and 
	$$w^{\ast,\eta}(z,\theta)\tri w^{\eta}_{\theta}\cdot w^{\ast}_z,~~~~w^{\ast,\lambda}(z,\theta)\tri w^{\lambda}_{\theta}\cdot w^{\ast}_z.$$
	\item[(4)]We introduce the differential operators  $$D_z\tri z\partial_z,~~~~D_\theta\tri\sin(2\theta)\partial_\theta.$$ 
	Define the $\mathcal{H}^{-1}([0,\infty)\times[0,\frac{\pi}{2}])$ norm as:
	\begin{align*}
		\|f\|^2_{\mathcal{H}^{-1}}
        \tri\Big(1-\frac{\beta}{2}\Big)\|fw^K\|^2_{L^2}+\frac{1-\beta}{2c_{\ast,\beta}\beta}\|L^{-1}_{z,K}(f)w_z(\Gamma^{\ast}_{\beta})^\frac{1}{2}\|^2_{L^2_z}
		+\frac{1}{2c_{\ast,\beta}\beta}\|L^{-1}_{z,K}(f)w_z(\Gamma^{\ast}_{\beta}L^{-1}_{z}(\Gamma^{\ast}_{\beta}))^\frac{1}{2}\|^2_{L^2_z}. 
	\end{align*}
	We also define the $\mathcal{H}^k([0,\infty)\times[0,\frac{\pi}{2}])$ norm as:
	\begin{align*}
		\|f\|^2_{\mathcal{H}^k}\tri \sum_{i=0}^{k}\|D^i_z fw^{\eta}\|^2_{L^2}+\sum_{0\leq i+j\leq k,j\geq 1}\|D^i_zD^j_{\theta} fw^{\lambda}\|^2_{L^2},
	\end{align*}
	and the $\mathcal{E}^k([0,\infty)\times[0,\frac{\pi}{2}])$ norm as:
	\begin{align*}
		\|f\|^2_{\mathcal{E}^k}\tri\alpha\left(\|fw^{\lambda}\|^2_{L^2}+\|D^k_z fw^{\lambda}\|^2_{L^2}\right).
	\end{align*}
	Moreover, we take the $\mathcal{H}^{k,\ast}([0,\infty)\times[0,\frac{\pi}{2}])$ norm as:
	\begin{align*}
		\|f\|^2_{\mathcal{H}^{k,\ast}}\tri\sum_{i=0}^{k}\|D^i_z fw^{\ast,\eta}\|^2_{L^2}+\sum_{0\leq i+j\leq k,j\geq 1}\|D^i_zD^j_{\theta} fw^{\ast,\lambda}\|^2_{L^2}.
	\end{align*}
	\item[(5)]Define the integral operators 
	$$L^{-1}_{z}(f)(z)\tri\int_{z}^{\infty}\frac{f(\rho)}{\rho} d\rho,$$
	and
	$$L^{-1}_{z,K}(f)(z)\tri\int_{z}^{\infty}\int_{0}^{\frac{\pi}{2}}f(\rho,\theta)\frac{K(\theta)}{\rho}d\theta d\rho.$$
	Note that $D_z(L^{-1}_{z}(f))=-f(z)$ and $D_z(L^{-1}_{z,K}(f))=-\langle f,K\rangle_{\theta}$.
	\item[(6)] 
	We use the notation $a\lesssim b$ (resp. $a\gtrsim b$) to indicate that there exists a uniform constant $C$, which may be different in each occurrence, such that $a\leq Cb$ (resp. $a\geq Cb$). Moreover, the symbol $a\approx b$ represents $a\lesssim b\lesssim a$.
	We also use the notation $a\leq C_{d}b$ to emphasize that the constant $C_{d}$ depends on some parameter $d$.
	Sometimes we will write $a\lesssim_{d}b$ if we want to emphasize that the implicit constant depends on some parameter $d$.
\end{itemize}

\section{The fundamental model}\label{sec:toy model}
Before investigating the blow-up phenomenon of the Euler equations \eqref{1eq2}, we first consider the fundamental model in this section.
More precisely, we will study the blow-up phenomenon of the fundamental model and the parameter stability of the fundamental self-similar model.
The stability of the parameters plays the crucial role in fixing the scaling index of the solutions of the Euler equations \eqref{1eq2}.
\subsection{Blow-up phenomenon of the fundamental model}
Firstly, we study the following fundamental model as:
\begin{align}\label{3eq1}
	\left\{\begin{array}{l}
		\partial_t \Omega^{\ast}=\frac{3}{2\alpha}L^{-1}_{R,K}(\Omega^{\ast})\Omega^{\ast},\quad t\in\mathbb{R}^+,~~(R,\theta)\in D,\\[1ex]
		\Omega^{\ast}|_{\partial D}=0. 
	\end{array}\right.
\end{align}
where $D\tri[0,\infty)\times[0,\frac{\pi}{2}]$ and $L^{-1}_{R,K}(\Omega^{\ast})(R)\tri\int_{R}^{\infty}\int_{0}^{\frac{\pi}{2}}\Omega^{\ast}(\rho,\theta)\frac{K(\theta)}{\rho}d\theta d\rho$.
\begin{prop}\label{3FM1}
	Assume that $\Gamma(\theta)$ satisfies $\Gamma(0)=\Gamma(\frac{\pi}{2})=0$, $K(\theta)\Gamma(\theta)\in L^1([0,\frac{\pi}{2}])$ and
	\begin{align*}
		c_{\ast}\tri\int_{0}^{\frac{\pi}{2}}K(\theta)\Gamma(\theta)d\theta\neq 0.
	\end{align*}
	Let $\alpha>0$, $\gamma>0$, then the fundamental model \eqref{3eq1} admits a family of self-similar solutions of the form 
	\begin{align}\label{form-omega}
		\Omega^{\ast}(R,\theta,t)=\frac {2\alpha}{3(1-t)}\frac{\Gamma(\theta)}{c_\ast}\Gamma^{\ast}_{\gamma}\left(\frac{R}{(1-t)^{\gamma}}\right),
	\end{align}
	where $\Gamma^{\ast}_{\gamma}(z)\tri\frac{2z^{\frac{1}{\gamma}}}{\gamma(1+z^{\frac{1}{\gamma}})^2}$. 
\end{prop}
\begin{proof}
	Taking $\overline{\Omega}^\ast\tri\frac{3}{2\alpha}\Omega^\ast$, we deduce from \eqref{3eq1} that
	\begin{align}\label{3f1}
		\partial_t \overline{\Omega}^{\ast}=L^{-1}_{R,K}(\overline{\Omega}^{\ast})\overline{\Omega}^{\ast}.
	\end{align}
	By virtue of the separated variable method, we consider $\overline{\Omega}^{\ast}$ as the following form:
	\begin{align}\label{form-bar-omega}
		\overline{\Omega}^{\ast}(t,R,\theta)=\frac {1}{1-t}\frac{\Gamma(\theta)}{c_\ast}\Gamma^{\ast}_{\gamma}(z),
	\end{align}
	where $z\tri\frac{R}{(1-t)^{\gamma}}$.
	For any $\Gamma(\theta)$, which satisfies $\Gamma(0)=\Gamma(\frac{\pi}{2})=0$ and $K(\theta)\Gamma(\theta)\in L^1([0,\frac{\pi}{2}])$, we derive from \eqref{3f1} that $\Gamma^{\ast}_{\gamma}(z)$ satisfies the following equation:
	\begin{align}\label{3f2}
		\Gamma^{\ast}_{\gamma}+\gamma D_z\Gamma^{\ast}_{\gamma}=L^{-1}_{z}(\Gamma^{\ast}_{\gamma})\Gamma^{\ast}_{\gamma},
	\end{align}
	where $L^{-1}_z(\Gamma^{\ast}_{\gamma})(z)$ is defined by
	\begin{align}\label{3f2'}
		L^{-1}_z(\Gamma^{\ast}_{\gamma})(z)\tri\int_z^{\infty}\frac{\Gamma^{\ast}_{\gamma}(\rho)}{\rho}d\rho.
	\end{align}
	We will proceed to solve the ordinary differential equation \eqref{3f2} with \eqref{3f2'}.
	
	By virtue of \eqref{3f2'}, it is easy to check that
	\begin{align}\label{est-DL1}
		D_zL^{-1}_{z}(\Gamma^{\ast}_{\gamma})(z)=-\Gamma^{\ast}_{\gamma}(z).
	\end{align}
	The combination of \eqref{3f2} and \eqref{est-DL1} leads us to get that
	\begin{align}\label{3f3}
		D_zL^{-1}_{z}(\Gamma^{\ast}_{\gamma})-\gamma D_z\Gamma^{\ast}_{\gamma}=-L^{-1}_{z}(\Gamma^{\ast}_{\gamma})\Gamma^{\ast}_{\gamma}.
	\end{align}
	Without loss of generality, we consider $\lim\limits_{z\rightarrow +\infty}\Gamma^{\ast}_{\gamma}(z)=0$.
	It then follows from \eqref{3f3} that
	\begin{align}\label{3f4}
		L^{-1}_{z}(\Gamma^{\ast}_{\gamma})-\gamma \Gamma^{\ast}_{\gamma}
		&=\int_z^{\infty}L^{-1}_{\rho}(\Gamma^{\ast}_{\gamma})\frac{\Gamma^{\ast}_{\gamma}}{\rho}d\rho =\int_z^{\infty}L^{-1}_{\rho}(\Gamma^{\ast}_{\gamma})(-\partial_{\rho}L^{-1}_{\rho}(\Gamma^{\ast}_{\gamma}))d\rho =\frac{1}{2}(L^{-1}_{z}(\Gamma^{\ast}_{\gamma}))^2. 
	\end{align}
	Again thanks to \eqref{est-DL1}, one can deduce from \eqref{3f4} that
	\begin{align}\label{3f4'}
		\gamma z\partial_z L^{-1}_{z}(\Gamma^{\ast}_{\gamma})=\left(\frac12L^{-1}_{z}(\Gamma^{\ast}_{\gamma})-1\right)L^{-1}_{z}(\Gamma^{\ast}_{\gamma}).
	\end{align}
	
	One can see that \eqref{3f4'} is an ordinary differential equation of $L^{-1}_{z}(\Gamma^{\ast}_{\gamma})$.
	It is easy to check that $L^{-1}_{z}(\Gamma^{\ast}_{\gamma})\equiv0$ and $L^{-1}_{z}(\Gamma^{\ast}_{\gamma})\equiv2$ are two trivial solutions of \eqref{3f4'}.
	Then $\Gamma^{\ast}_{\gamma}(z)=-D_zL^{-1}_{z}(\Gamma^{\ast}_{\gamma})\equiv0$ for all $z\in(0,+\infty)$.
	We find that $L^{-1}_{z}(\Gamma^{\ast}_{\gamma})\equiv2$ and $\Gamma^{\ast}_{\gamma}(z)\equiv0$ do not satisfy \eqref{3f2'}.
	Thus, only $L^{-1}_{z}(\Gamma^{\ast}_{\gamma})=\Gamma^{\ast}_{\gamma}(z)\equiv0$ is the trivial solution for the ordinary equation \eqref{3f2} with \eqref{3f2'}.
	But we do not consider the trivial solution here.
	
	Now, we aim to find the non-trivial solution for the ordinary equation \eqref{3f2} with \eqref{3f2'}.
	After some direct calculations, \eqref{3f4'} can be rewritten as follows:
	\begin{align}\label{ode-f}
		\left(\frac{1}{L^{-1}_{z}(\Gamma^{\ast}_{\gamma})-2}-\frac{1}{L^{-1}_{z}(\Gamma^{\ast}_{\gamma})}\right)\partial_zL^{-1}_{z}(\Gamma^{\ast}_{\gamma})=\frac{1}{\gamma z}.
	\end{align}	
	After detailed calculations, we arrive at
	\begin{align}
		&L^{-1}_{z}(\Gamma^{\ast}_{\gamma})(z)=\frac{2}{1+z^{\frac{1}{\gamma}}},
		\quad
		\Gamma^{\ast}_{\gamma}(z)=-D_z L^{-1}_{z}(\Gamma^{\ast}_{\gamma})=\frac{2 z^{\frac{1}{\gamma}}}{\gamma(1+z^{\frac{1}{\gamma}})^2}.
		\label{3f7}
	\end{align}
	Substituting \eqref{3f7} into \eqref{form-bar-omega}, we can obtain \eqref{form-omega} directly.
	We thus complete the proof of this proposition.
\end{proof}

\subsection{Parameter stability and energy estimate of the fundamental self-similar solutions}
In this subsection, we consider the fundamental self-similar model as follows:
\begin{align}\label{3eq2}
	\left\{\begin{array}{l}
		F^{\ast}_\gamma+\gamma D_zF^{\ast}_\gamma=\frac{3}{2\alpha}L^{-1}_{z,K}(F^{\ast}_\gamma)F^{\ast}_\gamma,
		\quad t\in\mathbb{R}^+,~~(z,\theta)\in D,
		\\[1ex]
		F^{\ast}_\gamma|_{\partial D}=0,
	\end{array}\right.
\end{align}
where $D\tri[0,\infty)\times[0,\frac{\pi}{2}]$ and $L^{-1}_{z,K}(F^{\ast}_\gamma)(z)\tri \int_{z}^{\infty}\int_{0}^{\frac{\pi}{2}}F^{\ast}_\gamma(\rho,\theta)\frac{K(\theta)}{\rho}d\theta d\rho$. 

Firstly, we study the parameter stability of the fundamental self-similar solutions for equation \eqref{3eq2}.
\begin{prop}\label{3FM2}
	Let $0<\alpha \leq \gamma$, then the fundamental self-similar model \eqref{3eq2} admits the solutions of the form 
	\begin{align}
		F^{\ast}_{\gamma}(z,\theta)=\frac {2\alpha}{3}\frac{\Gamma_{\theta,\gamma}(\theta)}{c_{\ast,\gamma}}\Gamma^{\ast}_{\gamma}(z),\label{form_F1}
	\end{align}
	where $c_{\ast,\gamma}$, $\Gamma_{\theta,\gamma}(\theta)$ and $\Gamma^{\ast}_{\gamma}(z)$ are defined by
	\begin{align}
		c_{\ast,\gamma}\tri \int_{0}^{\frac{\pi}{2}}K(\theta)\Gamma_{\theta,\gamma}(\theta)d\theta,\quad
		\Gamma_{\theta,\gamma}(\theta)\tri (K(\theta))^{\frac{\alpha}{3\gamma}},\quad 
		\Gamma^{\ast}_{\gamma}(z)\tri\frac{2z^{\frac{1}{\gamma}}}{\gamma(1+z^{\frac{1}{\gamma}})^2}.\label{def_Gama+c*}
	\end{align}
	Moreover, for 
	$\gamma_1,~\gamma_2\in(0,2)$ and $\delta_0\in(0,1)$, if $|\frac{1}{\gamma_1}-\frac{1}{\gamma_2}|\leq \min\{\frac{\ln(1-\delta_0)}{2\ln(\delta_0)},\delta_0\}$, then there holds
	\begin{align}
		\label{continu-F*}
		&\frac{3}{2\alpha}\|F^{\ast}_{\gamma_1}-F^{\ast}_{\gamma_2}\|_{L^\infty}\leq C\frac{\delta_0}{\gamma_1}.
	\end{align}
\end{prop}
\begin{proof}
	By virtue of Proposition \ref{3FM1}, we conclude that \eqref{form_F1} holds immediately.
	Next, we intend to prove \eqref{continu-F*}.
	If $|\frac{1}{\gamma_1}-\frac{1}{\gamma_2}|\leq \frac{\ln(1-\delta_0)}{2\ln(\delta_0)}$, for any $z\in[0,1]$, we claim that
	\begin{align}\label{3f8}
		|z^{\frac{1}{\gamma_1}}-z^{\frac{1}{\gamma_2}}|\leq \delta_0.
	\end{align}
	Without loss of generality,
	we assume that $\gamma_1\leq \gamma_2$, and the proof of the other case with $\gamma_1>\gamma_2$ is similar.
	For any $z\in[0,1]$, we have
	\begin{align*}
		|z^{\frac{1}{\gamma_1}}-z^{\frac{1}{\gamma_2}}|\leq z^{\frac{1}{\gamma_2}}(1-z^{\frac{1}{\gamma_1}-\frac{1}{\gamma_2}}).
	\end{align*}
	If $z\leq \delta^2_0$, then we get
	\begin{align}\label{est-1}
		|z^{\frac{1}{\gamma_1}}-z^{\frac{1}{\gamma_2}}|\leq z^{\frac{1}{\gamma_2}}\leq \delta_0.
	\end{align}
	Note that $1-z^{\frac{1}{\gamma_1}-\frac{1}{\gamma_2}}$ is a monotonically decreasing function. If $z\geq \delta^2_0$, then we obtain that
	\begin{align}\label{est-2}
		|z^{\frac{1}{\gamma_1}}-z^{\frac{1}{\gamma_2}}|\leq 1-z^{\frac{1}{\gamma_1}-\frac{1}{\gamma_2}}\leq 1-(\delta_0)^{2(\frac{1}{\gamma_1}-\frac{1}{\gamma_2})}\leq 1-(\delta_0)^{\frac{\ln(1-\delta_0)}{\ln(\delta_0)}}= \delta_0.
	\end{align}
	The combination of \eqref{est-1} and \eqref{est-2} shows that claim \eqref{3f8} is valid.

	With the help of \eqref{3f7} and \eqref{3f8}, one deduce that for any $z\in[0,1]$,
	\begin{align*}
		|L^{-1}_z(\Gamma^\ast_{\gamma_1})-L^{-1}_z(\Gamma^\ast_{\gamma_2})|=\frac{2}{(1+z^{\frac{1}{\gamma_1}})(1+z^{\frac{1}{\gamma_2}})}|z^{\frac{1}{\gamma_1}}-z^{\frac{1}{\gamma_2}}|\leq 2\delta_0.
	\end{align*}
	For another range that $z\geq 1$, we get that
	\begin{align*}
		|L^{-1}_z(\Gamma^\ast_{\gamma_1})-L^{-1}_z(\Gamma^\ast_{\gamma_2})|\leq\frac{2}{z^{\frac{1}{\gamma_1}}z^{\frac{1}{\gamma_2}}}|z^{\frac{1}{\gamma_1}}-z^{\frac{1}{\gamma_2}}|\leq 2|(z^{-1})^{\frac{1}{\gamma_1}}-(z^{-1})^{\frac{1}{\gamma_2}}|\leq 2\delta_0.
	\end{align*}
	Then we conclude that for any $z\geq 0$, there holds
	\begin{align}\label{3f9}
		|L^{-1}_z(\Gamma^\ast_{\gamma_1})-L^{-1}_z(\Gamma^\ast_{\gamma_2})|\leq 2\delta_0.
	\end{align}
	According to \eqref{3f4}, \eqref{3f7} and \eqref{3f9}, we find that for any $|\frac{1}{\gamma_1}-\frac{1}{\gamma_2}|\leq \delta_0$,
	\begin{align}\label{3f10}
		|\Gamma^\ast_{\gamma_1}-\Gamma^\ast_{\gamma_2}|
		\lesssim& \frac{1}{\gamma_1}\delta_0,
	\end{align}
	owing to $0<\gamma_1,\gamma_2<2$.
	Moreover, we infer from \eqref{3f8} and $(K(\theta))^{\frac{\alpha}{3}}\in[0,1]$ that 
	\begin{align}\label{3f11}
		|c_{\ast,\gamma_1}-c_{\ast,\gamma_2}|\lesssim \sup_{\theta\in[0,\frac{\pi}{2}]}|\Gamma_{\theta,\gamma_1}-\Gamma_{\theta,\gamma_2}|\lesssim \sup_{\theta\in[0,\frac{\pi}{2}]}|(K(\theta))^{\frac{\alpha}{3\gamma_1}}-(K(\theta))^{\frac{\alpha}{3\gamma_2}}|\lesssim \delta_0.
	\end{align}
	In view of \eqref{3f10} and \eqref{3f11}, we conclude that
	\begin{align*}
		\frac{3}{2\alpha}\|F^{\ast}_{\gamma_1}-F^{\ast}_{\gamma_2}\|_{L^\infty}
		&\lesssim\frac{\delta_0}{\gamma_1}.
	\end{align*}
	Therefore, we finish the proof of this proposition.
\end{proof}

In the last of this subsection, we intend to study the $\mathcal{H}^{k,\ast}$ norm of the fundamental self-similar solutions $F^\ast_{\gamma}$ and give the following lemma.
\begin{lemm}\label{7fm1}
	Let $\beta \in (0,1]$ and $\gamma = \frac{1+\mu}{1-\mu}\beta$.
	There exist constants $\alpha>0$ sufficiently small and $\eta(\beta)$, such that if $\alpha\ll 1-\eta\ll \beta$ and $|\mu|\leq \alpha^\frac{1}{2}$, then there holds 
	\begin{align*}
		\| F^\ast_{\gamma} \|_{\mathcal{H}^{k,\ast}}\leq C_{\beta}\alpha (1-\eta)^{-\frac{1}{2}},
	\end{align*}
	where integer $k$ satisfies $0\leq k\leq 4$.
\end{lemm}
\begin{proof}
	Firstly, by some direct calculations, we find that for any $0<a<1$,
	\begin{align}\label{est-sin}
		\int_0^{\frac{\pi}{2}}(\sin(2\theta))^{-a}d\theta \approx (1-a)^{-1},
	\end{align}
	which will be frequently used.
	We can deduce that
	\begin{align}\label{est-gamma-beta}
		|\gamma-\beta|=\frac{2|\mu|}{1-\mu}\beta\lesssim \alpha^{\frac12}\ll 1.
	\end{align}
	Recalling \eqref{def_K-theta}, \eqref{form_F1} and \eqref{def_Gama+c*}, then applying \eqref{est-sin}, \eqref{est-gamma-beta} and $K(\theta)\leq \sin(2\theta)$, one gets that
	\begin{align}\label{est-L2}
		\|F^\ast_{\gamma}w^{\ast,\eta}\|^2_{L^{2}}&\lesssim \left(\frac{\alpha}{\gamma}\right)^2\int_0^{\frac{\pi}{2}}(\sin(2\theta))^{\frac{2\alpha}{3\gamma}-\eta}d\theta\int_0^{\infty} \frac{z^{\frac{2}{\gamma}}}{(1+z^{\frac{1}{\gamma}})^4}\frac{(1+z^{\frac{1}{\beta}})^2}{z^{\frac{2}{\beta}+\frac{1}{2}}}dz \\ 
		&\lesssim \left(\frac{\alpha}{\beta}\right)^2(1-\eta)^{-1}\int_0^{\infty} \frac{1}{z^{\frac{3}{4}}(1+z)}dz 
		\lesssim \left(\frac{\alpha}{\beta}\right)^2(1-\eta)^{-1}.\notag
	\end{align}
	We note that, since $\eta\in(0,1)$, we use $\frac{2\alpha}{3\gamma}>0$, then apply \eqref{est-sin} to directly obtain the second inequality in \eqref{est-L2}.
	By virtue of \eqref{est-DL1}, \eqref{3f3} and \eqref{3f7}, we see that
	\begin{align}
		|D_z \Gamma^\ast_{\gamma}|&=\left|\frac{1}{\gamma}\Gamma^\ast_{\gamma}(L^{-1}_{z}(\Gamma^\ast_{\gamma})-1)\right|
		\lesssim \frac{1}{\gamma}\Gamma^\ast_{\gamma},
		\label{est-Dz-1}
		\\
		|D^2_z \Gamma^\ast_{\gamma}|
		&=\left|\frac{1}{\gamma^2}\Gamma^\ast_{\gamma}(L^{-1}_{z}(\Gamma^\ast_{\gamma})-1)^2-\frac{1}{\gamma}(\Gamma^\ast_{\gamma})^2\right| \lesssim \frac{1}{\gamma^2}\Gamma^\ast_{\gamma}.
		\label{est-Dz-2}
	\end{align}
	Thanks to \eqref{est-Dz-1} and \eqref{est-Dz-2}, it then follows from a similar way as \eqref{est-L2} that
	\begin{align*}
		\|D_zF^\ast_{\gamma}w^{\ast,\eta}\|_{L^{2}}^2
		\lesssim \frac{\alpha^2}{\gamma^4}(1-\eta)^{-1}\lesssim \frac{\alpha^2}{\beta^4}(1-\eta)^{-1},~~~\|D^2_zF^\ast_{\gamma}w^{\ast,\eta}\|_{L^{2}}^2
		\lesssim \frac{\alpha^2}{\beta^6}(1-\eta)^{-1}.
	\end{align*}
	Remembering the definition of $K(\theta)$ in \eqref{def_K-theta}, we infer that 
	\begin{align}
		|D_\theta \Gamma_{\theta,\gamma}|
		&=\left|\frac{\alpha}{3\gamma} (K(\theta))^{\frac{\alpha}{3\gamma}-1}D_\theta K(\theta)\right|
		=\left|\frac{2\alpha}{3\gamma} (K(\theta))^{\frac{\alpha}{3\gamma}}(\cos(2\theta)-(\sin\theta)^2)\right|
		\lesssim \frac{\alpha}{\gamma}\Gamma_{\theta,\gamma},
		\label{est-Dtheta-1}
		\\
		|D^2_\theta \Gamma_{\theta,\gamma}|
		&=\bigg|(K(\theta))^{\frac{\alpha}{3\gamma}}\Big(\Big(\frac{2\alpha}{3\gamma}\Big)^2 (\cos(2\theta)-(\sin\theta)^2)^2-\frac{2\alpha}{\gamma} \sin^2(2\theta)\Big)\bigg| 
		\lesssim  
		\frac{\alpha}{\gamma}\Gamma_{\theta,\gamma}\left(\frac{\alpha}{\gamma}+\sin(2\theta)\right).
		\label{est-Dtheta-2}
	\end{align}
	We recall the definition of $\lambda$ in \eqref{def-lambda}, then use the fact that $K(\theta) \leq \sin(2\theta)$ and \eqref{est-Dtheta-1}, to arrive at
	\begin{align}\label{est-H1-theta}
		\|D_{\theta}F^\ast_{\gamma}w^{\ast,\lambda}\|^2_{L^{2}}
		&\lesssim \Big(\frac{\alpha}{\gamma}\Big)^4\int_0^{\frac{\pi}{2}}(\sin(2\theta))^{\frac{2\alpha}{3\gamma}-\lambda}d\theta\int_0^{\infty} \frac{z^{\frac{2}{\gamma}}}{(1+z^{\frac{1}{\gamma}})^4}\frac{(1+z^{\frac{1}{\beta}})^2}{z^{\frac{2}{\beta}+\frac{1}{2}}}dz \\ 
		&\lesssim \Big(\frac{\alpha}{\beta}\Big)^3\int_0^{\infty} \frac{1}{z^{\frac{3}{4}}(1+z)}dz \lesssim \left(\frac{\alpha}{\beta}\right)^2.\notag
	\end{align}
	We point out that since $\lambda>1$, we can not drop out $\frac{2\alpha}{3\gamma}$ directly. But we have
	\begin{align}
		\lambda-\frac{2\alpha}{3\gamma}=1-\frac{\alpha}{\beta}\frac{17-23\mu}{30(1+\mu)}\in(0,1).
	\end{align}
	Thus, we can apply \eqref{est-sin} to obtain the second inequality in \eqref{est-H1-theta}.
	
	The same manner yields immediately
	\begin{align*}
		\|D^2_{\theta}F^\ast_{\gamma}w^{\ast,\lambda}\|^2_{L^{2}}
		\lesssim \left(\frac{\alpha}{\beta}\right)^2.
	\end{align*}
	According to separate variables of $F^\ast_{\gamma}$, the estimates \eqref{est-Dz-1} and \eqref{est-Dtheta-1}, we can deduce from a similar way as \eqref{est-L2} and \eqref{est-H1-theta} that
	\begin{align*}
		\|D_zD_{\theta}F^\ast_{\gamma}w^{\ast,\lambda}\|^2_{L^{2}}
		\lesssim \left(\frac{\alpha}{\beta}\right)^2.
	\end{align*}
	After some straightforward calculations, one can get that
	\begin{align}\label{est-Dz-3}
		|D^3_z \Gamma^\ast_{\gamma}|
		 \lesssim \frac{1}{\gamma^3}\Gamma^\ast_{\gamma},~~~~
	|D^4_z \Gamma^\ast_{\gamma}|
		 \lesssim \frac{1}{\gamma^4}\Gamma^\ast_{\gamma},~~~~
         |D^3_\theta \Gamma_{\theta,\gamma}|+|D^4_\theta \Gamma_{\theta,\gamma}|
		\lesssim \frac{\alpha^2}{\gamma^2}\Gamma_{\theta,\gamma}+\frac{\alpha}{\gamma}\Gamma_{\theta,\gamma}\sin(2\theta).
	\end{align}
	Owing to the separate variables of $F^\ast_{\gamma}$, we infer from the estimates \eqref{est-Dz-1}-\eqref{est-Dtheta-2} and \eqref{est-Dz-3} that
	\begin{align*}	
    &\|D^3_zF^\ast_{\gamma}w^{\ast,\eta}\|_{L^{2}}^2+\sum_{i+j= 3,j\geq 1}\|D^i_zD^j_{\theta} fw^{\ast,\lambda}\|_{L^2}^2\lesssim \frac{\alpha^2}{\beta^8}(1-\eta)^{-1},
	\\
		&\|D^4_z fw^{\ast,\eta}\|_{L^2}^2+\sum_{i+j= 4,j\geq 1}\|D^i_zD^j_{\theta} fw^{\ast,\lambda}\|_{L^2}^2\lesssim \frac{\alpha^2}{\beta^{10}}(1-\eta)^{-1}.
	\end{align*}
	The proof of this lemma is thus completed.
\end{proof}

\section{Self-similar variables and asymptotic analysis} 
\label{sec:analysis}

In this section, we aim to investigate the backward self-similar solution of system \eqref{1eq3}. For any given $\beta\in(0,1]$ and $\mu\in\mathbb{R}$ that would be determined later, we consider the following system:
\begin{align}\label{4eq3}
	\left\{\begin{array}{l} (1+\mu)F+(1+\mu)\beta D_{z}F+T=\mathcal{R}(\Phi)F,\quad (z,\theta)\in D,\\[1ex]
		L^{\alpha}_{z}(\Phi)+L_{\theta}(\Phi)=F, \\[1ex]
		F|_{\partial D}=0,~~\Phi|_{\partial D}=0, 
	\end{array}\right.
\end{align}
where the transport term $T$ and the stretching term $\mathcal{R}(\Phi)$ are defined by \eqref{def-T} and \eqref{def-R}, respectively.
The differential operators $L^{\alpha}_{z}$ and $L_{\theta}$ are defined by \eqref{def-L}. 
Here and throughout this section, we perform a refined analysis of system \eqref{4eq3}.

\subsection{Decomposition of potential equation}
In this subsection, we focus on the potential equation as:
\begin{align*}
	L^{\alpha}_{z}(\Phi)+L_{\theta}(\Phi)=F,
\end{align*}
and decompose $F$ based on the structure and properties of $L^{\alpha}_{z}$ and $L_{\theta}$.
It is worth highlighting that the operators $L^{\alpha}_z$ and $L_{\theta}$ are to some extent independent. 
Therefore, it is natural to investigate the properties of these two operators separately.

\medskip

\noindent\textbf{(1) Analysis of $L^{\alpha}_z$}
\medskip

Let $f\in C[0,\infty)$, we want to find $G(f)\in C^2(0,\infty)$, which vanishes as $z$ goes to infinity, such that
$$L^{\alpha}_z(G(f))=f.$$
Indeed, by some direct computations, $L^{\alpha}_z(G(f))$ can be rewritten as:
$$
L^{\alpha}_z(G(f)) = -\alpha^2 D_z\Big(z^{-\frac{5}{\alpha}}D_z\big(z^{\frac{5}{\alpha}}G(f)\big)\Big).
$$
It is easy to compute that
\begin{align}\label{est-f-1}
	f=-D_z L^{-1}_{z}(f).
\end{align}
Hence, we just need to solve
\begin{align}\label{est-G-1}
	\alpha^2z^{-\frac{5}{\alpha}}D_z(z^{\frac{5}{\alpha}}G(f))=L^{-1}_{z}(f),
\end{align}
where $L^{-1}_{z}(f)\in C^1(0,\infty)$ vanishes as $z\rightarrow \infty$.
Return to \eqref{est-G-1}, we assume that $L^{-1}_{\rho}(f)\rho^{\frac{5}{\alpha}}=o(1)$ as $\rho\rightarrow 0$, then employ \eqref{est-f-1} and integrate by parts, to get that
\begin{align}\label{4eq4}
	G(f) &= \frac{1}{\alpha^2}z^{-\frac{5}{\alpha}}\int_0^z \rho^{\frac{5}{\alpha}-1}L^{-1}_{\rho}(f)d\rho
	= \frac{1}{5\alpha}z^{-\frac{5}{\alpha}}\left(L^{-1}_{\rho}(f)\rho^{\frac{5}{\alpha}}\right)\bigg|_{\rho=0}^{\rho=z}+\frac{1}{5\alpha}z^{-\frac{5}{\alpha}}\int_0^z \rho^{\frac{5}{\alpha}-1} f d\rho\\ \notag
	&=\frac{1}{5\alpha}L^{-1}_{z}(f)+\frac{1}{5\alpha}z^{-\frac{5}{\alpha}}\int_0^z \rho^{\frac{5}{\alpha}-1} f d\rho
	\tri G^{*}(f) + \widetilde{G}(f).
\end{align}
Taking $f(0)=0$, we infer that
$$
\lim_{z\rightarrow 0^{+}} \widetilde{G}(f) = \frac{1}{25}\lim_{z\rightarrow 0^{+}} f(z) = 0,
\quad
\lim_{z\rightarrow 0^{+}} G^{*}(f)=\frac{1}{5\alpha}\int_0^{\infty}\frac{f(\rho)}{\rho} d\rho.
$$
We note that the following limit, in general, does not vanish.
But we can get that
\begin{align}\label{limit-z-0}
	\lim_{z\rightarrow0^+} \left|\frac{\widetilde{G}(f)}{G^{*}(f)}\right|\lesssim_{f} \alpha.
\end{align}
Moreover, one can verify that
\begin{align} \label{limit-z-infty}
	\lim_{z\rightarrow\infty}\frac{\widetilde{G}(f)}{G^{*}(f)}=\lim_{z\rightarrow\infty}\frac{\widetilde{G}(f)}{f}\lim_{z\rightarrow\infty}\frac{f}{G^{*}(f)}
	=-\frac{\lim\limits_{z\rightarrow\infty}\frac{D_zf}{f}}{\frac{5}{\alpha}+\lim\limits_{z\rightarrow\infty}\frac{D_zf}{f}}.
\end{align}
The estimates \eqref{limit-z-0} and \eqref{limit-z-infty} imply the crucial inequality below
$$
\max\limits_{z\in[0,\infty)}\left|\frac{\widetilde{G}(f)}{G^{*}(f)}\right| \lesssim_{f} \alpha.
$$
In conclusion, we decompose $G(f)$ into $G^*(f)$ and $\widetilde{G}(f)$.
Here, $G^{*}(f)$ constitutes the main component, while $\widetilde{G}(f)$ can be regarded as a perturbation of $G^{*}(f)$.

\medskip

\noindent\textbf{(2) Analysis of $L_{\theta}$}

\medskip

Denote $L^{*}_{\theta}$ as the conjugate operator of $L_{\theta}$. Then we obtain that 
$$\langle L_{\theta}(h),f \rangle_{\theta}=\langle h,L^{*}_{\theta}(f)\rangle_{\theta},~~~\forall~f,~h\in C^2\left(0,\frac{\pi}{2}\right).$$
In fact, the operator $L^{*}_{\theta}$ can be specifically written as
\begin{align*}
	L^{*}_{\theta}=-\partial_{\theta}^2-(\tan \theta )\partial_{\theta}-6.
\end{align*}
And it is easy to verify that
\begin{align}\label{est-Ltheta-1}
	L_{\theta}(\sin(2\theta))=L^{*}_{\theta}(K(\theta))=0.
\end{align}
For any given $f\in C[0,\infty)$, we consider the following model:
$$L_{\theta}(h)=f,$$
where $h\in C^2(0,\infty)$. 
We can derive that the following necessary condition holds:
$$\langle f,K \rangle_{\theta}=\langle L_{\theta}(h),K \rangle_{\theta}=\langle h,L^{*}_{\theta}(K) \rangle_{\theta}=0,$$
which indicates that $f$ should be orthogonal to $K(\theta)$. However, it is unrealistic to require $\langle F,K \rangle_{\theta}=0$. 
In order to solve the difficulty above, we regard $L^{\alpha}_{z}+L_{\theta}$ as a perturbation of $L^{\alpha}_{z}$, which would inherit the pleasing estimates from $L_{\theta}$. 

\medskip

\noindent\textbf{(3) Analysis of $L^{\alpha}_z+L_{\theta}$}

\medskip
We aim to find $\tilde{F}$, such that $\langle \tilde{F},K \rangle_{\theta}=0$. 
Firstly, we compute that
\begin{align*}
	\langle \sin(2\theta),K \rangle_{\theta}=\frac{4}{15}.
\end{align*}
So, we consider the decomposition of $F$ as follows:
$$
F=\bar{F}+\tilde{F},~~~\bar{F}\tri\frac{15}{4}\sin(2\theta)\langle F,K\rangle_{\theta},~~~\tilde{F}\tri F-\frac{15}{4}\sin(2\theta)\langle F,K\rangle_{\theta}.
$$
Hence, $\bar{F}$ and $\tilde{F}$ satisfy
$$\langle \bar{F},K\rangle_{\theta} = \langle F,K\rangle_{\theta},~~~\langle \tilde{F},K\rangle_{\theta} = 0.$$
Correspondingly, recalling \eqref{est-Ltheta-1}, we can split $\Phi$ as follows:
$$\Phi = G(z)\sin(2\theta)+\widetilde{\Phi},$$
where $G$ and $\widetilde{\Phi}$ satisfy
$$
L^{\alpha}_z(G)=\frac{15}{4}\langle F,K\rangle_{\theta},~~~
L^{\alpha}_z(\widetilde{\Phi})+L_{\theta}(\widetilde{\Phi})=\widetilde{F}.
$$
In the above decomposition, we observe that $\widetilde{\Phi}$ deserves better estimates since $\langle \tilde{F},K\rangle_{\theta} = 0$. 
According to \eqref{4eq4}, by taking $f=\frac{15}{4}\langle F,K\rangle_{\theta}$, we can further decompose $\Phi$ into the following three parts:
\begin{align}\label{decom-Phi}
	\Phi = \Phi_{\text{main}}+\widetilde{G}\sin(2\theta)+\widetilde{\Phi},
\end{align}
where $\Phi_{\text{main}}$ and $\widetilde{G}$ are defined by
$$\Phi_{\text{main}}\tri G^\ast\sin(2\theta),~~~G^\ast\tri \frac{3}{4\alpha}L^{-1}_{z,K}(F),~~~\widetilde{G}\tri \frac{3}{4\alpha}z^{-\frac{5}{\alpha}}\int_0^z \rho^{\frac{5}{\alpha}-1}\langle F,K\rangle_{\theta} d\rho.$$
We note that $\widetilde{G}$ and $\widetilde{\Phi}$ defined above both deserve better estimates than $\Phi_{\text{main}}$.
Thus, in the remaining analysis of this section, we will focus on the main term $\Phi_{\text{main}}$.

\subsection{Asymptotic analysis of stretching term}
The prevailing belief is that solutions to the 3D Euler equations blow up in finite time, as few global existence results are known.
And the stretching term $\mathcal{R}(\Phi)F$ in system \eqref{4eq3} definitely causes blow-up.
Our goal in this subsection is to derive the profile that generates the blow-up from the stretching term. 
As shown in \eqref{decom-Phi}, we compute that
\begin{align*}
	\mathcal{R}(\Phi) = \mathcal{R} (\Phi_{\text{main}}) +\mathcal{R}(\Phi-\Phi_{\text{main}})
	=\frac{3}{2\alpha}L^{-1}_{z,K}(F)-\frac{3
	}{2}\sin^2\theta\langle F,K\rangle_{\theta} + \mathcal{R}(\Phi-\Phi_{\text{main}}).
\end{align*}
Hence, $\eqref{4eq3}_1$ can be written as follows:
\begin{align}\label{4eq5}
	(1+\mu)F+(1+\mu)\beta D_{z}F+T=\frac{3}{2\alpha}L^{-1}_{z,K}(F)F -\frac{3
	}{2}\sin^2\theta\langle F,K\rangle_{\theta}F + \mathcal{R}(\Phi-\Phi_{\text{main}})F.
\end{align}
A natural idea is to regard \eqref{4eq5} as a perturbation of the following fundamental self-similar model:
\begin{align}\label{4eq6}
	F^{\ast}_{\gamma}+\gamma D_{z}F^{\ast}_{\gamma}=\frac{3}{2\alpha}L^{-1}_{z,K}(F^{\ast}_{\gamma})F^{\ast}_{\gamma},~~~\forall~\gamma\in(0,\infty),
\end{align}
which has been extensively studied in Section \ref{sec:toy model}. According to Proposition \ref{3FM2}, the fundamental self-similar model \eqref{4eq6} admits a solution $F^{\ast}_{\gamma}$. 
Then, we rewrite $F$ as 
\begin{align}\label{com-F}
	F=F^{\ast}_{\gamma}+g,
\end{align}
where $F^{\ast}_{\gamma}$ is regarded as the main part of $F$ and $g$ is regarded as the perturbation of $F^{\ast}_{\gamma}$.
This, along with \eqref{4eq5} and \eqref{4eq6}, leads us to derive that $g$ satisfies the following equation:
\begin{align}\label{4eq7}
	\mathcal{L}_{\Gamma}(g)=-T+R_0+R_1+R_2,
\end{align}
where the linear operator $\mathcal{L}_{\Gamma}(g)$ is defined by \eqref{def-LGammag},
and the remaining terms are defined by \eqref{def-R1}-\eqref{def-R2-1}.
Here, $R_0$ and $R_1$ contain the linear terms of $F^{\ast}_{\gamma}$ and $g$, respectively, while $R_2$ contains the nonlinear terms of $F^{\ast}_{\gamma}$ and $g$.
We point out that we will use different way to deal with these three terms.

According to the theory of stability analysis, we know that linear terms determine the nature of the partial differential equations.
Therefore, our core task is to derive the coercivity estimate from the linear operator.
For simplicity, we define the core part of the linear operator $\mathcal{L}_{\Gamma}(g)$ as follows:
\begin{align*}
	\mathcal{L}^{\beta}(g) & \tri g+\beta D_{z}g-\frac{3}{2\alpha}L^{-1}_{z,K}(F^{\ast}_{\beta})g
	\quad \text{( the core operator)},
	\\
	\mathcal{L}^{\beta}_{\Gamma}(g) & \tri 
	\mathcal{L}^{\beta}(g)
	-\frac{3}{2\alpha}L^{-1}_{z,K}(g)F^{\ast}_{\beta}
	\quad \text{(the main operator)}.
\end{align*}
Actually, the linear operator $\mathcal{L}_{\Gamma}(g)$ can be rewritten in the following perturbation form:
\begin{align}\label{4eq7-1}
	\mathcal{L}_{\Gamma}(g)  = \mathcal{L}^{\beta}_{\Gamma}(g)+ \text{Linear perturbation}.
\end{align}

\medskip

\noindent\textbf{(1) Analysis of the core operator $\mathcal{L}^{\beta}$}

\medskip

\noindent By choosing the suitable weight, we can obtain the following coercivity estimate:
$$
\langle\mathcal{L}^{\beta}(g),g(w_z)^2\rangle_z=\left(1-\frac{\beta}{2}\right)\|gw_z\|^2_{L^2_z}.
$$
Reader may refer to Proposition \ref{5co0} for more details.
In other words, it is necessary to establish the linear/nonlinear stability of $g$ in $L^2_z$ with weight $w_z$ satisfying
$$
w_z=O(z^{-\frac{2}{\beta}}),~~~\text{as}~z\rightarrow 0^{+};~~~w_z=O(1),~~~\text{as}~z\rightarrow \infty.
$$
This observation motivates us to construct $g$ with the following radial contour:
\begin{align}\label{4eq8}
	g=o(z^{\frac{2}{\beta}-\frac{1}{2}}),~~~\text{as}~z\rightarrow 0^{+};~~~g=o(z^{-\frac{1}{2}}),~~~\text{as}~z\rightarrow \infty.
\end{align}
If $f$ satisfies \eqref{4eq8}, then for convenience, we denote that
\begin{align}\label{def-g-Gamma}
	f=\widetilde{o}(\Gamma^{\ast}_{\beta}).
\end{align}
We note that if \eqref{4eq8} holds, then for $0<\beta\leq 1$, $g$ satisfies $
\lim\limits_{z\rightarrow0^+} \frac{g(z)}{\Gamma^{\ast}_{\beta}(z)}
=0.
$
But we can not directly deduce from \eqref{4eq8} that for $0<\beta\leq 1$,
$
\lim\limits_{z\rightarrow\infty} \frac{g(z)}{\Gamma^{\ast}_{\beta}(z)}=0.
$
Thus, we refer to the notation $g=\widetilde{o}(\Gamma^{\ast}_{\beta})$ as indicating that $g$ is a stable disturbance of $\Gamma^{\ast}_{\beta}$ at the origin. 

\medskip

\noindent\textbf{(2) Analysis of the main operator $\mathcal{L}^{\beta}_{\Gamma} $}

\medskip

\noindent 
We observe that $\mathcal{L}^{\beta}_{\Gamma}(g) $ contains an additional term $\frac{3}{2\alpha}L^{-1}_{z,K}(g)F^{\ast}_{\beta}$ compared to $\mathcal{L}^{\beta}(g)$. 
Thus, in what follows, we only need to estimate $\frac{3}{2\alpha}L^{-1}_{z,K}(g)F^{\ast}_{\beta}$.
As a component of linear operator $\mathcal{L}_{\Gamma}$, we still need to ensure
$$\frac{3}{2\alpha}L^{-1}_{z,K}(g)F^{\ast}_{\beta}=\widetilde{o}(\Gamma^{\ast}_{\beta}),$$
which requires that
\begin{align}\label{4eq9}
	L^{-1}_{z,K}(g) = o(z^{\frac{1}{\beta}-\frac{1}{2}}),~~~\text{as}~z\rightarrow 0^{+}.
\end{align}
Therefore, $g$ needs to meet the additional structural condition:
\begin{align}\label{4eq10}
	L^{-1}_{z,K}(g)(0) = \int_{0}^{\infty}\int_0^{\frac{\pi}{2}} g(\rho,\theta)\frac{K(\theta)}{\rho}d\theta d\rho = 0.
\end{align}
However, since the operator $L^{-1}_{z,K}$ is non-local, the structural condition \eqref{4eq10} cannot be derived from the equation \eqref{4eq7} naturally and directly. 
Therefore, we introduce the parameter $\mu$ to ensure that \eqref{4eq10} is satisfied, which will be detailed in Subsection \ref{subsec4:mu}.
With \eqref{4eq10} in hand, a new weight $w^{K}$ is introduced and the coercivity estimate of low frequency is obtained in Proposition \ref{5co1}. Specifically,
$$\langle\mathcal{L}^{\beta}_{\Gamma}(g), g(w^K)^2\rangle=\|g\|^2_{\mathcal{H}^{-1}},$$
which implies the better property for the low frequency of $L^{-1}_{z,K}(g)$:
\begin{align}\label{est-LzK-1g}
	L^{-1}_{z,K}(g) = o(z^{\frac{3}{2\beta}-\frac{1}{2}}),~~~\text{as}~z\rightarrow 0^{+}.
\end{align}

\medskip

\noindent\textbf{(3) Analysis of $R_0$}

\medskip

\noindent According to theory of nonlinear stability analysis, it is necessary to show that all terms on the right-hand side of \eqref{4eq7} are of the order $O(g)$. Given the assumptions regarding the perturbation terms, it is readily apparent that $R_1$ and $R_2$  both meet this requirement, except for $R_0$. 
Consequently, a more detailed analysis of $R_0$ is warranted.
Indeed, we deduce from \eqref{4eq6} that
\begin{align}\label{4eq11}
	\frac{R_0}{F^{\ast}_{\gamma}}&=-\mu+(\gamma-(1+\mu)\beta) \frac{D_zF^{\ast}_{\gamma}}{F^{\ast}_{\gamma}}
	= (1+\mu)\left(-1+\frac{\beta}{\gamma}\right)+\left(1-(1+\mu)\frac{\beta}{\gamma}\right)\frac{3}{2\alpha}L^{-1}_{z,K}(F^{\ast}_{\gamma}).
\end{align}
According to Proposition \ref{3FM2} and \eqref{3f7}, we have
\begin{align}\label{4eq12}
	\frac{3}{2\alpha}L^{-1}_{z,K}(F^{\ast}_{\gamma})(0) = L^{-1}_{z}(\Gamma^{\ast}_{\gamma})(0) 
	= 2.
\end{align}
Collecting \eqref{4eq11} and \eqref{4eq12} together, we find that
\begin{align*}
	\frac{R_0}{F^{\ast}_{\gamma}}\bigg|_{z=0}
	= 1-\mu - (1+\mu)\frac{\beta}{\gamma}.
\end{align*}
We choose
\begin{align}\label{4eq13}
	\gamma = \frac{1+\mu}{1-\mu}\beta,
\end{align}
which implies that 
$ \frac{R_0}{F^{\ast}_{\gamma}}\Big|_{z=0}=0.$
With \eqref{4eq13} at hand, if we choose $|\mu|\ll \beta\leq 1$, then $R_0$ can be rewritten as follows:
\begin{align}\label{4eq14}
	R_0
	=\mu \big(-2+\frac{3}{2\alpha}L^{-1}_{z,K}(F^{\ast}_{\gamma})\big)F^{\ast}_{\gamma}=-\frac{2\mu z^{\frac{1}{\gamma}}}{1+z^{\frac{1}{\gamma}}}F^{\ast}_{\gamma} =\widetilde{o}(\Gamma^{\ast}_{\beta}).
\end{align}

\noindent\textbf{(4) Conclusion}

\medskip

\noindent 
In view of \eqref{def-g-Gamma}, \eqref{est-LzK-1g} and \eqref{4eq14}, we finally conclude that
\begin{align}\label{anlysis-Lg}
	\mathcal{L}_{\Gamma}(g)=\widetilde{o}(\Gamma^{\ast}_{\beta}),~~~R_0+R_1+R_2=\widetilde{o}(\Gamma^{\ast}_{\beta}).
\end{align}

\subsection{The null structure of transport term}
Our goal in this subsection is to guarantee the ingredients appearing in the transport term $T$ are of $\widetilde{o}(\Gamma^{\ast}_{\beta})$.
The divergence free condition is well-known as a fundamental structure for the incompressible Euler equations \eqref{1eq1-1}-\eqref{1eq1-2} and plays a crucial role in its regularity-related issues.
Therefore, our next goal is to explore the relevant null structure under new variables and construct compatible Banach spaces to establish the global existence of the self-similar system \eqref{4eq3} with critical regularity.
Recalling the decomposition \eqref{com-F}, then transport term $T$ can be rewritten as follows:
\begin{align}\label{decom-T}
	T=T_{F^{\ast}_{\gamma}} + T_g,
\end{align}
where for any given $f$, which denotes either $F^{\ast}_{\gamma}$ or $g$, the notation $T_f$ is defined by
$$
T_f\tri\frac{U(\Phi)}{\sin(2\theta)} D_{\theta}f+V(\Phi)\alpha D_{z}f.
$$

\medskip

\noindent\textbf{(1) Analysis of $T_{F^{\ast}_{\gamma}}$}

\medskip

\noindent 
Based on the analysis of the linear operator $\mathcal{L}_{\Gamma}(g)$ in \eqref{anlysis-Lg}, we need to ensure that $T_{F^{\ast}_{\gamma}}=\widetilde{o}(\Gamma^{\ast}_{\beta})$. 
For any given $f$, which denotes either $F^{\ast}_{\gamma}$ or $g$, we denote that $G^\ast_f\tri\frac{3}{4\alpha}L^{-1}_{z,K}(f)$. 
According to \eqref{decom-Phi} and \eqref{com-F}, we can treat $\Phi$ as the perturbation of $G^\ast_{F^\ast_{\gamma}}\sin(2\theta)$. 
By some direct calculations, one obtains that
\begin{align}
	U(\Phi)&=U(G^\ast_{F^\ast_{\gamma}}\sin(2\theta))+U(\Phi-G^\ast_{F^\ast_{\gamma}}\sin(2\theta)) \label{est-U-1}\\
	&=-\frac{9}{4\alpha}\sin(2\theta)L^{-1}_{z,K}(F^{\ast}_{\gamma})+\frac{\alpha}{2}\sin(2\theta)\Gamma^{\ast}_{\gamma}+U(\Phi-G^\ast_{F^\ast_{\gamma}}\sin(2\theta)),\notag
\end{align}
and
\begin{align}
	V(\Phi)&= V(G^\ast_{F^\ast_{\gamma}}\sin(2\theta))+V(\Phi-G^\ast_{F^\ast_{\gamma}}\sin(2\theta))\label{est-V-1}\\
	&=\frac{3}{2\alpha}(\cos(2\theta)-\sin^2\theta)L^{-1}_{z,K}(F^{\ast}_{\gamma})+V(\Phi-G^\ast_{F^\ast_{\gamma}}\sin(2\theta)).\notag
\end{align}
Then, we can decompose $T_{F^{\ast}_{\gamma}} $ as
\begin{align*}
	T_{F^{\ast}_{\gamma}} 
	=T^1_{F^{\ast}_{\gamma}} + T^2_{F^{\ast}_{\gamma}},
\end{align*}
where $T^1_{F^{\ast}_{\gamma}}$ and $T^2_{F^{\ast}_{\gamma}}$ are defined by
\begin{align*}
	T^1_{F^{\ast}_{\gamma}}
	\tri& -\frac{9}{4\alpha}L^{-1}_{z,K}(F^{\ast}_{\gamma}) D_{\theta}F^{\ast}_{\gamma} + \frac{3}{2}(\cos(2\theta)-(\sin\theta)^2)L^{-1}_{z,K}(F^{\ast}_{\gamma})D_{z}F^{\ast}_{\gamma},
	\\
	T^2_{F^{\ast}_{\gamma}}
	\tri&\left(\frac{\alpha}{2}\Gamma^{\ast}_{\gamma}+\frac{U(\Phi-G^\ast_{F^\ast_{\gamma}}\sin(2\theta))}{\sin(2\theta)}\right) D_{\theta}F^{\ast}_{\gamma} 
	+  V(\Phi-G^\ast_{F^\ast_{\gamma}}\sin(2\theta))\alpha D_{z}F^{\ast}_{\gamma}.
\end{align*}
Under the assumptions \eqref{4eq8} and \eqref{4eq9}, we can verify that both $T^2_{F^{\ast}_{\gamma}}$ and $T_g$ satisfy the expected estimate, namely,
\begin{align*}
	T^2_{F^{\ast}_{\gamma}}=\widetilde{o}(\Gamma^{\ast}_{\beta}),\quad T_g=\widetilde{o}(\Gamma^{\ast}_{\beta}).
\end{align*}
But for $T^1_{F^{\ast}_{\gamma}}$, we need to conduct a more detailed analysis.
Firstly, we should first ensure that $T^1_{F^{\ast}_{\gamma}}\big|_{t=0}=0$.
Indeed, we deduce from \eqref{3f2} and \eqref{form_F1} that
\begin{align}\label{4eq15}
	\frac{T^1_{F^{\ast}_{\gamma}}}{L^{-1}_{z,K}(F^{\ast}_{\gamma})F^{\ast}_{\gamma}}
	&=-\frac{9}{4\alpha}\frac{D_{\theta}\Gamma_{\theta,\gamma}}{\Gamma_{\theta,\gamma}} + \frac{3}{2}\left(\cos(2\theta)-(\sin\theta)^2\right)\frac{D_{z}\Gamma^{\ast}_{\gamma}}{\Gamma^{\ast}_{\gamma}} \\ \notag
	&=-\frac{9}{4\alpha}\frac{D_{\theta}\Gamma_{\theta,\gamma}}{\Gamma_{\theta,\gamma}} + \frac{3}{2\gamma}\left(\cos(2\theta)-(\sin\theta)^2\right)\left(-1+L^{-1}_{z}(\Gamma^{\ast}_{\gamma})\right).
\end{align}
Choosing 
$\Gamma_{\theta,\gamma} = (K(\theta))^{\frac{\alpha}{3\gamma}}$, one gets, by \eqref{4eq12} and \eqref{4eq15}, that
\begin{align}\label{4eq16}
	\frac{T^1_{F^{\ast}_{\gamma}}}{L^{-1}_{z,K}(F^{\ast}_{\gamma})F^{\ast}_{\gamma}}\bigg|_{z=0}
	= -\frac{9}{4\alpha}\frac{D_{\theta}\Gamma_{\theta,\gamma}}{\Gamma_{\theta,\gamma}} + \frac{3}{2\gamma}\left(\cos(2\theta)-(\sin\theta)^2\right)=0.
\end{align}
Furthermore, combining \eqref{4eq15} and \eqref{4eq16}, then recalling that $\Gamma_{\theta,\gamma} = (K(\theta))^{\frac{\alpha}{3\gamma}}$, we can rewrite $T^1_{F^{\ast}_{\gamma}}$ as
\begin{align*}
	T^1_{F^{\ast}_{\gamma}}&=L^{-1}_{z,K}(F^{\ast}_{\gamma})F^{\ast}_{\gamma}\left(-\frac{9}{4\alpha}\frac{D_{\theta}\Gamma_{\theta,\gamma}}{\Gamma_{\theta,\gamma}} + \frac{3}{2\gamma}\left(\cos(2\theta)-(\sin\theta)^2\right)\left(-1+L^{-1}_{z}(\Gamma^{\ast}_{\gamma})\right)\right)
	\\ \notag
	&= -\frac{3}{\gamma}\left(\cos(2\theta)-\sin^2\theta\right)\frac{z^{\frac{1}{\gamma}}}{1+z^{\frac{1}{\gamma}}}L^{-1}_{z,K}(F^{\ast}_{\gamma})F^{\ast}_{\gamma}= \tilde{o}(\Gamma^{\ast}_{\beta}).
\end{align*}
Finally, we conclude that 
\begin{align}\label{4eq16-1}
	T_{F^{\ast}_{\gamma}}=\widetilde{o}(\Gamma^{\ast}_{\beta}).
\end{align}

\noindent\textbf{(2) Analysis of $T_g$}

\medskip

\noindent 
For $\mathcal{W}\in C((0,\infty)\times(0,\frac{\pi}{2}))$ that would be determined later, we perform the decomposition of $T_g$ below
\begin{align}\label{decom-Tg}
	T_g =\frac{1}{\mathcal{W}}T_{g\mathcal{W}}-\frac{1}{\mathcal{W}}T_{\mathcal{W}}\cdot g,
\end{align}
where the notations $\frac{1}{\mathcal{W}}T_{g\mathcal{W}}$ and $\frac{1}{\mathcal{W}}T_{\mathcal{W}}\cdot g$ are defined by
\begin{align*}
	\frac{1}{\mathcal{W}}T_{g\mathcal{W}} 
	\tri \frac{1}{\mathcal{W}}\left(\frac{U(\Phi)}{\sin(2\theta)} D_{\theta}(g\mathcal{W})+V(\Phi)\alpha D_{z}(g\mathcal{W})\right),
	\quad
	\frac{1}{\mathcal{W}}T_{\mathcal{W}}\cdot g 
	\tri \left(\frac{U(\Phi)}{\sin(2\theta)} \frac{D_{\theta}\mathcal{W}}{\mathcal{W}}+\alpha V(\Phi) \frac{D_{z}\mathcal{W}}{\mathcal{W}}\right)g.
\end{align*}
We point out that $\frac{1}{\mathcal{W}}T_{g\mathcal{W}}$ is the null part, while $\frac{1}{\mathcal{W}}T_{\mathcal{W}}\cdot g$ is the homogeneous part.
We aim to find the weight $\mathcal{W}$, such that the null part can inherit the structure from the divergence-free condition.
This enables us to treat $T_g$ as the null perturbation of the homogeneous part.

\medskip

\noindent\textbf{(2.1) Analysis of the null part}

\medskip

\noindent 
We note that the null structure is crucial for establishing global well-posedness with critical regularity.
After some direct computations, we derive that
\begin{align}\label{null-Tf-1}
	\big\langle T_f,fz^{\frac{3}{\alpha}-1}\cos\theta \big\rangle = 0.
\end{align}
Assume that $\widetilde{\mathcal{W}}(z,\theta)$ is the weight for establishing the $L^2$ stability of $g$, then the weight $\mathcal{W}$ is defined by 
\begin{align}\label{def-W}
	\mathcal{W}\tri\widetilde{\mathcal{W}}~z^{\frac{1}{2}-\frac{3}{2\alpha}}(\cos\theta)^{-\frac{1}{2}},
\end{align}
which, along with \eqref{null-Tf-1}, yields that
\begin{align}\label{4eqnull1}
	\big\langle \frac{1}{\mathcal{W}}T_{g\mathcal{W}}, g(\widetilde{\mathcal{W}})^2 \big\rangle = \big\langle T_{g\mathcal{W}},\big(g\mathcal{W}\big)z^{\frac{3}{\alpha}-1}\cos\theta \big\rangle =0.
\end{align}
For more details, one can refer to Lemma \ref{7div1} in Section \ref{sec:Est-T}.

\medskip

\noindent\textbf{(2.2) Analysis of the homogeneous part}

\medskip

\noindent 
Recalling \eqref{decom-Phi}, we only consider the main term $G^\ast_{F^\ast_{\gamma}}\sin(2\theta)$. With the help of \eqref{form_F1}, \eqref{est-U-1} and \eqref{est-V-1}, we infer that
\begin{align}\label{est-Twg-1}
	\frac{1}{\mathcal{W}}T_{\mathcal{W}}\cdot g &=
	\left(-\frac{3}{2}\frac{D_{\theta}\mathcal{W}}{\mathcal{W}}+ \alpha(\cos(2\theta)-(\sin\theta)^2) \frac{D_{z}\mathcal{W}}{\mathcal{W}}\right)L^{-1}_{z}(\Gamma^{\ast}_{\gamma})~g+\text{Nonlinear ~perturbation}.
\end{align}
We assume that weight $\widetilde{\mathcal{W}}$ satisfies the form of variable separation
\begin{align}\label{def-wW}
\widetilde{\mathcal{W}} = w_z~(\sin(2\theta))^{-\xi}.
\end{align}
According to the detailed analysis of $\mathcal{L}^{\beta}$, we apply \eqref{def-W}, to obtain that
\begin{align*}
	\frac{D_{z}\mathcal{W}}{\mathcal{W}}=\frac{D_{z}\widetilde{\mathcal{W}}_z}{\widetilde{\mathcal{W}}_z}+\frac{1}{2}-\frac{3}{2\alpha} = \frac{1}{2}-\frac{3}{2\alpha}-\frac{1}{\beta}L^{-1}_z(\Gamma^{\ast}_{\beta}).
\end{align*}
When $\alpha$ is small enough, $\frac{3}{2\alpha}$ occupies the majority in $\frac{D_z\mathcal{W}}{\mathcal{W}}$.
Thus, one can deduce that
\begin{align*}
	\frac{1}{\mathcal{W}}T_{\mathcal{W}}\cdot g &=
	-\frac{3}{2}\left(\frac{D_{\theta}\mathcal{W}}{\mathcal{W}}+(\cos(2\theta)(\sin\theta)^2) \right) L^{-1}_{z}(\Gamma^{\ast}_{\gamma})~g+\text{Linear/nonlinear ~perturbation}, 
\end{align*}
where
\begin{align}\label{4eqnull2}
	\frac{D_{\theta}\mathcal{W}}{\mathcal{W}}+\cos(2\theta)-(\sin\theta)^2
	=\frac{D_{\theta}\widetilde{\mathcal{W}}}{\widetilde{\mathcal{W}}}+\cos(2\theta) 
	= (1-2\xi)\cos(2\theta).
\end{align}
Note that the homogeneous part or even entire transport term $T$ acts as the perturbation of $\mathcal{L}^{\beta}$ if $\xi$ is approximately equal to $\frac{1}{2}$. 
Thus, we identify the structure \eqref{4eqnull1} and the identity \eqref{4eqnull2} with $\xi=\frac{1}{2}$ as the null structure of transport term $T$.
In view of \eqref{decom-Tg}, \eqref{4eqnull1}, \eqref{est-Twg-1} and \eqref{4eqnull2}, we can obtain that
\begin{align}\label{est-Tg}
	T_g=\text{Null~perturbation} + \text{Nonlinear~perturbation}.
\end{align}

\noindent\textbf{(3) Conclusion}

\medskip

\noindent 
In the end, we conclude from \eqref{4eq7}, \eqref{4eq7-1}, \eqref{anlysis-Lg}, \eqref{decom-T}, \eqref{4eq16-1} and \eqref{est-Tg} that
$$
\mathcal{L}^{\beta}(g)-\frac{3}{2\alpha}L^{-1}_{z,K}(g)F^{\ast}_{\beta}=\text{Linear/nonlinear~perturbation}+\text{Null~perturbation} = \tilde{o}(\Gamma^{\ast}_{\beta}).
$$

\subsection{Derivation of the self-similar system}\label{subsec4:mu}
In this subsection, we derive the equation of $\mu$, which ensures $L^{-1}_{z,K}(g)(0)=0$. To start with, we apply $L^{-1}_{z,K}$ to self-similar system \eqref{4eq7} and consider the value at $z=0$, to get that
\begin{align*}
	L^{-1}_{z,K}\big(\mathcal{L}_{\Gamma}(g)\big)(0)=-L^{-1}_{z,K}(T)(0)+L^{-1}_{z,K}(R_0)(0)+L^{-1}_{z,K}(R_1)(0)+L^{-1}_{z,K}(R_2)(0).
\end{align*}
Integrating by parts, we have
\begin{align*}
	L^{-1}_{z,K}(D_{z}g)(0) =0,~~~
	L^{-1}_{z,K}\big(L^{-1}_{z,K}(F^{\ast}_{\gamma})g\big)(0)= -L^{-1}_{z,K}\big(L^{-1}_{z,K}(g)F^{\ast}_{\gamma}\big)(0),
\end{align*}
which, along with \eqref{def-LGammag} and \eqref{def-R1}, implies that
\begin{align*}
	L^{-1}_{z,K}\big(\mathcal{L}_{\Gamma}(g)\big)(0) = L^{-1}_{z,K}(g)(0),~~~L^{-1}_{z,K}(R_1)(0)=-\mu L^{-1}_{z,K}(g)(0).
\end{align*}
Thanks to \eqref{form_F1} and \eqref{4eq14}, we see that
\begin{align*}
	L^{-1}_{z,K}(R_0)(0) = -2\mu \int_0^{\infty}\frac{ z^{\frac{1}{\gamma}-1}}{1+z^{\frac{1}{\gamma}}}\langle F^{\ast}_{\gamma},K\rangle_{\theta}dz= -\frac{8}{3}\frac{\alpha\mu}{\gamma}\int_0^{\infty}\frac{ z^{\frac{2}{\gamma}-1}}{(1+z^{\frac{1}{\gamma}})^3}dz = -\frac{4\alpha\mu}{3}.
\end{align*}
Then, we conclude from the three estimates above that
\begin{align}\label{4eq17}
	(1+\mu)L^{-1}_{z,K}(g)(0) &=L^{-1}_{z,K}(R_2-T)(0)-\frac{4\alpha\mu}{3}.
\end{align}
If we take $\mu$ satisfy 
\begin{align*}
	\mu = \frac{3}{4\alpha}L^{-1}_{z,K}(R_2-T)(0),
\end{align*}
then \eqref{4eq17} can be simplified as
\begin{align*}
	(1+\mu)L^{-1}_{z,K}(g)(0) = 0.
\end{align*}
Since $|\mu|<1$, which will be proven in Section \ref{sec:blow-up}, it follows that  
\begin{align*}
	L^{-1}_{z,K}(g)(0) = 0.
\end{align*}
This is consistent with \eqref{4eq10} and completes the analysis of $\mathcal{L}_{\Gamma}$ eventually. In summary, we rewrite \eqref{4eq3} as the following self-similar system 
\begin{align}\label{4eq18}
	\left\{\begin{array}{l} \mathcal{L}_{\Gamma}(g)=-T+R_0+R_1+R_2,\\[1ex]
		L^{\alpha}_{z}(\Phi)+L_{\theta}(\Phi)=F, \\[1ex]
		\mu= \frac{3}{4\alpha}L^{-1}_{z,K}(R_2-T)(0),~~\gamma = \frac{1+\mu}{1-\mu}\beta, \\[1ex]
		L^{-1}_{z,K}(g)(0)=0,~~g|_{\partial D}=0,~~\Phi|_{\partial D}=0, 
	\end{array}\right.
\end{align}
where $F=F^{\ast}_{\gamma}+g$, the operator $\mathcal{L}_{\Gamma}$ and the remaining terms $R_i$ $(i=0,1,2)$ is defined in \eqref{def-LGammag}-\eqref{def-R2-1}, respectively.
Guided by the detailed analysis in this section, we will establish the global estimate of the self-similar system \eqref{4eq18} with critical regularity.

\section{Coercivity}\label{sec:Coer}
Our goal in this section is to explore the coercivity of operator $\mathcal{L}_{\Gamma}$.
Firstly, for any function $f$, we define the main operator $\mathcal{L}^{\beta}_{\Gamma}$ as follows:
\begin{align}\label{def-main-oper}
	\mathcal{L}^{\beta}_{\Gamma}(f)\tri \mathcal{L}^{\beta}(f)-\frac{3}{2\alpha}L^{-1}_{z,K}(f)F^{\ast}_{\beta} ,
\end{align}
where the core operator $\mathcal{L}^{\beta}$ is defined by
\begin{align}\label{def-core-oper}
	\mathcal{L}^{\beta}(f) \tri f + \beta D_z f - L_z^{-1}(\Gamma^{\ast}_{\beta})f.
\end{align}
Recalling the definition of the linear operator $\mathcal{L}_{\Gamma}$ in \eqref{def-LGammag},
then $\mathcal{L}_{\Gamma}$ can be divided into the following two parts, namely, the main operator and the perturbation term:
\begin{align}\label{def-oper}
	\mathcal{L}_{\Gamma}(f)=\mathcal{L}^{\beta}_{\Gamma}(f)+\tilde{P}(f),
\end{align}
where perturbation term $\tilde{P}(f)$ is defined by
\begin{align}\label{def-wP}
	\tilde{P}(f)\tri L^{-1}_z(\Gamma^{\ast}_{\beta}-\Gamma^{\ast}_{\gamma})f+\frac{3}{2\alpha}L^{-1}_{z,K}(f)(F^{\ast}_{\beta}-F^{\ast}_{\gamma}).
\end{align}

Next, we recall the definition of $\mathcal{H}^{-1}([0,\infty)\times[0,\frac{\pi}{2}])$ norm.
For any $\beta\in (0,1]$, we remember that
\begin{align}\label{def_norm_H-1}
	\|f\|^2_{\mathcal{H}^{-1}}=&\frac{2-\beta}{2}\|fw^K\|^2_{L^2}+\frac{1-\beta}{2c_{\ast,\beta}\beta}\|L^{-1}_{z,K}(f)w_z(\Gamma^{\ast}_{\beta})^\frac{1}{2}\|^2_{L^2_z} 
	+\frac{1}{2c_{\ast,\beta}\beta}\|L^{-1}_{z,K}(f)w_z(\Gamma^{\ast}_{\beta}L^{-1}_{z}(\Gamma^{\ast}_{\beta}))^\frac{1}{2}\|^2_{L^2_z}.
\end{align} 
We then introduce the $\mathcal{H}_{\eta}^2([0,\infty)\times[0,\frac{\pi}{2}])$ norm related to $\eta$ as:
\begin{align}\label{def_norm_H2eta}
	\|f\|^2_{\mathcal{H}_{\eta}^2}\tri &(1-\eta)^2(\| fw^{\eta}\|^2_{L^2}+\|D_z fw^{\eta}\|^2_{L^2}+\|D_zD_{\theta} fw^{\lambda}\|^2_{L^2})+(1-\eta)^{-2}\|D_{\theta} fw^{\lambda}\|^2_{L^2} \\
	&+\|D^2_{\theta} fw^{\lambda}\|^2_{L^2}+(1-\eta)^4\|D^2_z fw^{\eta}\|^2_{L^2}.
	\notag
\end{align}
Moreover, we introduce the $\mathcal{E}^2_{\eta}([0,\infty)\times[0,\frac{\pi}{2}])$ norm related to $\eta$ as:
\begin{align}\label{def_norm_E2eta}
	\|f\|^2_{\mathcal{E}^2_{\eta}}\tri\alpha(1-\eta)^{2}\| fw^{\lambda}\|^2_{L^2}+\alpha(1-\eta)^{4}\|D^2_z fw^{\lambda}\|^2_{L^2}.
\end{align}
Here, the definitions of the weights $w^K$, $w_z$, $w^{\eta}$ and $w^{\lambda}$ can be found in Section \ref{sec:notation}.
Specifically, the key distinction between $w^{\lambda}$ and $w^{\eta}$ is their difference in angular weight, with $w^{\lambda}$ exhibiting higher angular singularity.
Actually, $\mathcal{E}^2_{\eta}$ can be viewed as a complement to the $\mathcal{H}_{\eta}^2$ norm, introducing extra angular weighting for radial derivatives.
While $\|\cdot\|_{\mathcal{E}^2_{\eta}}$ itself does not constitute a “good” norm, the intersection space $\mathcal{H}_{\eta}^2\cap\mathcal{E}^2_{\eta}$ is equipped with a “good” norm.
Furthermore, the notations $\langle \cdot, \cdot \rangle_{\mathcal{H}_{\eta}^2}$ and $\langle \cdot, \cdot \rangle_{\mathcal{E}_{\eta}^2}$ denote the corresponding inner products.

\subsection{Coercivity of main operator}
Firstly, we give the estimate of the core operator $\mathcal{L}^{\beta}$, which is defined in \eqref{def-core-oper}.
\begin{lemm}\label{5co0}
	For any $\beta\in(0,1]$, there holds
	\begin{align}\label{est-lem5.1}
		\langle\mathcal{L}^{\beta}(g),g(w_z)^2\rangle_z=\left(1-\frac{\beta}{2}\right)\|gw_z\|^2_{L^2_z}.
	\end{align}
\end{lemm}
\begin{proof}
	Recalling the natural radial weight $w_z(z)=2\beta^{-1}z^{-\frac{1}{\beta}}(\Gamma^{\ast}_{\beta})^{-1}$, we derive that
	\begin{align}\label{est-Dzwz-Gamma}
		\frac{D_z w_z}{w_z}=-\left(\frac{D_z \Gamma^{\ast}_{\beta}}{\Gamma^{\ast}_{\beta}}+\frac{1}{\beta}\right).
	\end{align}
	This equality, when combined with \eqref{3f2}, yields that
	\begin{align*}
		\mathcal{L}^{\beta}(g)w_z 
        &= gw_z + \beta D_z gw_z - L_z^{-1}(\Gamma^{\ast}_{\beta})gw_z  
		=\beta D_z (gw_z)+\left(1- \beta \frac{D_z w_z}{w_z}-
		L_z^{-1}(\Gamma^{\ast}_{\beta})\right)gw_z
        \\
		&=\beta D_z (gw_z)+gw_z.   
	\end{align*}
	This implies that 
	\begin{align*}
		\langle\mathcal{L}^{\beta}(g),g(w_z)^2\rangle_z&=\langle\beta D_z (gw_z)+gw_z,gw_z\rangle_z =\left(1-\frac{\beta}{2}\right)\|gw_z\|^2_{L^2_z}.
	\end{align*}
	Thus, we complete the proof of Lemma \ref{5co0}.
\end{proof}

Based on the estimate of the core operator in Lemma \ref{5co0}, we then investigate the coercivity of the main operator $\mathcal{L}^{\beta}_{\Gamma}$, which is defined in \eqref{def-main-oper}.
\begin{prop}\label{5co1}
	If $L^{-1}_{z,K}(g)(0)=0$, then there holds
	\begin{align}\label{est-H-1}
		\langle\mathcal{L}^{\beta}_{\Gamma}(g), g(w^K)^2\rangle=\|g\|^2_{\mathcal{H}^{-1}}.
	\end{align}
	Moreover, there exist constants $\alpha>0$ sufficiently small and $\eta(\beta)$, such that if $\alpha\ll 1-\eta\ll \beta$ and $|\mu|\leq \alpha^\frac{1}{2}$, then
	\begin{align}
		\langle\mathcal{L}^{\beta}_{\Gamma}(g), g\rangle_{\mathcal{H}_{\eta}^2}
		\geq& \frac{1}{3}\|g\|^2_{\mathcal{H}_{\eta}^2}-\frac{1}{100}\|g\|^2_{\mathcal{H}^{-1}},
		\label{est-H2eta}
		\\
		\langle\mathcal{L}^{\beta}_{\Gamma}(g), g\rangle_{\mathcal{E}_{\eta}^2}
		\geq& \frac{1}{3}\|g\|^2_{\mathcal{E}_{\eta}^2}-\frac{1}{100}\|g\|^2_{\mathcal{H}^{-1}}-\frac{1}{100}\|g\|^2_{\mathcal{H}_{\eta}^2}.
		\label{est-E2eta}
	\end{align}
\end{prop}
\begin{proof}
	According to the definitions of different norms, we will divide into the following three steps to prove \eqref{est-H-1}, \eqref{est-H2eta} and \eqref{est-E2eta}, respectively. 
	
	\medskip

	\noindent\textbf{Step 1: Estimate for $\mathcal{H}^{-1}$-norm.}
	
	\medskip
	
	\noindent 
	It follows from \eqref{3f2} and \eqref{est-Dzwz-Gamma} that
	\begin{align}\label{est-pz-1}
		\partial_z(z\Gamma^{\ast}_{\beta}w_z^2)
		=\left(1-\frac{1}{\beta}\right)\Gamma^{\ast}_{\beta}w_z^2-\frac{1}{\beta}L^{-1}_z(\Gamma^{\ast}_{\beta})\Gamma^{\ast}_{\beta}w_z^2.
	\end{align}
	Recall that the natural angular weight $w^K_{\theta}(\theta)=[K(\theta)]^{\frac{1}{2}(1-\frac{\alpha}{3\beta})}$ and $w^K(z,\theta)=w^K_{\theta} w_z$.
	By virtue of $L^{-1}_{z,K}(g)(0)=0$, we can deduce from \eqref{def-main-oper} that
	\begin{align}\label{est-Lbetag-1}
		\langle\mathcal{L}^{\beta}_{\Gamma}(g), g(w^K)^2\rangle
		&=\langle\mathcal{L}^{\beta}(g), g(w^K)^2\rangle-\langle\frac{3}{2\alpha}L^{-1}_{z,K}(g)F^{\ast}_{\beta}, g(w^K)^2\rangle.
	\end{align}
	It follows from \eqref{form_F1}, \eqref{def-main-oper}, \eqref{est-Dzwz-Gamma} and \eqref{est-pz-1} that
	\begin{align*}
		\langle\frac{3}{2\alpha}L^{-1}_{z,K}(g)F^{\ast}_{\beta}, g(w^K)^2\rangle
		&=\langle L^{-1}_{z,K}(g)\frac{\Gamma_{\theta,\beta}}{c_{\ast,\beta}}\Gamma^{\ast}_{\beta}, g(w^K)^2\rangle
		=\frac{1}{c_{\ast,\beta}}\langle L^{-1}_{z,K}(g)\Gamma^{\ast}_{\beta}, gKw_z^2\rangle
		\\
		&
		=-\frac{1}{c_{\ast,\beta}}\langle L^{-1}_{z,K}(g)\Gamma^{\ast}_{\beta}, D_zL^{-1}_{z,K}(g)w_z^2\rangle_z
		=\frac{1}{2c_{\ast,\beta}}\langle [L^{-1}_{z,K}(g)]^2, \partial_z(z\Gamma^{\ast}_{\beta}w_z^2)\rangle_z
		\\
		&
		=\frac{1}{2c_{\ast,\beta}}\langle [L^{-1}_{z,K}(g)]^2, \left(1-\frac{1}{\beta}\right)\Gamma^{\ast}_{\beta}w_z^2+\frac{1}{\beta}L^{-1}_z(\Gamma^{\ast}_{\beta})\Gamma^{\ast}_{\beta}w_z^2\rangle_z.
	\end{align*}
	Substituting the estimate above and \eqref{est-lem5.1} into \eqref{est-Lbetag-1}, we obtain \eqref{est-H-1} directly.
	
	\medskip

	\noindent\textbf{Step 2: Estimate for $\mathcal{H}^{2}_{\eta}$-norm.}
	
	\medskip
	
	\noindent 
	Recalling the definition of $\mathcal{H}_{\eta}^2$, we have
	\begin{align}
		\langle\mathcal{L}^{\beta}_{\Gamma}(g), g\rangle_{\mathcal{H}_{\eta}^2}
		=&
		(1-\eta)^2\langle\mathcal{L}^{\beta}_{\Gamma}(g), g(w^\eta)^2\rangle
		+\left[(1-\eta)^{-2}\langle D_{\theta}\mathcal{L}^{\beta}_{\Gamma}(g), D_{\theta}g(w^\eta)^2\rangle
		\right.
		\label{est-J}\\
		&
		\left.
		+(1-\eta)^2\langle D_{z}\mathcal{L}^{\beta}_{\Gamma}(g), D_zg(w^\eta)^2\rangle\right]
		+\left[\langle D_{\theta}^2\mathcal{L}^{\beta}_{\Gamma}(g), D_{\theta}^2g(w^\eta)^2\rangle
		\right.
		\notag\\
		&
		\left.
		+(1-\eta)^2\langle D_{\theta}D_{z}\mathcal{L}^{\beta}_{\Gamma}(g), D_{\theta}D_zg(w^\eta)^2\rangle
		+(1-\eta)^4\langle D_{z}^2\mathcal{L}^{\beta}_{\Gamma}(g), D_{z}^2g(w^\eta)^2\rangle\right]
		\notag\\
		\tri&J_1+J_2+J_3.
		\notag
	\end{align}
	We note that $J_1$, $J_2$ and $J_3$ are related to the $\mathcal{L}^{2}_{\eta}$,  $\dot{\mathcal{H}}^{1}_{\eta}$ and $\dot{\mathcal{H}}^{2}_{\eta}$ norms, respectively.
	
	\medskip   
	
	\noindent\textbf{Step 2.1: Estimate for $\mathcal{L}^{2}_{\eta}$-norm.} 
	
	\medskip
	
	\noindent 
	Recalling the definition of $\mathcal{L}^{\beta}_{\Gamma}$ in \eqref{def-main-oper}, we have
	\begin{align}\label{5in3-1}
		J_1
		&=(1-\eta)^2\langle\mathcal{L}^{\beta}(g), g(w^\eta)^2\rangle-(1-\eta)^2\langle\frac{3}{2\alpha}L^{-1}_{z,K}(g)F^{\ast}_{\beta}, g(w^\eta)^2\rangle
		\tri J_{11}+J_{12}.
	\end{align}
	By virtue of Lemma \ref{5co0}, we infer that
	\begin{align}\label{5in1}
		J_{11}
		&=\left(1-\frac{\beta}{2}\right)(1-\eta)^2\|gw^\eta\|^2_{L^2}.
	\end{align}
	In view of \eqref{form_F1}, one can get
	that
	\begin{align}\label{5in2}
		|J_{12}|
		\lesssim&(1-\eta)^2\|gw^\eta\|_{L^2}\|\frac{3}{2\alpha}L^{-1}_{z,K}(g)F^{\ast}_{\beta}w^\eta\|_{L^2} \\ \notag 
		\lesssim&(1-\eta)^2\|gw^\eta\|_{L^2}\|L^{-1}_{z,K}(g)\Gamma^{\ast}_{\beta}w_z\|_{L^2_z}\|\frac{\Gamma_{\theta,\beta}}{c_{\ast,\beta}}w_\theta^\eta\|_{L^2_\theta} \\ \notag
		\lesssim&(1-\eta)^2\|gw^\eta\|_{L^2}
		\beta^{-\frac12}\|L^{-1}_{z,K}(g)w_z(\Gamma^{\ast}_{\beta}L^{-1}_{z}(\Gamma^{\ast}_{\beta}))^\frac{1}{2}\|_{L^2_z}
		\beta^{\frac12}\|\frac{\Gamma^{\ast}_{\beta}}{L_{z}^{-1}(\Gamma^{\ast}_{\beta})}\|^{\frac{1}{2}}_{L^{\infty}_{z}}
		\|\Gamma_{\theta,\beta}w_\theta^\eta\|_{L^2_\theta} \\ \notag
		\lesssim&(1-\eta)^{\frac{3}{2}}\|gw^\eta\|_{L^2}\|g\|_{\mathcal{H}^{-1}} \\ \notag
		\lesssim& (1-\eta)^{\frac{5}{2}}\|gw^\eta\|^2_{L^2}+(1-\eta)^{\frac{1}{2}}\|g\|^2_{\mathcal{H}^{-1}},
	\end{align}
	where we have used the fact that
	\begin{align}
		&\|\frac{\Gamma^{\ast}_{\beta}}{L_{z}^{-1}(\Gamma^{\ast}_{\beta})}\|_{L^{\infty}_{z}}\lesssim \beta^{-1},
		\label{est-Gamma*-1}
		\\
		&\|\Gamma_{\theta,\beta}w_\theta^\eta\|_{L^2_\theta}^2
		\lesssim  \int_0^{\frac{\pi}{2}} \sin^{-\eta}(2\theta)d\theta\lesssim (1-\eta)^{-1}, \quad \forall~\eta\in(0,1).
		\label{est-sin-1}
	\end{align}
	The combination of \eqref{5in3-1}-\eqref{5in2} gives that
	\begin{align}\label{5in3}
		J_1
		&\geq \left(1-\frac{\beta}{2}-C(1-\eta)^{\frac{1}{2}}\right)(1-\eta)^2\|gw^\eta\|^2_{L^2}-C(1-\eta)^{\frac{1}{2}}\|g\|^2_{\mathcal{H}^{-1}}.
	\end{align}

	\noindent\textbf{Step 2.2: Estimate for $\dot{\mathcal{H}}^{1}_{\eta}$-norm.} 
	
	\medskip
	
	\noindent 
	After some calculations, one obtains that
	\begin{align}
		D_\theta \mathcal{L}^{\beta}_{\Gamma}(g)
		&=\mathcal{L}^{\beta}(D_\theta g)-\frac{3}{2\alpha}L^{-1}_{z,K}(g)D_\theta F^{\ast}_{\beta},
		\label{eq-D-theta}
		\\
		D_z \mathcal{L}^{\beta}_{\Gamma}(g)
		&=\mathcal{L}^{\beta}(D_z g)+\Gamma^{\ast}_{\beta}g-\frac{3}{2\alpha}L^{-1}_{z,K}(g)D_z F^{\ast}_{\beta}+\frac{3}{2\alpha}\langle g,K\rangle_\theta F^{\ast}_{\beta}.
		\label{eq-D-z}
	\end{align}
	Hence, we have
	\begin{align}
		J_2=&(1-\eta)^{-2}\langle\mathcal{L}^{\beta}(D_\theta g), D_\theta g(w^\lambda)^2\rangle
		+(1-\eta)^{2}\langle\mathcal{L}^{\beta}(D_z g), D_z g(w^\eta )^2\rangle
		\label{est-J2-1}\\
		&
		-(1-\eta)^{-2}\langle\frac{3}{2\alpha}L^{-1}_{z,K}(g)D_\theta F^{\ast}_{\beta},D_\theta g(w^\lambda)^2\rangle
		+(1-\eta)^{2}\langle \Gamma^{\ast}_{\beta}g, D_z g(w^\eta )^2\rangle
		\notag\\
		&
		-(1-\eta)^{2}\langle\frac{3}{2\alpha}L^{-1}_{z,K}(g)D_z F^{\ast}_{\beta},D_z g(w^\eta)^2\rangle
		+(1-\eta)^{2}\langle\frac{3}{2\alpha}\langle g,K\rangle_\theta F^{\ast}_{\beta},D_z g(w^\eta)^2\rangle
		\notag\\
		\tri& \sum_{i=1}^{6}J_{2i}.
		\notag
	\end{align}
	Next, we estimate $J_{2i}$ $(i=1,\cdots,6)$ term by term.
	By virtue of Lemma \ref{5co0}, we infer that
	\begin{align}
		J_{21}=\left(1-\frac{\beta}{2}\right)(1-\eta)^{-2}\|D_\theta gw^\lambda\|^2_{L^2},~~~~
		J_{22}=\left(1-\frac{\beta}{2}\right)(1-\eta)^{2}\|D_z gw^\eta\|^2_{L^2}.
		\label{5in4}
	\end{align}
	
	According to \eqref{est-Gamma*-1} and Proposition \ref{3FM2}, we have
	\begin{align}\label{5in5}
		|J_{23}|
		&\lesssim(1-\eta)^{-2}\|D_\theta gw^\lambda\|_{L^2}\|\frac{3}{2\alpha}L^{-1}_{z,K}(g)D_\theta F^{\ast}_{\beta}w^\lambda\|_{L^2} \\ \notag 
		&\lesssim\sqrt{\frac{\alpha}{\beta}}(1-\eta)^{-2}\|D_\theta gw^\lambda\|_{L^2}\|g\|_{\mathcal{H}^{-1}} \\ \notag
		&\lesssim\sqrt{\frac{\alpha}{\beta}}(1-\eta)^{-1}\left((1-\eta)^{-2}\|D_\theta gw^\lambda\|^2_{L^2}+\|g\|^2_{\mathcal{H}^{-1}}\right),
	\end{align}
	where we have used the fact that
	\begin{align}\label{est-sin-2}
		\|\Gamma_{\theta,\beta}w_\theta^\lambda\|_{L^2_\theta}^2\lesssim \int_0^{\frac{\pi}{2}} \sin^{-1+\frac{17\alpha}{30\beta}}(2\theta)d\theta\lesssim \frac{\beta}{\alpha}.
	\end{align}
	With the help of \eqref{est-pz-1}, we get that for any $0<\beta\leq 1$,
	\begin{align}\label{5in8}
		J_{24}
		&=\frac{1}{2}(1-\eta)^{2}\langle \Gamma^{\ast}_{\beta}, D_z g^2(w^\eta )^2\rangle
		=-\frac{1}{2}(1-\eta)^{2}\langle \partial_z(z\Gamma^{\ast}_{\beta}w_z^2),g^2(w_\theta^\eta)^2\rangle\\ \notag
		&=\frac{1-\beta}{2\beta}(1-\eta)^{2}\|g(\Gamma^\ast_\beta)^{\frac{1}{2}}w^\eta\|^2_{L^2}+\frac{1}{2\beta}(1-\eta)^{2}\|g(\Gamma^\ast_\beta L^{-1}_z(\Gamma^\ast_\beta))^{\frac{1}{2}}w^\eta\|^2_{L^2}\geq 0.
	\end{align}
	Then \eqref{est-Gamma*-1}, \eqref{est-sin-1} and Proposition \ref{3FM2} ensure that
	\begin{align}\label{5in9}
		|J_{25}|
		&\lesssim(1-\eta)^{2}\|D_z gw^\eta\|_{L^2}\|\frac{3}{2\alpha}L^{-1}_{z,K}(g)D_z F^{\ast}_{\beta}w^\eta\|_{L^2} 
        \\
		&\lesssim\frac{1}{\beta^{\frac32}}\big((1-\eta)^{\frac{5}{2}}\|D_z g w^\eta\|^2_{L^2}+(1-\eta)^{\frac{1}{2}}\|g\|^2_{\mathcal{H}^{-1}}\big).
        \notag
	\end{align}
	Thanks to \eqref{3f7} and Proposition \ref{3FM2}, we can deduce that
	\begin{align}\label{5in10}
		|J_{26}|
		&\lesssim(1-\eta)^{2}\|D_z gw^\eta\|_{L^2}\|\frac{3}{2\alpha}\langle g,K\rangle_\theta F^{\ast}_{\beta}w^\eta\|_{L^2} \\ \notag 
		&\lesssim(1-\eta)^{2}\|D_zgw^\eta\|_{L^2}\|\langle g,K\rangle_\theta\Gamma^{\ast}_{\beta}w_z\|_{L^2_z}\|\frac{\Gamma_{\theta,\beta}}{c_{\ast,\beta}}w_\theta^\eta\|_{L^2_\theta} \\ \notag
		&\lesssim\frac{1}{\beta}\left((1-\eta)^{\frac{5}{2}}\|D_z g w^\eta\|^2_{L^2}+(1-\eta)^{\frac{1}{2}}\|g\|^2_{\mathcal{H}^{-1}}\right).
	\end{align}
	Substituting \eqref{5in4}-\eqref{5in5}, \eqref{5in8}-\eqref{5in10} into \eqref{est-J2-1}, we conclude that
	\begin{align}\label{5in6}
		J_{2}
		\geq& \left(1-\frac{\beta}{2}-C\sqrt{\frac{\alpha}{\beta}}(1-\eta)^{-1}\right)(1-\eta)^{-2}\|D_\theta gw^\lambda\|^2_{L^2} -C\sqrt{\frac{\alpha}{\beta}}(1-\eta)^{-1}\|g\|^2_{\mathcal{H}^{-1}}
		\\
		&+\left(1-\frac{\beta}{2}-C(1-\eta)^{\frac{1}{2}}\frac{1}{\beta^{\frac32}}\right)(1-\eta)^2\|D_z gw^\eta\|^2_{L^2}
		-C(1-\eta)^{\frac{1}{2}}\frac{1}{\beta^{\frac32}}\|g\|^2_{\mathcal{H}^{-1}}.
		\notag
	\end{align}

	\noindent\textbf{Step 2.3: Estimate for $\dot{\mathcal{H}}^{2}_{\eta}$-norm.}
	
	\medskip
	
	\noindent
	Straightforward computations give that
	\begin{align*}
		D^2_\theta \mathcal{L}^{\beta}_{\Gamma}(g)
		=&\mathcal{L}^{\beta}(D^2_\theta g)-\frac{3}{2\alpha}L^{-1}_{z,K}(g)D^2_\theta F^{\ast}_{\beta},
		\\
		D_\theta D_z \mathcal{L}^{\beta}_{\Gamma}(g)
		=&\mathcal{L}^{\beta}(D_\theta D_z g)+\Gamma^{\ast}_{\beta}D_\theta g-\frac{3}{2\alpha}L^{-1}_{z,K}(g)D_\theta D_z F^{\ast}_{\beta}+\frac{3}{2\alpha}\langle g,K\rangle_\theta D_\theta F^{\ast}_{\beta},
	\end{align*}
	and
	\begin{align}\label{5in18}
		D^2_z \mathcal{L}^{\beta}_{\Gamma}(g)
		&=\mathcal{L}^{\beta}(D^2_z g)+2\Gamma^{\ast}_{\beta}D_z g+gD_z\Gamma^{\ast}_{\beta}-\frac{3}{2\alpha}L^{-1}_{z,K}(g)D^2_z F^{\ast}_{\beta}+\frac{3}{\alpha}\langle g,K\rangle_\theta D_z F^{\ast}_{\beta}
        +\frac{3}{2\alpha}\langle D_z g,K\rangle_\theta F^{\ast}_{\beta}.
	\end{align}
	Applying three equalities above, we find that
	\begin{align}\label{est-J3-1}
		J_3=&\left[\langle \mathcal{L}^{\beta}(D^2_\theta g), D^2_\theta g(w^\lambda)^2\rangle
		+(1-\eta)^{2}\langle\mathcal{L}^{\beta}(D_\theta D_z g), D_\theta D_z g(w^\lambda)^2\rangle
		\right.
		\\
		&
		\left.
		+(1-\eta)^{4}\langle\mathcal{L}^{\beta}(D^2_z g), D^2_z g(w^\eta )^2\rangle
		\right]
		+(1-\eta)^{2}\langle \Gamma^{\ast}_{\beta}D_\theta g, D_\theta D_z g(w^\lambda)^2\rangle
		\notag\\
		&
		+(1-\eta)^{4}\langle 2\Gamma^{\ast}_{\beta}D_z g, D^2_z g(w^\eta )^2\rangle
		-
		\langle \frac{3}{2\alpha}L^{-1}_{z,K}(g)D^2_\theta F^{\ast}_{\beta}, D^2_\theta g(w^\lambda)^2\rangle
		\notag\\
		&
		-(1-\eta)^{2}\langle\frac{3}{2\alpha}L^{-1}_{z,K}(g)D_\theta D_z F^{\ast}_{\beta},D_\theta D_z g(w^\lambda)^2\rangle
		-(1-\eta)^{4}\langle\frac{3}{2\alpha}L^{-1}_{z,K}(g)D^2_z F^{\ast}_{\beta},D^2_z g(w^\eta)^2\rangle
		\notag\\
		&
		+(1-\eta)^{2}\langle\frac{3}{2\alpha}\langle g,K\rangle_\theta D_\theta F^{\ast}_{\beta},D_\theta D_z g(w^\lambda)^2\rangle
		+(1-\eta)^{4}\langle gD_z \Gamma^{\ast}_{\beta},D^2_z g(w^\eta)^2\rangle
		\notag\\
		&
		+\Big[(1-\eta)^{4}\langle\frac{3}{\alpha}\langle g,K\rangle_\theta D_z F^{\ast}_{\beta},D^2_z g(w^\eta)^2\rangle 
		+(1-\eta)^{4}\langle\frac{3}{2\alpha}\langle D_z g,K\rangle_\theta F^{\ast}_{\beta},D^2_z g(w^\eta)^2\rangle\Big]
		\notag
		\\
		\tri&\sum_{i=1}^{9}J_{3i}.
		\notag
	\end{align}
	Now, we begin to estimate $J_{3i}$ $(i=1,\cdots,9)$ term by term.
	Again thanks to Lemma \ref{5co0}, one gets that
	\begin{align*}
		J_{31}=
		\left(1-\frac{\beta}{2}\right)
		\left(\|D_\theta^2 gw^\lambda\|^2_{L^2}
		+
		(1-\eta)^{2}\|D_\theta D_z gw^\lambda\|^2_{L^2}
		+
		(1-\eta)^{4}\|D_z^2 gw^\eta\|^2_{L^2}
		\right).
	\end{align*}
	It then follows from a similar way as \eqref{5in8} that for any $0<\beta\leq1$,
	\begin{align*}
		J_{32}+J_{33}
		\geq0.
	\end{align*}
	We then deduce from a similar way as \eqref{5in5} that
	\begin{align}\label{5in12}
		|J_{34}|
		&\lesssim \sqrt{\frac{\alpha}{\beta}}\left(\|D^2_\theta gw^\lambda\|^2_{L^2}+\|g\|^2_{\mathcal{H}^{-1}}\right).
	\end{align}
	In the same manner as \eqref{5in5} and \eqref{5in9}, one gets, by \eqref{est-sin-2} that
	\begin{align*}
		|J_{35}|
		&\lesssim\frac{\alpha^{\frac{1}{2}}}{\beta^{\frac{3}{2}}}(1-\eta)^{2}\|D_\theta D_z g w^\lambda\|_{L^2}\|g\|_{\mathcal{H}^{-1}} \lesssim\alpha^{\frac{1}{2}}\left((1-\eta)^{2}\|D_\theta D_z g w^\lambda\|^2_{L^2}+\|g\|^2_{\mathcal{H}^{-1}}\right).
	\end{align*}
	Combining this with \eqref{est-sin-1} and applying a derivation analogous to \eqref{5in9}, we obtain that
	\begin{align}\label{5in15-1}
		|J_{36}|
		&\lesssim\frac{1}{\beta^{2}}(1-\eta)^{\frac72}\|D_z^2 g w^\eta\|_{L^2}\|g\|_{\mathcal{H}^{-1}} \lesssim\frac{1}{\beta^{2}}(1-\eta)^{\frac12}\left((1-\eta)^{4}\|D_z^2 g w^\eta\|^2_{L^2}+\|g\|^2_{\mathcal{H}^{-1}}\right).
	\end{align}
	It then follows from a similar way as \eqref{5in10} that
	\begin{align}\label{5in16}
		|J_{37}|
		&\lesssim\frac{\alpha^{\frac12}}{\beta^{\frac32}}(1-\eta)^{2}\|D_\theta D_z g w^\lambda\|_{L^2}\|g\|_{\mathcal{H}^{-1}} 
		\lesssim\alpha^{\frac{1}{2}}((1-\eta)^{2}\|D_\theta D_z g w^\lambda\|^2_{L^2}+\|g\|^2_{\mathcal{H}^{-1}}).
	\end{align}
	Using Proposition \ref{3FM2}, we infer that
	\begin{align}\label{5in21}
		|J_{38}|
		\lesssim\frac{1}{\beta^2}(1-\eta)\left((1-\eta)^{4}\|D^2_z g w^\eta\|^2_{L^2}+(1-\eta)^{2}\|g w^\eta\|^2_{L^2}\right).
	\end{align}
	We then deduce through a method analogous to \eqref{5in10} that
	\begin{align}\label{5in22}
		|J_{39}|
		\lesssim\frac{1}{\beta^2}(1-\eta)^{\frac{1}{2}}((1-\eta)^{4}\|D^2_z g w^\eta\|^2_{L^2}+(1-\eta)^{2}\|D_z g w^\eta\|^2_{L^2}+\|g\|^2_{\mathcal{H}^{-1}}).
	\end{align}
	Inserting 
	the estimates of $J_{3i}$ with $i=1,\cdots,9$ into \eqref{est-J3-1}, we have
	\begin{align}\label{5in17}
		J_3
		&\geq 
		\left(1-\frac{\beta}{2}-C\sqrt{\frac{\alpha}{\beta}}\right)\|D_\theta^2 gw^\lambda\|^2_{L^2} 
		-C\sqrt{\frac{\alpha}{\beta}}\|g\|^2_{\mathcal{H}^{-1}}
		\\
		&+
		\left(1-\frac{\beta}{2}-C\alpha^{\frac{1}{2}}\right)(1-\eta)^2\|D_\theta D_z gw^\lambda\|^2_{L^2} 
		-C\alpha^{\frac{1}{2}}\|g\|^2_{\mathcal{H}^{-1}}
		\notag\\
		&+\left(1-\frac{\beta}{2}-C(1-\eta)^{\frac{1}{2}}\frac{1}{\beta^2}\right)(1-\eta)^4\|D^2_z gw^\eta\|^2_{L^2}  
		-C(1-\eta)^{\frac{1}{2}}\frac{1}{\beta^2}(\|g\|^2_{\mathcal{H}^{-1}}+\|g\|^2_{\mathcal{H}^{2}_\eta})
		.
		\notag
	\end{align}
	
	We substitute \eqref{5in3} in \textbf{Step 2.1}, \eqref{5in6} in \textbf{Step 2.2} and \eqref{5in17} in \textbf{Step 2.3} into \eqref{est-J}, and choose the suitable constants $\alpha$ and $1-\eta$ sufficiently small relative to $\beta$, to conclude that \eqref{est-H2eta} holds.
	
	\medskip

	\noindent\textbf{Step 3: Estimate for $\mathcal{E}^{2}_{\eta}$-norm.}
	
	\medskip
	
	\noindent
	We recall the definition of $\mathcal{E}^{2}_{\eta}$, to discover that
	\begin{align}\label{est-K-1}
		\langle\mathcal{L}^{\beta}_{\Gamma}(g), g\rangle_{\mathcal{E}_{\eta}^2}
		=&
		\alpha(1-\eta)^2\langle\mathcal{L}^{\beta}_{\Gamma}(g), g(w^\lambda)^2\rangle
		+\alpha(1-\eta)^4\langle D^2_z\mathcal{L}^{\beta}_{\Gamma}(g), D^2_z g(w^\lambda)^2\rangle
		\tri K_1+K_2,
	\end{align}
	where 
	$K_1$ is related to the $\mathcal{L}^{2}$ norm with the weight $w^\lambda$, while $K_2$ is related to the radial $\dot{\mathcal{H}}^{2}$ norm with the weight $w^\lambda$.
	
	\medskip   
	
	\noindent\textbf{Step 3.1: Estimate for $\mathcal{L}^{2}_{\eta}$-norm with the weight $w^\lambda$.} 
	
	\medskip
	
	\noindent 
	According to the definition of $\mathcal{L}^{\beta}_{\Gamma}$ in \eqref{def-main-oper}, $K_1$ can be divided into the following two terms:
	\begin{align}\label{est-K1-1}
		K_1=\alpha(1-\eta)^2\langle\mathcal{L}^{\beta}(g), g(w^\lambda)^2\rangle
		-\alpha(1-\eta)^2\langle\frac{3}{2\alpha}L^{-1}_{z,K}(g)F^{\ast}_{\beta}, g(w^\lambda)^2\rangle\tri K_{11}+K_{12}.
	\end{align}
	By virtue of Lemma \ref{5co0}, we deduce that
	\begin{align}\label{5in24}
		K_{11}
		=\left(1-\frac{\beta}{2}\right)\alpha(1-\eta)^2\|gw^\lambda\|^2_{L^2}.
	\end{align}
	With the help of \eqref{est-sin-2}, one can infer that
	\begin{align}\label{5in25}
		|K_{12}|
		&\lesssim\alpha(1-\eta)^2\|gw^\lambda\|_{L^2}\|\frac{3}{2\alpha}L^{-1}_{z,K}(g)F^{\ast}_{\beta}w^\lambda\|_{L^2} \\ \notag 
		&\lesssim\alpha(1-\eta)^2\|gw^\lambda\|_{L^2}\|L^{-1}_{z,K}(g)\Gamma^{\ast}_{\beta}w_z\|_{L^2_z}\|\frac{\Gamma_{\theta,\beta}}{c_{\ast,\beta}}w_\theta^\lambda\|_{L^2_\theta} \\ \notag
		&\lesssim (1-\eta)(\alpha(1-\eta)^{2}\|gw^\lambda\|^2_{L^2}+\|g\|^2_{\mathcal{H}^{-1}}).
	\end{align}
	Plugging \eqref{5in24} and \eqref{5in25} in \eqref{est-K1-1}, we conclude that
	\begin{align}\label{5in26}
		K_1
		&\geq \left(1-\frac{\beta}{2}-C(1-\eta)\right)\alpha(1-\eta)^2\|gw^\lambda\|^2_{L^2}-C(1-\eta)\|g\|^2_{\mathcal{H}^{-1}}.
	\end{align}

	\noindent\textbf{Step 3.2: Estimate for the radial $\dot{\mathcal{H}}^{2}_{\eta}$-norm with the weight $w^\lambda$.} 
	
	\medskip
	
	\noindent 
	Recalling \eqref{5in18} and the definition of $\mathcal{L}^{\beta}_{\Gamma}$ in \eqref{def-main-oper}, we divide $K_2$ into the following five terms:
	\begin{align}\label{est-K2-1}
		K_2&=\alpha(1-\eta)^{4}\langle\mathcal{L}^{\beta}(D^2_z g), D^2_z g(w^\lambda )^2\rangle
		+\alpha(1-\eta)^{4}\langle 2\Gamma^{\ast}_{\beta}D_z g, D^2_z g(w^\lambda)^2\rangle 
		\\
		&~~~
		+\alpha(1-\eta)^{4}\langle gD_z \Gamma^{\ast}_{\beta},D^2_z g(w^\lambda)^2\rangle
		-\alpha(1-\eta)^{4}\langle\frac{3}{2\alpha}L^{-1}_{z,K}(g)D^2_z F^{\ast}_{\beta},D^2_z g(w^\lambda)^2\rangle
		\notag\\
		&~~~
		+\alpha(1-\eta)^{4}\left(\langle\frac{3}{\alpha}\langle g,K\rangle_\theta D_z F^{\ast}_{\beta},D^2_z g(w^\lambda)^2\rangle +\langle\frac{3}{2\alpha}\langle D_z g,K\rangle_\theta F^{\ast}_{\beta},D^2_z g(w^\lambda)^2\rangle\right)
		\notag\\
		&\tri \sum_{i=1}^{5}K_{2i}.
		\notag
	\end{align}
	We now estimate $K_{2i}$ $(i=1,\cdots,5)$ term by term.
	Again thanks to Lemma \ref{5co0}, we obtain that
	\begin{align}\label{5in27}
		K_{21}
		=\left(1-\frac{\beta}{2}\right)\alpha(1-\eta)^{4}\|D^2_z gw^\lambda\|^2_{L^2}.
	\end{align}
	We use a similar derivation of \eqref{5in8} and \eqref{5in21}, to deduce that
	\begin{align}\label{5in28}
		K_{22}\geq0,~~~~
		|K_{23}|\lesssim \frac{1-\eta}{\beta^2}
		\left(\alpha(1-\eta)^{4}\|D^2_z g w^\lambda\|^2_{L^2}+\alpha(1-\eta)^{2}\|g w^\lambda\|^2_{L^2}\right).
	\end{align}
	In view of \eqref{est-sin-2}, we then deduce from a similar way as \eqref{5in15-1} that
	\begin{align}\label{5in29}
		|K_{24}|
		\lesssim \frac{1-\eta}{\beta^2}\left(\alpha(1-\eta)^{4}\|D^2_z g w^\lambda\|^2_{L^2}+\|g\|^2_{\mathcal{H}^{-1}}\right).
	\end{align}
	We employ \eqref{est-sin-2} again, then use a similar derivation as \eqref{5in22}, to find that
	\begin{align}\label{5in30}
		|K_{25}|
		\lesssim\frac{1-\eta}{\beta^2}(\alpha(1-\eta)^{4}\|D^2_z g w^\eta\|^2_{L^2}+(1-\eta)^{2}\|D_z g w^\eta\|^2_{L^2}+\|g\|^2_{\mathcal{H}^{-1}}).
	\end{align}
	We substitute \eqref{5in27}-\eqref{5in30} into \eqref{est-K2-1}, to conclude that
	\begin{align}\label{5in31}
		K_2
		&\geq \Big(1-\frac{\beta}{2}-C\frac{1-\eta}{\beta^2}\Big) \alpha(1-\eta)^4 \|D^2_z gw^\lambda\|^2_{L^2} -C\frac{1-\eta}{\beta^2}(\|g\|^2_{\mathcal{H}^{-1}}+\alpha(1-\eta)^{2}\|g w^\lambda\|^2_{L^2}+\|g\|^2_{\mathcal{H}^{2}_\eta}).
	\end{align}

	Inserting \eqref{5in26} in \textbf{Step 3.1} and \eqref{5in31} in \textbf{Step 3.2} into \eqref{est-K-1}, choosing $1-\eta$ sufficiently small relative to $\frac{1}{\beta}$, one can obtain \eqref{est-E2eta} directly.
\end{proof}

\subsection{Parameter stability of main operator}
With the coercivity of the main operator $\mathcal{L}^{\beta}_{\Gamma}$ shown in Proposition \ref{5co1} at hand, our goal in this subsection is to establish the coercivity of the operator $\mathcal{L}_{\Gamma}$. 
We note that the parameter stability plays a crucial role in this analysis.
To begin with, we introduce the following Hardy-type inequality.
\begin{lemm}\label{5ha1}
	If $L^{-1}_{z,K}(g)(0)=0$, then there holds 
	$$\|L^{-1}_{z,K}(g)w_z\|_{L^2_z}\leq C\|gw^K\|_{L^2}.$$ 
\end{lemm}
\begin{proof}
	Notice that
	$$w_z^2=\frac{(1+z^{\frac{1}{\beta}})^4}{z^{\frac{4}{\beta}}}\approx 1+z^{-\frac{4}{\beta}},$$
	hence, for any $k\neq 1$, we have
	\begin{align*}
		\int_0^{\infty} z^{-k}\left(L^{-1}_{z,K}(g)\right)^2dz&\lesssim\frac{1}{|k-1|}\left|\int_0^{\infty} \partial_z z^{1-k}\left(L^{-1}_{z,K}(g)\right)^2dz\right| \\
		&\lesssim\frac{1}{|k-1|}\left|\int_0^{\infty} z^{-k}L^{-1}_{z,K}(g)\int_0^{\frac{\pi}{2}}g(z,\theta)K(\theta)d\theta dz\right|\\
		&\lesssim\frac{1}{|k-1|}\left(\int_0^{\infty} z^{-k}\left(L^{-1}_{z,K}(g)\right)^2dz\right)^{\frac12}\left(\int_0^{\infty}\int_0^{\frac{\pi}{2}} z^{-k}g^2(w_\theta^K)^2 d\theta dz\right)^{\frac{1}{2}},
	\end{align*}
	which implies that
	\begin{align*}
		\int_0^{\infty} z^{-k}\left(L^{-1}_{z,K}(g)\right)^2dz\lesssim\frac{1}{|k-1|^2}\int_0^{\infty}\int_0^{\frac{\pi}{2}} z^{-k}g^2(w_\theta^K)^2 d\theta dz.
	\end{align*}
	By choosing $k=0$ and $k=\frac{4}{\beta}$, we thus complete the proof of Lemma \ref{5ha1}.
\end{proof}

Building upon Lemma \ref{5ha1}, we now proceed to establish the coercivity of the operator $\mathcal{L}_{\Gamma}$.
\begin{prop}\label{5co2}
	Let $L^{-1}_{z,K}(g)(0)=0$. There exist constants $\alpha>0$ sufficiently small and $\eta(\beta)$, such that if $\alpha\ll 1-\eta\ll \beta$ and $|\mu|\leq \alpha^\frac{1}{2}$, then
	\begin{align}
		&\langle\mathcal{L}_{\Gamma}(g), g(w^K)^2\rangle+\langle\mathcal{L}_{\Gamma}(g), g\rangle_{\mathcal{H}_{\eta}^2}\geq \frac{1}{4}(\|g\|^2_{\mathcal{H}^{-1}}+\|g\|^2_{\mathcal{H}_{\eta}^2}),
		\label{oper-L2-H2} \\
		&\langle\mathcal{L}_{\Gamma}(g), g\rangle_{\mathcal{E}_{\eta}^2}\geq \frac{1}{4}\|g\|^2_{\mathcal{E}_{\eta}^2}-\frac{1}{50}(\|g\|^2_{\mathcal{H}^{-1}}+\|g\|^2_{\mathcal{H}_{\eta}^2}).
		\label{oper-E2}
	\end{align}
\end{prop}
\begin{proof}
	Recalling \eqref{4eq13} and employing $|\mu|\leq \alpha^\frac{1}{2}$, for any $\delta_0>0$, we can choose $\alpha$ small enough, such that 
    \begin{align}\label{est-beta-gamma}
		\left|\frac{1}{\beta}-\frac{1}{\gamma}\right|\leq\frac{2\alpha^\frac{1}{2}}{\beta}\leq \min\left\{\frac{\ln(1-\delta_0)}{2\ln(\delta_0)},\delta_0\right\}.
	\end{align}
	By Proposition \ref{3FM2}, we can obtain that \eqref{continu-F*} and \eqref{3f9} hold for parameters $\beta$ and $\gamma$.
	Specifically, we have
	\begin{align}
		\label{continu-F*-1}
		\frac{3}{2\alpha}\|F^{\ast}_{\beta}-F^{\ast}_{\gamma}\|_{L^\infty}\lesssim \frac{\delta_0}{\beta},~~~~
		\|L^{-1}_z(\Gamma^\ast_{\beta})-L^{-1}_z(\Gamma^\ast_{\gamma})\|_{L^\infty_z}\lesssim \delta_0.
	\end{align}
	Then this proof will be divided into the following two steps.

	\medskip

	\noindent\textbf{Step 1: Estimate for $\mathcal{H}^{-1}\cap\mathcal{H}_{\eta}^2$-norm.}
	
	\medskip
	
	\noindent
	Recalling \eqref{def-oper} and the definition of $\mathcal{H}_{\eta}^2$, we find that
	\begin{align}
		&\langle\mathcal{L}_{\Gamma}(g), g(w^K)^2\rangle+\langle\mathcal{L}_{\Gamma}(g), g\rangle_{\mathcal{H}_{\eta}^2}
		\label{est-L-1}\\
		=&
		\left(\langle\mathcal{L}_{\Gamma}^{\beta}(g), g(w^K)^2\rangle
		+\langle\mathcal{L}_{\Gamma}^{\beta}(g), g\rangle_{\mathcal{H}_{\eta}^2}\right)
		+\langle\tilde{P}(g), g(w^K)^2\rangle
		+(1-\eta)^2\langle\tilde{P}(g), g(w^\eta)^2\rangle
		\notag\\
		&
		+\left[(1-\eta)^{-2}\langle D_{\theta}\tilde{P}(g), D_{\theta}g(w^\eta)^2\rangle
		+(1-\eta)^2\langle D_{z}\tilde{P}(g), D_zg(w^\eta)^2\rangle\right]
		\notag\\
		&
		+\left[\langle D_{\theta}^2\tilde{P}(g), D_{\theta}^2g(w^\eta)^2\rangle
		+(1-\eta)^2\langle D_{\theta}D_{z}\tilde{P}(g), D_{\theta}D_zg(w^\eta)^2\rangle
		+(1-\eta)^4\langle D_{z}^2\tilde{P}(g), D_{z}^2g(w^\eta)^2\rangle\right]
        \notag\\
		\tri& \sum_{i=1}^{5}L_i.
		\notag
	\end{align}
	Here $L_1$ is the $\mathcal{H}^{-1}\cap\mathcal{H}_{\eta}^2$ norm of the main operator, while $L_2$, $L_3$, $L_4$ and $L_5$ are the $\mathcal{H}^{-1}$, $\mathcal{L}_{\eta}^2$, $\dot{\mathcal{H}}_{\eta}^1$ and $\dot{\mathcal{H}}_{\eta}^2$ norms of the perturbation term, respectively. 
	
	\medskip

	\noindent\textbf{Step 1.1: Estimate for $\mathcal{H}^{-1}\cap\mathcal{H}_{\eta}^2$-norm of the main operator.}
	
	\medskip
	
	\noindent 
	According to \eqref{est-H-1} and \eqref{est-H2eta} in Proposition \ref{5co1}, we compute that
	\begin{align}\label{5p1}
		L_1\geq \frac{1}{3}(\|g\|^2_{\mathcal{H}^{-1}}+\|g\|^2_{\mathcal{H}_{\eta}^2}).
	\end{align}

	\noindent\textbf{Step 1.2: Estimate for $\mathcal{H}^{-1}$-norm of the perturbation term.}
	
	\medskip
	
	\noindent
	In view of \eqref{def-wP}, \eqref{continu-F*-1} and Lemma \ref{5ha1}, we find that
	\begin{align}\label{5p2}
		|L_2|
		&\lesssim \|gw^K\|_{L^2}
		\left(\|L^{-1}_z(\Gamma^\ast_{\beta})-L^{-1}_z(\Gamma^\ast_{\gamma})\|_{L^\infty_z}\|gw^K\|_{L^2}+\frac{3}{2\alpha}\|F^{\ast}_{\beta}-F^{\ast}_{\gamma}\|_{L^\infty}\|L^{-1}_{z,K}(g)w^K\|_{L^2}\right)
		\\
		&\lesssim \delta_0\left(1+
		\frac{1}{\beta}\right)\|gw^K\|_{L^2}^2
		\notag
	\end{align}

	\noindent\textbf{Step 1.3: Estimate for $\mathcal{L}^{2}_{\eta}$-norm of the perturbation term.}
	
	\medskip
	
	\noindent 
	We infer from \eqref{def-wP}, \eqref{est-sin-1}, \eqref{continu-F*-1} and Lemma \ref{5ha1} that
	\begin{align}\label{5p3}
		|L_3|
		&\lesssim\frac{\delta_0}{\beta}(1-\eta)^2\|gw^\eta\|_{L^2}\left(\|gw^\eta\|_{L^2}+\|L^{-1}_{z,K}(g)w^\eta\|_{L^2}\right)
		\\ \notag
		&\lesssim \frac{\delta_0}{\beta}(1-\eta)^2\|gw^\eta\|^2_{L^2}
		+\frac{\delta_0}{\beta}(1-\eta)\|g\|^2_{\mathcal{H}^{-1}}.
	\end{align}

	\noindent\textbf{Step 1.4: Estimate for $\dot{\mathcal{H}}^{1}_{\eta}$-norm of the perturbation term.}
	
	\medskip
	\noindent
	After some direct calculations, one obtains that
	$$
	D_{\theta}\tilde{P}(g)=
	L^{-1}_z(\Gamma^{\ast}_{\beta}-\Gamma^{\ast}_{\gamma})D_\theta g+\frac{3}{2\alpha}L^{-1}_{z,K}(g)D_\theta(F^{\ast}_{\beta}-F^{\ast}_{\gamma}),
	$$
	and
	$$D_z\tilde{P}(g)=
	L^{-1}_z(\Gamma^{\ast}_{\beta}-\Gamma^{\ast}_{\gamma})D_z g-(\Gamma^{\ast}_{\beta}-\Gamma^{\ast}_{\gamma})g+\frac{3}{2\alpha}L^{-1}_{z,K}(g)D_z(F^{\ast}_{\beta}-F^{\ast}_{\gamma})-\frac{3}{2\alpha}\langle g,K\rangle_\theta(F^{\ast}_{\beta}-F^{\ast}_{\gamma}).$$
	These give that
	\begin{align}
		L_4&=
		(1-\eta)^{-2}\left[\langle L^{-1}_z(\Gamma^{\ast}_{\beta}-\Gamma^{\ast}_{\gamma})D_\theta g+\frac{3}{2\alpha}L^{-1}_{z,K}(g)D_\theta(F^{\ast}_{\beta}-F^{\ast}_{\gamma}), D_\theta g(w^\lambda)^2\rangle\right]
		\label{est-L4-1}\\
		&~~~+(1-\eta)^{2}\left[\langle L^{-1}_z(\Gamma^{\ast}_{\beta}-\Gamma^{\ast}_{\gamma})D_z g-(\Gamma^{\ast}_{\beta}-\Gamma^{\ast}_{\gamma})g, D_z g(w^\eta)^2\rangle
		\right]
		\notag\\
		&~~~+(1-\eta)^{2}\left[\langle \frac{3}{2\alpha}L^{-1}_{z,K}(g)D_z(F^{\ast}_{\beta}-F^{\ast}_{\gamma})-\frac{3}{2\alpha}\langle g,K\rangle_\theta(F^{\ast}_{\beta}-F^{\ast}_{\gamma}), D_z g(w^\eta)^2\rangle
		\right]
		\notag\\
		&\tri\sum_{i=1}^{3}L_{4i}.
		\notag
	\end{align}
	Next, $L_{4i}$ $(i=1,2,3)$ will be controlled term by term.
	In order to control the first term $L_{41}$, we apply \eqref{continu-F*-1}, Proposition \ref{3FM2} and Lemma \ref{5ha1}, to discover that
	\begin{align}\label{5p4}
		|L_{41}|
		&\lesssim(1-\eta)^{-2}\|D_\theta gw^\lambda\|_{L^2}\left(\|L^{-1}_z
		(\Gamma^\ast_{\beta})-L^{-1}_z(\Gamma^\ast_{\gamma})\|_{L^\infty_z}\|D_\theta gw^\lambda\|_{L^2}
		\right.\\
		\notag
		&~~~\left.
		+\left(\|\frac{3}{2\alpha}D_\theta F^{\ast}_{\beta}w^\lambda_{\theta}\|_{L^\infty}+\|\frac{3}{2\alpha}D_\theta F^{\ast}_{\gamma}w^\lambda_{\theta}\|_{L^\infty}\right)
		\|L^{-1}_{z,K}(g)w_z\|_{L^2_z}\right) \\ \notag 
		&\lesssim\delta_0(1-\eta)^{-2}\|D_\theta gw^\lambda\|^2_{L^2}+\sqrt{\frac{\alpha}{\beta^3}}(1-\eta)^{-2}\|D_\theta gw^\lambda\|_{L^2}\|g w^{K}\|_{L^2} \\ \notag
		&\lesssim\left(\delta_0+\sqrt{\frac{\alpha}{\beta^3}}\right)(1-\eta)^{-2}\|D_\theta gw^\lambda\|^2_{L^2}+\sqrt{\frac{\alpha}{\beta^3}}(1-\eta)^{-2}\|g\|^2_{\mathcal{H}^{-1}}.
	\end{align}
	Taking advantage of \eqref{continu-F*-1}, Proposition \ref{3FM2} and Lemma \ref{5ha1} again, we achieve that
	\begin{align}\label{5p5}
		|L_{42}| 
		\lesssim\frac{\delta_0}{\beta}(1-\eta)^{2}\|D_z gw^\eta\|^2_{L^2}+\frac{\delta_0}{\beta}(1-\eta)^{2}\| gw^\eta\|^2_{L^2}.
	\end{align}
	With the help of \eqref{3eq2}, \eqref{est-beta-gamma}-\eqref{continu-F*-1}, one gets that
	\begin{align*}
		\frac{3}{2\alpha}\|D_z F^{\ast}_{\beta}-D_z F^{\ast}_{\gamma}\|_{L^\infty}
		&\lesssim\frac{\delta_0}{\beta^2},
	\end{align*} 
	which, together with \eqref{est-sin-1}, \eqref{continu-F*-1} and Lemma \ref{5ha1}, implies that
	\begin{align}\label{5p6}
		|L_{43}| 
		&\lesssim(1-\eta)^{\frac32}\|D_z gw^\eta\|_{L^2}
		\left(\frac{\delta_0}{\beta^2}\|L^{-1}_{z,K}(g)w_z\|_{L^2_z}+\frac{\delta_0}{\beta}\|gw^K\|_{L^2}\right) 
		\\ \notag
		&\lesssim\frac{\delta_0}{\beta^2}(1-\eta)^{2}\|D_z g w^\eta\|^2_{L^2}+\frac{\delta_0}{\beta^2}(1-\eta)\|g\|^2_{\mathcal{H}^{-1}}.
	\end{align}
	Substituting \eqref{5p4}-\eqref{5p6} into \eqref{est-L4-1}, we end up with
	\begin{align}
		|L_4|&\lesssim 
		\left(\delta_0+\sqrt{\frac{\alpha}{\beta^3}}\right)(1-\eta)^{-2}\|D_\theta gw^\lambda\|^2_{L^2}
		+\frac{\delta_0}{\beta^2}(1-\eta)^{2}\|D_z g w^\eta\|^2_{L^2}
		\label{est-L4-2}\\
		&~~~
		+\frac{\delta_0}{\beta}(1-\eta)^{2}\| gw^\eta\|^2_{L^2}
		+\left(\sqrt{\frac{\alpha}{\beta^3}}(1-\eta)^{-2}
		+\frac{\delta_0}{\beta^2}(1-\eta)\right)\|g\|^2_{\mathcal{H}^{-1}}.
		\notag
	\end{align}

	\noindent\textbf{Step 1.5: Estimate for $\dot{\mathcal{H}}^{2}_{\eta}$-norm of the perturbation term.}
	
	\medskip
	
	\noindent
	Straightforward computations immediately show that
	\begin{align}
		D^2_\theta \tilde{P}(g)&=
		L^{-1}_z(\Gamma^{\ast}_{\beta}-\Gamma^{\ast}_{\gamma})D^2_\theta g+\frac{3}{2\alpha}L^{-1}_{z,K}(g)D^2_\theta(F^{\ast}_{\beta}-F^{\ast}_{\gamma}),
		\notag\\
		D_\theta D_z\tilde{P}(g)&=
		L^{-1}_z(\Gamma^{\ast}_{\beta}-\Gamma^{\ast}_{\gamma})D_\theta D_z g
		-(\Gamma^{\ast}_{\beta}-\Gamma^{\ast}_{\gamma})D_\theta g 
		+\frac{3}{2\alpha}L^{-1}_{z,K}(g)D_\theta D_z(F^{\ast}_{\beta}-F^{\ast}_{\gamma})
		\notag\\ 
		&~~~-\frac{3}{2\alpha}\langle g,K\rangle_\theta D_\theta(F^{\ast}_{\beta}-F^{\ast}_{\gamma}),
		\notag\\
		D^2_z\tilde{P}(g) 
		&=L^{-1}_z(\Gamma^{\ast}_{\beta}-\Gamma^{\ast}_{\gamma})D^2_z g-2(\Gamma^{\ast}_{\beta}-\Gamma^{\ast}_{\gamma})D_z g-D_z(\Gamma^{\ast}_{\beta}-\Gamma^{\ast}_{\gamma})g+\frac{3}{2\alpha}L^{-1}_{z,K}(g)D^2_z(F^{\ast}_{\beta}-F^{\ast}_{\gamma}) 
		\label{eq-Dz-wP}\\
		&~~~-\frac{3}{\alpha}\langle g,K\rangle_\theta D_z(F^{\ast}_{\beta}-F^{\ast}_{\gamma}))-\frac{3}{2\alpha}\langle D_zg,K\rangle_\theta(F^{\ast}_{\beta}-F^{\ast}_{\gamma}),
		\notag
	\end{align}
	which implies that
	\begin{align}
		L_5=&
		\left[\langle L^{-1}_z(\Gamma^{\ast}_{\beta}-\Gamma^{\ast}_{\gamma})D^2_\theta g+\frac{3}{2\alpha}L^{-1}_{z,K}(g)D^2_\theta(F^{\ast}_{\beta}-F^{\ast}_{\gamma}), D^2_\theta g(w^\lambda)^2\rangle\right]
		\label{est-L5-1}\\
		&+(1-\eta)^{2}\left[
		\langle L^{-1}_z(\Gamma^{\ast}_{\beta}-\Gamma^{\ast}_{\gamma})D_\theta D_z g-(\Gamma^{\ast}_{\beta}-\Gamma^{\ast}_{\gamma})D_\theta g, D_\theta D_z g(w^\lambda)^2\rangle\right]
		\notag\\
		&+(1-\eta)^{2}\left[\langle \frac{3}{2\alpha}L^{-1}_{z,K}(g)D_\theta D_z(F^{\ast}_{\beta}-F^{\ast}_{\gamma})-\frac{3}{2\alpha}\langle g,K\rangle_\theta D_\theta(F^{\ast}_{\beta}-F^{\ast}_{\gamma}), D_\theta D_z g(w^\lambda)^2\rangle\right]
		\notag\\
		&+(1-\eta)^{4}\left[\langle L^{-1}_z(\Gamma^{\ast}_{\beta}-\Gamma^{\ast}_{\gamma})D^2_z g-2(\Gamma^{\ast}_{\beta}-\Gamma^{\ast}_{\gamma})D_z g-D_z(\Gamma^{\ast}_{\beta}-\Gamma^{\ast}_{\gamma})g, D_z^2 g(w^\eta)^2\rangle\right]
		\notag\\
		&+(1-\eta)^{4}\left[\langle \frac{3}{2\alpha}L^{-1}_{z,K}(g)D^2_z(F^{\ast}_{\beta}-F^{\ast}_{\gamma})-\frac{3}{\alpha}\langle g,K\rangle_\theta D_z(F^{\ast}_{\beta}-F^{\ast}_{\gamma})), D^2_z g(w^\eta)^2\rangle\right] 
		\notag\\
		&-(1-\eta)^{4}\langle \frac{3}{2\alpha}\langle D_zg,K\rangle_\theta(F^{\ast}_{\beta}-F^{\ast}_{\gamma}), D^2_z g(w^\eta)^2\rangle
		\notag\\
		\tri& \sum_{i=1}^{6}L_{5i}.
		\notag
	\end{align}
	Now, we deal with $L_{5i}$ $(i=1,\cdots,6)$ in turn.
	Bounding the first term $L_{51}$ by a similar way as \eqref{5p4}, we thus get that
	\begin{align}\label{5p7}
		|L_{51}| 
		\lesssim \left(\delta_0+\sqrt{\frac{\alpha}{\beta^3}}\right)\|D^2_\theta gw^\lambda\|^2_{L^2}+\sqrt{\frac{\alpha}{\beta^3}}\|g\|^2_{\mathcal{H}^{-1}}.
	\end{align}
	For the term $L_{52}$, we apply an analogous approach to  \eqref{5p5} and deduce that
	\begin{align}\label{5p8}
		|L_{52}|
		&\lesssim\frac{\delta_0}{\beta}(1-\eta)^{2}\|D_\theta D_z gw^\lambda\|^2_{L^2}+\frac{\delta_0}{\beta}(1-\eta)^{-2}\| D_\theta gw^\lambda\|^2_{L^2}.
	\end{align}
	We employ \eqref{est-sin-2}, Lemma \ref{5ha1} and Proposition \ref{3FM2}, to find that
	\begin{align}\label{5p9}
		|L_{53}| 
		&\lesssim(1-\eta)^{2}\|D_{\theta}D_z gw^\lambda\|_{L^2}
		\left(\sqrt{\frac{\alpha}{\beta^5}}\|L^{-1}_{z,K}(g)w_z\|_{L^2_z}
		+\sqrt{\frac{\alpha}{\beta^3}}\|g w^{K}\|_{L^2}
		\right) 
		\\ \notag
		&\lesssim\sqrt{\frac{\alpha}{\beta^5}}\left((1-\eta)^{2}\|D_\theta D_z g w^\lambda\|^2_{L^2}+\|g\|^2_{\mathcal{H}^{-1}}\right).
	\end{align}
	Along the similar line as \eqref{5p5}, it is enough to prove that
	\begin{align}\label{5p10}
		|L_{54}|
		&\lesssim\frac{\delta_0}{\beta^2}\left((1-\eta)^{4}\|D^2_z gw^\eta\|^2_{L^2}+(1-\eta)^{2}\|D_z gw^\eta\|^2_{L^2}+(1-\eta)^{2}\| gw^\eta\|^2_{L^2}\right).
	\end{align}
	The combination of \eqref{3eq2}, \eqref{est-beta-gamma}-\eqref{continu-F*-1} gives that
	\begin{align*}
		\frac{3}{2\alpha}\|D^2_z F^{\ast}_{\beta}-D^2_z F^{\ast}_{\gamma}\|_{L^\infty}
		\lesssim\frac{\delta_0}{\beta^3}.
		\notag
	\end{align*}
	This, following a derivation analogous to \eqref{5p6}, yields
	\begin{align}\label{5p11}
		|L_{55}|
		&\lesssim\frac{\delta_0}{\beta^3}\left((1-\eta)^{4}\|D^2_z g w^\eta\|^2_{L^2}+(1-\eta)^{3}\|g\|^2_{\mathcal{H}^{-1}}\right).
	\end{align}
	It then follows from a similar way as \eqref{5p6} that
	\begin{align}\label{5p12}
		|L_{56}|
		\lesssim\frac{\delta_0}{\beta}(1-\eta)^{\frac72}\|D^2_z gw^\eta\|_{L^2}\|D_z gw^\eta\|_{L^2} 
		\lesssim \frac{\delta_0}{\beta}\left((1-\eta)^{4}\|D^2_z g w^\eta\|^2_{L^2}+(1-\eta)^{3}\|D_z g w^\eta\|^2_{L^2}\right).
	\end{align}
	We then substitute \eqref{5p7}-\eqref{5p10}, \eqref{5p11} and \eqref{5p12} into \eqref{est-L5-1}, to find that
	\begin{align}
		|L_5|
		\lesssim 
		\left(\frac{\delta_0}{\beta}+\sqrt{\frac{\alpha}{\beta^5}}\right)\|g\|^2_{\mathcal{H}^2_{\eta}}
		+\left(\frac{\delta_0}{\beta^3}+\sqrt{\frac{\alpha}{\beta^5}}\right)\|g\|^2_{\mathcal{H}^{-1}}.
		\label{est-L5-2}
	\end{align}
	
	Inserting \eqref{5p1} in \textbf{Step 1.1}, \eqref{5p2} in \textbf{Step 1.2}, \eqref{5p3} in \textbf{Step 1.3}, \eqref{est-L4-2} in \textbf{Step 1.4} and \eqref{est-L5-2} in \textbf{Step 1.5} into \eqref{est-L-1}, choosing the suitable transitional constant $\eta$, $\delta_0$, and taking $\alpha$ small enough, we can conclude that \eqref{oper-L2-H2} holds.

	\medskip

	\noindent\textbf{Step 2: Estimate for $\mathcal{E}_{\eta}^2$-norm.}
	
	\medskip
	
	\noindent
	Recalling the definition of $\mathcal{E}^{2}_{\eta}$, we find that
	\begin{align}\label{est-M-1}
		\langle\mathcal{L}_{\Gamma}(g), g\rangle_{\mathcal{E}_{\eta}^2}
		&=\langle\mathcal{L}^{\beta}_{\Gamma}(g), g\rangle_{\mathcal{E}_{\eta}^2}
		+
		\alpha(1-\eta)^2\langle\tilde{P}(g), g(w^\lambda)^2\rangle
		+\alpha(1-\eta)^4\langle D^2_z\tilde{P}(g), D^2_z g(w^\lambda)^2\rangle
		\\
		&\tri M_1+M_2+M_3,
		\notag
	\end{align}
	where $M_1$ is the $\mathcal{E}^{2}_{\eta}$ norm of the main operator, while $M_2$ and $M_3$ are the $\mathcal{L}^{2}_{\eta}$ and the radial $\dot{\mathcal{H}}^{2}_{\eta}$ norms with the weight $w^\lambda$ of the perturbation term, respectively. 
	\medskip

	\noindent\textbf{Step 2.1: Estimate for $\mathcal{E}_{\eta}^2$-norm of the main operator.}
	
	\medskip
	
	\noindent
	By \eqref{est-E2eta}, it is easy to check that
	\begin{align}\label{est-M1}
		M_1
		\geq \frac{1}{3}\|g\|^2_{\mathcal{E}_{\eta}^2}-\frac{1}{100}\|g\|^2_{\mathcal{H}^{-1}}-\frac{1}{100}\|g\|^2_{\mathcal{H}_{\eta}^2}.
	\end{align}

	\noindent\textbf{Step 2.2: Estimate for $\mathcal{L}^{2}_{\eta}$-norm with the weight $w^\lambda$ of the perturbation term.} 
	
	\medskip
	
	\noindent 
	We infer from \eqref{continu-F*-1} and Lemma \ref{5ha1} that
	\begin{align}\label{5p13}
		|M_2|
		&\lesssim\frac{\delta_0}{\beta}\alpha(1-\eta)^2\|gw^\lambda\|_{L^2}\left(\|gw^\lambda\|_{L^2}+\|L^{-1}_{z,K}(g)[K(\theta)]^{\frac{\alpha}{10\beta}}w^\lambda\|_{L^2}\right)\\ \notag
		&\lesssim\frac{\delta_0}{\beta}\alpha(1-\eta)^2\|gw^\lambda\|_{L^2}\left(\|gw^\lambda\|_{L^2}+\sqrt{\frac{\beta}{\alpha}}\|L^{-1}_{z,K}(g)w_z\|_{L^2_z}\right)\\ \notag
		&\lesssim \frac{\delta_0}{\beta}\alpha(1-\eta)^2\|gw^\lambda\|^2_{L^2}+\delta_0(1-\eta)^2\|g\|^2_{\mathcal{H}^{-1}}.
	\end{align}

	\noindent\textbf{Step 2.3: Estimate for the radial $\dot{\mathcal{H}}^{2}_{\eta}$-norm with the weight $w^\lambda$ of the perturbation term.} 
	
	\medskip
	
	\noindent 
	Recalling \eqref{eq-Dz-wP}, $M_3$ can be writen as follows:
	\begin{align}
		M_3&=
		\alpha(1-\eta)^{4}\left[\langle L^{-1}_z(\Gamma^{\ast}_{\beta}-\Gamma^{\ast}_{\gamma})D^2_z g-2(\Gamma^{\ast}_{\beta}-\Gamma^{\ast}_{\gamma})D_z g-D_z(\Gamma^{\ast}_{\beta}-\Gamma^{\ast}_{\gamma})g, D^2_z g(w^\lambda)^2\rangle\right]
		\label{est-M3-1}\\
		&~~~+\alpha(1-\eta)^{4}\left[\langle \frac{3}{2\alpha}L^{-1}_{z,K}(g)D^2_z(F^{\ast}_{\beta}-F^{\ast}_{\gamma})-\frac{3}{\alpha}\langle g,K\rangle_\theta D_z(F^{\ast}_{\beta}-F^{\ast}_{\gamma})), D^2_z g(w^\lambda)^2\rangle\right]
		\notag\\
		&~~~-\alpha(1-\eta)^{4}\langle \frac{3}{2\alpha}\langle D_zg,K\rangle_\theta(F^{\ast}_{\beta}-F^{\ast}_{\gamma}), D^2_z g(w^\lambda)^2\rangle
		\notag\\
		&\tri M_{31}+M_{32}+M_{33}.
		\notag
	\end{align}
	Together with \eqref{est-sin-2}, it follows from a similar way as \eqref{5p10} that
	\begin{align}\label{5p14}
		|M_{31}|
		&\lesssim \frac{\delta_0}{\beta^2}\alpha(1-\eta)^{4}\|D^2_z gw^\lambda\|_{L^2}\left(\|D^2_z gw^\lambda\|_{L^2}
		+\|w_\theta^\lambda K^{\frac{\alpha}{10\beta}}\|_{L^2_\theta}\|D_z gw^\eta\|_{L^2}
		+\|gw^\lambda\|_{L^2}\right) \\ \notag
		&\lesssim \frac{\delta_0}{\beta^2}\left(\alpha(1-\eta)^{4}\|D^2_z gw^\lambda\|^2_{L^2}+\alpha(1-\eta)^{2}\| gw^\lambda\|^2_{L^2}\right)
		+\frac{\delta_0}{\beta}(1-\eta)^{2}\|D_z gw^\eta\|_{L^2}^2.
	\end{align}
	Along a similar way as \eqref{5p11}, we deduce from \eqref{est-sin-2} that 
	\begin{align}\label{5p15}
		|M_{32}|
		&\lesssim \frac{\delta_0}{\beta^3}\left(\alpha(1-\eta)^{4}\|D^2_z g w^\lambda\|^2_{L^2}+\|g\|^2_{\mathcal{H}^{-1}}\right).
	\end{align}
	Repeating the derivation scheme of \eqref{5p12} with \eqref{est-sin-2} gives
	\begin{align}\label{5p16}
		|M_{33}|
		&\lesssim \frac{\delta_0}{\beta}\left(\alpha(1-\eta)^{4}\|D^2_z g w^\lambda\|^2_{L^2}+(1-\eta)^{2}\|D_z g w^\eta\|^2_{L^2}\right).
	\end{align}
	Collecting the estimates from \eqref{est-M3-1} through \eqref{5p16}, we obtain that
	\begin{align}\label{est-M3-2}
		|M_3|
		\lesssim 
		\frac{\delta_0}{\beta^3}\|g\|_{\mathcal{E}^2_{\eta}}^2
		+\frac{\delta_0}{\beta}\|g\|^2_{\mathcal{H}_{\eta}^2}
		+\frac{\delta_0}{\beta^3}\|g\|^2_{\mathcal{H}^{-1}}.
	\end{align}

	By substituting \eqref{est-M1} in \textbf{Step 2.1}, \eqref{5p13} in \textbf{Step 2.2} and \eqref{est-M3-2} in \textbf{Step 2.3} into \eqref{est-M-1}, then choosing the suitable transitional constant $\eta$, $\delta_0$, finally taking $\alpha$ small enough, we directly derive \eqref{oper-E2}.
	Therefore, we finish the proof of this proposition.
\end{proof}

\section{Elliptic estimates}\label{sec:Elli}
In this section, we aim to establish elliptic estimates for the following boundary value problem across different Banach spaces:
\begin{align}\label{6model}
	\left\{\begin{array}{l} 
		L^{\alpha}_{z}(\phi)+L_{\theta}(\phi)=f, \\[1ex]
		\phi|_{\partial D}=0,~~~\langle f,K\rangle_{\theta}=0,
	\end{array}\right.
\end{align}
where the domain $D\tri[0,\infty)\times[0,\frac{\pi}{2}]$, the radial operator $L^{\alpha}_{z}(\phi)$ is defined by
\begin{align}\label{def-L-z}
	L^{\alpha}_{z}(\phi)=-\alpha^2 D^2_z\phi-5\alpha D_z\phi,
\end{align}
and the angular operator $L_{\theta}(\phi)$ is given by 
\begin{align}\label{def-L-theta}
	L_{\theta}(\phi)=-\partial_\theta^2 \phi+\partial_{\theta}\left(\tan\theta \phi\right)-6\phi.
\end{align}
Let $(L^{\alpha}_{z})^\ast$ and $L^{*}_{\theta}$ denote the adjoint operators of $L^{\alpha}_{z}$ and $L_{\theta}$, respectively.
Moreover, for $\xi\in[0,\infty)$, we define the weight function:
\begin{align*}
	\mathcal{W} \tri \mathcal{W}_z(z)\sin^{-\frac{\xi}{2}}(2\theta),
\end{align*}
where $\mathcal{W}_z$ denotes either $w_z$ or $w_z^\ast$. 
For any $n\in\mathbb{N}$, we denote that
$$\tilde{D}^n_\theta f \tri \sin^n(2\theta)\partial^n_{\theta} f.$$
It is easy to check that $\tilde{D}^n_\theta=D^n_\theta$ for $n=0,1$.
We also define the weighted $\mathcal{L}^2_{\mathcal{W}}(D)$-norm as follows:
\begin{align*}
	\|f\|_{\mathcal{L}^2_{\mathcal{W}}}\tri \|f\mathcal{W}\|_{L^2}.
\end{align*}

\subsection{Basic inequality}
Our goal in this subsection is to present some fundamental inequalities and lemmas, which is useful for establishing the elliptic estimates.
We begin with some equalities of the trigonometric functions.
\begin{lemm}\cite{E21}\label{6lem2}
	For any $n\in\mathbb{Z}$, there holds 
	\begin{align}\label{est-lem6.1}
		\int_0^{\frac{\pi}{2}}\sin(2n\theta)K(\theta)d\theta=(-1)^{n}~\frac{4n}{(4n^2-1)(4n^2-9)}.
	\end{align}
\end{lemm}

Then, we introduce some Hardy inequalities.
\begin{lemm}\label{6lem4}
	For any $n\in\mathbb{N}^{+}$ and $\xi\in[0,\infty)$, assume that $f\in H^n([0,\frac{\pi}{2}])$ and $f(0)=f(\frac{\pi}{2})=0$,
	then there holds 
	\begin{align}\label{est-lem6.2}
		\int_0^{\frac{\pi}{2}} \frac{f^2}{\sin^{\xi+2n}(2\theta)} d\theta \leq C\prod_{j=0}^{n-1}\left(\xi+2(n-j)-1\right)^{-2}\int_0^{\frac{\pi}{2}} \frac{(\partial^n_{\theta}f)^2}{\sin^{\xi}(2\theta)} d\theta.
	\end{align}    
\end{lemm}
\begin{rema}
	While Lemma $7.3$ in \cite{E21} established the Hardy inequality \eqref{est-lem6.2} for the specific case when $n=1$ and $\xi=0$, we now present its extension to all positive integers $n\in\mathbb{N}^{+}$ and parameters $\xi\in[0,\infty)$ in Lemma \ref{6lem4}. 
\end{rema}
Lemma \ref{6lem4} can be derived from the following lemma and mathematical induction, and the proof is omitted here.
\begin{lemm}\label{6lem3}
	Assume $f\in H^1([0,\frac{\pi}{2}])$ and $f(0)=f(\frac{\pi}{2})=0$, for any $\xi\in [0,\infty)$, there holds 
	\begin{align}\label{est-lem6.3}
		\int_0^{\frac{\pi}{2}} \frac{f^2}{\sin^{\xi+2}(2\theta)} d\theta \leq \frac{1}{(\xi+1)^2}\int_{0}^{\frac{\pi}{2}} \frac{(\partial_{\theta}f)^2}{\sin^{\xi}(2\theta)} d\theta + C\|f\|^2_{H^1}.
	\end{align}    
\end{lemm}
\begin{rema}
	By comparing Lemma \ref{6lem4} and Lemma \ref{6lem3}, we see that the size of the first constant on the right-hand side of \eqref{est-lem6.3} is sharp. 
	In the following elliptic estimates in this section, we will make full use of the smallness of the first term, so that it can be absorbed by the main term.
	We note that Corollary $7.7$ in \cite{E21} established the result for the parameter range $\xi\in[0,1]$.
	However, our analysis requires an extension to $\xi\in[0, 2]$ due to the angular weight $w_{\theta}^{\lambda}$ in the norms our introduced norm $\mathcal{H}^k$, $\mathcal{H}^{k,\ast}$ and $\mathcal{E}^2$, where $\lambda>1$ by construction.
	Remarkably, we can further generalize \eqref{est-lem6.3} to hold for all $\xi \in [0, +\infty)$.
\end{rema}
\begin{proof}[The proof of Lemma \ref{6lem3}]
	 We only prove the case that $\theta\in \left[0,\frac{\pi}{4}\right]$.
	Let $\xi\geq 0$. Integration by parts leads to
	\begin{align*}
		\int_0^{\frac{\pi}{4}} \frac{f^2}{(2\theta)^{\xi+2}} d\theta  &= -\frac{1}{2(\xi+1)}\int_0^{\frac{\pi}{4}} f^2 d(2\theta)^{-( \xi+1)}\\ \notag
		&=-\frac{1}{2(\xi+1)} \frac{f^2}{(2\theta)^{\xi+1}}\bigg|^{\theta=\frac{\pi}{4}}_{\theta=0}+\frac{1}{ \xi+1}\int_0^{\frac{\pi}{4}} \frac{f\partial_{\theta}f}{(2\theta)^{\xi+1}} d\theta\\ \notag
		&=-\frac{\left(\frac{2}{\pi}\right)^{\xi+1}}{2(\xi+1)} \int_0^{\frac{\pi}{4}}2f\partial_{\theta}fd\theta+\frac{1}{ \xi+1}\int_0^{\frac{\pi}{4}} \frac{f\partial_{\theta}f}{(2\theta)^{\xi+1}} d\theta \\ \notag
		&\leq C\|f\|_{H^1}^2+\frac{1}{2}\int_0^{\frac{\pi}{4}} \frac{f^2}{(2\theta)^{\xi+2}} d\theta+\frac{1}{2}\frac{1}{ (\xi+1)^2}\int_0^{\frac{\pi}{4}} \frac{(\partial_{\theta}f)^2}{(2\theta)^{\xi}} d\theta,
	\end{align*}
	which implies that 
	\begin{align}\label{est-f-1-1}
		\int_0^{\frac{\pi}{4}} \frac{f^2}{(2\theta)^{\xi+2}} d\theta\leq \frac{1}{ (\xi+1)^2}\int_0^{\frac{\pi}{4}} \frac{(\partial_{\theta}f)^2}{(2\theta)^{\xi}} d\theta+C\|f\|_{H^1}^2.
	\end{align}
	For any $\xi\geq 0$, we have
	\begin{align*}
		\left|\frac{1}{(2\theta)^{\xi}}-\frac{1}{\sin^{\xi}(2\theta)} \right|\lesssim 1,\quad \text{for any}~~\theta\in \left[0,\frac{\pi}{4}\right].
	\end{align*}
	This, together with \eqref{est-f-1-1}, yields directly that 
	\begin{align*}
		\int_0^{\frac{\pi}{4}} \frac{f^2}{\sin^{\xi+2}(2\theta)} d\theta \leq \frac{1}{(\xi+1)^2}\int_{0}^{\frac{\pi}{4}} \frac{(\partial_{\theta}f)^2}{\sin^{\xi}(2\theta)} d\theta + C\|f\|^2_{H^1}.
	\end{align*} 
	Thus, we finish the proof of this lemma.
\end{proof}

Moreover, based on the Hardy inequalities established in Lemma $8.26$ of \cite{E21} and Lemma \ref{6lem4}, we can derive the following proposition.
For the sake of simplicity, the proof is omitted here.
\begin{prop}\label{6prop0}
	Let $n\geq 0$. There holds
	\begin{align*}
		\|\frac{f}{\sin(2\theta)}\|_{\mathcal{H}^n}\leq C\|\partial_{\theta}f\|_{\mathcal{H}^n},
		\quad
		\|\frac{f}{\sin(2\theta)}\|_{\mathcal{H}^{n,\ast}}\leq C\|\partial_{\theta}f\|_{\mathcal{H}^{n,\ast}}.
	\end{align*} 
\end{prop}

Our next goal is to investigate the equivalent norms of the operators $D_{\theta}$ and $\tilde{D}_{\theta}$. 
\begin{lemm}\label{6lem5}
	Let $n\geq 1$. There holds
	\begin{align}\label{est-lem6.5}
		\|D^n_\theta f\|^2_{\mathcal{L}^2_{\mathcal{W}}} \leq C \sum_{k=1}^{n}\|\tilde{D}^k_\theta f\|^2_{\mathcal{L}^2_{\mathcal{W}}},\quad \|\tilde{D}^n_\theta f\|^2_{\mathcal{L}^2_{\mathcal{W}}} \leq C \sum_{k=1}^{n}\|D^k_\theta f\|^2_{\mathcal{L}^2_{\mathcal{W}}}.
	\end{align}
\end{lemm}
\begin{proof}
	Recalling the definitions of $D^n_\theta $ and $\tilde{D}^n_\theta$, it is easy to check that \eqref{est-lem6.5} holds for $n=1$.
	For $n\geq2$, there exists $\big\{S_k(\theta)\big\}_{k=1}^{n-1}\subset C^{\infty}([0,\frac{\pi}{2}])$ such that
	\begin{align}\label{inden-D-tD-f}
		D^{n}_\theta f =\tilde{D}^{n}_\theta f + \sum_{k=1}^{n-1} S_k(\theta) \tilde{D}^k_\theta f,
	\end{align}
	which implies that 
	\begin{align*}
		\|D^n_\theta f\|^2_{\mathcal{L}^2_{\mathcal{W}}} \lesssim \|\tilde{D}^n_\theta f\|^2_{\mathcal{L}^2_{\mathcal{W}}}+\sum_{k=1}^{n-1}\|S_k(\theta)\tilde{D}^k_\theta f\|^2_{\mathcal{L}^2_{\mathcal{W}}}\lesssim\sum_{k=1}^{n}\|\tilde{D}^k_\theta f\|^2_{\mathcal{L}^2_{\mathcal{W}}}.
	\end{align*}
	This gives the first inequality in \eqref{est-lem6.5}.

	While for proving the second inequality in \eqref{est-lem6.5}, we use the standard induction argument.
	Suppose that the second inequality in \eqref{est-lem6.5} holds for any $0<n<l$, then for $n=l+1$, we deduce from \eqref{inden-D-tD-f} that
	\begin{align*}
		\tilde{D}^{l+1}_\theta f=D^{l+1}_\theta f - \sum_{k=1}^{l} S_k(\theta) \tilde{D}^k_\theta f.
	\end{align*}
	Combining this with the induction assumption, we obtain that
	\begin{align*}
		\|\tilde{D}^{l+1}_\theta f\|^2_{\mathcal{L}^2_{\mathcal{W}}} \lesssim \|D^{l+1}_\theta f\|^2_{\mathcal{L}^2_{\mathcal{W}}}+\sum_{k=1}^{l}\|S_k(\theta)\tilde{D}^k_\theta f\|^2_{\mathcal{L}^2_{\mathcal{W}}}\lesssim\sum_{k=1}^{l+1}\|D^k_\theta f\|^2_{\mathcal{L}^2_{\mathcal{W}}}.
	\end{align*}
	Hence, the second inequality in \eqref{est-lem6.5} can be derive by the standard induction argument.
	We thus complete the proof of this lemma.
\end{proof}

From now on, for any function $f$, we define the adjoint operators of $D_\theta$ and $D_z$ as follows:
$$D^\ast_\theta (f)\tri \partial_{\theta}(\sin(2\theta) f),\quad 
D^\ast_z (f)\tri \partial_{z}(z f).$$
The following lemma provides an estimate bounding the $\mathcal{L}^2_{\mathcal{W}}$-norm of the mixed derivatives in terms of purely radial and angular derivatives.
\begin{lemm}\label{6lem6}
	Assume that for any $k=1,2$,
	$
	\sum\limits_{l=0}^{k}\big|\frac{D^{l}_{z}D^{k-l}_{\theta}\mathcal{W}}{\mathcal{W}}\big|\leq C_{\beta}$, where $C_{\beta}>0$ is a constant.
	Then for any $i,j,n\in\mathbb{N}$ such that $i+j=n$, there holds
	\begin{align}\label{61in1}
		\|D^i_zD^j_{\theta}f\|^2_{\mathcal{L}^2_{\mathcal{W}}} \leq C_{\beta} \sum_{k=0}^{n}\big(\|D^k_zf\|^2_{\mathcal{L}^2_{\mathcal{W}}} + \|D^k_{\theta}f\|^2_{\mathcal{L}^2_{\mathcal{W}}}\big).
	\end{align}
\end{lemm}
\begin{proof}
	We will establish \eqref{61in1} by induction on $n$, the cases $n=0$ and $n=1$ follow directly.
	For $n=2$, it is easy to check that \eqref{61in1} holds for $i=0$ or $j=0$.
	For $i=j=1$, one can deduce that
	\begin{align*}
		\|D_zD_{\theta}f\|^2_{\mathcal{L}^2_{\mathcal{W}}} 
		&= -\langle D_zf,D_zD^2_{\theta}f\mathcal{W}^2\rangle-\langle D_zf,D_zD_{\theta}fD^{\ast}_{\theta}(\mathcal{W}^2)\rangle\\ \notag
		&=\langle D^2_zf,D^2_{\theta}f\mathcal{W}^2\rangle+\langle D_zf,D^2_{\theta}fD^{\ast}_z(\mathcal{W}^2)\rangle+\frac{1}{2}\langle (D_zf)^2,(D^{\ast}_{\theta})^2(\mathcal{W}^2)\rangle\\ \notag
		&\lesssim_\beta \|D^2_zf\|^2_{\mathcal{L}^2_{\mathcal{W}}} + \|D^2_{\theta}f\|^2_{\mathcal{L}^2_{\mathcal{W}}} + \|D_zf\|^2_{\mathcal{L}^2_{\mathcal{W}}},
	\end{align*}
	which implies that \eqref{61in1} holds for $n=2$. 
	Assume that the following inequality is valid for $i+j \leq n$ as:
	\begin{align}\label{61in2}
	\|D^i_zD^j_{\theta}f\|^2_{\mathcal{L}^2_{\mathcal{W}}} \lesssim_\beta \sum_{k=0}^{i+j}\big(\|D^{k}_zf\|^2_{\mathcal{L}^2_{\mathcal{W}}} + \|D^{k}_{\theta}f\|^2_{\mathcal{L}^2_{\mathcal{W}}}\big).
	\end{align}
	For $i+j=n+1$, \eqref{61in1} holds for $i=0$ or $j=0$.
	With the induction assumption \eqref{61in2}, we deduce that for any $i+j=n+1$ and $i,j\neq0$,
	\begin{align}\label{61in3}
\|D^i_zD^j_{\theta}f\|^2_{\mathcal{L}^2_{\mathcal{W}}} 
		&=\langle D^{i+1}_zD^{j-1}_{\theta}f,D^{i-1}_zD^{j+1}_{\theta}f\mathcal{W}^2\rangle+\langle D^{i}_zD^{j-1}_{\theta}f,D^{i-1}_zD^{j+1}_{\theta}fD^{\ast}_z(\mathcal{W}^2)\rangle\\ \notag
		&~~~+\frac{1}{2}\langle (D^i_zD^{j-1}_{\theta}f)^2,(D^{\ast}_{\theta})^2(\mathcal{W}^2)\rangle\\ \notag
		&\lesssim_\beta\|D^{i+1}_zD^{j-1}_{\theta}f\|^2_{\mathcal{L}^2_{\mathcal{W}}} + \|D^{i-1}_zD^{j+1}_{\theta}f\|^2_{\mathcal{L}^2_{\mathcal{W}}} + \|D^i_zD^{j-1}_{\theta}f\|^2_{\mathcal{L}^2_{\mathcal{W}}}\\ \notag
		&\lesssim_\beta\|D^{i+1}_zD^{j-1}_{\theta}f\|^2_{\mathcal{L}^2_{\mathcal{W}}} + \|D^{i-1}_zD^{j+1}_{\theta}f\|^2_{\mathcal{L}^2_{\mathcal{W}}} + \sum_{k=0}^{n}\left(\|D^k_zf\|^2_{\mathcal{L}^2_{\mathcal{W}}} + \|D^k_{\theta}f\|^2_{\mathcal{L}^2_{\mathcal{W}}}\right).
	\end{align}
	This confirms that \eqref{61in1} is valid for $i=1$ and $j=1$.
	Performing the finite iterations, we can conclude that \eqref{61in1} for any $i+j=n+1$ by similar calculations as \eqref{61in3}. 
	This completes the proof of Lemma \ref{6lem6}.
\end{proof}

Furthermore, we can give the following lemma by similar argument as the proof of Lemma \ref{6lem6}.
\begin{lemm}\label{6lem7}
	Let $\alpha\in(0,1]$. Assume that for any $k=1,2$, 
	$
	\sum\limits_{l=0}^{k}\big|\frac{D^{l}_{z}D^{k-l}_{\theta}\mathcal{W}}{\mathcal{W}}\big|\leq C_{\beta}$, 
	where $C_{\beta}>0$ is a constant. 
	Then there holds
	\begin{align*}
		\alpha^2\|\partial_{\theta}D_{z}f\|^2_{\mathcal{L}^2_{\mathcal{W}}} \leq C_\beta( \|\partial^2_{\theta}f\|^2_{\mathcal{L}^2_{\mathcal{W}}} + \alpha^4\|D^2_{z}f\|^2_{\mathcal{L}^2_{\mathcal{W}}}).
	\end{align*}
\end{lemm}
\begin{proof}
	With the help of integrating by parts and Lemma \ref{6lem4}, we get that
	\begin{align*}
		\alpha^2\|\partial_{\theta}D_{z}f\|^2_{\mathcal{L}^2_{\mathcal{W}}} 	&=\alpha^2\langle\partial^2_{\theta}f,D^2_{z}f\mathcal{W}^2\rangle+\alpha^2\langle\partial_{\theta}f,D^2_{z}f\partial_{\theta}(\mathcal{W}^2)\rangle+\frac{\alpha^2}{2}\langle(\partial_{\theta}f)^2,(D^{\ast}_{z})^2(\mathcal{W}^2)\rangle \\ \notag
		&\lesssim_\beta \big\|\partial^2_{\theta}f\big\|^2_{\mathcal{L}^2_{\mathcal{W}}} + \alpha^4\big\|D^2_{z}f\big\|^2_{\mathcal{L}^2_{\mathcal{W}}} + \big\|\frac{\partial_{\theta}f}{\sin(2\theta)}\big\|^2_{\mathcal{L}^2_{\mathcal{W}}}\\ \notag
		&\lesssim_\beta
		\big\|\partial^2_{\theta}f\big\|^2_{\mathcal{L}^2_{\mathcal{W}}} + \alpha^4\big\|D^2_{z}f\big\|^2_{\mathcal{L}^2_{\mathcal{W}}}.
	\end{align*}
	We thus complete this proof.
\end{proof}

\subsection{Estimates for radial operator}
Our goal in this subsection is to study the radial operator and establish some estimates for it.
From now on, we assume that $\xi\in[0,2]$.
We can check that the radial derivative $D_z$ commutes with the operators $L_z^{\alpha}$ and $L_{\theta}$.
We thus briefly derive the following two propositions.
\begin{prop}\label{6prop3}
	Assume that $\big|\frac{D_z\mathcal{W}}{\mathcal{W}}\big| \leq C_\beta,$ where $C_{\beta}>0$ is a constant. 
	Then there holds
	\begin{align*}
		-\langle L^{\alpha}_z (\phi),(\alpha D_{z})^2\phi\mathcal{W}^2\rangle
		\geq
		\|(\alpha D_{z})^2\phi\|^2_{\mathcal{L}^2_{\mathcal{W}}}
		-C_\beta\alpha^3\|D_{z}\phi\|^2_{\mathcal{L}^2_{\mathcal{W}}}.
	\end{align*}
\end{prop}
\begin{proof}
	Recalling \eqref{def-L-z}, we can get easily that
	\begin{align*}
		-\langle L^{\alpha}_z (\phi),(\alpha D_{z})^2\phi\mathcal{W}^2\rangle
		=
		\|(\alpha D_{z})^2\phi\|^2_{\mathcal{L}^2_{\mathcal{W}}}
		-\frac{5}{2}\alpha^3
		\langle (D_z\phi)^2,D_z^\ast(\mathcal{W}^2)\rangle
		\geq 
		\|(\alpha D_{z})^2\phi\|^2_{\mathcal{L}^2_{\mathcal{W}}}
		-C_\beta\alpha^3\|D_{z}\phi\|^2_{\mathcal{L}^2_{\mathcal{W}}}.
	\end{align*}
	We thus finish the proof of this proposition.
\end{proof}

\begin{prop}\label{6prop4}
	Assume that for any $k=1,2$,
	$
	\sum\limits_{l=0}^{k}\big|\frac{D^{l}_{z}D^{k-l}_{\theta}\mathcal{W}}{\mathcal{W}}\big|\leq C_{\beta}$, where $C_{\beta}>0$ is a constant. 
	Then there exist sufficiently small constants $\alpha>0$ and $\varepsilon>0$, such that
	\begin{align}\label{est-prop6.13-1}
		-\langle L^{\alpha}_z (\phi),\partial^2_{\theta}\phi\mathcal{W}^2\rangle
		\geq& \frac{1-\xi}{1+\xi}\alpha^2\|D_{z}\partial_{\theta}\phi\|^2_{\mathcal{L}^2_{\mathcal{W}}} -\varepsilon\|\partial^2_{\theta}\phi\|^2_{\mathcal{L}^2_{\mathcal{W}}}-C_\beta\alpha^2\|D_{z}\partial_{\theta}\phi\|^2_{\mathcal{L}^2_{\mathcal{W}_z}}
		-C_{\beta,\varepsilon}\alpha^2\|D_{z}\phi\|^2_{\mathcal{L}^2_{\mathcal{W}}}.
	\end{align}
	Moreover, for any $n\in\mathbb{N}^{+}$, there holds
	\begin{align}\label{est-prop6.13-2}
		-\langle \tilde{D}^n_{\theta}L^{\alpha}_z (\phi),\tilde{D}^n_{\theta}\partial^2_{\theta}\phi\mathcal{W}^2\rangle
		\geq \alpha^2\|D_{z}\tilde{D}^n_{\theta}\partial_{\theta}\phi\|^2_{\mathcal{L}^2_{\mathcal{W}}} 
		-\varepsilon\|\tilde{D}^{n}_{\theta}\partial^2_{\theta}\phi\|^2_{\mathcal{L}^2_{\mathcal{W}}}- C_{\beta,\varepsilon}\alpha^2\|D_{z}\tilde{D}^{n-1}_{\theta}\partial_{\theta}\phi\|^2_{\mathcal{L}^2_{\mathcal{W}}}.
	\end{align}
\end{prop}
\begin{proof}
	We recall \eqref{def-L-z}, to discover that
	\begin{align}\label{est-6-M-1}
		-\langle L^{\alpha}_z (\phi),\partial^2_{\theta}\phi\mathcal{W}^2\rangle
		=\langle (\alpha D_{z})^2\phi,\partial^2_{\theta}\phi\mathcal{W}^2\rangle
		+\langle 5\alpha D_{z}\phi,\partial^2_{\theta}\phi\mathcal{W}^2\rangle
		\tri \mathcal{J}_1+\mathcal{J}_2.
	\end{align}
	Integration by parts leads to
	\begin{align}\label{63in1}
		\mathcal{J}_1
		&=-\alpha^2\langle D_{z}\phi,\partial^2_{\theta}D_{z}\phi\mathcal{W}^2\rangle-\alpha^2\langle D_{z}\phi,\partial^2_{\theta}\phi D^{\ast}_{z}(\mathcal{W}^2)\rangle\\ \notag  
		&=\alpha^2\|D_{z}\partial_{\theta}\phi\|^2_{\mathcal{L}^2_{\mathcal{W}}}-\frac{\alpha^2}{2}\langle (D_{z}\phi)^2,\partial^2_{\theta}(\mathcal{W}^2)\rangle-\alpha^2\langle D_{z}\phi,\partial^2_{\theta}\phi D^{\ast}_{z}(\mathcal{W}^2)\rangle\\ \notag
		&\geq \alpha^2\|D_{z}\partial_{\theta}\phi\|^2_{\mathcal{L}^2_{\mathcal{W}}} -
		\frac{\alpha^2}{2}\langle (D_{z}\phi)^2,\partial^2_{\theta}(\mathcal{W}^2)\rangle-\frac{\varepsilon}{2}\|\partial^2_{\theta}\phi\|^2_{\mathcal{L}^2_{\mathcal{W}}}-C_{\beta,\varepsilon}\alpha^2\|D_{z}\phi\|^2_{\mathcal{L}^2_{\mathcal{W}}},
	\end{align}
	where $\varepsilon>0$ is a small constant. 
	Notice that
	\begin{align}\label{63in3}
		\frac{\partial^2_{\theta}(\mathcal{W}^2)}{\mathcal{W}^2}-4\xi(\xi+1)\cot^2(2\theta)=4\xi,\quad\text{for any}~~\xi\in[0,2].
	\end{align}
	It then follows from \eqref{63in3}, Lemma \ref{6lem4} and Lemma \ref{6lem3} that
	\begin{align*}
		\frac{\alpha^2}{2}\langle (D_{z}\phi)^2,\partial^2_{\theta}(\mathcal{W}^2)\rangle
		&\leq2\alpha^2\xi(\xi+1)\langle (D_{z}\phi)^2,\cot^2(2\theta)\mathcal{W}^2\rangle+C_\beta\alpha^2\|D_{z}\phi\|^2_{\mathcal{L}^2_{\mathcal{W}}}\\ \notag
		&\leq \frac{2\alpha^2\xi}{\xi+1}\|D_{z}\partial_{\theta}\phi\|^2_{\mathcal{L}^2_{\mathcal{W}}}
		+C_\beta\alpha^2\left(\|D_{z}\phi\|^2_{\mathcal{L}^2_{\mathcal{W}_z}}
		+\|D_{z}\partial_{\theta}\phi\|^2_{\mathcal{L}^2_{\mathcal{W}_z}}
		\right)
		+C_\beta\alpha^2\|D_{z}\phi\|^2_{\mathcal{L}^2_{\mathcal{W}}}
		\\
		\notag
		&\leq \frac{2\alpha^2\xi}{\xi+1}\|D_{z}\partial_{\theta}\phi\|^2_{\mathcal{L}^2_{\mathcal{W}}}
		+C_\beta\alpha^2\|D_{z}\partial_{\theta}\phi\|^2_{\mathcal{L}^2_{\mathcal{W}_z}}
		+C_\beta\alpha^2\|D_{z}\phi\|^2_{\mathcal{L}^2_{\mathcal{W}}}.
	\end{align*}
	Plugging the estimate above into \eqref{63in1}, we can get that
	\begin{align*}
		\mathcal{J}_1
		\geq 
		\frac{1-\xi}{1+\xi}\alpha^2\|D_{z}\partial_{\theta}\phi\|^2_{\mathcal{L}^2_{\mathcal{W}}} -\frac{\varepsilon}{2}\|\partial^2_{\theta}\phi\|^2_{\mathcal{L}^2_{\mathcal{W}}}-C_\beta\alpha^2\|D_{z}\partial_{\theta}\phi\|^2_{\mathcal{L}^2_{\mathcal{W}_z}}-C_{\beta,\varepsilon}\alpha^2\|D_{z}\phi\|^2_{\mathcal{L}^2_{\mathcal{W}}}.
	\end{align*}
	A direct computation yields directly that
	\begin{align}\label{63in2}
		\mathcal{J}_2
		\geq - \frac{\varepsilon}{2}\|\partial^2_{\theta}\phi\|^2_{\mathcal{L}^2_{\mathcal{W}}}-C_{\beta,\varepsilon}\alpha^2\|D_{z}\phi\|^2_{\mathcal{L}^2_{\mathcal{W}}}.
	\end{align}
	We substitute the estimates of $\mathcal{J}_1$ and $\mathcal{J}_2$ into \eqref{est-6-M-1}, to obtain \eqref{est-prop6.13-1} directly.

	Again thanks to \eqref{def-L-z}, we can find that
	\begin{align}\label{est-6-M-2}
		-\langle \tilde{D}^n_{\theta}L^{\alpha}_z (\phi),\tilde{D}^n_{\theta}\partial^2_{\theta}\phi\mathcal{W}^2\rangle
		=\langle \tilde{D}^n_{\theta} (\alpha D_{z})^2\phi,\tilde{D}^n_{\theta}\partial^2_{\theta}\phi\mathcal{W}^2\rangle
		+5\alpha\langle \tilde{D}^n_{\theta} D_{z}\phi,\tilde{D}^n_{\theta}\partial^2_{\theta}\phi\mathcal{W}^2\rangle
		\tri \mathcal{J}_3+\mathcal{J}_4.
	\end{align}
	We infer from $\big|\frac{D^j_{z,\theta}\mathcal{W}}{\mathcal{W}}\big|\lesssim_\beta 1$ with $j=1,2$ that
	\begin{align*}
		\mathcal{J}_3
		=&-\alpha^2\langle D_{z}\partial^n_{\theta}\phi D_{z}\partial^{n+2}_{\theta}\phi,\sin^{2n}(2\theta)\mathcal{W}^2\rangle
		-\alpha^2\langle D_{z}\partial^n_{\theta}\phi\partial^{n+2}_{\theta}\phi,\sin^{2n}(2\theta)D^{\ast}_{z}(\mathcal{W}^2)\rangle
		\\ \notag
		=&\alpha^2\|D_{z}\tilde{D}^n_{\theta}\partial_{\theta}\phi\|^2_{\mathcal{L}^2_{\mathcal{W}}} -\frac{\alpha^2}{2} \langle (D_{z}\partial^n_{\theta}\phi)^2,\partial^2_{\theta}\big(\sin^{2n}(2\theta)\mathcal{W}^2\big)\rangle 
		-\alpha^2\langle D_{z}\partial^n_{\theta}\phi\partial^{n+2}_{\theta}\phi,\sin^{2n}(2\theta)D^{\ast}_{z}(\mathcal{W}^2)\rangle
		\\ \notag
		\geq&\alpha^2\|D_{z}\tilde{D}^n_{\theta}\partial_{\theta}\phi\|^2_{\mathcal{L}^2_{\mathcal{W}}}-\frac{\varepsilon}{2}\|\tilde{D}^{n}_{\theta}\partial^2_{\theta}\phi\|^2_{\mathcal{L}^2_{\mathcal{W}}}- C_{\beta,\varepsilon}\alpha^2\|D_{z}\tilde{D}^{n-1}_{\theta}\partial_{\theta}\phi\|^2_{\mathcal{L}^2_{\mathcal{W}}}.
	\end{align*}
	It is not difficult to deduce that
	\begin{align*}
		\mathcal{J}_4
		&\geq -\frac{\varepsilon}{2}\|\tilde{D}^{n}_{\theta}\partial^2_{\theta}\phi\|^2_{\mathcal{L}^2_{\mathcal{W}}} - C_{\beta,\varepsilon}\alpha^2\|D_{z}\tilde{D}^{n-1}_{\theta}\partial_{\theta}\phi\|^2_{\mathcal{L}^2_{\mathcal{W}}}.
	\end{align*}
	Inserting the estimates $\mathcal{J}_3$ and $\mathcal{J}_4$ into \eqref{est-6-M-2}, we obtain \eqref{est-prop6.13-2} immediately, and thus complete the proof of this proposition.
\end{proof}

\subsection{Estimates for angular operator}
In this subsection, we aim to investigate the angular operator and derive some estimates for it.
One can easily verify that the radial operator $L_z^{\alpha}$ commutes with the angular derivative $D_{\theta}$, whereas the angular operator $L_{\theta}$ does not.
This occurs because the term $\partial_{\theta}\left(\tan\theta \phi\right)$ in $L_{\theta}$ does not commute with $D_{\theta}$.
Thus, we establish some estimates for $\partial_{\theta}\left(\tan\theta \phi\right)$ in this section.

We define
$\widetilde{\phi}\tri \frac{\phi}{\cos(\theta)}$
and take the notation $D^{-1}_\theta f= 0$, and thereby derive the following proposition.
\begin{prop}\label{6prop1}
	Assume that for $k=1,2$, $\left|\frac{D^k_{\theta}\mathcal{W}}{\mathcal{W}}\right|\leq C$, where $C>0$ is a constant.
	Then for $n\in\mathbb{N}$, there holds
	\begin{align}\label{est-6.2-1}
		-\langle\tilde{D}^n_{\theta}\partial_{\theta}(\tan\theta\phi),\tilde{D}^n_\theta\partial^{2}_{\theta}\phi\mathcal{W}^2\rangle
		\geq
		\left(\frac{4}{3}-\frac{\xi}{2}\right)\|\tilde{D}^n_\theta\partial_{\theta}\widetilde{\phi}\|^2_{\mathcal{L}^2_{\mathcal{W}}}
		- C\|\tilde{D}^{n-1}_{\theta}\partial_{\theta}\widetilde{\phi}\|^2_{\mathcal{L}^2_{\mathcal{W}}}
		-C\sum_{k=0}^{n}\|\tilde{D}^k_\theta\widetilde{\phi}\|^2_{\mathcal{L}^2_{\mathcal{W}}}.
	\end{align}
\end{prop}
\begin{proof}
	We apply the Leibniz rule, to obtain that
	\begin{align}\label{eq-tantheta-phi}
		&\partial^{n+1}_{\theta}(\tan\theta \phi)
		=\partial^{n+1}_{\theta}(\sin\theta \widetilde{\phi})=(\sin\theta)\partial^{n+1}_{\theta} \widetilde{\phi}+(n+1)(\cos\theta)\partial^{n}_{\theta} \widetilde{\phi} + \sum_{k=0}^{n-1}P_{k}(\theta)\partial^{k}_{\theta} \widetilde{\phi},
		\\
		&\partial^{n+2}_{\theta}\phi=\partial^{n+2}_{\theta}((\cos\theta )\widetilde{\phi})=\cos\theta\partial^{n+2}_{\theta} \widetilde{\phi}-(n+2)(\sin\theta)\partial^{n+1}_{\theta} \widetilde{\phi} + \sum_{m=0}^{n}Q_{m}(\theta)\partial^{m}_{\theta} \widetilde{\phi},
		\notag
	\end{align}
	for some functions $P_k(\theta)$, $Q_m(\theta)\in C^{\infty}([0,\frac{\pi}{2}])$. 
	These two identities yield that
	\begin{align}\label{est-6-J-1}
		&~~~-\langle\tilde{D}^n_{\theta}\partial_{\theta}(\tan\theta\phi),\tilde{D}^n_\theta\partial^{2}_{\theta}\phi\mathcal{W}^2\rangle
		\\
		&=
		-\langle\partial_{\theta}^{n+1}(\tan\theta\phi)\partial^{n+2}_{\theta}\phi,\sin^{2n}(2\theta)\mathcal{W}^2\rangle
		\notag\\
		&=
		-\frac{1}{2}\langle\partial^{n+1}_{\theta} \widetilde{\phi}D_{\theta}\partial^{n+1}_{\theta} \widetilde{\phi},\sin^{2n}(2\theta)\mathcal{W}^2\rangle
		+(n+2)\langle (\sin\theta)^2(\partial^{n+1}_{\theta} \widetilde{\phi})^2,\sin^{2n}(2\theta)\mathcal{W}^2\rangle
		\notag\\
		&~~~-(n+1)\langle \cos^2\theta\partial^n_{\theta}\widetilde{\phi}\partial^{n+2}_{\theta}\widetilde{\phi},\sin^{2n}(2\theta)\mathcal{W}^2\rangle
		+\frac{(n+1)(n+2)}{2}
		\langle \sin(2\theta)\partial^n_{\theta}\tilde{\phi}\partial^{n+1}_{\theta}\tilde{\phi},\sin^{2n}(2\theta)\mathcal{W}^2\rangle
		\notag\\
		&~~~-\sum_{k=0}^{n-1}
		\langle S_k(\theta)\partial^k_{\theta}\tilde{\phi}\partial^{n+2}_{\theta}\tilde{\phi}
		,\sin^{2n}(2\theta)\mathcal{W}^2\rangle
		-\sum_{i=0}^{n+1}\sum_{j=0}^{n}\langle 
		S_{ij}(\theta)\partial^i_{\theta}\tilde{\phi}\partial^j_{\theta}\tilde{\phi}
		,\sin^{2n}(2\theta)\mathcal{W}^2\rangle
		\notag\\
		&\tri\sum_{i=1}^6\mathcal{K}_i,
		\notag
	\end{align}
	for some functions $S_k(\theta)$, $S_{i,j}(\theta)\in C^{\infty}([0,\frac{\pi}{2}])$.
	We deal with the terms $\mathcal{K}_i$ $(i=1,\dots,6)$ in turn, beginning with $\mathcal{K}_1$. 
	In view of integration by parts, we infer that
	\begin{align*}
		\mathcal{K}_1&=-\frac{1}{4}\langle D_{\theta} (\partial^{n+1}_{\theta} \widetilde{\phi})^2,\sin^{2n}(2\theta)\mathcal{W}^2\rangle
		=\frac{1}{4}\langle (\partial^{n+1}_{\theta} \widetilde{\phi})^2,D^{\ast}_{\theta}\big(\sin^{2n}(2\theta)\mathcal{W}^2\big)\rangle
		=\langle (\tilde{D}^n_{\theta}\partial_{\theta} \widetilde{\phi})^2,\mathcal{A}_1\mathcal{W}^2\rangle,
	\end{align*}
	where $D^{\ast}_{\theta}$ is the adjoint operators of $D_{\theta}$, and $\mathcal{A}_1 \tri \left(n+\frac{1}{2}\right)
	\cos(2\theta) + \frac{1}{2}\frac{D_{\theta}\mathcal{W}}{\mathcal{W}}$.
	For the term $\mathcal{K}_2$, it is easy to check that
	\begin{align*}
		\mathcal{K}_2=\langle (\tilde{D}^n_{\theta}\partial_{\theta}\widetilde{\phi})^2,\mathcal{A}_2\mathcal{W}^2\rangle,
	\end{align*}
	where $\mathcal{A}_2 \tri (n+2)\sin^2\theta.$
	As for the term $\mathcal{K}_3$, we deduce from $\left|\frac{D^j_{\theta}\mathcal{W}}{\mathcal{W}}\right|\lesssim1$ with $j=1,2$ and integration by parts that
	\begin{align*}
		\mathcal{K}_3 
		&=\langle (\tilde{D}^n_{\theta}\partial_{\theta}\widetilde{\phi})^2,\mathcal{A}_3\mathcal{W}^2\rangle-\frac{n+1}{2}\langle(\partial^n_{\theta}\widetilde{\phi})^2,\partial^2_{\theta}\left[\cos^2\theta\sin^{2n}(2\theta)\mathcal{W}^2\right]\rangle \\ \notag
		&\geq \langle (\tilde{D}^n_{\theta}\partial_{\theta}\widetilde{\phi})^2,\mathcal{A}_3\mathcal{W}^2\rangle - C\langle(\partial^n_{\theta}\widetilde{\phi})^2,\sin^{2(n-1)}(2\theta)\mathcal{W}^2\rangle \\ \notag
		&\geq \langle (\tilde{D}^n_{\theta}\partial_{\theta}\widetilde{\phi})^2,\mathcal{A}_3\mathcal{W}^2\rangle - C\|\tilde{D}^{n-1}_{\theta}\partial_{\theta}\widetilde{\phi}\|^2_{\mathcal{L}^2_{\mathcal{W}}},
	\end{align*}
	where $\mathcal{A}_3 \tri (n+1)\cos^2\theta.$ 
	For the term $\mathcal{K}_4$, since $\left|\frac{D_{\theta}\mathcal{W}}{\mathcal{W}}\right|\lesssim1$, we see that
	\begin{align*}
		\mathcal{K}_4=-\frac{1}{4}(n+1)(n+2)\big\langle\big(\partial^n_{\theta}\widetilde{\phi}\big)^2,D^{\ast}_{\theta}\left[\sin^{2n}(2\theta)\mathcal{W}^2\right]\big\rangle\gtrsim-\|\tilde{D}^{n}_{\theta}\widetilde{\phi}\|^2_{\mathcal{L}^2_{\mathcal{W}}}.
	\end{align*}
	In order to address the term $\mathcal{K}_5$, we apply once more $\left|\frac{D_{\theta}\mathcal{W}}{\mathcal{W}}\right|\lesssim1$ and Cauchy-Schwartz inequality, to find that
	\begin{align*}
		\mathcal{K}_5=& \sum_{k=0}^{n-1}\langle S_k(\theta)\partial^{k+1}_{\theta}\tilde{\phi}\partial^{n+1}_{\theta}\tilde{\phi},\sin^{2n}(2\theta)\mathcal{W}^2\rangle
		+\sum_{k=0}^{n-1}\big\langle \partial^{k}_{\theta}\tilde{\phi}\partial^{n+1}_{\theta}\tilde{\phi},D^{\ast}_{\theta}\left[S_k(\theta)\sin^{2n-1}(2\theta)\mathcal{W}^2\right]\big\rangle\\ \notag
		\gtrsim& -\varepsilon\|\tilde{D}^{n}_{\theta}\partial_{\theta}\widetilde{\phi}\|^2_{\mathcal{L}^2_{\mathcal{W}}}-C_\varepsilon\sum_{k=0}^{n}\|\tilde{D}^{k}_{\theta}\widetilde{\phi}\|^2_{\mathcal{L}^2_{\mathcal{W}}}.
	\end{align*}
	For the last term $\mathcal{K}_6$, again thanks to Cauchy-Schwartz inequality, we get that
	\begin{align*}
		\mathcal{K}_6
		\gtrsim -\varepsilon\|\tilde{D}^{n}_{\theta}\partial_{\theta}\widetilde{\phi}\|^2_{\mathcal{L}^2_{\mathcal{W}}}-C_\varepsilon\sum_{k=0}^{n}\|\tilde{D}^{k}_{\theta}\widetilde{\phi}\|^2_{\mathcal{L}^2_{\mathcal{W}}}.
	\end{align*}
	Moreover, it is easy to verify that
	\begin{align}\label{est-A123}
		\mathcal{A}_1+\mathcal{A}_2+\mathcal{A}_3
		=2n\cos^2(\theta)+\frac{3}{2}+\frac{1}{2}\frac{D_{\theta}\mathcal{W}}{\mathcal{W}}\geq \frac{3-\xi}{2},
	\end{align}
	where we have used the fact that
	\begin{align}\label{est-Dtheta-W}
		-\xi\leq \frac{D_{\theta}\mathcal{W}}{\mathcal{W}}\leq \xi.
	\end{align}
	Substituting the estimates for $\mathcal{K}_i$ $(i=1,\cdots,6)$ into \eqref{est-6-J-1}, and applying \eqref{est-A123}, we can obtain \eqref{est-6.2-1} immediately. 
\end{proof}

\begin{prop}\label{6prop2}
	Assume that for any $k=1,2$,
	$
	\sum\limits_{l=0}^{k}\left|\frac{D^{l}_{z}D^{k-l}_{\theta}\mathcal{W}}{\mathcal{W}}\right|\leq C_{\beta}$ 
	where $C_{\beta}>0$ is a constant.
	Then there exist sufficiently small constants $0<\alpha \leq 1$ and $0<\varepsilon \leq 1$ such that
	\begin{align}\label{est-prop6.11-1}
		-\langle\partial_{\theta}((\tan\theta) \phi),(\alpha D_{z})^2\phi\mathcal{W}^2\rangle
		\geq \frac{1-\xi}{2}\alpha^2\| D_{z}\widetilde{\phi}\|^2_{\mathcal{L}^2_{\mathcal{W}}}-\varepsilon\alpha^2\| D_{z}\widetilde{\phi}\|^2_{\mathcal{L}^2_{\mathcal{W}}}-C_{\beta,\varepsilon}\|\partial_{\theta}\phi\|^2_{\mathcal{L}^2_{\mathcal{W}}},
	\end{align}
	and for any $n\in\mathbb{N}^{+}$, there holds 
	\begin{align}\label{est-prop6.11-2}
		-\langle\tilde{D}^n_\theta\partial_{\theta}((\tan\theta) \phi),(\alpha D_{z})^2\tilde{D}^n_\theta \phi\mathcal{W}^2\rangle
		\geq& \frac{1-\xi}{2}\alpha^2\|\tilde{D}^n_\theta D_{z}\widetilde{\phi}\|^2_{\mathcal{L}^2_{\mathcal{W}}}
		-C_{\beta,\varepsilon}\alpha^2\sum_{k=0}^{n-1}\sum_{j=0}^{1}\|D^j_{z}\tilde{D}^k_\theta \widetilde{\phi}\|^2_{\mathcal{L}^2_{\mathcal{W}}}
		\\
		&
		-\left(2\varepsilon+C_{\beta}\alpha\right) 
		\sum_{i+j=0}^{1}\|(\alpha D_{z})^i\tilde{D}^n_\theta\partial^{j}_{\theta}\widetilde{\phi}\|^2_{\mathcal{L}^2_{\mathcal{W}}}.
		\notag
	\end{align}
\end{prop}
\begin{proof}
	Firstly, we infer that 
	\begin{align*}
		\partial_{\theta}((\tan\theta) \phi)(\alpha D_{z})^2\phi=\frac{1}{2}D_{\theta} \widetilde{\phi}(\alpha D_{z})^2\widetilde{\phi}
		+ \phi (\alpha D_{z})^2\phi.
	\end{align*}
	This gives that
	\begin{align}\label{est-6-K-1}
		-\langle\partial_{\theta}((\tan\theta) \phi),(\alpha D_{z})^2\phi\mathcal{W}^2\rangle
		=&-\frac12\langle D_{\theta} \widetilde{\phi}(\alpha D_{z})^2\widetilde{\phi},\mathcal{W}^2\rangle
		-\langle \phi (\alpha D_{z})^2\phi,\mathcal{W}^2\rangle
		\tri \mathcal{L}_1+ \mathcal{L}_2.
	\end{align}
	For the first term $\mathcal{L}_1$, in view of $\left|\frac{D_z\mathcal{W}}{\mathcal{W}}\right| \lesssim_{\beta} 1$ and integration by parts, we find that
	\begin{align*}
		\mathcal{L}_1
		&=\frac{\alpha^2}{2}\langle D_{\theta}D_{z}\widetilde{\phi}D_{z}\widetilde{\phi},\mathcal{W}^2\rangle+\frac{\alpha^2}{2}\langle D_{\theta}\widetilde{\phi}D_{z}\widetilde{\phi},D^{\ast}_{z}(\mathcal{W}^2)\rangle \\ \notag
		&\geq - \frac{\alpha^2}{4} \langle \big(D_{z}\widetilde{\phi}\big)^2,D^{\ast}_{\theta}(\mathcal{W}^2)\rangle 
		- \varepsilon\alpha^2\|D_{z}\widetilde{\phi}\|^2_{\mathcal{L}^2_{\mathcal{W}}}-C_{\beta,\varepsilon}\|D_{\theta}\widetilde{\phi}\|^2_{\mathcal{L}^2_{\mathcal{W}}},
	\end{align*}
	where $\varepsilon>0$ is a small constant.
	While for the term $\mathcal{L}_2$, we infer from $\left|\frac{D^j_z\mathcal{W}}{\mathcal{W}}\right| \lesssim_{\beta} 1$ with $j=1,2$ and Lemma \ref{6lem4} that
	\begin{align*}
		\mathcal{L}_2&=\alpha^2\|D_{z}\phi\|^2_{\mathcal{L}^2_{\mathcal{W}}}-\frac{\alpha^2}{2}\langle \phi^2,(D^{\ast}_{z})^2(\mathcal{W}^2)\rangle
		\geq \alpha^2\|D_{z}\phi\|^2_{\mathcal{L}^2_{\mathcal{W}}}-C_\beta\|\partial_{\theta}\phi\|^2_{\mathcal{L}^2_{\mathcal{W}}}
		\\
		&\geq \alpha^2\|(\cos \theta )D_{z}\tilde{\phi}\|^2_{\mathcal{L}^2_{\mathcal{W}}}-C_\beta\|\partial_{\theta}\phi\|^2_{\mathcal{L}^2_{\mathcal{W}}}.
	\end{align*}
	Notice that
	\begin{align}\label{62in1}
		(\cos\theta)^2 \mathcal{W}^2-\frac{1}{4}D^{\ast}_{\theta}(\mathcal{W}^2)=\left[(\cos\theta)^2+(\sin\theta)^2 - \frac{1}{2}\left(1+\frac{D_{\theta}\mathcal{W}}{\mathcal{W}}\right)\right]\mathcal{W}^2\geq\frac{1-\xi}{2}\mathcal{W}^2,
	\end{align}
	where we have used \eqref{est-Dtheta-W}.
	We continue by substituting the estimates of $\mathcal{L}_1$ and $\mathcal{L}_2$ into \eqref{est-6-K-1}, and employing \eqref{62in1}, eventually to find that \eqref{est-prop6.11-1} holds. 
	
	With the help of the Leibniz rule, one gets easily that 
	\begin{align*}
		\partial^{n}_{\theta}(\alpha D_{z})^2\phi
		=(\cos\theta)\partial^{n}_{\theta} (\alpha D_{z})^2\widetilde{\phi}-(n+1)(\sin\theta)\partial^{n-1}_{\theta} (\alpha D_{z})^2\widetilde{\phi}+\sum_{m=0}^{n-2}Q_{m}(\theta)\partial^{m}_{\theta} (\alpha D_{z})^2\widetilde{\phi},
	\end{align*}
	for some $Q_m(\theta)\in C^{\infty}([0,\frac{\pi}{2}])$. 
	This, together with \eqref{eq-tantheta-phi}, yields directly that
	\begin{align}
		&~~~-\langle\tilde{D}^n_\theta\partial_{\theta}((\tan\theta )\phi),(\alpha D_{z})^2\tilde{D}^n_\theta \phi\mathcal{W}^2\rangle
		\label{est-6-L-1}\\
		&=-\frac{1}{2}\langle D_{\theta}\partial^{n}_{\theta} \widetilde{\phi}\partial^{n}_{\theta} (\alpha D_{z})^2\widetilde{\phi}, \sin^{2n}(2\theta)\mathcal{W}^2\rangle
		-(n+1)\langle \cos^2\theta\partial^{n}_{\theta} \widetilde{\phi}\partial^{n}_{\theta} (\alpha D_{z})^2\widetilde{\phi},\sin^{2n}(2\theta)\mathcal{W}^2\rangle
		\notag\\
		&~~~+(n+1)\langle \sin^2\theta\partial^{n+1}_{\theta}\widetilde{\phi}\partial^{n-1}_{\theta}(\alpha D_{z})^2\widetilde{\phi}, \sin^{2n}(2\theta)\mathcal{W}^2\rangle
		-\sum_{k=0}^{n-1}\langle S_{k}(\theta)\partial^{k}_{\theta} \widetilde{\phi}\partial^{n}_{\theta} (\alpha D_{z})^2\widetilde{\phi} , \sin^{2n}(2\theta)\mathcal{W}^2\rangle
		\notag\\
		&~~~-\sum_{r=0}^{1}\sum_{i=0}^{n+r}\sum_{j=0}^{n-r-1}\langle S_{rij}(\theta)\partial^{i}_{\theta} \widetilde{\phi}\partial^{j}_{\theta} (\alpha D_{z})^2\widetilde{\phi}, \sin^{2n}(2\theta)\mathcal{W}^2\rangle
		\notag\\
		&\tri\sum_{i=1}^{5} \mathcal{M}_i,
		\notag
	\end{align}
	for some $S_k(\theta)$, $S_{rij}(\theta)\in C^{\infty}([0,\frac{\pi}{2}])$. 
	Recalling that $\left|\frac{D_z\mathcal{W}}{\mathcal{W}}\right|\lesssim_{\beta}1$, the first term $\mathcal{M}_1$ is controlled by
	\begin{align*}
		\mathcal{M}_1
		&=-\frac{\alpha^2}{4}\langle (\partial^{n}_{\theta} D_{z}\widetilde{\phi})^2,D^{\ast}_{\theta}\big(\sin^{2n}(2\theta)\mathcal{W}^2\big)\rangle+\frac{\alpha^2}{2}\langle \tilde{D}^{n+1}_{\theta} \widetilde{\phi}\tilde{D}^{n}_{\theta} D_{z}\widetilde{\phi},D^{\ast}_{z}(\mathcal{W}^2)\rangle
		\\
		&\geq
		\alpha^2\langle(\tilde{D}^n_{\theta}D_{z}\widetilde{\phi})^2,\tilde{\mathcal{A}}_1\mathcal{W}^2\rangle
		-C_\beta\alpha^3\|\tilde{D}^n_{\theta}D_{z}\widetilde{\phi}\|^2_{\mathcal{L}^2_{\mathcal{W}}}
		-C_\beta\alpha\|\tilde{D}^{n+1}_{\theta}\widetilde{\phi}\|^2_{\mathcal{L}^2_{\mathcal{W}}}.
	\end{align*}
	where $\tilde{\mathcal{A}}_1$ is defined by
	\begin{align*}
		\tilde{\mathcal{A}}_1
		=-\left(n+\frac12\right)\cos(2\theta)-\frac12\frac{D_{\theta}\mathcal{W}}{\mathcal{W}}
		=(n+1)(\sin\theta)^2-n(\cos\theta)^2- \frac{1}{2}\left(1+\frac{D_{\theta}\mathcal{W}}{\mathcal{W}}\right).
	\end{align*}
	For the term $\mathcal{M}_2$, we deduce from $\left|\frac{D^j_z\mathcal{W}}{\mathcal{W}}\right|\lesssim_{\beta}1$ with $j=1,2$ that
	\begin{align*}
		\mathcal{M}_2
		&=\alpha^2\langle (\tilde{D}^{n}_{\theta} D_{z}\widetilde{\phi})^2,\tilde{\mathcal{A}}_2\mathcal{W}^2\rangle
		-\frac{\alpha^2}{2}(n+1)\langle (\tilde{D}^{n}_{\theta} \widetilde{\phi})^2,(\cos\theta)^2 (D^{\ast}_{z})^2(\mathcal{W}^2)\rangle\\ \notag
		&\geq\alpha^2\langle (\tilde{D}^{n}_{\theta} D_{z}\widetilde{\phi})^2,\tilde{\mathcal{A}}_2\mathcal{W}^2\rangle - C_\beta\alpha^2\|\tilde{D}^{n}_{\theta} \widetilde{\phi}\|^2_{\mathcal{L}^2_{\mathcal{W}}},
	\end{align*}
	where $\tilde{\mathcal{A}}_2=(n+1)(\cos\theta)^2.$ 
	We now estimate the third term $\mathcal{M}_3$, we employ $\left|\frac{D_{z}\mathcal{W}}{\mathcal{W}}\right|+\left|\frac{D_{\theta}\mathcal{W}}{\mathcal{W}}\right|\lesssim_{\beta}1$ and $\alpha\leq 1$ to deduce that
	\begin{align*}
		\mathcal{M}_3&=\alpha^2(n+1)\langle\partial^{n}_{\theta}D_{z}\widetilde{\phi}\partial^{n-1}_{\theta}D_{z}\widetilde{\phi},D^{\ast}_{\theta}((\sin\theta)^2 \sin^{2n-1}(2\theta)\mathcal{W}^2)\rangle\\ \notag
		&~~~-\alpha^2(n+1)\langle\tilde{D}^{n+1}_{\theta}\widetilde{\phi}\tilde{D}^{n-1}_{\theta}D_{z}\widetilde{\phi},(\sin\theta)^2 D^{\ast}_{z}(\mathcal{W}^2)\rangle+\alpha^2\langle(\tilde{D}^{n}_{\theta}D_{z}\widetilde{\phi})^2,\tilde{\mathcal{A}}_3\mathcal{W}^2\rangle\\ \notag
		&\geq \alpha^2\langle(\tilde{D}^{n}_{\theta}D_{z}\widetilde{\phi})^2,\tilde{\mathcal{A}}_3\mathcal{W}^2\rangle - 
		\varepsilon\sum_{i+j=1}\|(\alpha D_{z})^{i}\tilde{D}^{n}_{\theta}\partial^j_{\theta}\widetilde{\phi}\|^2_{\mathcal{L}^2_{\mathcal{W}}}-C_\beta\alpha^2\|D_{z}\tilde{D}^{n-1}_{\theta}\widetilde{\phi}\|^2_{\mathcal{L}^2_{\mathcal{W}}},
	\end{align*}
	where $\tilde{\mathcal{A}}_3=(n+1)(\sin\theta)^2.$
	We then address the term $\mathcal{M}_4$ by using $\left|\frac{D_z\mathcal{W}}{\mathcal{W}}\right|\lesssim_{\beta}1$ that
	\begin{align*}
		\mathcal{M}_4
		&=\alpha^2\sum_{k=0}^{n-1}\langle S_{k}(\theta)\tilde{D}^{k}_{\theta} D_{z}\widetilde{\phi}\tilde{D}^{n}_{\theta} D_{z}\widetilde{\phi},\sin^{n-k}(2\theta)\mathcal{W}^2\rangle
		+\alpha^2\sum_{k=0}^{n-1}\langle S_{k}(\theta)\tilde{D}^{k}_{\theta} \widetilde{\phi}\tilde{D}^{n}_{\theta} D_{z}\widetilde{\phi},\sin^{n-k}(2\theta)D^{\ast}_{z}(\mathcal{W}^2)\rangle\\ \notag
		&\geq-\varepsilon\alpha^2 \|D_{z}\tilde{D}^n_\theta\widetilde{\phi}\|^2_{\mathcal{L}^2_{\mathcal{W}}} -C_\beta\alpha^2 \sum_{k=0}^{n-1}\sum_{j=0}^{1}\|D^j_{z}\tilde{D}^k_\theta \widetilde{\phi}\|^2_{\mathcal{L}^2_{\mathcal{W}}}.
	\end{align*}
	To handle the term $\mathcal{M}_5$, we derive from $\left|\frac{D_z\mathcal{W}}{\mathcal{W}}\right|\lesssim_{\beta}1$ and $\alpha\leq 1$ that
	\begin{align*}
		\mathcal{M}_5
		&=\alpha^2\sum_{r=0}^{1}\sum_{i=0}^{n+r}\sum_{j=0}^{n-r-1}\langle S_{rij}(\theta)\partial^{i}_{\theta} D_{z}\widetilde{\phi}\partial^{j}_{\theta} D_{z}\widetilde{\phi},\sin^{2n}(2\theta)\mathcal{W}^2\rangle\\ \notag
		&~~~+\alpha^2\sum_{r=0}^{1}\sum_{i=0}^{n+r}\sum_{j=0}^{n-r-1}\langle S_{rij}(\theta)\tilde{D}^{i}_{\theta} \widetilde{\phi}\tilde{D}^{j}_{\theta} D_{z}\widetilde{\phi},\sin^{2n-i-j}(2\theta)D^{\ast}_{z}(\mathcal{W}^2)\rangle\\ \notag
		&=-\alpha^2\sum_{r=0}^{1}\sum_{i=0}^{n+r-1}\sum_{j=0}^{n-r}\langle S_{rij}(\theta)\partial^{i}_{\theta} D_{z}\widetilde{\phi}\partial^{j}_{\theta} D_{z}\widetilde{\phi},\sin^{2n}(2\theta)\mathcal{W}^2)\rangle\\ \notag
		&~~~-\alpha^2\sum_{r=0}^{1}\sum_{i=0}^{n+r-1}\sum_{j=0}^{n-r-1}\langle S_{rij}(\theta)\partial^{i}_{\theta} D_{z}\widetilde{\phi}\partial^{j}_{\theta} D_{z}\widetilde{\phi},D^{\ast}_{\theta}(\sin^{2n-1}(2\theta)\mathcal{W}^2)\rangle\\ \notag
		&~~~+\alpha^2\sum_{r=0}^{1}\sum_{i=0}^{n+r}\sum_{j=0}^{n-r-1}\langle S_{rij}(\theta)\tilde{D}^{i}_{\theta} \widetilde{\phi}\tilde{D}^{j}_{\theta} D_{z}\widetilde{\phi},\sin^{n-i-j}(2\theta)D^{\ast}_{z}(\mathcal{W}^2)\rangle
		\\ \notag
		&\geq -\varepsilon\sum_{i+j=1}\|(\alpha D_{z})^{i}\tilde{D}^{n}_{\theta}\partial^j_{\theta}\widetilde{\phi}\|^2_{\mathcal{L}^2_{\mathcal{W}}}-C_\beta\alpha^2\sum_{k=0}^{n-1}\sum_{j=0}^{1}\|D^j_{z}\tilde{D}^k_\theta \widetilde{\phi}\|^2_{\mathcal{L}^2_{\mathcal{W}}}.
	\end{align*}
	Moreover, it is easy to verify that
	\begin{align}\label{62in2}
		\tilde{\mathcal{A}}_1+\tilde{\mathcal{A}}_2+\tilde{\mathcal{A}}_3
		=1+(2n+1)\sin^2\theta-\frac{1}{2}\bigg(1+\frac{D_{\theta}\mathcal{W}}{\mathcal{W}}\bigg)\geq\frac{1-\xi}{2}\mathcal{W}^2.
	\end{align}
	Substituting the estimates of $\mathcal{M}_i$ $(i=1,\cdots,5)$ into \eqref{est-6-L-1}, then using the fact that $0<\alpha\leq 1$ and $0<\varepsilon\leq 1$, finally using \eqref{62in2}, we get \eqref{est-prop6.11-2} directly. 
\end{proof}

\subsection{Elliptic estimates with purely radial weight}
In this subsection, we intend to establish the weighted elliptic estimates with purely radial weight.
\begin{lemm}\label{6lem8}
	Assume that $\langle f,K\rangle_{\theta}=0$ and that for $j=1,2$, $\big|\frac{D^j_z\mathcal{W}_z}{\mathcal{W}_z}\big|\leq C_\beta $, 
	where $C_{\beta}>0$ is a constant. Then for any $n\in \mathbb{N}$, there exists a sufficiently small constant $\alpha>0$ such that
	\begin{align*}
		\alpha^2\|D^{n+1}_z\phi\|^2_{\mathcal{L}^2_{\mathcal{W}_z}}+\|\partial_\theta D^n_z\phi\|^2_{\mathcal{L}^2_{\mathcal{W}_z}} \leq C_\beta\|D^n_zf\|^2_{\mathcal{L}^2_{\mathcal{W}_z}}.
	\end{align*}
\end{lemm}

\begin{proof}
	Recalling the definition of $L_{\theta}$ in \eqref{def-L-theta}, it is easy to verify that its adjoint operator $L^{*}_{\theta}$ satisfies
	$$L^{*}_{\theta}(K(\theta))=0.$$ 
	For any $\alpha>0$, there holds 
	\begin{align*}
		L^{\alpha}_z(\langle D^n_z\phi,K\rangle_{\theta})=\langle D^n_z(L^{\alpha}_z(\phi)+L_{\theta}(\phi)),K\rangle_{\theta} = D^n_z\langle f,K\rangle_{\theta}=0.
	\end{align*}
	This, combining a similar derivation of \eqref{4eq4}, leads us to get that
	\begin{align}\label{64in1}
		\langle D^n_z\phi,K\rangle_{\theta}=0.
	\end{align}
	According to the completeness of the Fourier series, we obtain the odd extension of $D_{z}^n\phi$ $(n\geq0)$ on $[0,\pi]$ as:
	\begin{align}\label{Four-Ser-phi}
		D^n_z\phi(z,\theta) = \sum_{k=0}^{\infty} \phi^n_k(z)\sin(2k\theta),~~~
		\phi_k^n(z)\tri\int_0^{\pi} D^n_z\phi(z,\theta)\sin(2k\theta)d\theta.
	\end{align} 
	With this odd extension at hand, the unique identity is derived from \eqref{64in1} and Lemma \ref{6lem2} that
	\begin{align}\label{64in2}
		\phi^n_1 = \sum_{k=2}^{\infty} (-1)^{k}\frac{15k}{(4k^2-1)(4k^2-9)} \phi^n_k.
	\end{align}
	We employ \eqref{64in2} and Cauchy-Schwartz inequality, to get that
	\begin{align}\label{est-phi-1}
		\|\phi^n_1\|^2_{L^2_{\mathcal{W}_z}} \leq \sum_{k=2}^{\infty} \frac{225k^2}{(4k^2-1)^2(4k^2-9)^2} \|\phi^n_k\|^2_{L^2_{\mathcal{W}_z}}\leq \sum_{k=2}^{\infty}  \|\phi^n_k\|^2_{L^2_{\mathcal{W}_z}}.
	\end{align}
	On the other hand, the equation \eqref{6model} gives that
	\begin{align}\label{64in4-1}
		\langle D^n_zL^{\alpha}_z(\phi),D^n_z\phi\mathcal{W}^2_z\rangle
		+\langle D^n_zL_{\theta}(\phi),D^n_z\phi\mathcal{W}^2_z\rangle
		=\langle f,D^n_z\phi\mathcal{W}^2_z\rangle.
	\end{align}
	Recalling \eqref{def-L-z} and \eqref{def-L-theta}, then applying integration by parts, we achieve that
	\begin{align*}
		\langle D^n_zL^{\alpha}_z(\phi),D^n_z\phi\mathcal{W}^2_z\rangle&= \alpha^2\|D^{n+1}_z\phi\|^2_{\mathcal{L}^2_{\mathcal{W}_z}}-\frac{\alpha^2}{2}\langle (D^n_z\phi)^2,(D^{\ast}_{z})^2(\mathcal{W}^2_z)\rangle + \frac{5\alpha}{2}\langle (D^n_z\phi)^2,D^{\ast}_{z}(\mathcal{W}^2_z)\rangle\\ \notag
		&\geq\alpha^2\|D^{n+1}_z\phi\|^2_{\mathcal{L}^2_{\mathcal{W}_z}}-C_{\beta}\alpha\|D^n_z\phi\|^2_{\mathcal{L}^2_{\mathcal{W}_z}},
        \\
		\langle D^n_zL_{\theta}(\phi),D^n_z\phi\mathcal{W}^2_z\rangle
		&=
		\|\partial_\theta D^n_z\phi\|^2_{\mathcal{L}^2_{\mathcal{W}_z}} 
		-6\|D^n_z\phi\|^2_{\mathcal{L}^2_{\mathcal{W}_z}} 
		+ \frac{1}{2}\|(\sec\theta) D^n_z\phi\|^2_{\mathcal{L}^2_{\mathcal{W}_z}}.
	\end{align*}
	We combine these two estimates above and \eqref{64in4-1} together, then take $\alpha$ and $\varepsilon$ small enough, to discover that
	\begin{align}\label{64in4}
		\alpha^2\|D^{n+1}_z\phi\|^2_{\mathcal{L}^2_{\mathcal{W}_z}}
		+\|\partial_\theta D^n_z\phi\|^2_{\mathcal{L}^2_{\mathcal{W}_z}} 
		-7\|D^n_z\phi\|^2_{\mathcal{L}^2_{\mathcal{W}_z}} 
		+
		\frac{1}{2}\|(\sec\theta )D^n_z\phi\|^2_{\mathcal{L}^2_{\mathcal{W}_z}}
		\lesssim_{\beta} \|D_z^nf\|^2_{\mathcal{L}^2_{\mathcal{W}_z}}.
	\end{align}
	We then deduce from \eqref{Four-Ser-phi}, the orthogonality of $\left\{\sin(2k\theta)\right\}_{k\in\mathbb{Z}}$ and $\left\{\cos(2k\theta)\right\}_{k\in\mathbb{Z}}$ that
	\begin{align}\label{express-phi}
		\|D^n_z\phi\|^2_{\mathcal{L}^2_{\mathcal{W}_z}} = \mathop{\sum}_{k=1}^{\infty} \|\phi^n_k\|^2_{L^2_{\mathcal{W}_z}},
		\quad
		\|\partial_{\theta} D^n_z\phi\|^2_{\mathcal{L}^2_{\mathcal{W}_z}} = \mathop{\sum}_{k=1}^{\infty} 4k^2\|\phi^n_k\|^2_{\mathcal{L}^2_{\mathcal{W}_z}}.
	\end{align}
	In view of \eqref{est-phi-1} and \eqref{express-phi}, we have
	\begin{align*}
		\|\partial_{\theta} D^n_z\phi\|^2_{\mathcal{L}^2_{\mathcal{W}_z}} \leq 2\mathop{\sum}_{k=2}^{\infty} 4k^2\|\phi^n_k\|^2_{\mathcal{L}^2_{\mathcal{W}_z}}.
	\end{align*}
	Combining this with \eqref{est-phi-1}, \eqref{64in4} and \eqref{express-phi}, we get that
	\begin{align}\label{64in4-2}
		\|\partial_\theta D^n_z\phi\|^2_{\mathcal{L}^2_{\mathcal{W}_z}} - 7\|D^n_z\phi\|^2_{\mathcal{L}^2_{\mathcal{W}_z}}&=-3\|\phi^n_1\|^2_{\mathcal{L}^2_{\mathcal{W}_z}} + \mathop{\sum}_{k=2}^{\infty} (4k^2-7)\|\phi^n_k\|^2_{\mathcal{L}^2_{\mathcal{W}_z}} 
		\geq \mathop{\sum}_{k=2}^{\infty} (4k^2-10)\|\phi^n_k\|^2_{\mathcal{L}^2_{\mathcal{W}_z}} \\ \notag
		&\geq \frac{3}{8}\mathop{\sum}_{k=2}^{\infty} 4k^2\|\phi^n_k\|^2_{\mathcal{L}^2_{\mathcal{W}_z}}
		\geq \frac{3}{16}\|\partial_\theta D^n_z\phi\|^2_{\mathcal{L}^2_{\mathcal{W}_z}}.
	\end{align}
	It follows that
	\begin{align}\label{64in5-1}
		7\|D^n_z\phi\|^2_{\mathcal{L}^2_{\mathcal{W}_z}}\leq\frac{13}{16}\|\partial_\theta D^n_z\phi\|^2_{\mathcal{L}^2_{\mathcal{W}_z}}.
	\end{align}
	Substituting \eqref{64in4-2} and \eqref{64in5-1} into \eqref{64in4}, we find that
	\begin{align*}
		\|D^n_z\phi\|^2_{\mathcal{L}^2_{\mathcal{W}_z}}+\alpha^2\|D^{n+1}_z\phi\|^2_{\mathcal{L}^2_{\mathcal{W}_z}}
		+\|\partial_\theta D^n_z\phi\|^2_{\mathcal{L}^2_{\mathcal{W}_z}} 
		+\|(\sec\theta) D^n_z\phi\|^2_{\mathcal{L}^2_{\mathcal{W}_z}}
		\lesssim_{\beta}\|D_z^nf\|^2_{\mathcal{L}^2_{\mathcal{W}_z}}.
	\end{align*}
	This completes the proof of Lemma \ref{6lem8}.
\end{proof}
For any $n\in\mathbb{N}$, we introduce $\mathcal{H}^n_{\mathcal{W}_z}([0,\infty)\times[0,\frac{\pi}{2}])$ equipped with the Sobolev norm containing exclusively radial weight as:
\begin{align*}
	\|f\|_{\mathcal{H}^n_{\mathcal{W}_z}}
	\tri \sum_{i+j=0}^{n}\|D^i_zD^j_{\theta}f\|^2_{\mathcal{L}^2_{\mathcal{W}_z}}.
\end{align*}

\begin{prop}\label{6prop5}
	Assume that $\langle f,K\rangle_{\theta}=0$ and that for $j=1,2$, $\left|\frac{D^j_z\mathcal{W}_z}{\mathcal{W}_z}\right|\leq C_\beta $ where $C_\beta>0$ is a constant. Then for any $n\in\mathbb{N}$, there exists a sufficiently small constant $\alpha>0$ such that
	\begin{align}\label{64eq0}
		\alpha^4\|D^2_{z}\phi\|^2_{\mathcal{H}^n_{\mathcal{W}_z}} + 
		\alpha^2\|D_{z}\partial_{\theta}\phi\|^2_{\mathcal{H}^n_{\mathcal{W}_z}} 
		+\|\partial^2_{\theta}\phi\|^2_{\mathcal{H}^n_{\mathcal{W}_z}} \leq C_{\beta} \|f\|^2_{\mathcal{H}^n_{\mathcal{W}_z}}.
	\end{align}
\end{prop}
\begin{proof}
	We will prove this proposition by induction on $n$.
	
	\medskip
	
	\noindent
	\textbf{Step 1: For the base case $n=0$.} 
	Along a similar way as \eqref{63in1} with $\partial^2_{\theta}(\mathcal{W}_z^2)=0$ and \eqref{63in2}, we can deduce that
	\begin{align}\label{64eq-zalpha}
		-\langle L^{\alpha}_z (\phi),\partial^2_{\theta}\phi\mathcal{W}^2_z\rangle
		\geq \alpha^2\|\partial_{\theta}D_{z}\phi\|^2_{\mathcal{L}^2_{\mathcal{W}_z}} -\varepsilon\|\partial^2_{\theta}\phi\|^2_{\mathcal{L}^2_{\mathcal{W}_z}}-C_{\beta,\varepsilon}\alpha^2\|D_{z}\phi\|^2_{\mathcal{L}^2_{\mathcal{W}_z}}.
	\end{align}
	Recalling the definition of $L_\theta$ in \eqref{def-L-theta}, and taking $n=0$ and $\xi=0$ in Proposition \ref{6prop1}, we then have
	\begin{align}\label{64eq-theta}
		-\langle L_\theta (\phi),\partial^{2}_{\theta}\phi\mathcal{W}^2_z\rangle
		&=\|\partial^2_{\theta}\phi\|^2_{\mathcal{L}^2_{\mathcal{W}_z}}-\langle\partial_{\theta}(\tan\theta\phi),\partial^{2}_{\theta}\phi\mathcal{W}^2_z\rangle+6\langle \phi,\partial^{2}_{\theta}\phi\mathcal{W}^2_z\rangle\\ \notag
		&\geq\frac{1}{2}\|\partial^2_{\theta}\phi\|^2_{\mathcal{L}^2_{\mathcal{W}_z}}+\frac{4}{3}\|\partial_{\theta}\widetilde{\phi}\|^2_{\mathcal{L}^2_{\mathcal{W}_z}}-C_{\beta}\|\widetilde{\phi}\|^2_{\mathcal{L}^2_{\mathcal{W}_z}}.
	\end{align}
	Then, we recall the equation \eqref{6model}, to discover that
	\begin{align}\label{64eq1-1}
		\langle L^{\alpha}_z(\phi),\partial^2_{\theta}\phi\mathcal{W}^2_z\rangle
		+\langle L_{\theta}\phi,\partial^2_{\theta}\phi\mathcal{W}^2_z\rangle
		=\langle f,\partial^2_{\theta}\phi\mathcal{W}^2_z\rangle.
	\end{align}
	Plugging \eqref{64eq-zalpha} and \eqref{64eq-theta} into \eqref{64eq1-1}, choosing $\varepsilon$ suitably small, and applying Lemma \ref{6lem4}, one gets that
	\begin{align}\label{64eq1}
		\alpha^2\|\partial_{\theta}D_{z}\phi\|^2_{\mathcal{L}^2_{\mathcal{W}_z}} +\|\partial^2_{\theta}\phi\|^2_{\mathcal{L}^2_{\mathcal{W}_z}}
		+\|\partial_{\theta}\widetilde{\phi}\|^2_{\mathcal{L}^2_{\mathcal{W}_z}}
		\lesssim_\beta&\|f\|^2_{\mathcal{L}^2_{\mathcal{W}_z}}+\alpha^2\|D_{z}\phi\|^2_{\mathcal{L}^2_{\mathcal{W}_z}}+\|\widetilde{\phi}\|^2_{\mathcal{L}^2_{\mathcal{W}_z}}
		\\
		\lesssim_{\beta}&\|f\|^2_{\mathcal{L}^2_{\mathcal{W}_z}}+\alpha^2\|D_{z}\phi\|^2_{\mathcal{L}^2_{\mathcal{W}_z}}+\|\partial_{\theta}\phi\|^2_{\mathcal{L}^2_{\mathcal{W}_z}}.
		\notag
	\end{align}

	The same argument as Proposition \ref{6prop3}, we can deduce that
	\begin{align}\label{est-Lzalpha-2}
		-\langle L^{\alpha}_z (\phi),(\alpha D_{z})^2\phi\mathcal{W}^2_z\rangle
		\geq\|(\alpha D_{z})^2\phi\|^2_{\mathcal{L}^2_{\mathcal{W}_z}}-C_\beta\alpha^3\|D_{z}\phi\|^2_{\mathcal{L}^2_{\mathcal{W}_z}}.
	\end{align}
	We then deduce from Proposition \ref{6prop2} with $\xi=0$ and $\varepsilon=\frac14$ that
	\begin{align*}
		-\langle\partial_{\theta}((\tan\theta)\phi),(\alpha D_{z})^2\phi\mathcal{W}^2_z\rangle
		\geq&\frac{\alpha^2}{4}\|D_{z}\widetilde{\phi}\|^2_{\mathcal{L}^2_{\mathcal{W}_z}}-C_\beta\|\partial_\theta\phi\|^2_{\mathcal{L}^2_{\mathcal{W}_z}},
	\end{align*}
	which, together with Lemma \ref{6lem4}, implies that
	\begin{align}\label{est-Ltheta-2}
		-\langle L_{\theta}(\phi),(\alpha D_{z})^2\phi\mathcal{W}^2_z\rangle
		&=\langle \partial^2_{\theta}\phi, (\alpha D_{z})^2\phi\mathcal{W}^2_z\rangle-\langle\partial_{\theta}((\tan\theta)\phi),\partial^{2}_{\theta}\phi\mathcal{W}^2_z\rangle+6\langle \phi,(\alpha D_{z})^2\phi\mathcal{W}^2_z\rangle
		\\ 
		&\geq \frac{\alpha^2}{4}\|D_{z}\widetilde{\phi}\|^2_{\mathcal{L}^2_{\mathcal{W}_z}}-\frac{\alpha^4}{2}\|D^2_{z}\phi\|^2_{\mathcal{L}^2_{\mathcal{W}_z}}-C_\beta\|\partial^2_{\theta}\phi\|^2_{\mathcal{L}^2_{\mathcal{W}_z}}-C_\beta\|\partial_{\theta}\phi\|^2_{\mathcal{L}^2_{\mathcal{W}_z}}.
		\notag
	\end{align}
	Combining the equation \eqref{6model}, \eqref{est-Lzalpha-2} and \eqref{est-Ltheta-2} together, we find that
	\begin{align}
		\label{64eq2}
		\alpha^4\|D^2_{z}\phi\|^2_{\mathcal{L}^2_{\mathcal{W}_z}}+\alpha^2\|D_{z}\widetilde{\phi}\|^2_{\mathcal{L}^2_{\mathcal{W}_z}}\lesssim_\beta \|f\|^2_{\mathcal{L}^2_{\mathcal{W}_z}}+\|\partial^2_{\theta}\phi\|^2_{\mathcal{L}^2_{\mathcal{W}_z}}+\alpha^2\|D_{z}\phi\|^2_{\mathcal{L}^2_{\mathcal{W}_z}}+\|\partial_{\theta}\phi\|^2_{\mathcal{L}^2_{\mathcal{W}_z}}.
	\end{align}
	By virtue of \eqref{64eq1}, \eqref{64eq2} and Lemma \ref{6lem8}, we can deduce that
	\begin{align}\label{64eq3}
		\alpha^4\|D^2_{z}\phi\|^2_{\mathcal{L}^2_{\mathcal{W}_z}}+\alpha^2\|D_{z}\widetilde{\phi}\|^2_{\mathcal{L}^2_{\mathcal{W}_z}}+\alpha^2\|\partial_{\theta}D_{z}\phi\|^2_{\mathcal{L}^2_{\mathcal{W}_z}} +\|\partial^2_{\theta}\phi\|^2_{\mathcal{L}^2_{\mathcal{W}_z}}+\|\partial_{\theta}\widetilde{\phi}\|^2_{\mathcal{L}^2_{\mathcal{W}_z}}\lesssim_\beta \|f\|^2_{\mathcal{L}^2_{\mathcal{W}_z}},
	\end{align}
	which implies that \eqref{64eq0} holds true for $n=0$.
	\medskip
	
	\noindent
	\textbf{Step 2: For the case $n\geq1$.} We continue our proof by induction argument on $n\in\mathbb{N}^{+}$. First of all, assume that \eqref{64eq0} holds for $n-1$, namely
	\begin{align}\label{indu-assum-1}
		\alpha^4\|D^2_{z}\phi\|^2_{\mathcal{H}^{n-1}_{\mathcal{W}_z}} 
		+\alpha^2\|D_{z}\partial_{\theta}\phi\|^2_{\mathcal{H}^{n-1}_{\mathcal{W}_z}}
		+ \|\partial^2_{\theta}\phi\|^2_{\mathcal{H}^{n-1}_{\mathcal{W}_z}}  \lesssim_\beta \|f\|^2_{\mathcal{H}^{n-1}_{\mathcal{W}_z}}.
	\end{align}
	With the help of Lemmas \ref{6lem5}-\ref{6lem7} and the induction assumption \eqref{indu-assum-1}, we find that
	\begin{align*}
		&~~~~~~\alpha^4\|D^2_{z}\phi\|^2_{\mathcal{H}^{n}_{\mathcal{W}_z}} 
		+\alpha^2\|D_{z}\partial_{\theta}\phi\|^2_{\mathcal{H}^{n}_{\mathcal{W}_z}}
		+\|\partial^2_{\theta}\phi\|^2_{\mathcal{H}^{n}_{\mathcal{W}_z}} 
		\\
		&\lesssim_\beta \|f\|^2_{\mathcal{H}^{n-1}_{\mathcal{W}_z}}
		+\sum_{i+j=n}\left(\alpha^4\|D^{i+2}_{z}D_{\theta}^{j}\phi\|^2_{\mathcal{L}^{2}_{\mathcal{W}_z}} 
		+\alpha^2\|D_{z}^{i+1}D_{\theta}^{j}\partial_{\theta}\phi\|^2_{\mathcal{L}^{2}_{\mathcal{W}_z}}
		+ \|D_{z}^{i}D_{\theta}^{j}\partial^2_{\theta}\phi\|^2_{\mathcal{L}^{2}_{\mathcal{W}_z}}
		\right) 
		\\
		&\lesssim_\beta 
		\|f\|^2_{\mathcal{H}^{n-1}_{\mathcal{W}_z}}
		+\sum_{k=0}^{n}\left(
		\alpha^4
		\left(\|D^{k+2}_{z}\phi\|^2_{\mathcal{L}^{2}_{\mathcal{W}_z}} 
		+\|D^{k}_{\theta}D_{z}^2\phi\|^2_{\mathcal{L}^{2}_{\mathcal{W}_z}} 
		\right)
		+\alpha^2\left(\|D^{k+1}_{z}\partial_{\theta}\phi\|^2_{\mathcal{L}^{2}_{\mathcal{W}_z}} 
		+\|D^{k}_{\theta}D_{z}\partial_{\theta}\phi\|^2_{\mathcal{L}^{2}_{\mathcal{W}_z}} 
		\right)
		\right.
		\\
		&~~~~~\left.
		+ \left(\|D_{z}^{k}\partial^{2}_{\theta}\phi\|^2_{\mathcal{L}^{2}_{\mathcal{W}_z}}
		+\|D_{\theta}^{k}\partial^{2}_{\theta}\phi\|^2_{\mathcal{L}^{2}_{\mathcal{W}_z}}
		\right)
		\right) 
		\\
		&\lesssim_\beta 
		\|f\|^2_{\mathcal{H}^{n-1}_{\mathcal{W}_z}}
		+
		\alpha^4
		\left(\|D^{n+2}_{z}\phi\|^2_{\mathcal{L}^{2}_{\mathcal{W}_z}} 
		+\|D^{n}_{\theta}D_{z}^2\phi\|^2_{\mathcal{L}^{2}_{\mathcal{W}_z}} 
		\right)
		+\alpha^2\left(\|D^{n+1}_{z}\partial_{\theta}\phi\|^2_{\mathcal{L}^{2}_{\mathcal{W}_z}} 
		+\|D^{n}_{\theta}D_{z}\partial_{\theta}\phi\|^2_{\mathcal{L}^{2}_{\mathcal{W}_z}} 
		\right)
		\\
		&~~~~~
		+ \left(\|D_{z}^{n}\partial^{2}_{\theta}\phi\|^2_{\mathcal{L}^{2}_{\mathcal{W}_z}}
		+\|D_{\theta}^{n}\partial^{2}_{\theta}\phi\|^2_{\mathcal{L}^{2}_{\mathcal{W}_z}}
		\right)
		\\
		&\lesssim_{\beta}
		\|f\|^2_{\mathcal{H}^{n-1}_{\mathcal{W}_z}}
		+\sum_{i+j=2}\|(\alpha D_{z})^i\tilde{D}^n_{\theta}\partial^j_{\theta}\phi\|^2_{\mathcal{L}^{2}_{\mathcal{W}_z}}+\alpha^4\|D^{n+2}_z\phi\|^2_{\mathcal{L}^{2}_{\mathcal{W}_z}} + \|\partial^2_{\theta}D^n_z\phi\|^2_{\mathcal{L}^{2}_{\mathcal{W}_z}}.
	\end{align*}
	Hence, we only need to prove that
	\begin{align}\label{est-goal-1}
		\sum_{i+j=2}\|(\alpha D_{z})^i\tilde{D}^n_{\theta}\partial^j_{\theta}\phi\|^2_{\mathcal{L}^{2}_{\mathcal{W}_z}}+\alpha^4\|D^{n+2}_z\phi\|^2_{\mathcal{L}^{2}_{\mathcal{W}_z}} + \|\partial^2_{\theta}D^n_z\phi\|^2_{\mathcal{L}^{2}_{\mathcal{W}_z}}
		\lesssim_\beta \|f\|^2_{\mathcal{H}^{n}_{\mathcal{W}_z}}.
	\end{align}
	We divide into the following two steps to prove \eqref{est-goal-1}.
	\medskip
	
	\noindent
	\textbf{Step 2.1.} 
	Our targe in to prove that
	\begin{align}\label{64eq6-1}
		\|(\alpha D_{z})^2D^n_z\phi\|^2_{\mathcal{L}^{2}_{\mathcal{W}_z}} + \|\partial^2_{\theta}D^n_z\phi\|^2_{\mathcal{L}^{2}_{\mathcal{W}_z}}\lesssim_\beta\|D^n_z f\|_{\mathcal{L}^{2}_{\mathcal{W}_z}}\lesssim_\beta \|f\|^2_{\mathcal{H}^{n}_{\mathcal{W}_z}}.
	\end{align}
	Notice that the operators $D_z$ and $L^\alpha_z+L_\theta$ are commutative, \eqref{64eq6-1} can be deduced by a similar way as \eqref{64eq3}.
	\medskip
	
	\noindent
	\textbf{Step 2.2.} Our goal is to show that 
	\begin{align}\label{64eq6}
		\sum_{i+j=2}\|(\alpha D_{z})^i\tilde{D}^n_{\theta}\partial^j_{\theta}\phi\|^2_{\mathcal{L}^2_{\mathcal{W}_z}}\lesssim_\beta\|f\|^2_{\mathcal{H}^{n}_{\mathcal{W}_z}}.
	\end{align}
	We can deduce from Lemma \ref{6lem5}, Proposition \ref{6prop4} and the induction assumption \eqref{indu-assum-1} that for $n\geq 1$,
	\begin{align}\label{64in7}
		-\langle \tilde{D}^n_{\theta}L^{\alpha}_z (\phi),\tilde{D}^n_{\theta}\partial^2_{\theta}\phi\mathcal{W}^2_z\rangle
		\geq \alpha^2\|D_{z}\tilde{D}^n_{\theta}\partial_{\theta}\phi\|^2_{\mathcal{L}^2_{\mathcal{W}_z}} -\varepsilon\|\tilde{D}^{n}_{\theta}\partial^2_{\theta}\phi\|^2_{\mathcal{L}^2_{\mathcal{W}_z}}- C_{\beta,\varepsilon}\|f\|^2_{\mathcal{H}^{n-1}_{\mathcal{W}_z}}.
	\end{align}
	On the other hand, some direct computations give that
	\begin{align}\label{est-Ltheta-n-1}
		-\langle \tilde{D}^n_{\theta}L_{\theta} (\phi),\tilde{D}^n_{\theta}\partial^2_{\theta}\phi\mathcal{W}^2_z\rangle
		=\|\tilde{D}^n_{\theta}\partial_{\theta}^2\phi\|_{L^2_{\mathcal{W}_z}}
		-\langle\tilde{D}^n_{\theta}\partial_{\theta}(\tan(\theta)\phi),\tilde{D}^n_\theta\partial^{2}_{\theta}\phi\mathcal{W}^2_z\rangle
		+6\langle\tilde{D}^n_{\theta}\phi,\tilde{D}^n_\theta\partial^{2}_{\theta}\phi\mathcal{W}^2_z\rangle.
	\end{align}
	We infer from Proposition \ref{6prop1} with $\xi=0$ that
	\begin{align}\label{est-Dznphi-0}
		-\langle\tilde{D}^n_{\theta}\partial_{\theta}((\tan\theta)\phi),\tilde{D}^n_\theta\partial^{2}_{\theta}\phi\mathcal{W}^2_z\rangle
		&\geq\|\tilde{D}^n_\theta\partial_{\theta}\widetilde{\phi}\|^2_{\mathcal{L}^2_{\mathcal{W}_z}}
		-C\|\tilde{D}^{n-1}_{\theta}\partial_{\theta}\widetilde{\phi}\|^2_{\mathcal{L}^2_{\mathcal{W}}}
		-C\sum_{k=0}^{n}\|\tilde{D}^k_\theta\widetilde{\phi}\|^2_{\mathcal{L}^2_{\mathcal{W}_z}}.
	\end{align}
	For any $1\leq m\leq n$, it is easy to check that
	\begin{align}\label{est-Dm-theta-wphi-1}
		\|\tilde{D}^m_\theta\widetilde{\phi}\|^2_{\mathcal{L}^2_{\mathcal{W}_z}} &\lesssim \sum_{j=0}^{m}\|\sin^{m}(2\theta)(\cos\theta)^{j-m-1}
		\partial^j_{\theta}\phi\|^2_{\mathcal{L}^2_{\mathcal{W}_z}}\lesssim\sum_{j=0}^{m}\|\sin^{j-1}(2\theta)\partial^j_{\theta}\phi\|^2_{\mathcal{L}^2_{\mathcal{W}_z}}.
	\end{align}
	We then employ the induction assumption \eqref{indu-assum-1}, Lemma \ref{6lem4} and Lemma \ref{6lem5}, to find that
	\begin{align*}
		&\|\frac{\partial_{\theta}^j\phi}{\sin^{2-j}(2\theta)}\|^2_{\mathcal{L}^2_{\mathcal{W}_z}}
		\lesssim \|\partial_{\theta}^2\phi\|^2_{\mathcal{L}^2_{\mathcal{W}_z}}
		\lesssim_\beta \|f\|^2_{\mathcal{L}^{2}_{\mathcal{W}_z}},\quad \text{for}~~j=0,1,\\
		&\|\sin^{j-2}(2\theta)\partial^j_{\theta}\phi\|^2_{\mathcal{L}^2_{\mathcal{W}_z}}
		\lesssim
		\|\tilde{D}^{j-2}_{\theta}\partial^2_{\theta}\phi\|^2_{\mathcal{L}^2_{\mathcal{W}_z}}
		\lesssim\|D^{j-2}_{\theta}\partial^2_{\theta}\phi\|^2_{\mathcal{L}^2_{\mathcal{W}_z}} 
		\lesssim_\beta \|f\|^2_{\mathcal{H}^{j-2}_{\mathcal{W}_z}},\quad \text{for any}~~2\leq j\leq m.
	\end{align*}
	Substituting these three estimates above into \eqref{est-Dm-theta-wphi-1}, then employing the induction assumption \eqref{indu-assum-1} and Lemma \ref{6lem5}, one obtains that for any $1\leq m\leq n$,
	\begin{align}\label{est-Dm-1}
		\|\tilde{D}^m_\theta\widetilde{\phi}\|^2_{\mathcal{L}^2_{\mathcal{W}_z}}
		\lesssim_\beta \|f\|^2_{\mathcal{H}^{m-1}_{\mathcal{W}_z}}.
	\end{align}
	For any $1\leq m\leq n$, by some direct calculations, we have
	\begin{align}\label{est-Dm-theta-wphi-2}
		\|\tilde{D}_{\theta}^{m-1}\partial_{\theta}\widetilde{\phi}\|^2_{\mathcal{L}^2_{\mathcal{W}_z}}
		\lesssim& \sum_{j=0}^{m}\|\sin^{m-1}(2\theta)(\cos\theta)^{j-m-1}
		\partial^j_{\theta}\phi\|^2_{\mathcal{L}^2_{\mathcal{W}_z}}\lesssim \sum_{j=0}^{m}\|\sin^{j-2}(2\theta)\partial^j_{\theta}\phi\|^2_{\mathcal{L}^2_{\mathcal{W}_z}}.
	\end{align}
	Using \eqref{est-Dm-theta-wphi-2}, we discover that for any $1\leq m\leq n$,
	\begin{align}\label{est-Dm-2}
		\|\tilde{D}_{\theta}^{m-1}\partial_{\theta}\widetilde{\phi}\|^2_{\mathcal{L}^2_{\mathcal{W}_z}}
		\lesssim_\beta \|f\|^2_{\mathcal{H}^{m-1}_{\mathcal{W}_z}}.
	\end{align}
	Substituting \eqref{est-Dm-1} and \eqref{est-Dm-2} into \eqref{est-Dznphi-0}, we find that
	\begin{align}\label{est-Dznphi-1}
		-\langle\tilde{D}^n_{\theta}\partial_{\theta}((\tan\theta)\phi),\tilde{D}^n_\theta\partial^{2}_{\theta}\phi\mathcal{W}^2_z\rangle
		&\geq\|\tilde{D}^n_\theta\partial_{\theta}\widetilde{\phi}\|^2_{\mathcal{L}^2_{\mathcal{W}_z}}-C_\beta\|f\|^2_{\mathcal{H}^{n-1}_{\mathcal{W}_z}}.
	\end{align}
	With the help of integration by parts, we can deduce that
	\begin{align}\label{est-Dznphi-2}
		\langle\tilde{D}^n_{\theta}\phi,\tilde{D}^n_\theta\partial^{2}_{\theta}\phi\mathcal{W}^2_z\rangle
		=&-\langle\sin^{2n}(2\theta)\left(\partial^{n+1}_{\theta}\phi\right)^2,\mathcal{W}^2_z\rangle
		+\frac12 \langle\partial_{\theta}^2\left(\sin^{2n}(2\theta)\right)\left(\partial^{n}_{\theta}\phi\right)^2,\mathcal{W}^2_z\rangle
		\\
		\lesssim& \|\tilde{D}^n_{\theta}\partial_{\theta}\phi\|^2_{\mathcal{L}^2_{\mathcal{W}_z}}+\|\widetilde{D}^{n-1}_{\theta}\partial_{\theta}\phi\|^2_{\mathcal{L}^2_{\mathcal{W}_z}}.
		\notag
	\end{align}
	Again applying Lemma \ref{6lem4}, Lemma \ref{6lem5} and the induction assumption \eqref{indu-assum-1}, and following a derivation analogous to \eqref{est-Dm-1} and \eqref{est-Dm-2}, we obtain that
	\begin{align}\label{est-Dznphi-3}
		\langle\tilde{D}^n_{\theta}\phi,\tilde{D}^n_\theta\partial^{2}_{\theta}\phi\mathcal{W}^2_z\rangle
		\lesssim_\beta  \|f\|^2_{\mathcal{H}^{n-1}_{\mathcal{W}_z}}.
	\end{align}
	The combination of \eqref{est-Ltheta-n-1}, \eqref{est-Dznphi-1} and \eqref{est-Dznphi-3}, yields directly that
	\begin{align}\label{64in6}
		-\langle\tilde{D}^n_{\theta}L_{\theta}(\phi),\tilde{D}^n_\theta\partial^{2}_{\theta}\phi\mathcal{W}^2_z\rangle
		\geq\|\tilde{D}^n_\theta\partial^2_{\theta}\phi\|^2_{\mathcal{L}^2_{\mathcal{W}_z}}+\|\tilde{D}^n_\theta\partial_{\theta}\widetilde{\phi}\|^2_{\mathcal{L}^2_{\mathcal{W}_z}}-C_\beta\|f\|^2_{\mathcal{H}^{n-1}_{\mathcal{W}_z}}.
	\end{align}
	Recalling the equation \eqref{6model}, applying \eqref{64in7} and \eqref{64in6}, and taking $\varepsilon$ suitably small, we infer that
	\begin{align}\label{64eq4}
		\|\tilde{D}^n_\theta\partial^2_{\theta}\phi\|^2_{\mathcal{L}^2_{\mathcal{W}_z}}
		+\|\tilde{D}^n_\theta\partial_{\theta}\widetilde{\phi}\|^2_{\mathcal{L}^2_{\mathcal{W}_z}}
		+\alpha^2\|D_{z}\tilde{D}^n_{\theta}\partial_{\theta}\phi\|^2_{\mathcal{L}^2_{\mathcal{W}_z}}\lesssim_\beta\|f\|^2_{\mathcal{H}^{n}_{\mathcal{W}_z}}.
	\end{align}

	Along a similar way as Proposition \ref{6prop3}, it then follows from Lemma \ref{6lem5} and the induction assumption \eqref{indu-assum-1} that for $n\geq 1$,
	\begin{align}\label{64in8}
		-\langle \tilde{D}^n_\theta L^{\alpha}_z (\phi),(\alpha D_z)^2\tilde{D}^n_\theta \phi\mathcal{W}^2_z\rangle 
		&\geq\|(\alpha D_z)^2\tilde{D}^n_\theta \phi\|^2_{\mathcal{L}^2_{\mathcal{W}_z}}-C_\beta\alpha^3\|D_{z}\tilde{D}^n_\theta \phi\|^2_{\mathcal{L}^2_{\mathcal{W}_z}} \\ \notag
		&\geq\|(\alpha D_z)^2\tilde{D}^n_\theta \phi\|^2_{\mathcal{L}^2_{\mathcal{W}_z}}-C_\beta\alpha^3\|D_{z}\tilde{D}^{n-1}_\theta \partial_{\theta}\phi\|^2_{\mathcal{L}^2_{\mathcal{W}_z}}\\ \notag
		&\geq\|(\alpha D_z)^2\tilde{D}^n_\theta \phi\|^2_{\mathcal{L}^2_{\mathcal{W}_z}}-C_\beta\|f\|^2_{\mathcal{H}^{n-1}_{\mathcal{W}_z}}.
	\end{align}
	While for $L_{\theta}$, we have
	\begin{align}\label{est-6-15-theta-1}
		-\langle\tilde{D}^n_\theta L_{\theta}(\phi),(\alpha D_{z})^2\tilde{D}^n_\theta \phi\mathcal{W}^2_z\rangle
		&=\langle \tilde{D}^n_{\theta}\partial_{\theta}^2\phi,(\alpha D_{z})^2\tilde{D}^n_\theta \phi\mathcal{W}^2_z\rangle
		-\langle\tilde{D}^n_{\theta}\partial_{\theta}((\tan\theta)\phi),(\alpha D_{z})^2\tilde{D}^n_\theta \phi\mathcal{W}^2_z\rangle
		\\
		&~~~+6\langle\tilde{D}^n_{\theta}\phi,(\alpha D_{z})^2\tilde{D}^n_\theta \phi\mathcal{W}^2_z\rangle
		\notag
		\\
		&\tri\mathcal{N}_1+\mathcal{N}_2+\mathcal{N}_3.
		\notag
	\end{align}
	We now control three terms in the right-hand side of \eqref{est-6-15-theta-1} term by term.
	The application of \eqref{64eq4} yields directly that
	\begin{align}\label{est-6-15-theta-2}
		\mathcal{N}_1
		\geq -\frac{\varepsilon}{2}\|(\alpha D_{z})^2\tilde{D}^n_\theta \phi\|^2_{\mathcal{L}^2_{\mathcal{W}_z}} -C_{\varepsilon} \|\tilde{D}^n_{\theta}\partial_{\theta}^2\phi\|^2_{\mathcal{L}^2_{\mathcal{W}_z}}
		\geq-\frac{\varepsilon}{2}\|(\alpha D_{z})^2\tilde{D}^n_\theta \phi\|^2_{\mathcal{L}^2_{\mathcal{W}_z}} -C_{\varepsilon} \|f\|^2_{\mathcal{H}^n_{\mathcal{W}_z}}.
	\end{align}
	For the second term, employing Proposition \ref{6prop2} with $\xi=0$, we get that for $n\geq 1$,
	\begin{align}\label{est-6-15-0}
		\mathcal{N}_2
		\geq -(2\varepsilon+C_{\beta}\alpha) \sum^1_{i+j=0}\|(\alpha D_{z})^i\tilde{D}^n_\theta\partial^{j}_{\theta}\widetilde{\phi}\|^2_{\mathcal{L}^2_{\mathcal{W}_z}} 
		-C_{\beta,\varepsilon}\alpha^2 \sum_{j=0}^{1}\sum_{k=0}^{n-1}\|D_{z}^j\tilde{D}^k_\theta \widetilde{\phi}\|^2_{\mathcal{L}^2_{\mathcal{W}_z}}.
	\end{align}
	We then estimate the terms in the right-hand side of \eqref{est-6-15-0} in turn.
	Firstly, it follows from Lemma \ref{6lem4} that
	\begin{align}\label{est-Dn-patheta-1}
		\sum_{i+j=1}\|(\alpha D_{z})^i\tilde{D}^n_\theta\partial^{j}_{\theta}\widetilde{\phi}\|^2_{\mathcal{L}^2_{\mathcal{W}_z}} 
		&\lesssim \sum_{i+j=1}\sum_{k=0}^{n+j}\|(\alpha D_{z})^i\sin^n(2\theta)(\cos\theta)^{k-n-j-1}\partial^k_{\theta}\phi\|^2_{\mathcal{L}^2_{\mathcal{W}_z}}\\ \notag
		&\lesssim \sum_{i+j=1}\sum_{k=0}^{n} \|(\alpha D_{z})^i(\cos\theta)^{-j-1}\tilde{D}^k_{\theta}\phi\|^2_{\mathcal{L}^2_{\mathcal{W}_z}}+\|(\cos\theta)^{-1}\tilde{D}^{n}_{\theta}\partial_{\theta}\phi\|^2_{\mathcal{L}^2_{\mathcal{W}_z}}
		\\ \notag
		&\lesssim \sum_{i+j=1}\sum_{k=0}^{n} \|(\alpha D_{z})^i\partial_{\theta}^{j+1}\tilde{D}^k_{\theta}\phi\|^2_{\mathcal{L}^2_{\mathcal{W}_z}}+\|\partial_{\theta}\tilde{D}^{n}_{\theta}\partial_{\theta}\phi\|^2_{\mathcal{L}^2_{\mathcal{W}_z}}
		\\\notag
		&\lesssim \sum_{i+j=1}\sum_{k=0}^{n}\sum_{l=0}^{j+1} \|(\alpha D_{z})^i(\sin(2\theta))^{k+l-j-1}\partial_{\theta}^{k+l}\phi\|^2_{\mathcal{L}^2_{\mathcal{W}_z}}.
	\end{align}
	For any $0\leq l\leq j+1-k$, it follows from Lemma \ref{6lem4} that
	\begin{align*}
		\|(\alpha D_{z})^i(\sin(2\theta))^{k+l-j-1}\partial_{\theta}^{k+l}\phi\|^2_{\mathcal{L}^2_{\mathcal{W}_z}}
		\lesssim \|(\alpha D_{z})^i\partial_{\theta}^{j+1}\phi\|^2_{\mathcal{L}^2_{\mathcal{W}_z}}.
	\end{align*}
	For any $j+2-k\leq l\leq j+1$, one can deduce directly that
	\begin{align*}
		\|(\alpha D_{z})^i(\sin(2\theta))^{k+l-j-1}\partial_{\theta}^{k+l}\phi\|^2_{\mathcal{L}^2_{\mathcal{W}_z}}
		\lesssim\|(\alpha D_{z})^i\tilde{D}_{\theta}^{k+l-j-1}\partial_{\theta}^{j+1}\phi\|^2_{\mathcal{L}^2_{\mathcal{W}_z}}
		\lesssim \sum_{m=0}^{k}\|(\alpha D_{z})^i\tilde{D}_{\theta}^{m}\partial_{\theta}^{j+1}\phi\|^2_{\mathcal{L}^2_{\mathcal{W}_z}}.
	\end{align*}
	We plug these two estimates into \eqref{est-Dn-patheta-1}, then employ Lemma \ref{6lem5}, the induction assumption \eqref{indu-assum-1} and \eqref{64eq4}, to discover that
	\begin{align}\label{est-6-15-1}
		\sum_{i+j=1}\|(\alpha D_{z})^i\tilde{D}^n_\theta\partial^{j}_{\theta}\widetilde{\phi}\|^2_{\mathcal{L}^2_{\mathcal{W}_z}} 
		&\lesssim \sum_{i+j=1}\sum_{k=0}^{n} \|(\alpha D_{z})^i\tilde{D}^{k}_{\theta}\partial_{\theta}^{j+1}\phi\|^2_{\mathcal{L}^2_{\mathcal{W}_z}}
		\lesssim_\beta  \|f\|^2_{\mathcal{H}^{n}_{\mathcal{W}_z}}.
	\end{align}
	After some direct calculations, one gets, by Lemma \ref{6lem4}, that
	\begin{align}\label{est-DzjDtheta-1}
		\sum_{j=0}^{1}\sum_{k=0}^{n-1}\|(\alpha D_{z})^j\tilde{D}^k_\theta \widetilde{\phi}\|^2_{\mathcal{L}^2_{\mathcal{W}_z}}&\lesssim \sum_{j=0}^{1}\sum_{k=0}^{n-1}\sum_{i=0}^{k}\|(\alpha D_{z})^j\sin^{k}(2\theta)(\cos\theta)^{i-k-1}\partial^{i}_{\theta}\phi\|^2_{\mathcal{L}^2_{\mathcal{W}_z}}\\ \notag
		& \lesssim\sum_{j=0}^{1}\sum_{i=0}^{n-1}\|(\cos\theta)^{-1}(\alpha D_{z})^j\widetilde{D}^{i}_{\theta}\phi\|^2_{\mathcal{L}^2_{\mathcal{W}_z}}
		\\ \notag
		& \lesssim\sum_{j=0}^{1}\sum_{i=0}^{n-1}\|\partial_{\theta}(\alpha D_{z})^j\tilde{D}^{i}_{\theta}\phi\|^2_{\mathcal{L}^2_{\mathcal{W}_z}}.
	\end{align}
	For $j=0$, taking advantage of Lemma \ref{6lem4}, Lemma \ref{6lem5} and the induction assumption \eqref{indu-assum-1}, we can deduce from a similar way as \eqref{est-Dm-1} and \eqref{est-Dm-2} that
	\begin{align}\label{est-DzjDtheta-2}
		\sum_{i=0}^{n-1}\|\partial_{\theta}\tilde{D}^{i}_{\theta}\phi\|^2_{\mathcal{L}^2_{\mathcal{W}_z}}
		\lesssim_\beta \|f\|^2_{\mathcal{H}^{n-1}_{\mathcal{W}_z}}.
	\end{align}
	For $j=1$, again thanks to Lemma \ref{6lem4}, Lemma \ref{6lem5} and the induction assumption \eqref{indu-assum-1}, it is enough to prove that
	\begin{align}\label{est-DzjDtheta-3}
		\sum_{i=0}^{n-1}\|\alpha D_{z}\partial_{\theta}\tilde{D}^{i}_{\theta}\phi\|^2_{\mathcal{L}^2_{\mathcal{W}_z}}
		&\lesssim
		\|\alpha D_{z}\partial_{\theta}\phi\|^2_{\mathcal{L}^2_{\mathcal{W}_z}}
		+\sum_{i=1}^{n-1}\left(\|\alpha D_{z}\tilde{D}^{i}_{\theta}\partial_{\theta}\phi\|^2_{\mathcal{L}^2_{\mathcal{W}_z}}
		+\|\alpha D_{z}\tilde{D}^{i-1}_{\theta}\partial_{\theta}\phi\|^2_{\mathcal{L}^2_{\mathcal{W}_z}}
		\right)
		\\
		&\lesssim \sum_{i=0}^{n-1}\|\alpha D_{z}\tilde{D}^{i}_{\theta}\partial_{\theta}\phi\|^2_{\mathcal{L}^2_{\mathcal{W}_z}}
		\lesssim_\beta \|f\|^2_{\mathcal{H}^{n-1}_{\mathcal{W}_z}}.
		\notag
	\end{align}
	Substituting \eqref{est-DzjDtheta-2} and \eqref{est-DzjDtheta-3} into \eqref{est-DzjDtheta-1}, we have
	\begin{align}\label{est-6-15-2}
		\sum_{j=0}^{1}\sum_{k=0}^{n-1}\|(\alpha D_{z})^j\tilde{D}^k_\theta \widetilde{\phi}\|^2_{\mathcal{L}^2_{\mathcal{W}_z}}
		&
		\lesssim_\beta \|f\|^2_{\mathcal{H}^{n-1}_{\mathcal{W}_z}}.
	\end{align}
	Inserting \eqref{est-Dm-1}, \eqref{est-6-15-1} and \eqref{est-6-15-2} into \eqref{est-6-15-0}, one obtains that
	\begin{align}\label{est-6-15-theta-3}
		\mathcal{N}_2
		\gtrsim_\beta -\|f\|^2_{\mathcal{H}^{n}_{\mathcal{W}_z}}.
	\end{align}
	As for the last term, we can deduce from the induction assumption \eqref{indu-assum-1} and Lemma \ref{6lem4} that
	\begin{align}\label{est-6-15-theta-4}
		\mathcal{N}_3
		&\geq-\|(\alpha D_{z})^2\tilde{D}^n_\theta \phi\|_{\mathcal{L}^{2}_{\mathcal{W}_z}}
		\|\tilde{D}^n_{\theta}\phi\|_{\mathcal{L}^{2}_{\mathcal{W}_z}}
		\geq -\|(\alpha D_{z})^2\tilde{D}^n_\theta \phi\|_{\mathcal{L}^{2}_{\mathcal{W}_z}}
		\left(\sum_{k=0}^{n-1}\|\tilde{D}^k_{\theta}\partial_{\theta}^2\phi\|_{\mathcal{L}^{2}_{\mathcal{W}_z}}\right)
		\\
		\notag
		&\geq -\frac{\varepsilon}{2}\|(\alpha D_{z})^2\tilde{D}^n_\theta \phi\|^2_{\mathcal{L}^{2}_{\mathcal{W}_z}}
		-C_{\varepsilon}\|f\|^2_{\mathcal{H}^{n-1}_{\mathcal{W}_z}}.
	\end{align}
	We substitute \eqref{est-6-15-theta-2}, \eqref{est-6-15-theta-3} and \eqref{est-6-15-theta-4} into \eqref{est-6-15-theta-1}, to discover that
	\begin{align}\label{64in9}
		-\langle\tilde{D}^n_\theta L_{\theta}(\phi),(\alpha D_{z})^2\tilde{D}^n_\theta \phi\mathcal{W}^2_z\rangle
		&\geq -\varepsilon \|(\alpha D_{z})^2\tilde{D}^{n}_{\theta}\phi\|^2_{\mathcal{L}^2_{\mathcal{W}_z}}-C_{\beta,\varepsilon}\|f\|^2_{\mathcal{H}^{n}_{\mathcal{W}_z}}.
	\end{align}
	Combining the equation \eqref{6model}, \eqref{64in8} and \eqref{64in9}, and taking $\varepsilon$ suitably small, we infer that
	\begin{align}\label{64eq5}
		\|(\alpha D_{z})^2\tilde{D}^n_\theta \phi\|^2_{\mathcal{L}^2_{\mathcal{W}_z}}\lesssim_\beta\|f\|^2_{\mathcal{H}^{n}_{\mathcal{W}_z}}.
	\end{align}
	In view of \eqref{64eq4} and \eqref{64eq5}, we conclude that \eqref{64eq6} holds true.

	The combination of \eqref{64eq6-1} in \textbf{Step 2.1} and \eqref{64eq6} in \textbf{Step 2.2} yields \eqref{est-goal-1} immediately.
	We thus complete the proof of Proposition \ref{6prop5} by standard induction argument.
\end{proof}
\subsection{Elliptic estimates with mixed weight}
In the previous section, we obtained the elliptic estimates with purely radial weights. 
Building upon this foundation, we now focus on the mixed weight and attempt to establish the elliptic estimate with mixed weight.
Firstly, for any parameters $\alpha$, $\beta$, $\eta$ and $\lambda$, which are defined in \eqref{def-lambda}, we denote that 
\begin{align*}
	\|f\|^2_{\tilde{\mathcal{H}}^n}\tri
	\sum_{i=0}^{n}\|D^i_z f\mathcal{W}^{\eta}\|^2_{L^2}+\sum_{0\leq i+j\leq n,j\geq 1}\|D^i_z\tilde{D}^j_{\theta} f\mathcal{W}^{\lambda}\|^2_{L^2},
\end{align*}
where $\mathcal{W}^{\eta}\tri\mathcal{W}_z\sin^{-\frac{\eta}{2}}(2\theta)$ and $\mathcal{W}^{\lambda}\tri\mathcal{W}_z\sin^{-\frac{\lambda}{2}}(2\theta)$.
Then, we derive the following elliptic estimates.

\begin{prop}\label{6prop6}
	Assume that $\langle f,K\rangle_{\theta}=0$ and  that $j=1,2$, $\left|\frac{D^j_z\mathcal{W}_z}{\mathcal{W}_z}\right|\leq C_\beta $, where $C_\beta>0$ is a constant.
	Then for any $n\in\mathbb{N}$, there exists a sufficiently small constant $\alpha>0$ such that
	\begin{align}\label{65eq0}
		\|(\alpha D_{z})^2\phi\|^2_{\tilde{\mathcal{H}}^n} 
		+\|(\alpha D_{z})\partial_{\theta}\phi\|^2_{\tilde{\mathcal{H}}^n} 
		+\|\partial^2_{\theta}\phi\|^2_{\tilde{\mathcal{H}}^n} \leq C_\beta  \|f\|^2_{\tilde{\mathcal{H}}^n}.
	\end{align}
\end{prop}
\begin{proof}
	We will prove \eqref{65eq0} by induction on $n$. 
	\medskip
	
	\noindent
	\textbf{Step 1: For the base case $n=0$.}
	Since $0<\eta<1$, we deduce from Lemma \ref{6lem4} and Proposition \ref{6prop4} and Proposition \ref{6prop5} that
	\begin{align}\label{65eq3}
		-\langle L^{\alpha}_z \phi,\partial^2_{\theta}\phi(\mathcal{W}^\eta)^2\rangle
		&\geq \frac{1-\eta}{1+\eta}\alpha^2\|D_{z}\partial_{\theta}\phi\|^2_{\mathcal{L}^2_{\mathcal{W}^\eta}} -\varepsilon\|\partial^2_{\theta}\phi\|^2_{\mathcal{L}^2_{\mathcal{W}^\eta}}-C_{\beta,\varepsilon} \alpha^2\|D_{z}\partial_{\theta}\phi\|^2_{\mathcal{L}^2_{\mathcal{W}_z}}\\ \notag
		&\geq \frac{1-\eta}{1+\eta}\alpha^2\|D_{z}\partial_{\theta}\phi\|^2_{\mathcal{L}^2_{\mathcal{W}^\eta}} -\varepsilon\|\partial^2_{\theta}\phi\|^2_{\mathcal{L}^2_{\mathcal{W}^\eta}}-C_{\beta,\varepsilon} \|f\|^2_{\mathcal{L}^2_{\mathcal{W}_z}}.
	\end{align}
	While for $L_{\theta}$, it is easy to check that
	\begin{align}\label{65in1-1}
		-\langle L_{\theta}(\phi),\partial^{2}_{\theta}\phi(\mathcal{W}^\eta)^2\rangle
		=&\langle \partial_{\theta}^2\phi,\partial^{2}_{\theta}\phi(\mathcal{W}^\eta)^2\rangle
		-\langle\partial_{\theta}((\tan\theta)\phi),\partial^{2}_{\theta}\phi(\mathcal{W}^\eta)^2\rangle
		+6\langle\phi,\partial^{2}_{\theta}\phi(\mathcal{W}^\eta)^2\rangle.
	\end{align}
	Again utilizing the condition $0<\eta<1$, we apply Proposition \ref{6prop1} and Proposition \ref{6prop5} with $n=0$ and $\xi=\eta$, combined with Lemma \ref{6lem4}, to obtain that
	\begin{align}\label{65eq1}
		-\langle\partial_{\theta}(\tan(\theta)\phi),\partial^{2}_{\theta}\phi(\mathcal{W}^\eta)^2\rangle
		&\geq\frac{1}{2}\|\partial_{\theta}\widetilde{\phi}\|^2_{\mathcal{L}^2_{\mathcal{W}^\eta}}-C\|\widetilde{\phi}\|^2_{\mathcal{L}^2_{\mathcal{W}^\eta}} \geq \frac{1}{2}\|\partial_{\theta}\widetilde{\phi}\|^2_{\mathcal{L}^2_{\mathcal{W}^\eta}}-C\|\partial^2_{\theta}\phi\|^2_{\mathcal{L}^2_{\mathcal{W}_z}}\\ \notag
		&\geq \frac{1}{2}\|\partial_{\theta}\widetilde{\phi}\|^2_{\mathcal{L}^2_{\mathcal{W}^\eta}}-C_\beta \|f\|^2_{\mathcal{L}^2_{\mathcal{W}_z}}.
	\end{align}
	It then follows from $0<\eta<1$, Lemma \ref{6lem4} and Proposition \ref{6prop5} that
	\begin{align}\label{65eq1-1}
		6\langle\phi,\partial^{2}_{\theta}\phi(\mathcal{W}^\eta)^2\rangle
		&\geq -\varepsilon \|\partial^{2}_{\theta}\phi\|^2_{\mathcal{L}^2_{\mathcal{W}^\eta}}-C_{\varepsilon}\|\phi\|^2_{\mathcal{L}^2_{\mathcal{W}^\eta}}
		\geq -\varepsilon \|\partial^{2}_{\theta}\phi\|^2_{\mathcal{L}^2_{\mathcal{W}^\eta}}-C_{\varepsilon}\|\partial_{\theta}^2\phi\|^2_{\mathcal{L}^2_{\mathcal{W}_z}}
		\\
		&\leq -\varepsilon \|\partial^{2}_{\theta}\phi\|^2_{\mathcal{L}^2_{\mathcal{W}^\eta}}-C_{\beta,\varepsilon}\|f\|^2_{\mathcal{L}^2_{\mathcal{W}_z}}.
		\notag
	\end{align}
	Substituting \eqref{65eq1} and \eqref{65eq1-1} into \eqref{65in1-1}, and choosing $\varepsilon$ suitably small, one obtains that
	\begin{align}\label{65eq2}
		-\langle L_{\theta}(\phi),\partial^{2}_{\theta}\phi(\mathcal{W}^\eta)^2\rangle
		&\geq \frac{1}{2} \|\partial^2_{\theta}\phi\|^2_{\mathcal{L}^2_{\mathcal{W}^\eta}}+\frac{1}{2}\|\partial_{\theta}\widetilde{\phi}\|^2_{\mathcal{L}^2_{\mathcal{W}^\eta}}-C_\beta \|f\|^2_{\mathcal{L}^2_{\mathcal{W}_z}}.
	\end{align}
	Recalling \eqref{6model}, then employing \eqref{65eq3} and \eqref{65eq2}, finally choosing $\varepsilon$ suitably small, we infer that
	\begin{align}\label{65eq4}
		\|\partial^2_{\theta}\phi\|^2_{\mathcal{L}^2_{\mathcal{W}^\eta}}+\|\partial_{\theta}\widetilde{\phi}\|^2_{\mathcal{L}^2_{\mathcal{W}^\eta}}\lesssim_\beta  \|f\|^2_{\mathcal{L}^2_{\mathcal{W}^\eta}}.
	\end{align}
	
	Owing to $0<\eta<1$, by virtue of Lemma \ref{6lem4} and Proposition \ref{6prop3}, Proposition \ref{6prop5} with $n=0$, we get that
	\begin{align}\label{65in2}
		-\langle L^{\alpha}_z (\phi),(\alpha D_{z})^2\phi(\mathcal{W}^\eta)^2\rangle
		&\geq\|(\alpha D_{z})^2\phi\|^2_{\mathcal{L}^2_{\mathcal{W}^\eta}}-C_\beta \alpha^3 \|D_{z}\phi\|^2_{\mathcal{L}^2_{\mathcal{W}^\eta}}\\ \notag
		&\geq\|(\alpha D_{z})^2\phi\|^2_{\mathcal{L}^2_{\mathcal{W}^\eta}}-C_\beta \alpha^3\|D_{z}\partial_{\theta}\phi\|^2_{\mathcal{L}^2_{\mathcal{W}_z}}\\ \notag
		&\geq\|(\alpha D_{z})^2\phi\|^2_{\mathcal{L}^2_{\mathcal{W}^\eta}}-C_{\beta} \|f\|^2_{\mathcal{L}^2_{\mathcal{W}_z}}.
	\end{align}
	On the other hand, it is not difficult to find that
	\begin{align}\label{65in1-2}
		-\langle L_{\theta}(\phi),(\alpha D_{z})^2\phi(\mathcal{W}^\eta)^2\rangle
		&=\langle \partial_{\theta}^2\phi,(\alpha D_{z})^2\phi(\mathcal{W}^\eta)^2\rangle
		-\langle\partial_{\theta}((\tan\theta)\phi),(\alpha D_{z})^2\phi(\mathcal{W}^\eta)^2\rangle
		\\
		&~~~+6\langle\phi,(\alpha D_{z})^2\phi(\mathcal{W}^\eta)^2\rangle.
		\notag
	\end{align}
	We can derive that
	\begin{align}\label{65eq1-4}
		\langle \partial_{\theta}^2\phi,(\alpha D_{z})^2\phi(\mathcal{W}^\eta)^2\rangle
		\geq -\varepsilon \|(\alpha D_{z})^2\phi\|^2_{\mathcal{L}^2_{\mathcal{W}^\eta}}-C_\varepsilon \|\partial^2_{\theta}\phi\|^2_{\mathcal{L}^2_{\mathcal{W}^\eta}}.
	\end{align}
	Then, we deduce from Lemma \ref{6lem4}, Lemma \ref{6lem7}, Proposition \ref{6prop2} and Proposition \ref{6prop5} with $n=0$ that
	\begin{align}\label{65eq1-3}
		-\langle\partial_{\theta}(\tan(\theta)\phi),(\alpha D_{z})^2\phi(\mathcal{W}^\eta)^2\rangle
		&\geq-\varepsilon\alpha^2\|D_{z}\widetilde{\phi}\|^2_{\mathcal{L}^2_{\mathcal{W}^\eta}}-C_{\beta,\varepsilon} \|\partial_\theta\phi\|^2_{\mathcal{L}^2_{\mathcal{W}^\eta}}\\ \notag
		&\geq-\varepsilon\alpha^2\|\partial_{\theta}D_{z}\phi\|^2_{\mathcal{L}^2_{\mathcal{W}^\eta}}-C_{\beta,\varepsilon} \|\partial^2_{\theta}\phi\|^2_{\mathcal{L}^2_{\mathcal{W}_z}}\\ \notag
		&\geq -\varepsilon\big(\|(\alpha D_{z})^2\phi\|^2_{\mathcal{L}^2_{\mathcal{W}^\eta}}+\|\partial^2_{\theta}\phi\|^2_{\mathcal{L}^2_{\mathcal{W}^\eta}}\big)-C_{\beta,\varepsilon} \|f\|^2_{\mathcal{L}^2_{\mathcal{W}_z}}.
	\end{align}
	Using a similar derivation of \eqref{65eq1-1}, we obtain that
	\begin{align}\label{65eq1-2}
		6\langle\phi,(\alpha D_{z})^2\phi(\mathcal{W}^\eta)^2\rangle
		\leq& -\varepsilon \|(\alpha D_{z})^2\phi\|^2_{\mathcal{L}^2_{\mathcal{W}^\eta}}-C_{\beta,\varepsilon}\|f\|^2_{\mathcal{L}^2_{\mathcal{W}_z}}.
	\end{align}
	Inserting \eqref{65eq1-4}-\eqref{65eq1-2} into \eqref{65in1-2}, one gets that
	\begin{align}\label{65in1}
		-\langle L_{\theta}(\phi),(\alpha D_{z})^2\phi(\mathcal{W}^\eta)^2\rangle\geq -\varepsilon\|(\alpha D_{z})^2\phi\|^2_{\mathcal{L}^2_{\mathcal{W}^\eta}}-C_{\beta,\varepsilon} \|\partial^2_{\theta}\phi\|^2_{\mathcal{L}^2_{\mathcal{W}^\eta}}
		-C_{\beta,\varepsilon} \|f\|^2_{\mathcal{L}^2_{\mathcal{W}_z}}.
	\end{align}
	Applying \eqref{6model}, \eqref{65in2} and \eqref{65in1}, and choosing $\varepsilon$ suitably small, we deduce that
	\begin{align}\label{65eq5}
		\|(\alpha D_{z})^2\phi\|^2_{\mathcal{L}^2_{\mathcal{W}^\eta}}\lesssim_\beta \|\partial^2_{\theta}\phi\|^2_{\mathcal{L}^2_{\mathcal{W}^\eta}}+\|f\|^2_{\mathcal{L}^2_{\mathcal{W}^\eta}}.
	\end{align}
	Therefore, we conclude from \eqref{65eq4} and \eqref{65eq5} that
	\begin{align}\label{est-case-n0}
		\|(\alpha D_{z})^2\phi\|^2_{\mathcal{L}^2_{\mathcal{W}^\eta}}+\|\partial^2_{\theta}\phi\|^2_{\mathcal{L}^2_{\mathcal{W}^\eta}}+\|\partial_{\theta}\widetilde{\phi}\|^2_{\mathcal{L}^2_{\mathcal{W}^\eta}}\lesssim_\beta \|f\|^2_{\mathcal{L}^2_{\mathcal{W}^\eta}}.
	\end{align}
	This completes the base case $(n=0)$ of the induction, with \eqref{65eq0} following directly from Lemma \ref{6lem7}.
	\medskip

	\noindent
	\textbf{Step 2: For the case $n\geq1$.} 
	We now proceed with the induction argument on the parameter $n\in\mathbb{N}$. 
	As our induction hypothesis, we assume the validity of estimate \eqref{65eq0} for the case $n-1$, which explicitly takes the form:
	\begin{align}\label{indu-assum-2}
		\|(\alpha D_{z})^2\phi\|^2_{\tilde{\mathcal{H}}^{n-1}} 
		+\alpha^2\|D_{z}\partial_{\theta}\phi\|^2_{\tilde{\mathcal{H}}^{n-1}}
		+ \|\partial^2_{\theta}\phi\|^2_{\tilde{\mathcal{H}}^{n-1}}  \lesssim _\beta \|f\|^2_{\tilde{\mathcal{H}}^{n-1}}.
	\end{align}
	We intend to show that \eqref{65eq0} holds for $n$.
	Actually, we can deduce from Lemma \ref{6lem5}, Lemma \ref{6lem7} and \eqref{indu-assum-2} that
	\begin{align*}
		\|(\alpha D_{z})^2\phi\|^2_{\tilde{\mathcal{H}}^{n}}
		+\alpha^2\|D_{z}\partial_{\theta}\phi\|^2_{\tilde{\mathcal{H}}^{n}}+ \|\partial^2_{\theta}\phi\|^2_{\tilde{\mathcal{H}}^{n}}
		&\lesssim _\beta \|f\|^2_{\tilde{\mathcal{H}}^{n-1}}
		+\|(\alpha D_{z})^2D^n_z \phi\|^2_{L^2_{\mathcal{W}^{\eta}}}
		+\|D^n_z \partial_{\theta}^2\phi\|^2_{L^2_{\mathcal{W}^{\eta}}}
		\notag
		\\
		&~~~~
		+\sum_{i+j= n,j\geq 1}
		\left(\|D^i_z\tilde{D}^j_{\theta} (\alpha D_{z})^2\phi\|^2_{L^2_{\mathcal{W}^{\lambda}}}
		+\|D^i_z\tilde{D}^j_{\theta} \partial_{\theta}^2\phi\|^2_{L^2_{\mathcal{W}^{\lambda}}}
		\right).
	\end{align*}
	Hence, we only need to show that for any $n\geq1$,
	\begin{align}\label{est-case-n1}
		&\sum_{i+j= n,j\geq 1}
		\left(\|D^i_z\tilde{D}^j_{\theta} (\alpha D_{z})^2\phi\|^2_{L^2_{\mathcal{W}^{\lambda}}}
		+\|D^i_z\tilde{D}^j_{\theta} \partial_{\theta}^2\phi\|^2_{L^2_{\mathcal{W}^{\lambda}}}
		\right)
		\\
		&\quad+\|(\alpha D_{z})^2D^n_z \phi\|^2_{L^2_{\mathcal{W}^{\eta}}}
		+\|D^n_z \partial_{\theta}^2\phi\|^2_{L^2_{\mathcal{W}^{\eta}}}
		\lesssim _\beta \|f\|^2_{\tilde{\mathcal{H}}^{n}}.~
		\notag
	\end{align}
We divide the proof into the following three steps.
\medskip

\noindent
\textbf{Step 2.1.} 
In this step, we aim to prove that for any $n\geq1$,
\begin{align}\label{65eq8-1}
	\|(\alpha D_{z})^2D^n_z\phi\|^2_{\mathcal{L}^2_{\mathcal{W}^\eta}}+\|\partial^2_{\theta}D^n_z\phi\|^2_{\mathcal{L}^2_{\mathcal{W}^\eta}}+\|\partial_{\theta}D^n_z\widetilde{\phi}\|^2_{\mathcal{L}^2_{\mathcal{W}^\eta}}\lesssim_\beta  \|D^n_zf\|^2_{\mathcal{L}^2_{\mathcal{W}^\eta}}\lesssim _\beta \|f\|^2_{\tilde{\mathcal{H}}^n}.
\end{align}
Since the operators $D_z$ and $L^\alpha_z+L_\theta$ are commutative, we can use a similar derivation of \eqref{est-case-n0}, to deduce that \eqref{65eq8-1} holds.

\medskip

\noindent
\textbf{Step 2.2.} 
Our goal in this step is to prove that for any $n\geq1$,
\begin{align}\label{65eq8}
	\sum_{i+j=2}\|(\alpha D_{z})^i\tilde{D}^n_{\theta}\partial^j_{\theta}\phi\|^2_{\mathcal{L}^2_{\mathcal{W}^\lambda}}\lesssim_\beta \|f\|^2_{\tilde{\mathcal{H}}^{n}}.
\end{align}
For the case $n=1$, we reapplying Proposition \ref{6prop4} in combination with Lemma \ref{6lem4} and \eqref{65eq8-1}, we establish that
\begin{align}\label{65in5-2}
	-\langle \tilde{D}_{\theta}L^{\alpha}_z (\phi),\tilde{D}_{\theta}\partial^2_{\theta}\phi(\mathcal{W}^\lambda)^2\rangle
	&\geq \alpha^2\|D_{z}\tilde{D}_{\theta}\partial_{\theta}\phi\|^2_{\mathcal{L}^2_{\mathcal{W}^\lambda}} -\varepsilon\|\tilde{D}_{\theta}\partial^2_{\theta}\phi\|^2_{\mathcal{L}^2_{\mathcal{W}^\lambda}}
	- C_{\beta,\varepsilon} \alpha^2\|D_{z}\partial_{\theta}\phi\|^2_{\mathcal{L}^2_{\mathcal{W}^\lambda}}\\ \notag
	&\geq \alpha^2\|D_{z}\tilde{D}_{\theta}\partial_{\theta}\phi\|^2_{\mathcal{L}^2_{\mathcal{W}^\lambda}} -\varepsilon\|\tilde{D}_{\theta}\partial^2_{\theta}\phi\|^2_{\mathcal{L}^2_{\mathcal{W}^\lambda}}- C_{\beta,\varepsilon} \alpha^2\|D_{z}\partial_{\theta}^2\phi\|^2_{\mathcal{L}^2_{\mathcal{W}^\eta}}
	\\ \notag
	&\geq 
	\alpha^2\|D_{z}\tilde{D}_{\theta}\partial_{\theta}\phi\|^2_{\mathcal{L}^2_{\mathcal{W}^\lambda}} -\varepsilon\|\tilde{D}_{\theta}\partial^2_{\theta}\phi\|^2_{\mathcal{L}^2_{\mathcal{W}^\lambda}}
	-
	C_{\beta,\varepsilon} \|f\|^2_{\tilde{\mathcal{H}}^{1}}.
\end{align}
One deduces from Proposition \ref{6prop4} and the induction assumption \eqref{indu-assum-2} that for any $n\geq2$,
\begin{align}\label{65in5-1}
	-\langle \tilde{D}^n_{\theta}L^{\alpha}_z (\phi),\tilde{D}^n_{\theta}\partial^2_{\theta}\phi(\mathcal{W}^\lambda)^2\rangle
	&\geq \alpha^2\|D_{z}\tilde{D}^n_{\theta}\partial_{\theta}\phi\|^2_{\mathcal{L}^2_{\mathcal{W}^\lambda}} -\varepsilon\|\tilde{D}^{n}_{\theta}\partial^2_{\theta}\phi\|^2_{\mathcal{L}^2_{\mathcal{W}^\lambda}}
	- C_{\beta,\varepsilon} \alpha^2\|D_{z}\tilde{D}^{n-1}_{\theta}\partial_{\theta}\phi\|^2_{\mathcal{L}^2_{\mathcal{W}^\lambda}}\\ \notag
	&\geq  \alpha^2\|D_{z}\tilde{D}^n_{\theta}\partial_{\theta}\phi\|^2_{\mathcal{L}^2_{\mathcal{W}^\lambda}} -\varepsilon\|\tilde{D}^{n}_{\theta}\partial^2_{\theta}\phi\|^2_{\mathcal{L}^2_{\mathcal{W}^\lambda}}-C_{\beta,\varepsilon} \|f\|^2_{\tilde{\mathcal{H}}^{n-1}}.
\end{align}
Combining \eqref{65in5-2} and \eqref{65in5-1} together, we find that for any $n\geq1$,
\begin{align}\label{65in5}
	-\langle \tilde{D}^n_{\theta}L^{\alpha}_z (\phi),\tilde{D}^n_{\theta}\partial^2_{\theta}\phi(\mathcal{W}^\lambda)^2\rangle
	\geq \alpha^2\|D_{z}\tilde{D}^n_{\theta}\partial_{\theta}\phi\|^2_{\mathcal{L}^2_{\mathcal{W}^\lambda}} -\varepsilon\|\tilde{D}^{n}_{\theta}\partial^2_{\theta}\phi\|^2_{\mathcal{L}^2_{\mathcal{W}^\lambda}}-C_{\beta,\varepsilon} \|f\|^2_{\tilde{\mathcal{H}}^{n}}.
\end{align}
For the angular operator $L_{\theta}$, we recall \eqref{def-L-theta}, to get that
\begin{align}\label{65in4-0}
	-\langle\tilde{D}^n_{\theta}L_{\theta}(\phi),\tilde{D}^n_\theta\partial^{2}_{\theta}\phi(\mathcal{W}^\lambda)^2\rangle
	&=
	\langle\tilde{D}^n_{\theta}\partial_{\theta}^2\phi,\tilde{D}^n_\theta\partial^{2}_{\theta}\phi(\mathcal{W}^\lambda)^2\rangle
	-\langle\tilde{D}^n_{\theta}\partial_{\theta}(\tan\theta\phi),\tilde{D}^n_\theta\partial^{2}_{\theta}\phi(\mathcal{W}^\lambda)^2\rangle
	\\
	&~~~
	+6\langle\tilde{D}^n_{\theta}\phi,\tilde{D}^n_\theta\partial^{2}_{\theta}\phi(\mathcal{W}^\lambda)^2\rangle.
	\notag
\end{align}
Applying Proposition \ref{6prop1} with $n\geq 1$ and $\xi=\lambda$, we infer that
\begin{align}\label{65in4-1}
	-\langle\tilde{D}^n_{\theta}\partial_{\theta}(\tan\theta\phi),\tilde{D}^n_\theta\partial^{2}_{\theta}\phi(\mathcal{W}^\lambda)^2\rangle
	\geq\frac{1}{2}\|\tilde{D}^n_\theta\partial_{\theta}\widetilde{\phi}\|^2_{\mathcal{L}^2_{\mathcal{W}^\lambda}}-C \|\tilde{D}^{n-1}_\theta\partial_{\theta}\widetilde{\phi}\|^2_{\mathcal{L}^2_{\mathcal{W}^\lambda}}-C \sum_{k=0}^{n}\|\tilde{D}^k_\theta\widetilde{\phi}\|^2_{\mathcal{L}^2_{\mathcal{W}^\lambda}}.
\end{align}
In view of Lemma \ref{6lem4} and the induction assumption \eqref{indu-assum-2}, we obtain that for any $1\leq m\leq n$, 
\begin{align}
	\label{65in3}
	\|\tilde{D}^m_\theta\widetilde{\phi}\|^2_{\mathcal{L}^2_{\mathcal{W}^\lambda}} 
	&\lesssim \sum_{j=0}^{m}\|\sin^{m}(2\theta)(\cos\theta)^{j-m-1}\partial^j_{\theta}\phi\|^2_{\mathcal{L}^2_{\mathcal{W}^\lambda}}
    \lesssim \sum_{j=0}^{m}\|\sin^{j-1}(2\theta)\partial^j_{\theta}\phi\|^2_{\mathcal{L}^2_{\mathcal{W}^\lambda}}
    \\ \notag
	&\lesssim \big\|\frac{\phi}{\sin(2\theta)}\big\|^2_{\mathcal{L}^2_{\mathcal{W}^\lambda}}+\big\|\partial_{\theta}\phi\big\|^2_{\mathcal{L}^2_{\mathcal{W}^\lambda}}+\sum_{j=1}^{m}\|\sin^{j}(2\theta)\partial^{j+1}_{\theta}\phi\|^2_{\mathcal{L}^2_{\mathcal{W}^\lambda}} 
	\\ \notag
	&\lesssim \|\partial^2_{\theta}\phi\|^2_{\mathcal{L}^2_{\mathcal{W}^\eta}} 
	+\sum_{j=1}^{m-1}\|\tilde{D}^j_{\theta}\partial^{2}_{\theta}\phi\|^2_{\mathcal{L}^2_{\mathcal{W}^\lambda}}
	\\ \notag
	&\lesssim_\beta  \|f\|^2_{\tilde{\mathcal{H}}^{m-1}}.
\end{align}
Similar argument leads us to get that
\begin{align*}
	\|\tilde{D}^{m-1}_\theta\partial_{\theta}\widetilde{\phi}\|^2_{\mathcal{L}^2_{\mathcal{W}^\lambda}} 
	\lesssim \sum_{j=0}^{m-1}\|\sin^{m-1}(2\theta)(\cos\theta)^{j-m} \partial^j_{\theta}\phi\|^2_{\mathcal{L}^2_{\mathcal{W}^\lambda}} 
	\lesssim \sum_{j=0}^{m-1}\|\sin^{j-1}(2\theta)\partial^j_{\theta}\phi\|^2_{\mathcal{L}^2_{\mathcal{W}^\lambda}} 
	\lesssim_\beta  \|f\|^2_{\tilde{\mathcal{H}}^{m-1}}.
\end{align*}
We then substitute two estimates above into \eqref{65in4-1}, to find that
\begin{align}\label{65in4-2}
	-\langle\tilde{D}^n_{\theta}\partial_{\theta}(\tan\theta\phi),\tilde{D}^n_\theta\partial^{2}_{\theta}\phi(\mathcal{W}^\lambda)^2\rangle
	&\geq\frac{1}{2}\|\tilde{D}^n_\theta\partial_{\theta}\widetilde{\phi}\|^2_{\mathcal{L}^2_{\mathcal{W}^\lambda}}-C_\beta \|f\|^2_{\tilde{\mathcal{H}}^{n-1}}.
\end{align}
It is easy to verify that
\begin{align}\label{65in4-3}
	6\langle\tilde{D}^n_{\theta}\phi,\tilde{D}^n_\theta\partial^{2}_{\theta}\phi(\mathcal{W}^\lambda)^2\rangle
	\geq&-\varepsilon\|\tilde{D}^n_\theta\partial^2_{\theta}\phi\|^2_{\mathcal{L}^2_{\mathcal{W}^\lambda}}
	-C_{\varepsilon}\|\tilde{D}^n_\theta\phi\|^2_{\mathcal{L}^2_{\mathcal{W}^\lambda}}
	\\
	\geq&-\varepsilon\|\tilde{D}^n_\theta\partial^2_{\theta}\phi\|^2_{\mathcal{L}^2_{\mathcal{W}^\lambda}}
	-C_{\varepsilon}\|\partial_{\theta}^2\phi\|^2_{\mathcal{L}^2_{\mathcal{W}^\eta}}
	-C_{\varepsilon}\sum_{k=1}^{n-1}\|\tilde{D}^k_\theta\partial_{\theta}^2\phi\|^2_{\mathcal{L}^2_{\mathcal{W}^\lambda}}
	\notag\\
	\geq&-\varepsilon\|\tilde{D}^n_\theta\partial^2_{\theta}\phi\|^2_{\mathcal{L}^2_{\mathcal{W}^\lambda}}
	-C_{\beta,\varepsilon} \|f\|^2_{\tilde{\mathcal{H}}^{n-1}}.
	\notag
\end{align}
Plugging \eqref{65in4-2} and \eqref{65in4-3} into \eqref{65in4-0}, and choosing $\varepsilon$ suitably small, we find that
\begin{align}\label{65in4}
	-\langle\tilde{D}^n_{\theta}L_{\theta}(\phi),\tilde{D}^n_\theta\partial^{2}_{\theta}\phi(\mathcal{W}^\lambda)^2\rangle
	\geq\frac{1}{2}\|\tilde{D}^n_\theta\partial^2_{\theta}\phi\|^2_{\mathcal{L}^2_{\mathcal{W}^\lambda}}+\frac{1}{2}\|\tilde{D}^n_\theta\partial_{\theta}\widetilde{\phi}\|^2_{\mathcal{L}^2_{\mathcal{W}^\lambda}}-C_\beta \|f\|^2_{\tilde{\mathcal{H}}^{n-1}}.
\end{align}
Recalling the equation \eqref{6model}, then employing \eqref{65in5} and \eqref{65in4}, finally taking $\varepsilon$ suitably small, we infer that
\begin{align}\label{65eq6}
	\alpha^2\|D_{z}\tilde{D}^n_{\theta}\partial_{\theta}\phi\|^2_{\mathcal{L}^2_{\mathcal{W}^\lambda}} +\|\tilde{D}^{n}_{\theta}\partial^2_{\theta}\phi\|^2_{\mathcal{L}^2_{\mathcal{W}^\lambda}}+\|\tilde{D}^n_\theta\partial_{\theta}\widetilde{\phi}\|^2_{\mathcal{L}^2_{\mathcal{W}^\lambda}}
	\lesssim_\beta \|f\|^2_{\tilde{\mathcal{H}}^{n}}.
\end{align}

It follows from a similar way as Proposition \ref{6prop3} that
\begin{align}\label{65in6-0}
	-\langle \tilde{D}^{n}_{\theta}L^{\alpha}_z (\phi),(\alpha D_{z})^2\tilde{D}^{n}_{\theta}\phi(\mathcal{W}^\lambda)^2\rangle
	&\geq\|(\alpha D_{z})^2\tilde{D}^{n}_{\theta}\phi\|^2_{\mathcal{L}^2_{\mathcal{W}^\lambda}}
	-C_\beta \alpha^3\|D_{z
	}\tilde{D}^{n}_{\theta}\phi\|^2_{\mathcal{L}^2_{\mathcal{W}^\lambda}}
	\\
	&\geq\|(\alpha D_{z})^2\tilde{D}^{n}_{\theta}\phi\|^2_{\mathcal{L}^2_{\mathcal{W}^\lambda}}-C_\beta \alpha^3\|D_{z}\tilde{D}^{n-1}_{\theta}\partial_{\theta}\phi\|^2_{\mathcal{L}^2_{\mathcal{W}^\lambda}}.
	\notag
\end{align}
As for $n=1$, applying \eqref{65eq8-1} and \eqref{65in6-0}, we can get that
\begin{align*}
	-\langle \tilde{D}_{\theta}L^{\alpha}_z (\phi),(\alpha D_{z})^2\tilde{D}_{\theta}\phi(\mathcal{W}^\lambda)^2\rangle
	&\geq 
	\|(\alpha D_{z})^2\tilde{D}_{\theta}\phi\|^2_{\mathcal{L}^2_{\mathcal{W}^\lambda}}
	-C_{\beta}\alpha^3\|D_z\partial_{\theta}^2\phi\|^2_{\mathcal{L}^2_{\mathcal{W}^\eta}}
	\\
	&\geq\|(\alpha D_{z})^2\tilde{D}_{\theta}\phi\|^2_{\mathcal{L}^2_{\mathcal{W}^\lambda}}-C_\beta \|f\|^2_{\tilde{\mathcal{H}}^{1}}.
\end{align*}
Using the induction assumption \eqref{indu-assum-2}, we can deduce from \eqref{65in6-0} that for any $n\geq2$
\begin{align*}
	-\langle \tilde{D}^{n}_{\theta}L^{\alpha}_z (\phi),(\alpha D_{z})^2\tilde{D}^{n}_{\theta}\phi(\mathcal{W}^\lambda)^2\rangle
	&\geq\|(\alpha D_{z})^2\tilde{D}^{n}_{\theta}\phi\|^2_{\mathcal{L}^2_{\mathcal{W}^\lambda}}-C_\beta \|f\|^2_{\tilde{\mathcal{H}}^{n-1}}.
\end{align*}
A combination of two estimates above yields that
\begin{align}\label{65in6}
	-\langle \tilde{D}^{n}_{\theta}L^{\alpha}_z (\phi),(\alpha D_{z})^2\tilde{D}^{n}_{\theta}\phi(\mathcal{W}^\lambda)^2\rangle
	&\geq\|(\alpha D_{z})^2\tilde{D}^{n}_{\theta}\phi\|^2_{\mathcal{L}^2_{\mathcal{W}^\lambda}}-C_\beta \|f\|^2_{\tilde{\mathcal{H}}^{n}}.
\end{align}
While for $L_{\theta}$ defined in \eqref{def-L-theta}, one deduces that
\begin{align}\label{65in9-0}
	\langle\tilde{D}^n_\theta L_{\theta}(\phi),(\alpha D_{z})^2\tilde{D}^n_\theta \phi(\mathcal{W}^\lambda)^2\rangle
	&=-\langle\tilde{D}^n_\theta \partial^2_{\theta}\phi,(\alpha D_{z})^2\tilde{D}^n_\theta \phi(\mathcal{W}^\lambda)^2\rangle
	\\
	&~~~+\langle\tilde{D}^n_\theta \partial_{\theta}((\tan\theta)\phi),(\alpha D_{z})^2\tilde{D}^n_\theta \phi(\mathcal{W}^\lambda)^2\rangle
	\notag
	\\
	&~~~-6\langle\tilde{D}^n_\theta \phi,(\alpha D_{z})^2\tilde{D}^n_\theta \phi(\mathcal{W}^\lambda)^2\rangle.
	\notag
\end{align}
With the help of \eqref{65eq6}, we obtain that
\begin{align}\label{65in9-1}
	\left|\langle\tilde{D}^n_\theta \partial^2_{\theta}\phi,(\alpha D_{z})^2\tilde{D}^n_\theta \phi(\mathcal{W}^\lambda)^2\rangle\right|
	&\leq \frac{\varepsilon}{2}\|(\alpha D_{z})^2\tilde{D}^{n}_{\theta}\phi\|^2_{\mathcal{L}^2_{\mathcal{W}^\lambda}}
	+C_{\varepsilon}\|\tilde{D}^n_\theta \partial^{2}_{\theta}\phi\|^2_{\mathcal{L}^2_{\mathcal{W}^\lambda}}
	\\
	&\leq \frac{\varepsilon}{2}\|(\alpha D_{z})^2\tilde{D}^{n}_{\theta}\phi\|^2_{\mathcal{L}^2_{\mathcal{W}^\lambda}}
	+C_{\beta,\varepsilon}\|f\|^2_{\tilde{\mathcal{H}}^{n}}.
	\notag
\end{align}
Using Proposition \ref{6prop2} with $\xi=\lambda$, we find that
\begin{align}\label{65in9-2}
	\langle\tilde{D}^n_\theta\partial_{\theta}(\tan(\theta)\phi),(\alpha D_{z})^2\tilde{D}^n_\theta \phi(\mathcal{W}^\lambda)^2\rangle
	&\lesssim_\beta  \sum^1_{i+j=0}\|(\alpha D_{z})^i\tilde{D}^n_\theta\partial^{j}_{\theta}\widetilde{\phi}\|^2_{\mathcal{L}^2_{\mathcal{W}^\lambda}} 
	+\sum_{k=0}^{n-1}\sum_{j=0}^{1}\|(\alpha D_{z})^j\tilde{D}^k_\theta \widetilde{\phi}\|^2_{\mathcal{L}^2_{\mathcal{W}^\lambda}}.
\end{align}
Thanks to \eqref{65eq6}, the induction assumption \eqref{indu-assum-2} and Lemma \ref{6lem4}, we see that
\begin{align}\label{65in7}
	\sum_{i+j=1}\|(\alpha D_{z})^i\tilde{D}^n_\theta\partial^{j}_{\theta}\widetilde{\phi}\|^2_{\mathcal{L}^2_{\mathcal{W}^\lambda}}
	&\lesssim \sum_{i+j=1}\sum_{k=0}^{n-1} \|(\alpha D_{z})^i\tilde{D}^k_{\theta}\partial_{\theta}^{j+1}\phi\|^2_{\mathcal{L}^2_{\mathcal{W}^\lambda}}
	\\
	&\lesssim 
	\sum_{i+j=1}\|(\alpha D_{z})^i\tilde{D}_{\theta}\partial_{\theta}^{j+1}\phi\|^2_{\mathcal{L}^2_{\mathcal{W}^\lambda}}
	+\sum_{i+j=1}\sum_{k=1}^{n-1} \|(\alpha D_{z})^i\tilde{D}^k_{\theta}\partial_{\theta}^{j+1}\phi\|^2_{\mathcal{L}^2_{\mathcal{W}^\lambda}}
	\notag
	\\
	&\lesssim_\beta  \|f\|^2_{\tilde{\mathcal{H}}^{n}}.
	\notag
\end{align}
It then follows from Lemma \ref{6lem4}, the induction assumption \eqref{indu-assum-2} and \eqref{65eq8-1} that
\begin{align}\label{65in8}
	\sum_{j=0}^{1}\sum_{k=0}^{n-1}\|(\alpha D_{z})^j\tilde{D}^k_\theta \widetilde{\phi}\|^2_{\mathcal{L}^2_{\mathcal{W}^\lambda}}
	&\lesssim \sum_{j=0}^{1}\sum_{i=0}^{n-1}\|(\alpha D_{z})^j\sin^{i-1}(2\theta)\partial^i_\theta \phi\|^2_{\mathcal{L}^2_{\mathcal{W}^\lambda}}
	\\
	\notag
	&\lesssim \sum_{j=0}^{1}\sum_{i=0}^{n-1}\|\partial_{\theta}(\alpha D_{z})^j\tilde{D}^{i}_{\theta}\phi\|^2_{\mathcal{L}^2_{\mathcal{W}^\lambda}}\\ \notag
	&\lesssim \sum_{j=0}^{1}\sum_{i=0}^{n-1}\|(\alpha D_{z})^j\tilde{D}^{i}_{\theta}\partial_{\theta}\phi\|^2_{\mathcal{L}^2_{\mathcal{W}^\lambda}}\\ \notag
	&\lesssim 
	\sum_{j=0}^{1}\|(\alpha D_{z})^j\partial_{\theta}^2\phi\|^2_{\mathcal{L}^2_{\mathcal{W}^\eta}}
	+
	\sum_{i=1}^{n-1}\sum_{k+j=2}\|(\alpha D_{z})^j\tilde{D}^{i}_{\theta}\partial^k_{\theta}\phi\|^2_{\mathcal{L}^2_{\mathcal{W}^\lambda}}\\ \notag
	&\lesssim_\beta  \|f\|^2_{\tilde{\mathcal{H}}^{n}}.
\end{align}
Substituting \eqref{65in3} with $m=n$, \eqref{65in7} and \eqref{65in8} into \eqref{65in9-2}, we find that
\begin{align}\label{65in9-3}
	\langle\tilde{D}^n_\theta\partial_{\theta}(\tan(\theta)\phi),(\alpha D_{z})^2\tilde{D}^n_\theta \phi(\mathcal{W}^\lambda)^2\rangle
	&\lesssim_\beta \|f\|^2_{\tilde{\mathcal{H}}^{n}}.
\end{align}
It is easy to verify that
\begin{align}\label{65in9-4}
	\left|
	6\langle\tilde{D}^n_\theta \phi,(\alpha D_{z})^2\tilde{D}^n_\theta \phi(\mathcal{W}^\lambda)^2\rangle
	\right|    
	&\leq \frac{\varepsilon}{2}\|(\alpha D_{z})^2\tilde{D}^{n}_{\theta}\phi\|^2_{\mathcal{L}^2_{\mathcal{W}^\lambda}}
	+C_{\varepsilon}\|\tilde{D}^n_\theta\phi\|^2_{\mathcal{L}^2_{\mathcal{W}^\lambda}}
	\\
	&\lesssim \frac{\varepsilon}{2}\|(\alpha D_{z})^2\tilde{D}^{n}_{\theta}\phi\|^2_{\mathcal{L}^2_{\mathcal{W}^\lambda}}
	+C_{\varepsilon}\sum_{k=0}^{n-1}\|\tilde{D}^k_\theta\partial_{\theta}^2\phi\|^2_{\mathcal{L}^2_{\mathcal{W}^\eta}}
	\notag\\
	&\lesssim \frac{\varepsilon}{2}\|(\alpha D_{z})^2\tilde{D}^{n}_{\theta}\phi\|^2_{\mathcal{L}^2_{\mathcal{W}^\lambda}}
	+C_{\beta,\varepsilon}\|f\|^2_{\tilde{\mathcal{H}}^{n-1}}.
	\notag
\end{align}
We insert \eqref{65in9-1}, \eqref{65in9-3} and \eqref{65in9-4} into \eqref{65in9-0}, to discover that
\begin{align}\label{65in9}
	\left|\langle\tilde{D}^n_\theta L_{\theta}(\phi),(\alpha D_{z})^2\tilde{D}^n_\theta \phi(\mathcal{W}^\lambda)^2\rangle\right|
	\leq\varepsilon\|(\alpha D_{z})^2\tilde{D}^{n}_{\theta}\phi\|^2_{\mathcal{L}^2_{\mathcal{W}^\lambda}}+C_\beta \|f\|^2_{\tilde{\mathcal{H}}^{n}}.
\end{align}
Combining \eqref{65in6} and \eqref{65in9} together, and taking $\varepsilon$ suitably small, we infer that
\begin{align}\label{65eq7}
	\|(\alpha D_{z})^2\tilde{D}^{n}_{\theta}\phi\|^2_{\mathcal{L}^2_{\mathcal{W}^\lambda}}\lesssim_\beta \|f\|^2_{\tilde{\mathcal{H}}^{n}}.
\end{align}
Therefore, we conclude from \eqref{65eq6} and \eqref{65eq7} that \eqref{65eq8} holds.

\medskip

\noindent
\textbf{Step 2.3.} In this step, we intend to show that
\begin{align}\label{65eq7-1}
	\sum_{i+j=n,i,1\leq j\leq n-1}\|(\alpha D_{z})^2\tilde{D}^j_{\theta}(D^i_z\phi)\|^2_{\mathcal{L}^2_{\mathcal{W}^\lambda}}+\|\partial^2_{\theta}\tilde{D}^j_{\theta}(D^i_z\phi)\|^2_{\mathcal{L}^2_{\mathcal{W}^\lambda}}
	\lesssim_\beta  \|f\|^2_{\tilde{\mathcal{H}}^n}.
\end{align}
We have shown in \eqref{65eq8} that for any $1\leq k\leq n$,
\begin{align*}
	\sum_{i+j=2}\|(\alpha D_z)^i\tilde{D}^k_{\theta}\partial^j_{\theta}\phi\|^2_{\mathcal{L}^2_{\mathcal{W}^\lambda}}\lesssim_\beta \|f\|^2_{\tilde{\mathcal{H}}^{k}}.
\end{align*}
Along a similar way as \eqref{65eq8-1}, we can deduce from the estimate above that
\begin{align*}
	\sum_{i+j=2}\|(\alpha D_z)^i\tilde{D}^k_{\theta}\partial^j_{\theta}D_z^{n-k}\phi\|^2_{\mathcal{L}^2_{\mathcal{W}^\lambda}}
	\lesssim_\beta \|D^{n-k}_z f\|^2_{\tilde{\mathcal{H}}^{k}}
	\lesssim_\beta \|f\|^2_{\tilde{\mathcal{H}}^{n}}.
\end{align*}
This implies that \eqref{65eq7-1} holds.

Combining \eqref{65eq8-1} in \textbf{Step 2.1}, \eqref{65eq8} in \textbf{Step 2.2} and \eqref{65eq7-1} in \textbf{Step 2.3}, we can obtain \eqref{est-case-n1} immediately.
As a consequence, we complete the proof of this proposition by the standard induction argument.
\end{proof}

Recalling Lemma \ref{6lem5} and Proposition \ref{6prop6}, and taking $\mathcal{W}_z=w_z$ or $w_z^\ast$, we can derive the following two corollaries.
\begin{coro}\label{6cor1}
Assume that $\langle f,K\rangle_{\theta}=0$ and $f\in\mathcal{H}^n$, there exists a sufficiently small constant $\alpha>0$ such that
\begin{align*}
	\|(\alpha D_{z})^2\phi\|^2_{\mathcal{H}^n}
	+\|(\alpha D_{z})\partial_{\theta}\phi\|^2_{\mathcal{H}^n}
	+\|\partial^2_{\theta}\phi\|^2_{\mathcal{H}^n}
	\lesssim_\beta \|f\|^2_{\mathcal{H}^n}.
\end{align*}
\end{coro}
\begin{coro}\label{6cor2}
Assume that $\langle f,K\rangle_{\theta}=0$ and $f\in\mathcal{H}^{n,\ast}$, there exists a sufficiently small constant $\alpha>0$ such that
\begin{align*}
	\|(\alpha D_{z})^2\phi\|^2_{\mathcal{H}^{n,\ast}}
	+\|(\alpha D_{z})\partial_{\theta}\phi\|^2_{\mathcal{H}^{n,\ast}}
	+\|\partial^2_{\theta}\phi\|^2_{\mathcal{H}^{n,\ast}}
	\lesssim_\beta \|f\|^2_{\mathcal{H}^{n,\ast}}.
\end{align*}
\end{coro}

Now, it remains to estimate $\mathcal{E}^{n}$-norm.
More precisely, we obtain the following proposition.
\begin{prop}\label{6prop6-19}
Assume that $\langle f,K\rangle_{\theta}=0$.
Then for any $n\in\mathbb{N}$, there exists a sufficiently small constant $\alpha>0$ such that
\begin{align*}
	\|(\alpha D_{z})^2D_z^n\phi\|^2_{\mathcal{L}^2_{w^{\lambda}}}
	+\|(\alpha D_{z})\partial_{\theta}D_z^n\phi\|^2_{\mathcal{L}^2_{w^{\lambda}}}
	+\|\partial^2_{\theta}D_z^n\phi\|^2_{\mathcal{L}^2_{w^{\lambda}}}\lesssim_\beta  \|D_z^nf\|^2_{\mathcal{L}^2_{w^{\lambda}}}.
\end{align*}
\end{prop}
\begin{proof}
It follows from Proposition \ref{6prop5} that
\begin{align}\label{65in10}
	\|(\alpha D_{z})^2\phi\|^2_{\mathcal{L}^{2}_{w_z}} + \|\partial^2_{\theta}\phi\|^2_{\mathcal{L}^{2}_{w_z}} \lesssim_\beta  \|f\|^2_{\mathcal{L}^2_{w_z}}.
\end{align}
Owing to $|\lambda-1|\leq\frac{\alpha}{\beta}$, in view of \eqref{65in10}, Lemma \ref{6lem4}, Lemma \ref{6lem7}, Proposition \ref{6prop4} with $\xi=\lambda$ and Proposition \ref{6prop6} with $n=0$, one obtains that
\begin{align*}
	-\langle L^{\alpha}_z (\phi),\partial^2_{\theta}\phi (w^{\lambda})^2\rangle
	&\geq -\frac{\alpha}{\beta}\|(\alpha D_{z})\partial_{\theta}\phi\|^2_{\mathcal{L}^2_{w^{\lambda}}} -C\varepsilon\|\partial^2_{\theta}\phi\|^2_{\mathcal{L}^2_{w^{\lambda}}}-C_\beta \alpha^2\|\partial_{\theta}D_z\phi\|^2_{\mathcal{L}^{2}_{w_z}}
	-C_{\beta,\varepsilon} \alpha^2\|\partial_{\theta}D_z\phi\|^2_{\mathcal{L}^{2}_{w^{\eta}}}
	\\
	&\geq -\frac{\alpha}{\beta}\|(\alpha D_{z})^2\phi\|^2_{\mathcal{L}^2_{w^{\lambda}}} -C\left(\frac{\alpha}{\beta}+\varepsilon\right)\|\partial^2_{\theta}\phi\|^2_{\mathcal{L}^2_{w^{\lambda}}}-C_{\beta,\varepsilon}\|f\|^2_{\mathcal{L}^{2}_{w^{\eta}}}.
	\notag
\end{align*}
By virtue of Lemma \ref{6lem4}, Proposition \ref{6prop1}, Proposition \ref{6prop6}, we infer that
\begin{align*}
	-\langle L_{\theta}(\phi),\partial^{2}_{\theta}\phi (w^{\lambda})^2\rangle
	\geq
	\|\partial^2_{\theta}\phi\|^2_{\mathcal{L}^2_{w^{\lambda}}}-C\|\partial_\theta^2\phi\|^2_{\mathcal{L}^2_{w^{\eta}}}-\varepsilon\|\partial^2_{\theta}\phi\|^2_{\mathcal{L}^2_{w^{\lambda}}}
	-C_{\varepsilon} \|\partial^2_{\theta}\phi\|^2_{\mathcal{L}^{2}_{w^{\eta}}}
    \geq \frac{1}{2}\|\partial^2_{\theta}\phi\|^2_{\mathcal{L}^2_{w^{\lambda}}}-C_\beta \|f\|^2_{\mathcal{L}^{2}_{w^{\eta}}}.
\end{align*}
These two estimates above, together with the equation \eqref{6model}, yields directly by taking $\varepsilon$ and $\alpha$ suitably small that
\begin{align}\label{65in11}
	\|\partial^2_{\theta}\phi\|^2_{\mathcal{L}^2_{w^{\lambda}}}\lesssim_\beta  \alpha\|(\alpha D_{z})^2\phi\|^2_{\mathcal{L}^2_{w^{\lambda}}}+\|f\|^2_{\mathcal{L}^{2}_{w^{\lambda}}}.
\end{align}

Applying Lemma \ref{6lem4}, Proposition \ref{6prop3} and Proposition \ref{6prop6}, we find that
\begin{align*}
	-\langle L^{\alpha}_z (\phi),(\alpha D_{z})^2\phi (w^{\lambda})^2\rangle
	\geq\|(\alpha D_{z})^2\phi\|^2_{\mathcal{L}^2_{w^{\lambda}}}-C_\beta \alpha^2\|(\alpha D_z)\partial_{\theta}\phi\|^2_{\mathcal{L}^2_{w^\eta}}
	\geq\|(\alpha D_{z})^2\phi\|^2_{\mathcal{L}^2_{w^{\lambda}}}-C_\beta \|f\|^2_{\mathcal{L}^2_{w^\eta}}.
\end{align*}
Since $|\lambda-1|\leq\frac{\alpha}{\beta}$, we then deduce from Proposition \ref{6prop2} with $\xi=\lambda$ and Lemma \ref{6lem4} that
\begin{align*}
	-\langle L_{\theta}(\phi),(\alpha D_{z})^2\phi (w^{\lambda})^2\rangle
	&\geq
	-\varepsilon\|(\alpha D_{z})^2\phi\|^2_{\mathcal{L}^2_{w^{\lambda} }}
	-C_{\beta,\varepsilon}\|\partial^2_{\theta}\phi\|^2_{w^{\lambda}}-\left(\frac{\alpha}{\beta}+\varepsilon\right) \|(\alpha D_{z})\partial_{\theta}\phi\|^2_{\mathcal{L}^2_{w^{\lambda} }}
	\\
	&\geq -\left(\frac{\alpha}{\beta}+\varepsilon\right)\|(\alpha D_{z})^2\phi\|^2_{\mathcal{L}^2_{w^{\lambda} }}-C_{\beta,\varepsilon} \|\partial^2_{\theta}\phi\|^2_{w^{\lambda}}.
\end{align*}
Collecting the two estimates above and the equation \eqref{6model}, and taking $\varepsilon$ suitably small, we find that
\begin{align}\label{65in12}
	\|(\alpha D_{z})^2\phi\|^2_{\mathcal{L}^2_{w^{\lambda}}}\lesssim_\beta  \|\partial^2_{\theta}\phi\|^2_{\mathcal{L}^2_{w^{\lambda}}}+\|f\|^2_{\mathcal{L}^2_{w^{\lambda}}}.
\end{align}
Combining Lemma \ref{6lem7} with estimates \eqref{65in11} and \eqref{65in12}, and selecting $\alpha$ sufficiently small, we obtain that
\begin{align*}
	\|(\alpha D_{z})^2\phi\|^2_{\mathcal{L}^2_{w^{\lambda}}}
	+\|(\alpha D_{z})\partial_{\theta}\phi\|^2_{\mathcal{L}^2_{w^{\lambda}}}
	+\|\partial^2_{\theta}\phi\|^2_{\mathcal{L}^2_{w^{\lambda}}}\lesssim_\beta  \|f\|^2_{\mathcal{L}^2_{w^{\lambda}}}.
\end{align*}
Since $D_z$ commutes with operators $L_{z}^{\alpha}$ and $L_{\theta}$, we can easily derive from the estimate above that for any $n\geq1$,
\begin{align*}
	\|(\alpha D_{z})^2D_z^n\phi\|^2_{\mathcal{L}^2_{w^{\lambda}}}
	+\|(\alpha D_{z})\partial_{\theta}D_z^n\phi\|^2_{\mathcal{L}^2_{w^{\lambda}}}
	+\|\partial^2_{\theta}D_z^n\phi\|^2_{\mathcal{L}^2_{w^{\lambda}}}\lesssim_\beta  \|D_z^nf\|^2_{\mathcal{L}^2_{w^{\lambda}}}.
\end{align*}
We thus complete the proof of Proposition \ref{6prop6-19}.
\end{proof}

Based on Proposition \ref{6prop6-19}, we obtain the following corollary directly by recalling the definition of $\mathcal{E}^n$.
\begin{coro}\label{6cor3}
Assume that $\langle f,K\rangle_{\theta}=0$ and $f\in\mathcal{E}^{n}$.
Then there exists a sufficiently small constant $\alpha>0$ such that
\begin{align*}
	\|(\alpha D_{z})^2\phi\|^2_{\mathcal{E}^{n}}
	+\|(\alpha D_{z})\partial_{\theta}\phi\|^2_{\mathcal{E}^{n}}
	+\|\partial^2_{\theta}\phi\|^2_{\mathcal{E}^{n}}\lesssim_\beta \|f\|^2_{\mathcal{E}^{n}}.
\end{align*}
\end{coro}

\section{Estimates of transport term}\label{sec:Est-T}
In the previous section, we have established the elliptic estimates for equation \eqref{6model}, namely Corollaries \ref{6cor1}, \ref{6cor2} and \ref{6cor3}. 
Building upon these corollaries, we now begin to study the transport term in this section:
$$T=\frac{U(\Phi)}{\sin(2\theta)}D_{\theta}F+V(\Phi)\alpha D_{z}F.$$
We stress that estimating this transport term is crucial for the analysis of the critical regularity problem. 
We need to employ some special structures, such as the null structure. 
Recall from \eqref{com-F} that the transport term $T$ admits the following decomposition as:
$$T=T_g+T_{F^{\ast}_{\gamma}},$$
where for any function $f$ (representing either $g$ or $F^{\ast}_{\gamma}$), the operator $T_f$ is defined as
\begin{align*}
T_f
\tri
\frac{U(\Phi)}{\sin(2\theta)}D_{\theta}f+V(\Phi)\alpha D_{z}f.
\end{align*}
We recall that the Sobolev spaces $\mathcal{H}^{-1}$, $\mathcal{H}^{2}_{\eta}$ and $\mathcal{E}^2_{\eta}$ are respectively defined by \eqref{def_norm_H-1}-\eqref{def_norm_E2eta}, whereas the definitions of $\mathcal{H}^{2}$, $\mathcal{H}^{*,2}$ and $\mathcal{E}^2$ are provided in Section \ref{sec:notation}.

\subsection{Basic lemmas and the null structure}
Let $D=[0,\infty)\times[0,\frac{\pi}{2}]$ and $\beta\in(0,1]$. 
Before proceeding, we first present some useful lemmas in this section, which are particularly crucial when dealing with the transport term later.
\begin{lemm}\label{7em1}
Let $f(z,\theta)|_{\partial D}=0$. Then we have
$$\|f\|_{L^{\infty}_{z}}\leq C\|D_{z}f w_{z}\|_{L^{2}_{z}},~~~~\|f\|_{L^{\infty}_{z}}\leq C\|D_{z}f w^{\ast}_{z}\|_{L^{2}_{z}},$$
and 
$$\|f\|_{L^{\infty}_{\theta}}\leq C\sqrt{\frac{\beta}{\alpha}}\|D_{\theta}f w_{\theta}^{\lambda}\|_{L^{2}_{\theta}}.$$   
\end{lemm}
\begin{proof}
Since $\beta\in(0,1]$, one gets, by recalling the definitions of $D_z$ and $w_z$ in Section \ref{sec:notation}, that
\begin{align}\label{7in1}
	\|f\|^2_{L^{\infty}_{z}}&\lesssim \left(\int_0^{\infty}|\partial_z f| dz\right)^2 \lesssim \int_0^{\infty}|D_z f|^2(w_z)^2 dz\int_0^{\infty}(zw_z)^{-2} dz \\ \notag
	&\lesssim \|D_{z}fw_{z}\|^2_{L^{2}_{z}}\int_0^{\infty}(1+z)^{-2} dz \lesssim \|D_{z}fw_{z}\|^2_{L^{2}_{z}}.
\end{align}
From the definition of $w^{\ast}_{z}$ in Section \ref{sec:notation}, the same manner yields immediately that
\begin{align}\label{7in2}
	\|f\|^2_{L^{\infty}_{z}}
	&\lesssim \|D_{z}fw^{\ast}_{z}\|^2_{L^{2}_{z}}\int_0^{\infty}(1+z)^{-\frac{3}{2}} dz 
	\lesssim \|D_{z}fw^{\ast}_{z}\|^2_{L^{2}_{z}}.
\end{align}
Following the definition \eqref{def-lambda}, the parameter $\lambda$ satisfies $\lambda=1+\frac{\alpha}{10\beta}>1$.
It then follows from a similar way that
\begin{align}\label{7in3}
	\|f\|^2_{L^{\infty}_{\theta}}
	&\lesssim \int_0^{\frac{\pi}{2}}|D_\theta f|^2(w^\lambda_\theta)^2 d\theta\int_0^{\frac{\pi}{2}}(\sin(2\theta)w^\lambda_\theta)^{-2} d\theta 
	\lesssim \frac{\beta}{\alpha}\|D_{\theta}fw^\lambda_\theta\|^2_{L^{2}_{\theta}}.
\end{align}
We thus complete the proof of this lemma.
\end{proof}

Remembering the definitions of some weight functions $w_{z}$, $w_{\theta}^{\eta}$, $w^{\eta}$, $w^{\ast,\eta}$, $w^{\lambda}$ and $w^{\ast,\lambda}$ in Section \ref{sec:notation}, we have the following corollary. 
Its proof is similar to that of Lemma \ref{7em1}, so it is omitted here.
\begin{coro}\label{7em2}
Let $f(z,\theta)|_{\partial D}=0$. Then we have
$$\|f w_{\theta}^{\eta}\|
_{L^{\infty}_{z}L^{2}_{\theta}}
\leq C\|D_{z}f w^\eta\|_{L^{2}},
~~~~
\|fw_{\theta}^{\eta}\|
_{L^{\infty}_{z}L^{2}_{\theta}}
\leq C\|D_{z}fw^{\ast,\eta}\|_{L^{2}},$$
and 
$$\|f w_z\|_{L^2_z L^{\infty}_{\theta}}\leq C\sqrt{\frac{\beta}{\alpha}}\|D_{\theta}f w^{\lambda}\|_{L^{2}}.$$
Moreover, there holds 
$$\|f\|_{L^{\infty}}\leq C\sqrt{\frac{\beta}{\alpha}}\|D_{\theta}D_z fw^{\lambda}\|_{L^{2}},
~~~
\|f\|_{L^{\infty}}\leq C\sqrt{\frac{\beta}{\alpha}}\|D_{\theta}D_z fw^{\ast,\lambda}\|_{L^{2}}.$$
\end{coro}

\begin{lemm}\label{7ha1}
Let $f(z,\theta)|_{\partial D}=0$ and $\tilde{G}_f(z)=\frac{3}{4\alpha}z^{-\frac{5}{\alpha}}\int_{0}^{z}\rho^{\frac{5}{\alpha}-1}\langle f, K\rangle_{\theta}d\rho$. Then we have
$$\|D^k_z\tilde{G}_fw_z\|_{L^{2}_{z}}\leq C\|f\|_{\mathcal{H}^{k}},~~~~\|D^k_z\tilde{G}_fw^{\ast}_z\|_{L^{2}_{z}}\leq C\|f\|_{\mathcal{H}^{k,\ast}},$$
and 
$$\alpha\|D^{k+1}_z\tilde{G}_fw_z\|_{L^{2}_{z}}\leq C\|f\|_{\mathcal{H}^{k}},~~~~\alpha\|D^{k+1}_z\tilde{G}_fw^{\ast}_z\|_{L^{2}_{z}}\leq C\|f\|_{\mathcal{H}^{k,\ast}}.$$   
\end{lemm}
\begin{proof}
After some direct calculations, it is not difficult to check that $w_z\approx 1+z^{-\frac{2}{\beta}}$. 
This, together with integration by parts, yields that
\begin{align*}
	\|\tilde{G}_fw_z\|^2_{L^{2}_{z}}&\lesssim \|\tilde{G}_f\|^2_{L^{2}_{z}}+ \|\tilde{G}_fz^{-\frac{2}{\beta}}\|^2_{L^{2}_{z}} \\ \notag
	&\lesssim \frac{1}{\alpha^2}\frac{\alpha}{|\alpha-10|}\left|\int_0^{\infty}\partial_z(z^{-\frac{10}{\alpha}+1})\left(\int_{0}^{z}\rho^{\frac{5}{\alpha}-1}\langle f, K\rangle_{\theta}d\rho\right)^2dz\right| \\ \notag
	&~~~+ \frac{1}{\alpha^2}\frac{\alpha\beta}{|\alpha\beta-10\beta-4\alpha|}\left|\int_0^{\infty}\partial_z(z^{-\frac{10}{\alpha}-\frac{4}{\beta}+1})\left(\int_{0}^{z}\rho^{\frac{5}{\alpha}-1}\langle f, K\rangle_{\theta}d\rho\right)^2dz\right| \\ \notag
	&\lesssim \left|\int_0^{\infty}\langle f, K\rangle_{\theta}\tilde{G}_fdz\right|+\left|\int_0^{\infty}z^{-\frac{4}{\beta}} \langle f, K\rangle_{\theta}\tilde{G}_fdz\right|  \\ \notag
	&\lesssim \|\tilde{G}_f\|_{L^{2}_{z}}\|fw_\theta^\eta\|_{L^{2}}+ \|\tilde{G}_fz^{-\frac{2}{\beta}}\|_{L^{2}_{z}}\|fz^{-\frac{2}{\beta}}w_\theta^\eta\|_{L^{2}}\\ \notag
	&\lesssim \|\tilde{G}_fw_z\|_{L^{2}_{z}}\|fw^\eta\|_{L^{2}},
\end{align*}
which directly implies that
\begin{align}\label{7in4}
	\|\tilde{G}_fw_z\|_{L^{2}_{z}}\lesssim \|fw^\eta\|_{L^{2}}.
\end{align}
Since $w^\ast_z\approx z^{-\frac{1}{4}}+z^{-\frac{1}{\beta}-\frac{1}{4}}$, it then follows from a similar way as \eqref{7in4} that
\begin{align}\label{7in5}
	\|\tilde{G}_fw^\ast_z\|_{L^{2}_{z}}\lesssim \|fw^{\ast,\eta}\|_{L^{2}}.
\end{align}
Since $D_z=z\partial_z$, one can get, by some straightforward computations, that
\begin{align}
	D_z \tilde{G}_f
	&=-\frac{15}{4\alpha^2}z^{-\frac{5}{\alpha}}\int_{0}^{z}\rho^{\frac{5}{\alpha}-1}\langle f, K\rangle_{\theta}d\rho+\frac{3}{4\alpha}\langle f, K\rangle_{\theta} \label{DzGf}\\
	&=-\frac{3}{4\alpha}z^{-\frac{5}{\alpha}}\int_{0}^{z}\partial_{\rho}(\rho^{\frac{5}{\alpha}})\langle f, K\rangle_{\theta}d\rho+\frac{3}{4\alpha}\langle f, K\rangle_{\theta} \notag\\
	&=\frac{3}{4\alpha}z^{-\frac{5}{\alpha}}\int_{0}^{z}\rho^{\frac{5}{\alpha}-1}\langle D_{\rho}f, K\rangle_{\theta}d\rho \notag\\
	&= \tilde{G}_{D_z f}.\notag
\end{align}
The above equality implies that the operators $D_z$ and $\tilde{G}$ commute.
Following the same derivation procedure as in \eqref{7in4} and \eqref{7in5}, we deduce from \eqref{DzGf} that
\begin{align}\label{7in6}
	\|D^k_z\tilde{G}_fw_z\|_{L^{2}_{z}}\leq C\|f\|_{\mathcal{H}^{k}},~~~~\|D^k_z\tilde{G}_fw^{\ast}_z\|_{L^{2}_{z}}\leq C\|f\|_{\mathcal{H}^{k,\ast}}.
\end{align}
It then follows from \eqref{DzGf} that
\begin{align*}
	\alpha D_z \tilde{G}_f
	&=-\frac{15}{4\alpha}z^{-\frac{5}{\alpha}}\int_{0}^{z}\rho^{\frac{5}{\alpha}-1}\langle f, K\rangle_{\theta}d\rho+\frac{3}{4}\langle f, K\rangle_{\theta} 
	=-5\tilde{G}_f+\frac{3}{4}\langle f,K\rangle_{\theta},
\end{align*}
which, along with \eqref{7in6}, gives that
\begin{align}\label{7in7}
	\alpha\|D^{k+1}_z\tilde{G}_fw_z\|_{L^{2}_{z}}\leq C\|f\|_{\mathcal{H}^{k}},~~~~\alpha\|D^{k+1}_z\tilde{G}_fw^{\ast}_z\|_{L^{2}_{z}}\leq C\|f\|_{\mathcal{H}^{k,\ast}}.
\end{align}
The proof of Lemma \ref{7ha1} is thereby completed.
\end{proof}

As is well known, when considering the solutions of the Euler equations \eqref{1eq1} in the $H^3$ framework, by exploiting the divergence-free condition when dealing with the transport term, one can obtain that $\langle u\cdot \nabla u,u \rangle=0$. This is the so-called null structure.
Our next goal is to find suitable weight functions so that a similar null structure can be achieved.
More precisely, we have the following lemma.
\begin{lemm}\label{7div1}
Let $f(z,\theta)|_{\partial D}=0$. 
For any $\xi\in[0,+\infty)$, we denote that 
\begin{align}\label{def-wxi}
	\overline{w}^\xi\tri w^\xi z^{\frac{1}{2}-\frac{3}{2\alpha}}(\cos\theta)^{-\frac{1}{2}},
\end{align}
then we have
\begin{align}
	\label{stru-Tf-1}
	\Big\langle \frac{1}{\overline{w}^\xi}T_{\overline{w}^\xi f},(w^\xi)^2f\Big\rangle=0.
\end{align}
\end{lemm}
\begin{rema}
In what follows, we mainly employ the cases $\xi=\eta$ or $\lambda$ from Lemma \ref{7div1}.
\end{rema}
\begin{proof}[The proof of Lemma \ref{7div1}]
By virtue of integrating by parts, we deduce that
\begin{align}\label{est-Tf-1}
	\Big\langle \frac{1}{\overline{w}^\xi}T_{\overline{w}^\xi f},(w^\xi)^2f\Big\rangle
	&=\Big\langle \frac{1}{\overline{w}^\xi}T_{\overline{w}^\xi f},(\overline{w}^\xi)^2f z^{\frac{3}{\alpha}-1}\cos\theta \Big\rangle
    =\frac{1}{2} \big\langle T_{(\overline{w}^\xi f)^2},z^{\frac{3}{\alpha}-1}\cos\theta\big\rangle
    \\
	&=\frac{1}{2}\big\langle (\overline{w}^\xi f)^2,\partial_\theta(U(\Phi)z^{\frac{3}{\alpha}-1}\cos\theta)+\alpha\partial_z(V(\Phi)z^{\frac{3}{\alpha}}\cos\theta)\big\rangle.
	\notag
\end{align}
Thanks to \eqref{def_U-V}, we infer that
\begin{align}\label{est-Tf-2}
	&~~~~\partial_\theta(U(\Phi)z^{\frac{3}{\alpha}-1}\cos\theta)+\alpha\partial_z(V(\Phi)z^{\frac{3}{\alpha}}\cos\theta) \\
	&=\partial_\theta((-3\Phi-\alpha D_z \Phi)z^{\frac{3}{\alpha}-1}\cos\theta)+\alpha\partial_z((\partial_\theta \Phi-\tan \theta\Phi)z^{\frac{3}{\alpha}}\cos\theta) \notag\\
	&=(-3\partial_\theta\Phi-\alpha \partial_\theta D_z \Phi)z^{\frac{3}{\alpha}-1}\cos\theta+(3\Phi+\alpha D_z \Phi)z^{\frac{3}{\alpha}-1}\sin\theta \notag\\
	&~~~~+\alpha(\partial_\theta D_z \Phi-\tan\theta D_z\Phi)z^{\frac{3}{\alpha}-1}\cos\theta+3(\partial_\theta \Phi-\tan\theta\Phi)z^{\frac{3}{\alpha}-1}\cos\theta \notag\\
	&=0.\notag
\end{align}
We then substitute \eqref{est-Tf-2} into \eqref{est-Tf-1}, to find that \eqref{stru-Tf-1} holds.
Thus, we finish the proof of this lemma.
\end{proof}

\subsection{Transport estimates}
Based on the elliptic estimates for equation \eqref{6model} shown in Corollaries \ref{6cor1}, \ref{6cor2} and \ref{6cor3}, the transport null structure shown in Lemma \ref{7div1}, we will give the \textit{a priori} energy estimates for $\mathcal{H}^{-1}$-norm, $\mathcal{H}^{2}_{\eta}$-norm and $\mathcal{E}^2_{\eta}$-norm of $T_g$.

\begin{prop}\label{7Tg}
Let $L^{-1}_{z,K}(g)(0)=0$, there exist constants $\alpha>0$ sufficiently small and $\eta(\beta)$, such that if $\alpha\ll 1-\eta\ll \beta$ and $|\mu|\leq \alpha^\frac{1}{2}$, then 
\begin{align}
	\big|\langle T_g, g(w^K)^2\rangle\big|
	\leq&~ \frac{1}{10}\left(\|g\|^2_{\mathcal{H}^{-1}}+\|g\|^2_{\mathcal{H}_{\eta}^2}\right)+C_\beta\alpha^{\frac{1}{2}}(1-\eta)^{-\frac{1}{2}}\|g\|^2_{\mathcal{H}^2}+C_\beta \alpha^{-1}\|g\|^3_{\mathcal{H}^2},
	\label{est-norm-H-1}
	\\
	\big|\langle T_g, g\rangle_{\mathcal{H}_{\eta}^2}\big|
	\leq&~ \frac{1}{10}\|g\|^2_{\mathcal{H}_{\eta}^2}
	+C_\beta\alpha^{\frac{1}{2}}(1-\eta)^{-\frac{5}{2}}\|g\|^2_{\mathcal{H}^2}
    +C_\beta\alpha^{-1}(1-\eta)^{-2}\|g\|^3_{\mathcal{H}^2}
	+C_\beta\alpha^{-1}
	\|g\|^2_{\mathcal{H}^{2}}\|g\|_{\mathcal{E}^{2}},
	\label{est-norm-H2eta}
	\\ 
	\big|\langle T_g, g\rangle_{\mathcal{E}_{\eta}^2}\big|
	\leq&~ \frac{1}{10}(\|g\|^2_{\mathcal{E}_{\eta}^2}+\|g\|^2_{\mathcal{H}_{\eta}^2})+C_\beta\alpha^{\frac{1}{2}}(1-\eta)^{-\frac{1}{2}}(\|g\|^2_{\mathcal{H}^2}+\|g\|^2_{\mathcal{E}^2}) \label{est-norm-E2eta}\\
	&+C_\beta
	\alpha^{-1}\left(\|g\|_{\mathcal{E}^{2}}+\|g\|_{\mathcal{H}^{2}}\right)
	\|g\|_{\mathcal{E}^{2}}\|g\|_{\mathcal{H}^{2}}.
	\notag
\end{align}
\end{prop}
\begin{proof}
We will divide this proof into three steps according to the definitions of different norms.

\medskip

\noindent\textbf{Step 1: Estimate for $\mathcal{H}^{-1}$-norm.}

\medskip

\noindent 
It is easy to check that the $\mathcal{H}^{-1}$-norm, which is defined by \eqref{def_norm_H-1}, can be decomposed as:
\begin{align}\label{est-bI-1}
	\langle T_g, g(w^K)^2\rangle
	=\langle T_g^l, g(w^K)^2\rangle
	+\langle T_g^n, g(w^K)^2\rangle
	\tri \bar{I}_1+\bar{I}_2,
\end{align}
where $T^l_g$ and $T^n_g$ are defined by
\begin{align*}
	T^l_g\tri&\frac{U(\frac{3}{4\alpha}L^{-1}_{z,K}(F^{\ast}_\gamma)\sin(2\theta))}{\sin(2\theta)}D_{\theta}g+V\bigg(\frac{3}{4\alpha}L^{-1}_{z,K}(F^{\ast}_\gamma)\sin(2\theta)\bigg)\alpha D_{z}g,
	\\
	T^n_g\tri&\frac{U(\Phi-\frac{3}{4\alpha}L^{-1}_{z,K}(F^{\ast}_\gamma)\sin(2\theta))}{\sin(2\theta)}D_{\theta}g+V\bigg(\Phi-\frac{3}{4\alpha}L^{-1}_{z,K}(F^{\ast}_\gamma)\sin(2\theta)\bigg)\alpha D_{z}g.
\end{align*}
Now, we estimate the first term $\bar{I}_1$.
In view of \eqref{form_F1}, we have
\begin{align}
	L^{-1}_{z,K}(F^{\ast}_\gamma)=\frac{2\alpha}{3}L^{-1}_{z}(\Gamma^{\ast}_{\gamma}).
\end{align}
The preceding equality, combined with \eqref{3f7} and \eqref{form_F1}, necessarily implies that
\begin{align}
	\frac{U(\frac{3}{4\alpha}L^{-1}_{z,K}(F^{\ast}_\gamma)\sin(2\theta))}{\sin(2\theta)}
	&= -\frac{3}{2}L^{-1}_{z}(\Gamma^{\ast}_{\gamma})-\frac{\alpha}{2}D_{z}L^{-1}_{z}(\Gamma^{\ast}_{\gamma})
	=-\frac{3}{2}L^{-1}_{z}(\Gamma^{\ast}_{\gamma})+\frac{\alpha}{2}\Gamma^{\ast}_{\gamma},
	\label{cal_U}
	\\
	V\bigg(\frac{3}{4\alpha}L^{-1}_{z,K}(F^{\ast}_\gamma)\sin(2\theta)\bigg)
	&= L^{-1}_{z}(\Gamma^{\ast}_{\gamma})(\cos(2\theta)-\sin^2(\theta)).
	\label{cal_V}
\end{align}
Applying \eqref{cal_U} and \eqref{cal_V}, we can simplify $T^l_g$ to:
\begin{align}\label{cal_Tgl}
	T^l_g
	&=\frac12\left(-3L^{-1}_{z}(\Gamma^{\ast}_{\gamma})+\alpha\Gamma^{\ast}_{\gamma}\right)D_{\theta}g+L^{-1}_{z}(\Gamma^{\ast}_{\gamma})
	\left(\cos(2\theta)-\sin^2(\theta)\right)\alpha D_{z}g.
\end{align}
Notice that 
\begin{align}\label{w-rela}
	|w^K|\leq|w^\eta|\leq|w^\lambda|,
\end{align} 
one gets, by employing \eqref{3f7} and \eqref{cal_Tgl}, that
\begin{align}\label{7g1}
	|\bar{I}_1|
	&\lesssim \left(1+\frac{\alpha}{\gamma}\right)\|g w^K\|_{L^2}\|D_\theta g w^\lambda\|_{L^2}+\alpha\|g w^K\|_{L^2}\|D_z g w^\eta\|_{L^2} \\ \notag
	&\lesssim \left(1+\frac{\alpha}{\beta}\right)\|g \|_{\mathcal{H}^{-1}}(1-\eta)\| g \|_{\mathcal{H}^2_{\eta}}+\alpha\|g\|_{\mathcal{H}^{-1}}(1-\eta)^{-1}\| g \|_{\mathcal{H}^2_{\eta}} \\ \notag
	&\leq \frac{1}{100}\|g \|^2_{\mathcal{H}^{-1}}+\frac{1}{100}\| g \|^2_{\mathcal{H}^2_{\eta}}.
\end{align}
Before estimating $\bar{I}_2$, we recall \eqref{decom-Phi} and \eqref{com-F}, to divide $\Phi$ into the following six terms:
\begin{align}\label{decom-Phi-1}   \Phi=\tilde{\Phi}_g+\tilde{G}_g\sin(2\theta)+G^\ast_g\sin(2\theta)+\tilde{\Phi}_{F^\ast_{\gamma}}+\tilde{G}_{F^\ast_{\gamma}}\sin(2\theta)+G^\ast_{F^\ast_{\gamma}}\sin(2\theta),\end{align}
where for any function $f$ (representing either $g$ or $F_\gamma^\ast$), $\tilde{G}_f$ and $G^\ast_f$ are defined as
\begin{align*}
	\tilde{G}_f(z)\tri\frac{3}{4\alpha}z^{-\frac{5}{\alpha}}\int_{0}^{z}\rho^{\frac{5}{\alpha}-1}\langle f, K\rangle_{\theta}d\rho,
	\quad
	G^\ast_f(z)=\frac{3}{4\alpha}L^{-1}_{z,K}(f).
\end{align*}
According to \eqref{decom-Phi-1}, we may split $\bar{I}_2$ into the following two terms:
\begin{align}\label{est-bI2-1}
	\bar{I}_2=
	\langle T_g^{n,1}, g(w^K)^2\rangle
	+\langle T_g^{n,2}, g(w^K)^2\rangle
	\tri \bar{I}_{21}+\bar{I}_{22},
\end{align}
where $T_g^{n,1}$ and $T_g^{n,2}$ are defined by
\begin{align*}
	T^{n,1}_g\tri&
	\left(\frac{U(\tilde{\Phi}_{F^\ast_{\gamma}})}{\sin(2\theta)}+U(\tilde{G}_{F^\ast_{\gamma}})\right)D_{\theta}g
	+V\left(\tilde{\Phi}_{F^\ast_{\gamma}}+\tilde{G}_{F^\ast_{\gamma}}\sin(2\theta)\right)\alpha D_{z}g,
	\\
	T^{n,2}_g\tri&\left(\frac{U(\tilde{\Phi}_g)}{\sin(2\theta)}+U(\tilde{G}_g)+U(G^\ast_g)\right)D_{\theta}g
	+V\left(\tilde{\Phi}_g+\tilde{G}_g\sin(2\theta)+G^\ast_g\sin(2\theta)\right)\alpha D_{z}g.
\end{align*}
With the help of Lemma \ref{7fm1}, Proposition \ref{6prop0}, Corollary \ref{6cor2} and Corollary \ref{7em2}, we infer that
\begin{align}\label{7g3}
	\bigg\|\frac{U(\tilde{\Phi}_{F^\ast_{\gamma}})}{\sin(2\theta)}\bigg\|_{L^{\infty}}
	\lesssim_{\beta} \alpha^{-\frac12}\|\partial_\theta U(\tilde{\Phi}_{F^\ast_{\gamma}})\|_{\mathcal{H}^{2,\ast}}\lesssim_\beta\alpha^{-\frac{1}{2}}\| F^\ast_{\gamma} \|_{\mathcal{H}^{2,\ast}}\lesssim_\beta \alpha^{\frac{1}{2}}(1-\eta)^{-\frac{1}{2}}.
\end{align}
In view of Lemma \ref{7fm1}, Lemma \ref{7em1} and Lemma \ref{7ha1}, one obtains that
\begin{align}\label{7g4}
	\|U(\tilde{G}_{F^\ast_{\gamma}})\|_{L^{\infty}}\lesssim \|D_z U(\tilde{G}_{F^\ast_{\gamma}})w^{\ast}_z\|_{L^{2}_z}  \lesssim\| F^\ast_{\gamma} \|_{\mathcal{H}^{1,\ast}}\lesssim_\beta \alpha(1-\eta)^{-\frac{1}{2}}.
\end{align}
The application of Lemma \ref{7fm1}, Proposition \ref{6prop0}, Corollary \ref{6cor2} and Corollary \ref{7em2} yields that
\begin{align}\label{7g5}
	\|V(\tilde{\Phi}_{F^\ast_{\gamma}})\|_{L^{\infty}}
	\lesssim_{\beta} \alpha^{-\frac12}\|\partial_\theta \tilde{\Phi}_{F^\ast_{\gamma}}\|_{\mathcal{H}^{2,\ast}} \lesssim_\beta\alpha^{-\frac{1}{2}}\| F^\ast_{\gamma} \|_{\mathcal{H}^{2,\ast}}\lesssim_\beta \alpha^{\frac{1}{2}}(1-\eta)^{-\frac{1}{2}}.
\end{align}
Again thanks to Lemma \ref{7fm1}, Lemma \ref{7em1} and Lemma \ref{7ha1}, we see that
\begin{align}\label{7g6}   \|V(\tilde{G}_{F^\ast_{\gamma}}\sin(2\theta)\|_{L^{\infty}}\lesssim \|D_z \tilde{G}_{F^\ast_{\gamma}}w^{\ast}_z\|_{L^{2}_z} \lesssim\| F^\ast_{\gamma} \|_{\mathcal{H}^{1,\ast}}\lesssim_\beta \alpha(1-\eta)^{-\frac{1}{2}}.
\end{align}
Combining \eqref{7g3}-\eqref{7g6} together, we get, by using \eqref{w-rela}, that
\begin{align}\label{7g7}
	|\bar{I}_{21}|
	&\lesssim_\beta \alpha^{\frac{1}{2}}(1-\eta)^{-\frac{1}{2}}\|g w^\eta\|_{L^2}(\|D_\theta g w^\lambda\|_{L^2}+\|D_z g w^\eta\|_{L^2}) 
	\lesssim_\beta \alpha^{\frac{1}{2}}(1-\eta)^{-\frac{1}{2}}\|g\|^2_{\mathcal{H}^2}.
\end{align}
According to Proposition \ref{6prop0}, Corollary \ref{6cor1} and Corollary \ref{7em2}, we infer that
\begin{align}\label{7g8}
	\bigg\|\frac{U(\tilde{\Phi}_{g})}{\sin(2\theta)}\bigg\|_{L^{\infty}}
	\lesssim_{\beta} \alpha^{-\frac12}\|\partial_\theta U(\tilde{\Phi}_{g})\|_{\mathcal{H}^{2}} 
	\lesssim_{\beta} \alpha^{-\frac12}(\|\partial_\theta \tilde{\Phi}_{g}\|_{\mathcal{H}^{2}}+\alpha\|\partial_\theta D_z\tilde{\Phi}_{g}\|_{\mathcal{H}^{2}}) \lesssim_\beta \alpha^{-\frac12} \|g\|_{\mathcal{H}^{2}}.
\end{align}
By virtue of Lemma \ref{7em1} and Lemma \ref{7ha1}, we get that
\begin{align}\label{7g9}
	\|U(\tilde{G}_{g})\|_{L^{\infty}}\lesssim \|D_z U(\tilde{G}_{g})w_z\|_{L^{2}_z} \lesssim \|D_z \tilde{G}_{g}w_z\|_{L^{2}_z}+\alpha\|D^2_z \tilde{G}_{g}w_z\|_{L^{2}_z}\lesssim\|g\|_{\mathcal{H}^{2}}.
\end{align}
We can deduce from Lemma \ref{7em1} that
\begin{align}\label{7g10}
	\|U(G^\ast_{g})\|_{L^{\infty}}
	\lesssim \|D_z U(G^\ast_{g})w_z\|_{L^{2}_z}  
	\lesssim \alpha^{-1}\|\langle g,K\rangle_{\theta}w_z\|_{L^{2}_z}+\|\langle D_z g,K\rangle_{\theta}w_z\|_{L^{2}_z}
	\lesssim\alpha^{-1}\|g\|_{\mathcal{H}^{2}}.
\end{align}
By using Proposition \ref{6prop0}, Corollary \ref{6cor1} and Corollary \ref{7em2}, it is easy to check that
\begin{align}\label{7g11}
	\|V(\tilde{\Phi}_{g})\|_{L^{\infty}}
	\lesssim_{\beta} \alpha^{-\frac12}\|\partial_\theta \tilde{\Phi}_{g}\|_{\mathcal{H}^{2}} \lesssim_\beta \alpha^{-\frac12}\|g\|_{\mathcal{H}^{2}}.
\end{align}
We employ Lemma \ref{7em1} and Lemma \ref{7ha1} once again, to find that
\begin{align}\label{7g12}
	\|V(\tilde{G}_{g}\sin(2\theta))\|_{L^{\infty}}&\lesssim \|D_z \tilde{G}_{g}w_z\|_{L^{2}_z}  \lesssim\|g\|_{\mathcal{H}^{2}}.
\end{align}
Taking advantage of Lemmas \ref{7em1}, we get that
\begin{align}\label{7g13}
	\|V(G^\ast_g\sin(2\theta))\|_{L^{\infty}}\lesssim \|D_z G^\ast_g w_z\|_{L^{2}_z}  
	\lesssim \alpha^{-1}\|\langle g,K\rangle_{\theta}w_z\|_{L^{2}_z} 
	\lesssim \alpha^{-1}\|g\|_{\mathcal{H}^{2}}.
\end{align}
Then the combination of \eqref{w-rela} and \eqref{7g8}-\eqref{7g13}, gives that
\begin{align}\label{7g14}
	|\bar{I}_{22}|
	&\lesssim_\beta \alpha^{-1}\|g\|_{\mathcal{H}^{2}}\|g w^\eta\|_{L^2}(\|D_\theta g w^\lambda\|_{L^2}+\|D_z g w^\eta\|_{L^2}) 
	\lesssim_\beta \alpha^{-1}\|g\|^3_{\mathcal{H}^2}.
\end{align}
Substituting \eqref{7g7} and \eqref{7g14} into \eqref{est-bI2-1}, we achieve that
\begin{align}\label{est-bI2-2}
	|\bar{I}_2| \lesssim_{\beta} \alpha^{\frac{1}{2}}(1-\eta)^{-\frac{1}{2}}\|g\|^2_{\mathcal{H}^2}
	+\alpha^{-1}\|g\|^3_{\mathcal{H}^2}.
\end{align}
Plugging \eqref{7g1} and \eqref{est-bI2-2} into \eqref{est-bI-1}, one obtains that \eqref{est-norm-H-1} holds.

\medskip

\noindent
\textbf{Step 2: Estimate for $\mathcal{H}^{2}_{\eta}$-norm.} 

\medskip

\noindent
We recall the definition of $\mathcal{H}^{2}_{\eta}$ in \eqref{def_norm_H2eta}, to find that
\begin{align}\label{est-bJ-1}
	\langle T_g, g\rangle_{\mathcal{H}^2_{\eta}}
	=&~ (1-\eta)^{2}\langle T_g, g(w^\eta)^2\rangle
	+\left[(1-\eta)^{-2}\langle D_{\theta}T_{g}, D_{\theta}g(w^\lambda)^2\rangle
    +(1-\eta)^{2}\langle D_{z}T_{g}, D_{z}g(w^\eta)^2\rangle\right]
	\\
	&~
	+\left[\langle D^2_{\theta}T_{g}, D^2_{\theta}g(w^\lambda)^2\rangle
	+(1-\eta)^{4}\langle D^2_{z}T_{g}, D^2_{z}g(w^\eta)^2\rangle
	+(1-\eta)^2\langle D_zD_{\theta}T_{g}, D_zD_{\theta}g(w^\lambda)^2\rangle\right]
	\notag\\
	\tri& \bar{J}_1+\bar{J}_2+\bar{J}_3.
	\notag
\end{align}
Here $\bar{J}_1$, $\bar{J}_2$ and $\bar{J}_3$ are related to the $\mathcal{L}^{2}_{\eta}$, $\dot{\mathcal{H}}^{1}_{\eta}$ and $\dot{\mathcal{H}}^{2}_{\eta}$ norms, respectively.

\medskip

\noindent
\textbf{Step 2.1: Estimate for $\mathcal{L}^{2}_{\eta}$-norm.}

\medskip

\noindent
We can split $\mathcal{L}^2_{\eta}$-norm into the following two parts:
\begin{align}\label{est-bJ1-1}
	\bar{J}_1
	=(1-\eta)^{2}\langle T_g^d, g(w^\eta)^2\rangle
	-(1-\eta)^{2}\langle T_g^e, g(w^\eta)^2\rangle
	\tri \bar{J}_{11}+\bar{J}_{12}.
\end{align}
Here $T^d_g$ and $T^e_g$ are defined by
\begin{align*}
	T^d_g\tri&\frac{1}{\overline{w}^\eta}\bigg[\frac{U(\Phi)}{\sin(2\theta)}D_{\theta}(g\overline{w}^\eta)+V(\Phi)\alpha D_{z}(g\overline{w}^\eta)\bigg]=\frac{1}{\overline{w}^\eta}T_{g\overline{w}^\eta},
	\\
	T^e_g\tri&\frac{1}{\overline{w}^\eta}g\bigg[\frac{U(\Phi)}{\sin(2\theta)}D_{\theta}(\overline{w}^\eta)+V(\Phi)\alpha D_{z}(\overline{w}^\eta)\bigg]=\frac{1}{\overline{w}^\eta}gT_{\overline{w}^\eta},
\end{align*}
where $\overline{w}^\eta$ is defined by \eqref{def-wxi} with $\xi=\eta$.
Whereas it follows from Lemma \ref{7div1} with $\xi=\eta$ that
\begin{align}\label{7g16}
	\bar{J}_{11}=0.
\end{align}
Then by exactly the same procedure as that in the decomposition of \eqref{est-bI-1}, we can deduce from \eqref{cal_U} and \eqref{cal_V} that
\begin{align}\label{est-bJ12-1}
	\bar{J}_{12}
	=-(1-\eta)^{2}\langle T_g^{e,l}, g(w^\eta)^2\rangle
	-(1-\eta)^{2}\langle T_g^{e,n}, g(w^\eta)^2\rangle
	\tri \bar{J}_{121}+\bar{J}_{122},
\end{align}
where $T^{e,l}_g$ and $T^{e,n}_g$ are defined by
\begin{align*}
	T^{e,l}_g
	\tri&\frac{1}{\overline{w}^\eta}g\left[\frac12\left(-3L^{-1}_{z}(\Gamma^{\ast}_{\gamma})+\alpha\Gamma^{\ast}_{\gamma}\right)D_{\theta}\overline{w}^\eta+L^{-1}_{z}(\Gamma^{\ast}_{\gamma})(\cos(2\theta)-\sin^2(\theta))\alpha D_{z}\overline{w}^\eta\right],
	\\
	T^{e,n}_g
	\tri&\frac{1}{\overline{w}^\eta}g\left[\frac{U(\Phi-\frac{3}{4\alpha}L^{-1}_{z,K}(F^{\ast}_\gamma)\sin(2\theta))}{\sin(2\theta)}D_{\theta}\overline{w}^\eta+V\bigg(\Phi-\frac{3}{4\alpha}L^{-1}_{z,K}(F^{\ast}_\gamma)\sin(2\theta)\bigg)\alpha D_{z}\overline{w}^\eta\right].
\end{align*}
Recalling the definition of $\overline{w}^\eta$ in \eqref{def-wxi}, one can infer that
\begin{align}
	\frac{D_{\theta}\overline{w}^\eta}{\overline{w}^\eta}
	=&\frac{D_{\theta}w^\eta}{w^\eta}+\frac{D_{\theta}(\cos^{-\frac{1}{2}}(\theta))}{\cos^{-\frac{1}{2}}(\theta)}
	=-\eta \cos(2\theta)+\sin^2(\theta),
	\label{7g17}
	\\
	\frac{D_{z}\overline{w}^\eta}{\overline{w}^\eta}
	=&\frac{D_{z}w^\eta}{w_\eta}+\frac{D_{z}(z^{\frac{1}{2}-\frac{3}{2\alpha}})}{z^{\frac{1}{2}-\frac{3}{2\alpha}}} 
	=-\frac{1}{\beta}L^{-1}_{z}(\Gamma^{\ast}_{\beta})+\frac{1}{2}-\frac{3}{2\alpha}.
	\label{7g18}
\end{align}
By applying \eqref{7g17} and \eqref{7g18}, we simplify $T^{e,l}_g$ as follows:
\begin{align}\label{7g19}
	T^{e,l}_g&=\frac12\left(-3L^{-1}_{z}(\Gamma^{\ast}_{\gamma})+\alpha\Gamma^{\ast}_{\gamma}\right)
	\left(-\eta \cos(2\theta)+\sin^2(\theta)\right)g \\ \notag
	&~~~+L^{-1}_{z}\left(\Gamma^{\ast}_{\gamma})(\cos(2\theta)-\sin^2(\theta)\right)\alpha \left(-\frac{1}{\beta}L^{-1}_{z}(\Gamma^{\ast}_{\beta})+\frac{1}{2}-\frac{3}{2\alpha}\right)g\\ \notag
	&=\frac{\alpha}{2}\Gamma^{\ast}_{\gamma}\left(-\eta \cos(2\theta)+\sin^2(\theta)\right)g-\frac{3(1-\eta)}{2}L^{-1}_{z}(\Gamma^{\ast}_{\gamma})\cos(2\theta)g \\ \notag
	&~~~+\alpha L^{-1}_{z}(\Gamma^{\ast}_{\gamma})\left(\cos(2\theta)-\sin^2(\theta)\right)
	\left(-\frac{1}{\beta}L^{-1}_{z}(\Gamma^{\ast}_{\beta})+\frac{1}{2}\right)g.
\end{align}
This fact gives that
\begin{align}\label{7g20}
	|\bar{J}_{121}|\lesssim \left(\frac{\alpha}{\beta}+(1-\eta)\right)(1-\eta)^2\|gw^\eta\|^2_{L^2}\leq \frac{1}{100}\| g \|^2_{\mathcal{H}^2_{\eta}}.
\end{align}
Recalling \eqref{decom-Phi-1} and employing \eqref{7g3}, \eqref{7g4} and \eqref{7g8}-\eqref{7g10}, we obtain that
\begin{align}\label{7g21}
	\bigg\|\frac{U(\Phi-\frac{3}{4\alpha}L^{-1}_{z,K}(F^{\ast}_\gamma)\sin(2\theta))}{\sin(2\theta)}\bigg\|_{L^{\infty}}\lesssim_\beta \alpha^{\frac{1}{2}}(1-\eta)^{-\frac{1}{2}}+\alpha^{-1}\|g\|_{\mathcal{H}^{2}}.
\end{align}
According \eqref{decom-Phi-1}, one can deduce from \eqref{7g5}, \eqref{7g6} and \eqref{7g11}-\eqref{7g13} that
\begin{align}\label{7g22}
	\bigg\|V\bigg(\Phi-\frac{3}{4\alpha}L^{-1}_{z,K}(F^{\ast}_\gamma)\sin(2\theta)\bigg)\bigg\|_{L^{\infty}}\lesssim_\beta \alpha^{\frac{1}{2}}(1-\eta)^{-\frac{1}{2}}+\alpha^{-1}\|g\|_{\mathcal{H}^{2}}.
\end{align}
By virtue of \eqref{7g17} and \eqref{7g18}, we infer that $$\bigg\|\frac{D_{\theta}\overline{w}^\eta}{\overline{w}^\eta}\bigg\|_{L^\infty}+\alpha\bigg\|\frac{D_{z}\overline{w}^\eta}{\overline{w}^\eta}\bigg\|_{L^\infty}\leq C.$$ 
Combining this with \eqref{7g21} and \eqref{7g22}, we obtain that
\begin{align}\label{7g23}
	|\bar{J}_{122}|&\lesssim_\beta \left(\alpha^{\frac{1}{2}}(1-\eta)^{-\frac{1}{2}}+\alpha^{-1}\|g\|_{\mathcal{H}^{2}}\right)(1-\eta)^2\|gw^\eta\|^2_{L^2} 
	\lesssim_\beta \alpha^{\frac{1}{2}}\|g\|^2_{\mathcal{H}^2}+\alpha^{-1}\|g\|^3_{\mathcal{H}^{2}}.
\end{align}
We then substitute \eqref{7g20} and \eqref{7g23} into \eqref{est-bJ12-1}, to discover that
\begin{align}\label{7g24-1}
	|\bar{J}_{12}|\leq \frac{1}{100}\|g\|^2_{\mathcal{H}_{\eta}^2}
	+C_\beta\alpha^{\frac{1}{2}}\|g\|^2_{\mathcal{H}^2}
	+C_\beta\alpha^{-1}\|g\|^3_{\mathcal{H}^2}.
\end{align}
Plugging \eqref{7g16} and \eqref{7g24-1} into \eqref{est-bJ1-1}, we find that
\begin{align}\label{7g24}
	|\bar{J}_{1}|\leq \frac{1}{100}\|g\|^2_{\mathcal{H}_{\eta}^2}+C_\beta\alpha^{\frac{1}{2}}\|g\|^2_{\mathcal{H}^2}+C_\beta\alpha^{-1}\|g\|^3_{\mathcal{H}^2}.
\end{align}

\noindent
\textbf{Step 2.2: Estimate for $\dot{\mathcal{H}}^{1}_{\eta}$-norm.} 

\medskip

\noindent
Recalling \eqref{est-bJ-1}, we split the $\dot{\mathcal{H}}^{1}_{\eta}$-norm into two terms as follows:
\begin{align}\label{est-bJ2-1}
	\bar{J}_2
	=(1-\eta)^{-2}\langle D_{\theta}T_{g}, D_{\theta}g(w^\lambda)^2\rangle
	+(1-\eta)^{2}\langle D_{z}T_{g}, D_{z}g(w^\eta)^2\rangle
	\tri \bar{J}_{21}+\bar{J}_{22}.
\end{align}
We note that $\bar{J}_{21}$ and $\bar{J}_{22}$ are related to $D_\theta T_g$ and $D_z T_g$, respectively.
\medskip

\noindent
\textbf{Step 2.2.1: Estimate for the term related to $D_\theta T_g$.} 
\medskip

\noindent
We recall \eqref{def-wxi} with $\xi=\lambda$ and using the similar way as \eqref{est-bJ1-1}, to divide $\bar{J}_{21}$ into the following three parts:
\begin{align}\label{est-bJ21-1}
	\bar{J}_{21}
	=&~(1-\eta)^{-2}\langle T_{D_\theta g}, D_{\theta}g(w^\lambda)^2\rangle
	+(1-\eta)^{-2}\langle \mathcal{R}_1, D_{\theta}g(w^\lambda)^2\rangle
	\\
	=&~(1-\eta)^{-2}\langle T^d_{D_{\theta}g}, D_{\theta}g(w^\lambda)^2\rangle
	-(1-\eta)^{-2}\langle T^e_{D_{\theta}g}, D_{\theta}g(w^\lambda)^2\rangle
	+(1-\eta)^{-2}\langle \mathcal{R}_1, D_{\theta}g(w^\lambda)^2\rangle
	\notag\\
	\tri&~ \bar{J}_{211}+\bar{J}_{212}+\bar{J}_{213}.
	\notag
\end{align}
Here $T^d_{D_{\theta}g}$, $T^e_{D_{\theta}g}$ and the commutator $\mathcal{R}_1$ are defined by
\begin{align*}
	&T^d_{D_{\theta}g}\tri \frac{1}     {\overline{w}^\lambda}T_{D_{\theta}g\overline{w}^\lambda},~~
	T^e_{D_{\theta}g}\tri \frac{1}{\overline{w}^\lambda}D_{\theta}gT_{\overline{w}^\lambda},~~
	\mathcal{R}_1\tri D_{\theta}
	\left(\frac{U(\Phi)}{\sin(2\theta)}\right)
	D_{\theta}g+D_{\theta}V(\Phi)\alpha D_{z} g.
\end{align*}
According to Lemma \ref{7div1}, we have
\begin{align}\label{7g25}
	\bar{J}_{211}=0.
\end{align}
As for $\bar{J}_{212}$, we deduce from a similar way as \eqref{est-bJ12-1} that
\begin{align}\label{est-bJ212-1}
	\bar{J}_{212}
	= -(1-\eta)^{-2}\langle T^{e,l}_{D_{\theta}g}, D_{\theta}g(w^\lambda)^2\rangle
	-(1-\eta)^{-2}\langle T^{e,n}_{D_{\theta}g}, D_{\theta}g(w^\lambda)^2\rangle
	\tri \bar{J}_{2121}+\bar{J}_{2122},
\end{align}
where $T^{e,l}_{D_{\theta}g}$ and $T^{e,n}_{D_{\theta}g}$ are defined by
\begin{align*}
	T^{e,l}_{D_{\theta}g}
	\tri& \frac{1}{\overline{w}^\lambda}D_{\theta}g\left[\frac12\left(-3L^{-1}_{z}(\Gamma^{\ast}_{\gamma})+\alpha\Gamma^{\ast}_{\gamma}\right)D_{\theta}\overline{w}^\lambda+L^{-1}_{z}(\Gamma^{\ast}_{\gamma})\left(\cos(2\theta)-\sin^2(\theta)\right)\alpha D_{z}\overline{w}^\lambda\right],
	\\
	T^{e,n}_{D_{\theta}g}
	\tri& \frac{1}{\overline{w}^\lambda}D_{\theta}g\left[\frac{U(\Phi-\frac{3}{4\alpha}L^{-1}_{z,K}(F^{\ast}_\gamma)\sin(2\theta))}{\sin(2\theta)}D_{\theta}\overline{w}^\lambda+V\bigg(\Phi-\frac{3}{4\alpha}L^{-1}_{z,K}(F^{\ast}_\gamma)\sin(2\theta)\bigg)\alpha D_{z}\overline{w}^\lambda\right].
\end{align*}
It then follows from the same derivation as \eqref{7g17} and \eqref{7g18} that
\begin{align}\label{7g26}
	\frac{D_{\theta}\overline{w}^\lambda}{\overline{w}^\lambda}
	=-\lambda \cos(2\theta)+\sin^2(\theta),
	~~
	\frac{D_{z}\overline{w}^\lambda}{\overline{w}^\lambda } 
	=-\frac{1}{\beta}L^{-1}_{z}(\Gamma^{\ast}_{\beta})+\frac{1}{2}-\frac{3}{2\alpha}.
\end{align}
According to \eqref{7g26}, $T^{e,l}_{D_{\theta}g}$ can be simplified using a method similar to \eqref{7g19} as
\begin{align*}
	T^{e,l}_{D_{\theta}g}
	=&\frac{\alpha}{2}\Gamma^{\ast}_{\gamma}\left(-\lambda \cos(2\theta)+\sin^2(\theta)\right)D_{\theta}g-\frac{3(1-\lambda)}{2}L^{-1}_{z}(\Gamma^{\ast}_{\gamma})\cos(2\theta)D_{\theta}g \\ \notag
	&+\alpha L^{-1}_{z}(\Gamma^{\ast}_{\gamma})(\cos(2\theta)-\sin^2(\theta))\left(-\frac{1}{\beta}L^{-1}_{z}(\Gamma^{\ast}_{\beta})+\frac{1}{2}\right)D_{\theta}g.
\end{align*}
This, along with the facts that $\lambda=1+\frac{\alpha}{10\beta}$ and $\alpha\ll1-\eta\ll\beta$, yields that
\begin{align}\label{7g29}
	|\bar{J}_{2121}|\lesssim \frac{\alpha}{\beta}(1-\eta)^{-2}\|D_{\theta}gw^\lambda\|^2_{L^2}\leq \frac{1}{200}\| g \|^2_{\mathcal{H}^2_{\eta}}.
\end{align}
By using \eqref{7g26}, it is easy to check that $$\bigg\|\frac{D_{\theta}\overline{w}^\lambda}{\overline{w}^\lambda}\bigg\|_{L^\infty}+\alpha\bigg\|\frac{D_{z}\overline{w}^\lambda}{\overline{w}^\lambda}\bigg\|_{L^\infty}\leq C,$$ 
which, together with \eqref{7g21} and \eqref{7g22}, ensures that
\begin{align}\label{7g30}
	|\bar{J}_{2122}|&\lesssim_\beta \left(\alpha^{\frac{1}{2}}(1-\eta)^{-\frac{1}{2}}+\alpha^{-1}\|g\|_{\mathcal{H}^{2}}\right)(1-\eta)^{-2}\|D_{\theta}gw^\lambda\|^2_{L^2} \\ \notag
	&\lesssim_\beta \alpha^{\frac{1}{2}}(1-\eta)^{-\frac{5}{2}}\|g\|^2_{\mathcal{H}^2}+\alpha^{-1}(1-\eta)^{-2}\|g\|^3_{\mathcal{H}^{2}}.
\end{align}
Substituting \eqref{7g29} and \eqref{7g30} into \eqref{est-bJ212-1}, we find that
\begin{align}\label{7g31}
	|\bar{J}_{212}|\leq \frac{1}{200}\|g\|^2_{\mathcal{H}_{\eta}^2}+C_\beta\alpha^{\frac{1}{2}}(1-\eta)^{-\frac{5}{2}}\|g\|^2_{\mathcal{H}^2}+C_\beta\alpha^{-1} (1-\eta)^{-2}\|g\|^3_{\mathcal{H}^2}.
\end{align}
Now, we deal with $\bar{J}_{213}$.
Along the same line to the decomposition of \eqref{est-bI-1}, we can deduce from \eqref{cal_U} and \eqref{cal_V} that
\begin{align}\label{est-bJ213-1}
	\bar{J}_{213}
	=(1-\eta)^{-2}\langle \mathcal{R}_1^l, D_{\theta}g(w^\lambda)^2\rangle
	+(1-\eta)^{-2}\langle \mathcal{R}_1^n, D_{\theta}g(w^\lambda)^2\rangle
	\tri \bar{J}_{2131}+\bar{J}_{2132},
\end{align}
where $\mathcal{R}_1^l$ and $\mathcal{R}_1^n$ are defined by
\begin{align}
	\mathcal{R}_1^l
	\tri&L^{-1}_{z}(\Gamma^{\ast}_\gamma)D_{\theta}(\cos(2\theta)-\sin^2(\theta))\alpha D_{z} g
	=-3L^{-1}_{z}(\Gamma^{\ast}_\gamma)\sin^2(2\theta)\alpha D_{z} g,
	\label{def-R1l}\\
	\mathcal{R}_1^n
	\tri&D_{\theta}\left[\frac{U(\Phi-\frac{3}{4\alpha}L^{-1}_{z,K}(F^{\ast}_\gamma)\sin(2\theta))}{\sin(2\theta)}\right]D_{\theta}g
	+D_{\theta}\left[V\bigg(\Phi-\frac{3}{4\alpha}L^{-1}_{z,K}(F^{\ast}_\gamma)\sin(2\theta)\bigg)\right]\alpha D_{z} g.
	\label{def-R1n}
\end{align}
One can deduce from \eqref{3f7} that
\begin{align}\label{7g32}
	|\bar{J}_{2131}|
	&\lesssim \alpha(1-\eta)^{-2}\|D_z g \sin^2(2\theta)w^\lambda\|_{L^2}\|D_{\theta}g w^\lambda\|_{L^2} \\ \notag
	&\lesssim \alpha(1-\eta)^{-2}\|D_z g w^\eta\|_{L^2}\|D_{\theta}g w^\lambda\|_{L^2} \\ \notag
	&\leq \frac{1}{200}\| g \|^2_{\mathcal{H}^2_{\eta}}.
\end{align}
According to \eqref{decom-Phi-1}, it follows from Lemma \ref{7fm1}, Proposition \ref{6prop0}, Corollary \ref{6cor1}, Corollary \ref{6cor2} and Corollary \ref{7em2}, we infer that
\begin{align}\label{7g33}
	&~~~~\bigg\|D_\theta\bigg[\frac{U(\Phi-\frac{3}{4\alpha}L^{-1}_{z,K}(F^{\ast}_\gamma)\sin(2\theta))}{\sin(2\theta)}\bigg]\bigg\|_{L^{\infty}}
	\\ \notag
	&\lesssim \bigg\|D_\theta\bigg(\frac{U(\tilde{\Phi}_{F^\ast_{\gamma}})}{\sin(2\theta)}\bigg)\bigg\|_{L^{\infty}}
	+\bigg\|D_\theta\bigg(\frac{U(\tilde{\Phi}_{g})}{\sin(2\theta)}\bigg)\bigg\|_{L^{\infty}} 
	\\ \notag
	&\lesssim \|\partial_\theta U(\tilde{\Phi}_{F^\ast_{\gamma}})\|_{L^{\infty}}+\|\partial_\theta U(\tilde{\Phi}_{g})\|_{L^{\infty}} +\bigg\|\frac{U(\tilde{\Phi}_{F^\ast_{\gamma}})}{\sin(2\theta)}\bigg\|_{L^{\infty}}+\bigg\|\frac{U(\tilde{\Phi}_{g})}{\sin(2\theta)}\bigg\|_{L^{\infty}}  \\ \notag
	&\lesssim_{\beta} \alpha^{-\frac12}(\|\partial_\theta U(\tilde{\Phi}_{F^\ast_{\gamma}})\|_{\mathcal{H}^{2,\ast}}+\|\partial_\theta U(\tilde{\Phi}_{g})\|_{\mathcal{H}^{2}}) \\  \notag
	&\lesssim_\beta \alpha^{\frac{1}{2}}(1-\eta)^{-\frac{1}{2}}+\alpha^{-\frac12}\|g\|_{\mathcal{H}^{2}}.
\end{align}
Taking advantage of Lemma \ref{7fm1}, Proposition \ref{6prop0}, Corollary \ref{6cor1}, Corollary \ref{6cor2} and Corollary \ref{7em2}, one obtains that
\begin{align}\label{7g34}
	\|\partial_\theta V(\tilde{\Phi}_{F^\ast_{\gamma}}+\tilde{\Phi}_{g})\|_{L^{\infty}}
	&\lesssim_{\beta} \alpha^{-\frac12}\left(\|\partial^2_\theta (\tilde{\Phi}_{F^\ast_{\gamma}}+\tilde{\Phi}_{g})\|_{\mathcal{H}^{2,\ast}}+\|\partial_\theta (\tan\theta(\tilde{\Phi}_{F^\ast_{\gamma}}+\tilde{\Phi}_{g}))\|_{\mathcal{H}^{2,\ast}}\right)\\
	&\lesssim_\beta\alpha^{-\frac12}(\| F^\ast_{\gamma} \|_{\mathcal{H}^{2,\ast}}+\|g\|_{\mathcal{H}^{2}})
	\lesssim_\beta \alpha^{\frac{1}{2}}(1-\eta)^{-\frac{1}{2}}+\alpha^{-\frac12}\|g\|_{\mathcal{H}^{2}}.
	\notag
\end{align}
As the functions $\tilde{G}_{F^\ast_{\gamma}}$, $\tilde{\Phi}_{g}$ and $G^\ast_g$ are independent of $\theta$, the same manners as \eqref{7g6}, \eqref{7g12} and \eqref{7g13}, yield immediately that
\begin{align}\label{7g35}   
	\|\partial_\theta V((\tilde{G}_{F^\ast_{\gamma}}+\tilde{G}_{g}+G^\ast_g)
	\sin(2\theta))\|_{L^{\infty}}
	\lesssim_\beta \alpha(1-\eta)^{-\frac{1}{2}}+\alpha^{-1}\|g\|_{\mathcal{H}^{2}}.
\end{align}
Combining \eqref{7g34} and \eqref{7g35} together, we have
\begin{align}\label{7g39}
	\bigg\|\partial_\theta \bigg[V\bigg(\Phi-\frac{3}{4\alpha}L^{-1}_{z,K}(F^{\ast}_{\gamma})\sin(2\theta)\bigg)\bigg]\bigg\|_{L^{\infty}}
	\lesssim_\beta \alpha^{\frac{1}{2}}(1-\eta)^{-\frac{1}{2}}+\alpha^{-1}\|g\|_{\mathcal{H}^{2}}.
\end{align}
The application of \eqref{7g33} and \eqref{7g39} gives that
\begin{align}\label{7g40}
	|\bar{J}_{2132}|
	&\lesssim_\beta \left(\alpha^{\frac{1}{2}}(1-\eta)^{-\frac{1}{2}}+\alpha^{-\frac12}\|g\|_{\mathcal{H}^{2}}\right)(1-\eta)^{-2}\|D_{\theta}g w^\lambda\|^2_{L^2} \\ \notag
	&~~~~+\left(\alpha^{\frac{1}{2}}(1-\eta)^{-\frac{1}{2}}+\alpha^{-1}\|g\|_{\mathcal{H}^{2}}\right)(1-\eta)^{-2}\alpha\|D_z g \sin(2\theta)w^\lambda\|_{L^2}\|D_{\theta}g w^\lambda\|_{L^2} \\ \notag
	&\lesssim_\beta \left(\alpha^{\frac{1}{2}}(1-\eta)^{-\frac{5}{2}}+\alpha^{-1}(1-\eta)^{-2}\|g\|_{\mathcal{H}^{2}}\right)\| g \|^2_{\mathcal{H}^2}.
\end{align}
We then substitute \eqref{7g32} and \eqref{7g40} into \eqref{est-bJ213-1}, to discover that
\begin{align}\label{7g41-1}
	|\bar{J}_{213}|\leq \frac{1}{200}\|g\|^2_{\mathcal{H}_{\eta}^2}+C_\beta\alpha^{\frac{1}{2}}(1-\eta)^{-\frac{5}{2}}\|g\|^2_{\mathcal{H}^2}+C_\beta\alpha^{-1} (1-\eta)^{-2}\|g\|^3_{\mathcal{H}^2}.
\end{align}
Plugging \eqref{7g25}, \eqref{7g31} and \eqref{7g41-1} into \eqref{est-bJ21-1}, we achieve that
\begin{align}\label{7g41}
	|\bar{J}_{21}|\leq \frac{1}{100}\|g\|^2_{\mathcal{H}_{\eta}^2}+C_\beta\alpha^{\frac{1}{2}}(1-\eta)^{-\frac{5}{2}}\|g\|^2_{\mathcal{H}^2}+C_\beta\alpha^{-1} (1-\eta)^{-2}\|g\|^3_{\mathcal{H}^2}.
\end{align}

\noindent
\textbf{Step 2.2.2: Estimate for the term related to $D_z T_g$.} 
\medskip

\noindent
Firstly, we split $\bar{J}_{22}$ into two terms as follows:
\begin{align}\label{est-J22-1}
	\bar{J}_{22}
	=(1-\eta)^{2}\langle T_{D_{z}g}, D_{z}g(w^\eta)^2\rangle
	+(1-\eta)^{2}\langle \mathcal{R}_2, D_{z}g(w^\eta)^2\rangle
	\tri \bar{J}_{221}+\bar{J}_{222},
\end{align}
where the commutator $\mathcal{R}_2$ is defined by
\begin{align}\label{def-R2}
	\mathcal{R}_2 
	\tri 
	D_{z}\left(\frac{U(\Phi)}{\sin(2\theta)}\right)D_{\theta}g+D_{z}V(\Phi)\alpha D_{z} g.
\end{align}
It then follows in a similar way to estimate \eqref{7g24} that
\begin{align}\label{7g46}
	|\bar{J}_{221}|
	&=(1-\eta)^{2}
	\left|\left\langle \frac{1}{\overline{w}^\eta}T_{D_{z}g\overline{w}^\eta}-\frac{1}{\overline{w}^\eta}D_{z}gT_{\overline{w}^\eta}, D_{z}g(w^\eta)^2\right\rangle\right| \\ \notag
	&\leq\left(\frac{\alpha}{\beta}+(1-\eta)\right)(1-\eta)^{2}\|D_{z}g w^\eta\|^2_{L^2} \\ \notag
	&~~~+C_\beta\left(\alpha^{\frac{1}{2}}(1-\eta)^{-\frac{1}{2}}+\alpha^{-1}\|g\|_{\mathcal{H}^{2}}\right)(1-\eta)^{2}\|D_{z}gw^\eta\|^2_{L^2}
	\\ \notag
	&\leq \frac{1}{200}\|g\|^2_{\mathcal{H}_{\eta}^2}+C_\beta\alpha^{\frac{1}{2}}\|g\|^2_{\mathcal{H}^2}+C_\beta\alpha^{-1}\|g\|^3_{\mathcal{H}^2}.
\end{align}
A direct calculation yields directly that
\begin{align}\label{est-Dz-Gamma}
	D_z L_{z}^{-1}(\Gamma^\ast_\gamma)=-\Gamma^\ast_\gamma.
\end{align}
By a similar decomposition of \eqref{est-bI-1}, one can deduce from \eqref{cal_U}, \eqref{cal_V} and \eqref{est-Dz-Gamma} that
\begin{align}\label{est-bJ222-1}
	\bar{J}_{222}
	=(1-\eta)^{2}\langle \mathcal{R}_2^l, D_{z}g(w^\eta)^2\rangle
	+(1-\eta)^{2}\langle \mathcal{R}_2^n, D_{z}g(w^\eta)^2\rangle
	\tri \bar{J}_{2221}+\bar{J}_{2222},
\end{align}
where $\mathcal{R}_2^l$ and $\mathcal{R}_2^n$ are defined by
\begin{align}
	\mathcal{R}_2^l
	&\tri \frac12\left(3\Gamma^{\ast}_{\gamma}+\alpha D_z\Gamma^{\ast}_{\gamma}\right)D_{\theta}g-\Gamma^{\ast}_{\gamma}\left(\cos(2\theta)-\sin^2(\theta)\right)\alpha D_{z}g,
	\label{def-R2l}
	\\
	\mathcal{R}_2^n
	&\tri D_{z}\left(\frac{U(\Phi-\frac{3}{4\alpha}L^{-1}_{z,K}(F^{\ast}_\gamma)\sin(2\theta))}{\sin(2\theta)}\right)D_{\theta}g+D_{z}\bigg(V\bigg(\Phi-\frac{3}{4\alpha}L^{-1}_{z,K}(F^{\ast}_\gamma)\sin(2\theta)\bigg)\bigg)\alpha D_{z} g.
	\label{def-R2n}
\end{align}
In view of \eqref{3f7} and $\alpha\ll1-\eta\ll\beta$, we find that
\begin{align}\label{7g47}
	|\bar{J}_{2221}|
	&\lesssim \frac{1}{\beta}(1-\eta)^{2}\|D_z g w^\eta\|_{L^2}\|D_{\theta}g w^\eta\|_{L^2}+\frac{\alpha}{\beta}(1-\eta)^{2}\|D_z g w^\eta\|^2_{L^2} \\ \notag
	&\lesssim \frac{1}{\beta}(1-\eta)^{2}\left((1-\eta)^{2}\|D_z g w^\eta\|^2_{L^2}+(1-\eta)^{-2}\|D_{\theta}g w^\lambda\|^2_{L^2}\right) +\frac{\alpha}{\beta}(1-\eta)^{2}\|D_z g w^\eta\|^2_{L^2} \\ \notag
	&\leq \frac{1}{200}\| g \|^2_{\mathcal{H}^2_{\eta}}.
\end{align}
Let us deal with $\bar{J}_{2222}$. 
It follows in a similar way to \eqref{7g3} that
\begin{align}\label{7g49}
	\bigg\|D_z \left(\frac{U(\tilde{\Phi}_{F^\ast_{\gamma}})}{\sin(2\theta)}\right)\bigg\|_{L^{\infty}}
	\lesssim \sqrt{\frac{\beta}{\alpha}}\|\partial_\theta U(\tilde{\Phi}_{F^\ast_{\gamma}})\|_{\mathcal{H}^{3,\ast}}
	\lesssim_\beta \alpha^{-\frac12}\| F^\ast_{\gamma} \|_{\mathcal{H}^{3,\ast}}\lesssim_\beta \alpha^{\frac{1}{2}}(1-\eta)^{-\frac{1}{2}}.
\end{align}
According to Lemma \ref{7fm1}, Lemma \ref{7em1} and Lemma \ref{7ha1}, we infer that
\begin{align}\label{7g50}   
	\|D_z U(\tilde{G}_{F^\ast_{\gamma}})\|_{L^{\infty}}\lesssim \|D^2_z \tilde{G}_{F^\ast_{\gamma}}w^{\ast}_z\|_{L^{2}_z}+\alpha\|D^3_z \tilde{G}_{F^\ast_{\gamma}}w^{\ast}_z\|_{L^{2}_z}  \lesssim\| F^\ast_{\gamma} \|_{\mathcal{H}^{2,\ast}}\lesssim_\beta\alpha(1-\eta)^{-\frac{1}{2}}.
\end{align}
While for the terms related to $g$, we can only controlled them by the $\mathcal{H}^{2}$-norm for closing the energy estimates.
Thus, by a different way as \eqref{7g49} and \eqref{7g50}, we only estimate the $L^2_zL^\infty_\theta$-norm.
In view of Proposition \ref{6prop0}, Corollary \ref{6cor1} and Corollary \ref{7em2}, we obtain that
\begin{align}\label{7g51}
	\bigg\|D_z \left(\frac{U(\tilde{\Phi}_{g})}{\sin(2\theta)}\right)w_z\bigg\|_{L_z^{2}L_\theta^{\infty}}
	\lesssim \sqrt{\frac{\beta}{\alpha}}\bigg\|\frac{U(\tilde{\Phi}_{g})}{\sin(2\theta)}\bigg\|_{\mathcal{H}^{2}}
	\lesssim_\beta \alpha^{-\frac12}\|g\|_{\mathcal{H}^{2}}.
\end{align}
With the help of Lemma \ref{7ha1}, we have
\begin{align}\label{7g52}
	\|D_z U(\tilde{G}_{g}+G^\ast_g)w_z\|_{L_z^{2}}
	\lesssim
	\|D_z (\tilde{G}_{g}+G^\ast_g) w_z\|_{L^{2}_z}+\alpha\|D^2_z (\tilde{G}_{g}+G^\ast_g) w_z\|_{L^{2}_z}
	\lesssim
	\alpha^{-1}\|g\|_{\mathcal{H}^{2}}.
\end{align}
Collecting \eqref{7g49}-\eqref{7g52}, we find that
\begin{align}\label{7g54}
	&~~~~\left|(1-\eta)^{2}\bigg\langle D_{z}\bigg(\frac{U(\Phi-\frac{3}{4\alpha}L^{-1}_{z,K}(F^{\ast}_\gamma)\sin(2\theta))}{\sin(2\theta)}\bigg) D_{\theta}g,D_{z}g(w^\eta)^2\bigg\rangle\right| \\ \notag
	&\lesssim_\beta \alpha^{\frac{1}{2}}(1-\eta)^{\frac{3}{2}}\|D_{z}gw^\eta\|_{L^2}\|D_{\theta}gw^\eta\|_{L^2}+\frac{1}{\alpha} (1-\eta)^{2}
	\|g\|_{\mathcal{H}^{2}}\|D_{z}gw^\eta\|_{L^2}
	\|D_{\theta}gw_\theta^\eta\|_{L_z^{\infty}L^2_\theta}\\ \notag
	&\lesssim_\beta \alpha^{\frac{1}{2}}\|g\|^2_{\mathcal{H}^2}+\alpha^{-1}\|g\|^3_{\mathcal{H}^{2}}.
\end{align}
Similar argument as \eqref{7g5} and \eqref{7g6} leads us to get that
\begin{align}
	\label{7g55}
	\|D_z V\left(\tilde{\Phi}_{F^\ast_{\gamma}}+\tilde{G}_{F^\ast_{\gamma}}\sin(2\theta)\right)\|_{L^{\infty}}
	&\lesssim \sqrt{\frac{\beta}{\alpha}}\|\partial_\theta \tilde{\Phi}_{F^\ast_{\gamma}}\|_{\mathcal{H}^{3,\ast}}
	+\|D^2_z \tilde{G}_{F^\ast_{\gamma}}w^{\ast}_z\|_{L^{2}_z} 
	\\
	&\lesssim_\beta\alpha^{-\frac12}\| F^\ast_{\gamma} \|_{\mathcal{H}^{3,\ast}}+\| F^\ast_{\gamma} \|_{\mathcal{H}^{2,\ast}}
	\lesssim_\beta \alpha^{\frac{1}{2}}(1-\eta)^{-\frac{1}{2}}.
	\nonumber
\end{align}
We employ Proposition \ref{6prop0}, Corollary \ref{6cor1} and Corollary \ref{7em2}, to discover that
\begin{align}\label{7g57}
	\|D_z V(\tilde{\Phi}_{g})w_z\|_{L_{z}^{2}L_\theta^{\infty}}
	\lesssim \sqrt{\frac{\beta}{\alpha}}\|\partial_\theta \tilde{\Phi}_{g}\|_{\mathcal{H}^{2}} \lesssim_\beta \alpha^{-\frac12}\|g\|_{\mathcal{H}^{2}}.
\end{align}
According to Lemma \ref{7ha1}, we have
\begin{align}\label{7g58}
	\|D_z V((\tilde{G}_{g}+G^\ast_g)\sin(2\theta))w_z\|_{L_{z}^{2}L_\theta^{\infty}}
	&\lesssim \|D_z \tilde{G}_{g}w_z\|_{L^{2}_z}  +\|D_z G^\ast_g w_z\|_{L^{2}_z}
	\lesssim \alpha^{-1} \|g\|_{\mathcal{H}^{2}}.
\end{align}
Combining \eqref{7g55}-\eqref{7g58} together, we conclude that
\begin{align}\label{7g60}
	&~~~~(1-\eta)^{2}\left|\left\langle D_{z}V\bigg(\Phi-\frac{3}{4\alpha}L^{-1}_{z,K}(F^{\ast}_\gamma)\sin(2\theta)\bigg)\alpha D_{z} g,D_{z}g(w^\eta)^2\right\rangle \right|
	\\ \notag
	&\lesssim_\beta \alpha^{\frac{3}{2}}(1-\eta)^{\frac{3}{2}}\|D_{z}gw^\eta\|^2_{L^2}+(1-\eta)^{2}\|g\|_{\mathcal{H}^{2}} \|D_{z}gw^\eta\|_{L^2}\|D_{z}gw_\theta^\eta\|_{L_z^{\infty}L^2_\theta}\\ \notag
	&\lesssim_\beta \alpha^{\frac{3}{2}}\|g\|^2_{\mathcal{H}^2}+\|g\|^3_{\mathcal{H}^{2}}.
\end{align}
Applying \eqref{7g54} and \eqref{7g60}, we deduce that
\begin{align}\label{7g61}
	|\bar{J}_{2222}|
	&\lesssim_\beta \alpha^{\frac{1}{2}}\|g\|^2_{\mathcal{H}^2}+\alpha^{-1}\|g\|^3_{\mathcal{H}^{2}}.
\end{align}
Inserting \eqref{7g47} and \eqref{7g61} into \eqref{est-bJ222-1}, we find that
\begin{align}\label{est-bJ222-2}
	|\bar{J}_{222}|
	\leq 
	\frac{1}{200}\|g\|^2_{\mathcal{H}_{\eta}^2}+
	C_\beta \alpha^{\frac{1}{2}}\|g\|^2_{\mathcal{H}^2}+C_\beta\alpha^{-1}\|g\|^3_{\mathcal{H}^2}.
\end{align}
We then substitute \eqref{7g46} and \eqref{est-bJ222-2} into \eqref{est-J22-1}, to discover that
\begin{align}\label{7g62-1}
	|\bar{J}_{22}|\leq \frac{1}{100}\|g\|^2_{\mathcal{H}_{\eta}^2}+C_\beta\alpha^{\frac{1}{2}}\|g\|^2_{\mathcal{H}^2}+C_\beta\alpha^{-1}\|g\|^3_{\mathcal{H}^2}.
\end{align}

Finally, we plug \eqref{7g41} in \textbf{Step 2.2.1} and \eqref{7g62-1} in \textbf{Step 2.2.2} into \eqref{est-bJ2-1}, to conclude that
\begin{align}\label{7g62}
	|\bar{J}_{2}|\leq \frac{1}{50}\|g\|^2_{\mathcal{H}_{\eta}^2}+C_\beta\alpha^{\frac{1}{2}}(1-\eta)^{-\frac52}\|g\|^2_{\mathcal{H}^2}+C_\beta\alpha^{-1}(1-\eta)^{-2}\|g\|^3_{\mathcal{H}^2}.
\end{align}

\noindent
\textbf{Step 2.3: Estimate for $\dot{\mathcal{H}}^{2}_{\eta}$-norm.}
\medskip

\noindent
We recall \eqref{est-bJ-1} and divide the $\dot{\mathcal{H}}^{2}_{\eta}$-norm into the following three terms:
\begin{align}\label{est-bJ3-1}
	\bar{J}_3
	=&\langle D^2_{\theta}T_{g}, D^2_{\theta}g(w^\lambda)^2\rangle
	+(1-\eta)^{4}\langle D^2_{z}T_{g}, D^2_{z}g(w^\eta)^2\rangle
	+(1-\eta)^2\langle D_zD_{\theta}T_{g}, D_z D_{\theta}g(w^\lambda)^2\rangle
	\\
	\tri& \bar{J}_{31}+\bar{J}_{32}+\bar{J}_{33}.
	\notag
\end{align}
Here $\bar{J}_{31}$, $\bar{J}_{32}$ and $\bar{J}_{33}$ correspond to $D^2_\theta T_g$, $D^2_z T_g$ and $D_zD_{\theta}T_{g}$, respectively.

\medskip

\noindent
\textbf{Step 2.3.1: Estimate for the term related to $D^2_\theta T_g$.} 
\medskip

\noindent   
Firstly, we divide $\bar{J}_{31}$ into two terms as follows:
\begin{align}\label{est-J31-1}
	\bar{J}_{31}
	=\langle T_{D^2_{\theta}g}, D^2_{\theta}g(w^\lambda)^2\rangle
	+\langle \mathcal{R}_3, D^2_{\theta}g(w^\lambda)^2\rangle
	\tri \bar{J}_{311}+\bar{J}_{312},
\end{align}
where the commutator $\mathcal{R}_3$ is defined by
\begin{align}\label{def-R3}
	\mathcal{R}_3\tri\bigg[D^2_{\theta},\frac{U(\Phi)}{\sin(2\theta)}\bigg]D_{\theta}g+\bigg[D^2_{\theta},V(\Phi)\bigg]\alpha D_{z} g.
\end{align}
By the similar derivation as \eqref{7g31}, we obtain that
\begin{align}\label{7g67}
	|\bar{J}_{311}|
	&\leq \frac{1}{200}\|g\|^2_{\mathcal{H}_{\eta}^2}+C_\beta\alpha^{\frac{1}{2}}(1-\eta)^{-\frac{1}{2}}\|g\|^2_{\mathcal{H}^2}+C_\beta\alpha^{-1}\|g\|^3_{\mathcal{H}^2}.
\end{align}
We apply the same decomposition as \eqref{est-bJ213-1}, to discover that
\begin{align}\label{est-J312-1}
	\bar{J}_{312}
	=\langle \mathcal{R}_3^l,D^2_{\theta}g(w^\lambda)^2\rangle
	+\langle \mathcal{R}_3^n,D^2_{\theta}g(w^\lambda)^2\rangle
	\tri \bar{J}_{3121}+\bar{J}_{3122},
\end{align}
where $\mathcal{R}_3^l$ and $\mathcal{R}_3^n$ are defined by
\begin{align}
	\mathcal{R}_3^l
	&\tri -6L^{-1}_{z}(\Gamma^{\ast}_\gamma)\sin^2(2\theta)\alpha D_\theta D_{z} g-12L^{-1}_{z}(\Gamma^{\ast}_\gamma)\sin^2(2\theta)\cos(2\theta)\alpha D_{z} g,
	\label{def-R3l}
	\\
	\mathcal{R}_3^n
	&\tri \bigg[D^2_{\theta},\frac{U(\Phi-\frac{3}{4\alpha}L^{-1}_{z,K}(F^{\ast}_\gamma)\sin(2\theta))}{\sin(2\theta)}\bigg]D_{\theta}g+\bigg[D^2_{\theta},V\bigg(\Phi-\frac{3}{4\alpha}L^{-1}_{z,K}(F^{\ast}_\gamma)\sin(2\theta)\bigg)\bigg]\alpha D_{z} g.
	\label{def-R3n}
\end{align}
By using a similar way as \eqref{7g32}, we can find that $\bar{J}_{3121}$ can be controlled by
\begin{align}\label{7g68}
	|\bar{J}_{3121}|
	&\lesssim 
	\alpha\|D_\theta D_z g w^\lambda\|_{L^2}\|D^2_{\theta}g w^\lambda\|_{L^2} 
	+\alpha\|D_z g w^\eta\|_{L^2}\|D^2_{\theta}g w^\lambda\|_{L^2}
	\leq \frac{1}{200}\| g \|^2_{\mathcal{H}^2_{\eta}}.
\end{align}
Now, we begin to deal with another term $\bar{J}_{3122}$.
We employ \eqref{7g33} and \eqref{7g39}, to find that
\begin{align}
	&~~~~\bigg|\bigg\langle D_\theta\bigg[\frac{U(\Phi-\frac{3}{4\alpha}L^{-1}_{z,K}(F^{\ast}_\gamma)\sin(2\theta))}{\sin(2\theta)}\bigg]D^2_{\theta}g
	\label{7g69}
	\\
	&~~~~+D_{\theta}\bigg[V\bigg(\Phi-\frac{3}{4\alpha}L^{-1}_{z,K}(F^{\ast}_\gamma)\sin(2\theta)\bigg)\bigg]\alpha D_{\theta} D_{z} g
	,D^2_{\theta}g(w^\lambda)^2\bigg\rangle \bigg|
	\notag
	\\
	&\lesssim_\beta \left(\alpha^{\frac{1}{2}}(1-\eta)^{-\frac{1}{2}}+\alpha^{-\frac12}\|g\|_{\mathcal{H}^{2}}\right)\| g \|^2_{\mathcal{H}^2}.
	\notag
\end{align}
According to Lemma \ref{7fm1}, Proposition \ref{6prop0}, Corollary \ref{6cor1}, Corollary \ref{6cor2} and Corollary \ref{7em2}, we infer that
\begin{align*}
	&~~~~\bigg\|D^2_\theta\bigg(\frac{U(\Phi-\frac{3}{4\alpha}L^{-1}_{z,K}(F^{\ast}_\gamma)\sin(2\theta))}{\sin(2\theta)}\bigg)w_\theta^\lambda\bigg\|_{L_z^{\infty}L_\theta^2}
	\\ \notag
	&\lesssim \bigg\|D^2_\theta\frac{U(\tilde{\Phi}_{F^\ast_{\gamma}})}{\sin(2\theta)}w_\theta^\lambda+D^2_\theta\frac{U(\tilde{\Phi}_{g})}{\sin(2\theta)}w_\theta^\lambda\bigg\|_{L_z^{\infty}L_\theta^2} 
	\\ \notag
	&\lesssim \|\partial_\theta D_\theta U(\tilde{\Phi}_{F^\ast_{\gamma}})w_\theta^\lambda\|_{L_z^{\infty}L_\theta^2}+\|\partial_\theta D_\theta U(\tilde{\Phi}_{g})w_\theta^\lambda\|_{L_z^{\infty}L_\theta^2} +\bigg\|\frac{D_\theta U(\tilde{\Phi}_{F^\ast_{\gamma}})}{\sin(2\theta)}w_\theta^\lambda\bigg\|_{L_z^{\infty}L_\theta^2}
	\\ \notag
	&~~~+\bigg\|\frac{D_\theta U(\tilde{\Phi}_{g})}{\sin(2\theta)}w_\theta^\lambda\bigg\|_{L_z^{\infty}L_\theta^2} +\bigg\|\frac{U(\tilde{\Phi}_{F^\ast_{\gamma}})}{\sin(2\theta)}w_\theta^\lambda\bigg\|_{L_z^{\infty}L_\theta^2}+\bigg\|\frac{U(\tilde{\Phi}_{g})}{\sin(2\theta)}w_\theta^\lambda\bigg\|_{L_z^{\infty}L_\theta^2}  
	\\ \notag
	&\lesssim \|\partial_\theta U(\tilde{\Phi}_{F^\ast_{\gamma}})\|_{\mathcal{H}^{2,\ast}}+\|\partial^2_\theta \tilde{\Phi}_{F^\ast_{\gamma}}\|_{\mathcal{H}^{2,\ast}} +\|\partial_\theta U(\tilde{\Phi}_{g})\|_{\mathcal{H}^{2}}+ \|\partial^2_\theta \tilde{\Phi}_{g}\|_{\mathcal{H}^{2}}
	\\  \notag
	&\lesssim_\beta\alpha(1-\eta)^{-\frac{1}{2}}+\|g\|_{\mathcal{H}^{2}},
\end{align*}
which implies that
\begin{align}\label{7g72}
	&~~~~\bigg|\bigg\langle D^2_\theta\bigg(\frac{U(\Phi-\frac{3}{4\alpha}L^{-1}_{z,K}(F^{\ast}_\gamma)\sin(2\theta))}{\sin(2\theta)}\bigg)D_{\theta}g,D^2_{\theta}g(w^\lambda)^2\bigg\rangle \bigg|
	\\ \notag
	&\lesssim_\beta (\alpha(1-\eta)^{-\frac{1}{2}}+\|g\|_{\mathcal{H}^{2}})\|D_{\theta}g w_z\|_{L^2_z L^\infty_\theta} \|D^2_{\theta}g w^\lambda\|_{L^2} 
	\\ \notag
	&\lesssim_\beta
	\left(\alpha^{\frac{1}{2}}(1-\eta)^{-\frac{1}{2}}+\alpha^{-\frac{1}{2}}\|g\|_{\mathcal{H}^{2}}\right)\| g \|^2_{\mathcal{H}^2}.
\end{align}
Thanks to Lemma \ref{7fm1}, Proposition \ref{6prop0}, Corollary \ref{6cor1}, Corollary \ref{6cor2} and Corollary \ref{7em2}, it is enough to prove that
\begin{align}\label{7g73}
	\alpha\|D^2_\theta V(\tilde{\Phi}_{F^\ast_{\gamma}}+\tilde{\Phi}_{g})w_\theta^\lambda\|_{L_z^{\infty}L_\theta^2}
	&\lesssim \alpha\|D^2_\theta D_z V(\tilde{\Phi}_{F^\ast_{\gamma}}+\tilde{\Phi}_{g})w^{\lambda}\|_{L^2}
	\\
	&\lesssim \alpha\|\partial_\theta D_z \tilde{\Phi}_{F^\ast_{\gamma}}\|_{\mathcal{H}^{2,\ast}}
	+\alpha\|\partial_\theta D_z \tilde{\Phi}_{g}\|_{\mathcal{H}^{2}}
	\notag\\
	&\lesssim_\beta\| F^\ast_{\gamma} \|_{\mathcal{H}^{2,\ast}}
	+\| g \|_{\mathcal{H}^{2}}
	\notag\\
	&\lesssim_\beta \alpha(1-\eta)^{-\frac{1}{2}}
	+\| g \|_{\mathcal{H}^{2}}.
	\notag
\end{align}
It then follows from Lemma \ref{7fm1}, Lemma \ref{7em1} and Lemma \ref{7ha1} that
\begin{align}\label{7g74}   
	\alpha\|D^2_\theta V((\tilde{G}_{F^\ast_{\gamma}}+\tilde{G}_{g}+G^\ast_{g})\sin(2\theta))w_\theta^\lambda\|_{L_z^{\infty}L_\theta^2}
	&\lesssim \alpha\|D_z \tilde{G}_{F^\ast_{\gamma}}w^{\ast}_z\|_{L^{2}_z}  
	+\alpha\|D_z (\tilde{G}_{g}+G^\ast_{g})w_z\|_{L^{2}_z}
	\\
	&\lesssim\alpha\| F^\ast_{\gamma} \|_{\mathcal{H}^{1,\ast}}
	+\alpha\| g \|_{\mathcal{H}^{1}}
	\notag\\
	&\lesssim_\beta \alpha^2(1-\eta)^{-\frac{1}{2}}+\alpha\| g \|_{\mathcal{H}^{1}}.
	\notag
\end{align}
We collect \eqref{7g73} and \eqref{7g74}, to deduce directly that
\begin{align}\label{7g78}
	\alpha\bigg\|D^2_\theta V\bigg(\Phi-\frac{3}{4\alpha}L^{-1}_{z,K}(F^{\ast})\sin(2\theta)\bigg)w_\theta^\lambda\bigg\|_{L_z^{\infty}L_\theta^2}\lesssim_\beta \alpha(1-\eta)^{-\frac{1}{2}}+\|g\|_{\mathcal{H}^{2}}.
\end{align}
By using \eqref{7g78} and Corollary \ref{7em2}, we have
\begin{align}\label{7g79}
	&~~~~~\bigg|\bigg\langle D^2_\theta V\bigg(\Phi-\frac{3}{4\alpha}L^{-1}_{z,K}(F^{\ast}_\gamma)\sin(2\theta)\bigg)\alpha D_z g,D^2_{\theta}g(w^\lambda)^2\bigg\rangle \bigg|
	\\ \notag
	&\lesssim_\beta \left(\alpha(1-\eta)^{-\frac{1}{2}}+\|g\|_{\mathcal{H}^{2}}\right)\|D_{z}g w_z\|_{L^2_z L^\infty_\theta} \|D^2_{\theta}g w^\lambda\|_{L^2} \\ \notag
	&\lesssim_\beta \left(\alpha^{\frac{1}{2}}(1-\eta)^{-\frac{1}{2}}+\alpha^{-\frac12}\|g\|_{\mathcal{H}^{2}}\right)\| g \|^2_{\mathcal{H}^2}.
\end{align}
The combination of \eqref{7g69}, \eqref{7g72} and \eqref{7g79}, yields that
\begin{align}\label{7g80}
	|\bar{J}_{3122}|
	\lesssim_\beta (\alpha^{\frac{1}{2}}(1-\eta)^{-\frac{1}{2}}+\alpha^{-\frac12}\|g\|_{\mathcal{H}^{2}})\| g \|^2_{\mathcal{H}^2}.
\end{align}
Substituting \eqref{7g68} and \eqref{7g80} into \eqref{est-J312-1}, one obtains that
\begin{align}\label{est-J312-2}
	|\bar{J}_{312}|
	\leq
	\frac{1}{200}\|g\|^2_{\mathcal{H}_{\eta}^2}+C_\beta\alpha^{\frac{1}{2}}(1-\eta)^{-\frac{1}{2}}\|g\|^2_{\mathcal{H}^2}+C_\beta\alpha^{-\frac12}\|g\|^3_{\mathcal{H}^2}.
\end{align}
Plugging \eqref{7g67} and \eqref{est-J312-2} into \eqref{est-J31-1}, we see that
\begin{align}\label{7g81}
	|\bar{J}_{31}|\leq \frac{1}{100}\|g\|^2_{\mathcal{H}_{\eta}^2}+C_\beta\alpha^{\frac{1}{2}}(1-\eta)^{-\frac{1}{2}}\|g\|^2_{\mathcal{H}^2}+C_\beta\alpha^{-1}\|g\|^3_{\mathcal{H}^2}.
\end{align}

\noindent
\textbf{Step 2.3.2: Estimate for the term related to $D^2_z T_g$.} 
\medskip

\noindent
Similar to \eqref{est-J31-1}, $\bar{J}_{32}$ can be divided into the following two parts:
\begin{align}\label{est-J32-1}
	\bar{J}_{32}
	=(1-\eta)^{4}\langle T_{D^2_{z}g}, D^2_{z}g(w^\eta)^2\rangle
	+(1-\eta)^{4}\langle \mathcal{R}_4, D^2_{z}g(w^\eta)^2\rangle
	\tri \bar{J}_{321}+\bar{J}_{322},
\end{align}
where the commutator $\mathcal{R}_4$ is defined by
\begin{align}\label{def-R4}
	\mathcal{R}_4
	=\bigg[D^2_{z},\frac{U(\Phi)}{\sin(2\theta)}\bigg]D_{\theta}g+\bigg[D^2_{z},V(\Phi)\bigg]\alpha D_{z} g.
\end{align}
Then by exactly the same procedure as that in the proof of \eqref{7g24}, we can deduce from Lemma \ref{7div1}, \eqref{7g21} and \eqref{7g22} that
\begin{align}\label{7g85}
	|\bar{J}_{321}|
	\leq \frac{1}{200}\|g\|^2_{\mathcal{H}_{\eta}^2}+C_\beta\alpha^{\frac{1}{2}}\|g\|^2_{\mathcal{H}^2}+C_\beta\alpha^{-1}\|g\|^3_{\mathcal{H}^2}.
\end{align}
By the similar way as \eqref{est-bJ222-1}, we can decompose $\bar{J}_{322}$ as follows:
\begin{align}\label{est-bJ322-1}
	\bar{J}_{322}
	=(1-\eta)^{4}\langle \mathcal{R}_4^{l}, D^2_{z}g(w^\eta)^2\rangle
	+(1-\eta)^{4}\langle \mathcal{R}_4^{n}, D^2_{z}g(w^\eta)^2\rangle
	\tri \bar{J}_{3221}+\bar{J}_{3222}.
\end{align}
Here $\mathcal{R}_4^{l}$ and $\mathcal{R}_4^{n}$ are defined by
\begin{align}
	\mathcal{R}_4^{l}
	&\tri\left(3\Gamma^{\ast}_{\gamma}+\alpha D_z\Gamma^{\ast}_{\gamma}\right)D_z D_{\theta}g
	+\left(\frac{3}{2}D_z\Gamma^{\ast}_{\gamma}+\frac{\alpha}{2}D^2_z\Gamma^{\ast}_{\gamma}\right)D_{\theta}g 
	\label{def_R4l}\\
	&~~~~-2\Gamma^{\ast}_{\gamma}\left(\cos(2\theta)-\sin^2(\theta)\right)\alpha D^2_{z}g-D_z\Gamma^{\ast}_{\gamma}\left(\cos(2\theta)-\sin^2(\theta)\right)\alpha D_{z}g,
	\notag\\
	\mathcal{R}_4^{n}
	&\tri\bigg[D^2_{z},\frac{U(\Phi-\frac{3}{4\alpha}L^{-1}_{z,K}(F^{\ast}_\gamma)\sin(2\theta))}{\sin(2\theta)}\bigg]D_{\theta}g
	+\bigg[D^2_{z},V\bigg(\Phi-\frac{3}{4\alpha}L^{-1}_{z,K}(F^{\ast}_\gamma)\sin(2\theta)\bigg)\bigg]\alpha D_{z} g.
	\label{def_R4n}
\end{align}
Along a similar way as \eqref{7g47}, we find that
\begin{align}\label{7g86}
	|\bar{J}_{3221}|
	&\lesssim \frac{1}{\beta^2}(1-\eta)^{4}\|D^2_z g w^\eta\|_{L^2}(\|D_{\theta}D_z g w^\eta\|_{L^2}+\|D_{\theta} g w^\eta\|_{L^2})\\ \notag
	&~~~+\frac{\alpha}{\beta^2}(1-\eta)^{4}(\|D^2_z g w^\eta\|^2_{L^2}+\|D_z g w^\eta\|^2_{L^2}) \\ \notag
	&\leq \frac{1}{300}\left((1-\eta)^{4}\|D^2_z g w^\eta\|^2_{L^2}+(1-\eta)^{2}\|D_\theta D_z g w^\lambda\|^2_{L^2}+(1-\eta)^{-2}\|D_{\theta}g w^\lambda\|^2_{L^2}\right) \\ \notag
	&~~~+C\frac{\alpha}{\beta^2}(1-\eta)^{4}(\|D^2_z g w^\eta\|^2_{L^2}+\|D_z g w^\eta\|^2_{L^2})\\ \notag
	&\leq \frac{1}{200}\| g \|^2_{\mathcal{H}^2_{\eta}}.
\end{align}
One can deduce from a similar way as \eqref{7g9} and \eqref{7g10} that
\begin{align*}
	\|D_z U(\tilde{G}_{g}+G^\ast_g)\|_{L^{\infty}}
	\lesssim \|D^2_z (\tilde{G}_{g}+G^\ast_g) w_z\|_{L^{2}_z}+\alpha\|D^3_z (\tilde{G}_{g}+G^\ast_g) w_z\|_{L^{2}_z}\lesssim \alpha^{-1} \|g\|_{\mathcal{H}^{2}}.
\end{align*}
This, together with \eqref{7g49} and \eqref{7g50}, implies that
\begin{align}\label{7g89}
	\|D_z U(\tilde{G}_{g}+G^\ast_g)\|_{L^{\infty}}
	+\bigg\|D_z \bigg(\frac{U(\tilde{\Phi}_{F^\ast_{\gamma}})}{\sin(2\theta)}\bigg)\bigg\|_{L^{\infty}}
	+\|D_z U(\tilde{G}_{F^\ast_{\gamma}})\|_{L^{\infty}}
	\lesssim_\beta \alpha^{-1}\|g\|_{\mathcal{H}^{2}}+\alpha^{\frac{1}{2}}(1-\eta)^{-\frac{1}{2}}.
\end{align}
With the help of Proposition \ref{6prop0}, Corollary \ref{6cor1} and Corollary \ref{7em2}, one obtains that
\begin{align}\label{7g90}
	\bigg\|D_z \bigg(\frac{U(\tilde{\Phi}_{g})}{\sin(2\theta)}\bigg)\frac{w_\theta^\eta}{w_\theta^\lambda}\bigg\|_{L^{\infty}}
	&\lesssim \|D^2_z U(\tilde{\Phi}_{g})w_z(\sin(2\theta))^{-\frac{1}{2}+\frac{\alpha}{20\beta}-\frac{\eta}{2}}\|_{L_z^{2}L_\theta^{\infty}} 
	\\  \notag
	&\lesssim \bigg(\|\partial_\theta D^2_z U(\tilde{\Phi}_{g})w^\eta\|_{L^{2}}+\bigg\|D^2_z \left(\frac {U(\tilde{\Phi}_{g})}{\sin(2\theta)}\right)w^\eta\bigg\|_{L^{2}}\bigg)
	\|(\sin(2\theta))^{-\frac{1}{2}+\frac{\alpha}{20\beta}-\frac{\eta}{2}}(w_\theta^\eta)^{-1}\|_{L^{2}_\theta} 
	\notag \\ \notag
	&\lesssim_\beta \alpha^{-\frac12}\|g\|_{\mathcal{H}^{2}}.
\end{align}
Applying Lemma \ref{7fm1}, it follows from the similar way as \eqref{7g49}, \eqref{7g50} and \eqref{7g52} that
\begin{align}\label{7g91}
	&~~~~~~\|D^2_z U(\tilde{G}_{g})w_z\|_{L_z^2 }+\|D^2_z U(G^\ast_g)w_z\|_{L_z^2}+\bigg\|D^2_z \left(\frac{U(\tilde{\Phi}_{F^\ast_{\gamma}})}{\sin(2\theta)}\right)\bigg\|_{L^{\infty}}+\|D^2_z U(\tilde{G}_{F^\ast_{\gamma}})\|_{L^{\infty}} \\ \notag
	&\lesssim_\beta \alpha^{-1}\|g\|_{\mathcal{H}^{2}}+\alpha^{-\frac12}\| F^\ast_{\gamma} \|_{\mathcal{H}^{4,\ast}}
	\lesssim_\beta \alpha^{-1}\|g\|_{\mathcal{H}^{2}}+\alpha^{\frac{1}{2}}(1-\eta)^{-\frac{1}{2}}.
\end{align}
While by exactly the same procedure as that in the proof of \eqref{7g90}, it is enough to show that
\begin{align}\label{7g92}
	\bigg\|D^2_z \left(\frac{U(\tilde{\Phi}_{g})}{\sin(2\theta)}\right)w_z\frac{w_\theta^\eta}{w_\theta^\lambda}\bigg\|_{L_z^2 L_\theta^{\infty}}
	\lesssim_\beta \alpha^{-\frac12}  \|g\|_{\mathcal{H}^{2}}.
\end{align}
Combining \eqref{7g89}-\eqref{7g92} together, we obtain that
\begin{align}\label{7g93}
	&~~~~(1-\eta)^4 \bigg|\bigg\langle \bigg[D^2_{z},\frac{U(\Phi-\frac{3}{4\alpha}L^{-1}_{z,K}(F^{\ast}_\gamma)\sin(2\theta))}{\sin(2\theta)}\bigg]D_{\theta}g, D_z^2g (w^{\eta})^2 \bigg\rangle\bigg| 
	\\ 
	\notag
	&\lesssim_\beta (1-\eta)^{4}\left(\alpha^{-1}\|g\|_{\mathcal{H}^{2}}+\alpha^{\frac{1}{2}}(1-\eta)^{-\frac{1}{2}}\right)\|D^2_{z}gw^\eta\|_{L^2}\left(\|D_z D_{\theta}g w^\lambda\|_{L^2}
	+
	\|D_{\theta}gw_\theta^\lambda\|_{L_z^{\infty}L^2_\theta}+\|D_{\theta}gw^\eta\|_{L^2}\right)
	\\ \notag
	&\lesssim _\beta \alpha^{\frac{1}{2}}\|g\|^2_{\mathcal{H}^2}+\alpha^{-1}\|g\|^3_{\mathcal{H}^{2}}.
\end{align}
According to \eqref{7g55}, Corollary \ref{6cor1} and Lemma \ref{7ha1}, we infer that
\begin{align}\label{7g94}
	&~~~~~\bigg\|D_z \bigg[V\bigg(\Phi-\frac{3}{4\alpha}L^{-1}_{z,K}(F^{\ast}_\gamma)\sin(2\theta)\bigg)\bigg]\bigg\|_{L^{\infty}}
	\\ \notag
	&\lesssim_\beta \alpha^{\frac{1}{2}}(1-\eta)^{-\frac{1}{2}}+\alpha^{-\frac12}\|\partial_\theta D_z\tilde{\Phi}_{g}\|_{\mathcal{H}^{2}}  
	+\|D^2_z \tilde{G}_{g}w_z\|_{L^{2}_z}+\|D^2_z G^\ast_g w_z\|_{L^{2}_z} \\ \notag
	&\lesssim_\beta\alpha^{\frac{1}{2}}(1-\eta)^{-\frac{1}{2}}+\alpha^{-\frac{3}{2}}\|g\|_{\mathcal{H}^{2}}.
\end{align}
Along a similar way as \eqref{7g55}-\eqref{7g58}, one gets that
\begin{align}\label{7g95}
	&~~~~\|D^2_z V\left(\tilde{\Phi}_{g}+(\tilde{G}_{g}+G^\ast_g) \sin(2\theta)\right)w_z\|_{L_z^2L_\theta^{\infty}}+\|D^2_z V(\tilde{\Phi}_{F^\ast_{\gamma}}+\tilde{G}_{F^\ast_{\gamma}}\sin(2\theta))\|_{L^{\infty}} 
	\\ \notag
	&\lesssim_\beta \alpha^{\frac{1}{2}}(1-\eta)^{-\frac{1}{2}}+\alpha^{-\frac{3}{2}}\|g\|_{\mathcal{H}^{2}}.
\end{align}
We then collect \eqref{7g94}-\eqref{7g95}, to discover that
\begin{align}\label{7g96}
	&~~~~~(1-\eta)^{4}\left|\bigg\langle \bigg[D^2_{z},V\bigg(\Phi-\frac{3}{4\alpha}L^{-1}_{z,K}(F^{\ast}_\gamma)\sin(2\theta)\bigg)\bigg]\alpha D_{z} g,D^2_{z}g(w^\eta)^2\bigg\rangle \right|
	\\ \notag
	&\lesssim_\beta (1-\eta)^{4}\left(\alpha^{\frac{3}{2}}(1-\eta)^{-\frac{1}{2}}+\alpha^{-\frac{1}{2}}\|g\|_{\mathcal{H}^{2}}\right)\|D^2_{z}gw^\eta\|_{L^2} 
	\\ \notag
	&~~~~~\times\left(\|D^2_{z}gw^\eta\|_{L^2}+\|D_{z}gw_\theta^\eta\|_{L_z^{\infty}L^2_\theta}+\|D_{z}gw^\eta\|_{L^2}\right) 
	\\ \notag
	&\lesssim_\beta \alpha^{\frac{3}{2}}\|g\|^2_{\mathcal{H}^2}+\alpha^{-\frac{1}{2}}\|g\|^3_{\mathcal{H}^{2}}.
\end{align}
According to \eqref{7g93} and \eqref{7g96}, we deduce that
\begin{align}\label{7g97}
	|\bar{J}_{3222}|
	\lesssim_\beta \alpha^{\frac{1}{2}}\|g\|^2_{\mathcal{H}^2}+\alpha^{-1}\|g\|^3_{\mathcal{H}^{2}}.
\end{align}
We insert \eqref{7g86} and \eqref{7g97} into \eqref{est-bJ322-1}, to obtain that
\begin{align}\label{est-bJ322-2}
	|\bar{J}_{322}| \leq \frac{1}{200}\|g\|^2_{\mathcal{H}_{\eta}^2}+C_\beta\alpha^{\frac{1}{2}}\|g\|^2_{\mathcal{H}^2}+C_\beta \alpha^{-1}\|g\|^3_{\mathcal{H}^2}.
\end{align}
Substituting \eqref{7g85} and \eqref{est-bJ322-2} into \eqref{est-J32-1}, we finally get that
\begin{align}\label{7g98}
	|\bar{J}_{32}|\leq 
	\frac{1}{100}\|g\|^2_{\mathcal{H}_{\eta}^2}
	+C_\beta\alpha^{\frac{1}{2}}\|g\|^2_{\mathcal{H}^2}
	+C_\beta \alpha^{-1}\|g\|^3_{\mathcal{H}^2}.
\end{align}

\noindent
\textbf{Step 2.3.3: Estimate for the term related to $D_zD_\theta  T_g$.} 
\medskip

\noindent 
Firstly, we rewrite $D_zD_\theta  T_g$ as the following two terms:
\begin{align}\label{est-bJ33-1}
	\bar{J}_{33}
	=(1-\eta)^2\langle T_{D_zD_{\theta}g}, D_zD_{\theta}g(w^\lambda)^2\rangle
	+(1-\eta)^2\langle \mathcal{R}_5, D_zD_{\theta}g(w^\lambda)^2\rangle
	\tri \bar{J}_{331}+\bar{J}_{332},
\end{align}
where the commutator $\mathcal{R}_5$ is defined by
\begin{align}\label{def-R5}
	\mathcal{R}_5=\left[D_zD_{\theta},\frac{U(\Phi)}{\sin(2\theta)}\right]D_{\theta}g+\left[D_zD_{\theta},V(\Phi)\right]\alpha D_{z} g.
\end{align}
Along a similar way as \eqref{7g24}, we deduce from Lemma \ref{7div1}, \eqref{7g21} and \eqref{7g22} that
\begin{align}\label{7g102}
	|\bar{J}_{331}|
	\leq \frac{1}{200}\|g\|^2_{\mathcal{H}_{\eta}^2}+C_\beta\alpha^{\frac{1}{2}}(1-\eta)^{-\frac{1}{2}}\|g\|^2_{\mathcal{H}^2}+C_\beta\alpha^{-1}\|g\|^3_{\mathcal{H}^2}.
\end{align}
We control $\bar{J}_{332}$ through a decomposition similar to \eqref{est-bJ213-1}, specifically:
\begin{align}\label{est-bJ332-1}
	\bar{J}_{332}
	=(1-\eta)^2\langle \mathcal{R}_5^l, D_zD_{\theta}g(w^\lambda)^2\rangle
	+(1-\eta)^2\langle \mathcal{R}_5^n, D_zD_{\theta}g(w^\lambda)^2\rangle
	\tri \bar{J}_{3321}+\bar{J}_{3322},
\end{align}
where $\mathcal{R}_5^l$ and $\mathcal{R}_5^n$ are defined by
\begin{align}
	\mathcal{R}_5^l
	&\tri\frac12\left(3\Gamma^{\ast}_{\gamma}+\alpha D_z\Gamma^{\ast}_{\gamma}\right)
	D^2_{\theta}g
	-3\alpha\left(L^{-1}_{z}(\Gamma^{\ast}_\gamma)  D^2_{z} g-\Gamma^{\ast}_\gamma D_{z} g
	\right)\sin^2(2\theta)
	\label{def-R5l}
	\\
	\notag
	&~~~~-\Gamma^{\ast}_\gamma
	\left(\cos(2\theta)-\sin^2(\theta)\right)\alpha D_{z}D_\theta g,
\end{align}
and
\begin{align}
	\mathcal{R}_5^n
	&\tri\bigg[D_zD_{\theta},\frac{U(\Phi-\frac{3}{4\alpha}L^{-1}_{z,K}(F^{\ast}_\gamma)\sin(2\theta))}{\sin(2\theta)}\bigg]D_{\theta}g+\bigg[D_zD_{\theta},V\bigg(\Phi-\frac{3}{4\alpha}L^{-1}_{z,K}(F^{\ast}_\gamma)\sin(2\theta)\bigg)\bigg]\alpha D_{z} g.
	\label{def-R5n}
\end{align}
Adapting the approach developed for \eqref{7g32} and \eqref{7g47}, we derive
\begin{align}\label{7g103}
	|\bar{J}_{3321}|
	&\leq \frac{1}{200}\| g \|^2_{\mathcal{H}^2_{\eta}}.
\end{align}
By combining \eqref{7g33} and \eqref{7g39}, and following a derivation analogous to \eqref{7g40}, we establish that
\begin{align}\label{7g104}
	&~~~~(1-\eta)^2\bigg|\bigg\langle D_\theta\bigg(\frac{U(\Phi-\frac{3}{4\alpha}L^{-1}_{z,K}(F^{\ast}_\gamma)\sin(2\theta))}{\sin(2\theta)}\bigg)D_zD_{\theta}g,D_zD_{\theta}g(w^\lambda)^2\bigg\rangle\bigg| \\ \notag
	&~~~~+(1-\eta)^2\bigg|\bigg\langle D_{\theta}\left(V(\Phi-\frac{3}{4\alpha}L^{-1}_{z,K}(F^{\ast}_\gamma)\sin(2\theta))\right)\alpha D^2_{z} g,D_zD_{\theta}g(w^\lambda)^2\bigg\rangle\bigg|
	\\ \notag
	&\lesssim_\beta (\alpha^{\frac{1}{2}}(1-\eta)^{-\frac{1}{2}}+\alpha^{-\frac12}\|g\|_{\mathcal{H}^{2}})\| g \|^2_{\mathcal{H}^2}.
\end{align}
Observing that $\tilde{G}_{F^\ast_{\gamma}}$, $\tilde{G}_{g}$ and $G^\ast_g$ are independent of $\theta$, we derive through direct computations that
\begin{align}\label{est-DzDtheta-1}
	D_zD_\theta\bigg(\frac{U(\Phi-\frac{3}{4\alpha}L^{-1}_{z,K}(F^{\ast}_\gamma)\sin(2\theta))}{\sin(2\theta)}\bigg)
	=D_zD_\theta\bigg(\frac{U(\tilde{\Phi}_{F^\ast_{\gamma}})}{\sin(2\theta)}\bigg)
	+D_zD_\theta \bigg(\frac{U(\tilde{\Phi}_{g})}{\sin(2\theta)}\bigg).
\end{align}
We deduce from Proposition \ref{6prop0}, Corollaries \ref{6cor1}, \ref{6cor2} and \ref{7em2} that
\begin{align}
	\label{7g106}
	&\bigg\|D_zD_\theta\left(\frac{U(\tilde{\Phi}_{F^\ast_{\gamma}})}{\sin(2\theta)}\right)w_\theta^\lambda\bigg\|_{L_z^{\infty}L_\theta^2}
	\lesssim
	\|\partial_\theta U(\tilde{\Phi}_{F^\ast_{\gamma}})\|_{\mathcal{H}^{3,\ast}} \lesssim_\beta\alpha(1-\eta)^{-\frac{1}{2}},
	\\
	&\bigg\|D_zD_\theta\left(\frac{U(\tilde{\Phi}_{g})}{\sin(2\theta)}\right)w_z\bigg\|_{L_z^2L_\theta^{\infty}}
	\lesssim \sqrt{\frac{\beta}{\alpha}}\left(\|\partial_\theta U(\tilde{\Phi}_{g})\|_{\mathcal{H}^{2}}+\|\partial^2_\theta \tilde{\Phi}_{g}\|_{\mathcal{H}^{2}}\right) \lesssim_\beta\alpha^{-\frac12}\|g\|_{\mathcal{H}^{2}}.
	\label{7g107}
\end{align}
Combining \eqref{est-DzDtheta-1}-\eqref{7g107} together, it comes out
\begin{align}
	\label{7g108}
	&~~~~(1-\eta)^2\bigg|\bigg\langle D_zD_\theta\bigg(\frac{U(\Phi-\frac{3}{4\alpha}L^{-1}_{z,K}(F^{\ast}_\gamma)\sin(2\theta))}{\sin(2\theta)}\bigg)D_{\theta}g,D_zD_{\theta}g(w^\lambda)^2\bigg\rangle\bigg| 
	\\ \notag
	&\lesssim_\beta \left(\alpha(1-\eta)^{-\frac{1}{2}}\|D_{\theta}g w_z\|_{L^2_z L^\infty_\theta}+\alpha^{-\frac12}\|g\|_{\mathcal{H}^{2}}\|D_{\theta}g w_\theta^\lambda\|_{L^\infty_z L^2_\theta}\right) \|D_zD_{\theta}g w^\lambda\|_{L^2} \\ \notag
	&\lesssim_\beta (\alpha^{\frac{1}{2}}(1-\eta)^{-\frac{1}{2}}+\alpha^{-\frac12}\|g\|_{\mathcal{H}^{2}})\| g \|^2_{\mathcal{H}^2}.
\end{align}
We then deduce from the similar way as \eqref{7g73}-\eqref{7g78} that
\begin{align*}
	\alpha\bigg\|D_zD_\theta V\bigg(\Phi-\frac{3}{4\alpha}L^{-1}_{z,K}(F^{\ast})\sin(2\theta)\bigg)w_\theta^\lambda\bigg\|_{L_z^{\infty}L_\theta^2}\lesssim_\beta \alpha(1-\eta)^{-\frac{1}{2}}+\|g\|_{\mathcal{H}^{2}}.
\end{align*}
By combining this with Corollary \ref{7em2}, we obtain that
\begin{align}\label{7g115}
	&~~~~~(1-\eta)^2\bigg|\bigg\langle D_zD_\theta V\bigg(\Phi-\frac{3}{4\alpha}L^{-1}_{z,K}(F^{\ast}_\gamma)\sin(2\theta)\bigg)\alpha D_z g,D_zD_{\theta}g(w^\lambda)^2\bigg\rangle \bigg|
	\\ \notag
	&\lesssim_\beta \left(\alpha(1-\eta)^{-\frac{1}{2}}+\|g\|_{\mathcal{H}^{2}}\right)\|D_{z}g w_z\|_{L^2_z L^\infty_\theta} \|D_zD_{\theta}g w^\lambda\|_{L^2} 
	\\ \notag
	&\lesssim_\beta \left(\alpha^{\frac{1}{2}}(1-\eta)^{-\frac{1}{2}}+\alpha^{-\frac12}\|g\|_{\mathcal{H}^{2}}\right)\| g \|^2_{\mathcal{H}^2}.
\end{align}
By virtue of \eqref{7g94}, it follows that
\begin{align}\label{7g116}
	&(1-\eta)^2\bigg|\bigg\langle D_z V\bigg(\Phi-\frac{3}{4\alpha}L^{-1}_{z,K}(F^{\ast}_\gamma)\sin(2\theta)\bigg)\alpha D_zD_\theta g,D_zD_{\theta}g(w^\lambda)^2\bigg\rangle\bigg| \\ \notag
	\lesssim_\beta& \left(\alpha^{\frac{3}{2}}(1-\eta)^{-\frac{1}{2}}+\alpha^{-\frac{1}{2}}\|g\|_{\mathcal{H}^{2}}\right)\| g \|^2_{\mathcal{H}^2}.
\end{align}
Controlling $\|D_z \Big(\frac{U(\tilde{\Phi}_{g})}{\sin(2\theta)}\Big)\|_{L^\infty}$ within $\mathcal{H}^2$ presents a fundamental difficulty. Our approach introduces the new norm of $D^2_z g$ with an additional weight that avoids increasing regularity. 
More precisely, we introduce the $\mathcal{E}^{2}$-norm. 
According to Corollaries \ref{7em2}, \ref{6cor3} and Lemma \ref{6lem4}, we infer that
\begin{align}\label{7g117}
	\bigg\|D_z \bigg(\frac{U(\tilde{\Phi}_{g})}{\sin(2\theta)}\bigg)\bigg\|_{L^{\infty}}
	&\lesssim \sqrt{\frac{\beta}{\alpha}}\bigg\|D^2_zD_\theta \bigg(\frac{U(\tilde{\Phi}_{g})}{\sin(2\theta)}\bigg)w^\lambda\bigg\|_{L^{2}} 
	\lesssim \sqrt{\frac{\beta}{\alpha}}\|\partial_\theta D^2_z U(\tilde{\Phi}_{g})w^\lambda\|_{L^{2}} 
	\lesssim_\beta \alpha^{-1}\|g\|_{\mathcal{E}^{2}}.
\end{align}
The application of \eqref{7g89} and \eqref{7g117}, yields that
\begin{align}
	\label{7g118}
	\bigg\|D_z \bigg(\frac{U(\Phi-\frac{3}{4\alpha}L^{-1}_{z,K}(F^{\ast}_\gamma)\sin(2\theta))}{\sin(2\theta)}\bigg)\bigg\|_{L^{\infty}}
	\lesssim_\beta \alpha^{-1}\|g\|_{\mathcal{E}^{2}}
	+\alpha^{-1}\|g\|_{\mathcal{H}^{2}}
	+\alpha^{\frac{1}{2}}(1-\eta)^{-\frac{1}{2}},
\end{align}
which implies that
\begin{align}\label{7g119}
	&~~~~(1-\eta)^2\bigg|\bigg\langle D_z \bigg(\frac{U(\Phi-\frac{3}{4\alpha}L^{-1}_{z,K}(F^{\ast}_\gamma)\sin(2\theta))}{\sin(2\theta)}\bigg)D^2_\theta g,D_zD_{\theta}g(w^\lambda)^2\bigg\rangle\bigg| \\ \notag
	&\lesssim_\beta \left(\alpha^{-1}\|g\|_{\mathcal{E}^{2}}+\alpha^{-1}\|g\|_{\mathcal{H}^{2}}+\alpha^{\frac{1}{2}}(1-\eta)^{-\frac{1}{2}}\right)\| g \|^2_{\mathcal{H}^2}.
\end{align}
Taking advantage of \eqref{7g104}, \eqref{7g108}-\eqref{7g116} and \eqref{7g119}, it is easy to deduce that
\begin{align}\label{7g120}
	|\bar{J}_{3322}|
	\lesssim_\beta \left(\alpha^{-1}\|g\|_{\mathcal{E}^{2}}+\alpha^{-1}\|g\|_{\mathcal{H}^{2}}+\alpha^{\frac{1}{2}}(1-\eta)^{-\frac{1}{2}}\right)\| g \|^2_{\mathcal{H}^2}.
\end{align}
Substituting \eqref{7g103} and \eqref{7g120} into \eqref{est-bJ332-1}, we find that
\begin{align}\label{est-bJ332-2}
	|\bar{J}_{332}|
	\leq \frac{1}{200}\|g\|^2_{\mathcal{H}_{\eta}^2}+C_\beta\left(\alpha^{-1}\|g\|_{\mathcal{E}^{2}}+\alpha^{-1}\|g\|_{\mathcal{H}^{2}}+\alpha^{\frac{1}{2}}(1-\eta)^{-\frac{1}{2}}\right)\| g \|^2_{\mathcal{H}^2}.
\end{align}
We insert \eqref{7g102} and \eqref{est-bJ332-2} into \eqref{est-bJ33-1}, to discover that
\begin{align}\label{7g121}
	|\bar{J}_{33}| 
	\leq&\frac{1}{100}\|g\|^2_{\mathcal{H}_{\eta}^2}+C_\beta\left(\alpha^{-1}\|g\|_{\mathcal{E}^{2}}+\alpha^{-1}\|g\|_{\mathcal{H}^{2}}+\alpha^{\frac{1}{2}}(1-\eta)^{-\frac{1}{2}}\right)\| g \|^2_{\mathcal{H}^2}.
\end{align}
We plug \eqref{7g81} in \textbf{Step 2.3.1}, \eqref{7g98} in \textbf{Step 2.3.2} and \eqref{7g121} in \textbf{Step 2.3.3} into \eqref{est-bJ3-1}, to find that
\begin{align}\label{est-bJ3-2}
	|\bar{J}_3| \leq 
	\frac{3}{100}\|g\|^2_{\mathcal{H}_{\eta}^2}+C_\beta\left(\alpha^{-1}\|g\|_{\mathcal{E}^{2}}+\alpha^{-1}\|g\|_{\mathcal{H}^{2}}+\alpha^{\frac{1}{2}}(1-\eta)^{-\frac{1}{2}}\right)\| g \|^2_{\mathcal{H}^2}.
\end{align}

Consequently, collecting \eqref{7g24} in \textbf{Step 2.1}, \eqref{7g62} in \textbf{Step 2.2} and \eqref{est-bJ3-2} in \textbf{Step 2.3}, we get \eqref{est-norm-H2eta} immediately.

\medskip

\noindent
\textbf{Step 3: Estimate for $\mathcal{E}^{2}_{\eta}$-norm.} 
\medskip

\noindent
Recalling the definition of $\mathcal{E}^{2}_{\eta}$ in \eqref{def_norm_E2eta}, we get that
\begin{align}\label{est-bK-1}
	\langle T_g, g\rangle_{\mathcal{E}^{2}_{\eta}}
	=\alpha(1-\eta)^2\langle T_g, g(w^\lambda)^2\rangle
	+\alpha(1-\eta)^4\langle D_z^2T_g, D_z^2g(w^\lambda)^2\rangle
	\tri \bar{K}_1+\bar{K}_2.
\end{align}
By virtue of Lemma \ref{7div1}, \eqref{7g21} and \eqref{7g22}, we obtain that
\begin{align}\label{7e4}
	|\bar{K}_1|\leq \frac{1}{400}\|g\|^2_{\mathcal{E}_{\eta}^2}+C_\beta
	\left(\alpha^{\frac{1}{2}}(1-\eta)^{-\frac{1}{2}}+\alpha^{-1}\|g\|_{\mathcal{H}^{2}}\right)\|g\|^2_{\mathcal{E}^2}.
\end{align}
We recall the definitions of $\mathcal{R}_4^l$ and $\mathcal{R}_4^n$ in \ref{def_R4l} and \eqref{def_R4n}, respectively, to find that
\begin{align}\label{est-bK2-1}
	\bar{K}_2
	=&\alpha(1-\eta)^4\langle T_{D^2_z g}, D_z^2g(w^\lambda)^2\rangle
	+\alpha(1-\eta)^4\langle \mathcal{R}_4^l, D_z^2g(w^\lambda)^2\rangle
	+\alpha(1-\eta)^4\langle \mathcal{R}_4^n, D_z^2g(w^\lambda)^2\rangle
	\\
	\tri& \bar{K}_{21}+\bar{K}_{22}+\bar{K}_{23}.
	\notag
\end{align}
With the help of Lemma \ref{7div1}, it follows from \eqref{7g21} and \eqref{7g22} that
\begin{align}\label{7e8}
	|\bar{K}_{21}|\leq \frac{1}{200}\|g\|^2_{\mathcal{E}_{\eta}^2}+C_\beta(\alpha^{\frac{1}{2}}(1-\eta)^{-\frac{1}{2}}+\alpha^{-1}\|g\|_{\mathcal{H}^{2}})\|g\|^2_{\mathcal{E}^2}.
\end{align}
Along the similar way as \eqref{7g86}, one can deduce from Lemma \ref{6lem4} that
\begin{align}\label{7e9}
	|\bar{K}_{22}|&\lesssim \frac{1}{\beta^2}(1-\eta)^{4}\alpha\|D^2_z g w^\lambda\|_{L^2}(\|D_{\theta}D_z g w^\lambda\|_{L^2}+\|D_{\theta} g w^\lambda\|_{L^2})\\ \notag
	&~~~+\frac{\alpha}{\beta^2}(1-\eta)^{4}\alpha(\|D^2_z g w^\lambda\|^2_{L^2}+\|D_z g w^\lambda\|^2_{L^2}) \\ \notag
	&\leq \frac{1}{200}(\| g \|^2_{\mathcal{E}^2_{\eta}}+\| g \|^2_{\mathcal{H}^2_{\eta}}).
\end{align}
By using the similar derivation as \eqref{7g117}, we have
\begin{align}\label{7e10}
	\bigg\|D^2_z \bigg(\frac{U(\tilde{\Phi}_{g})}{\sin(2\theta)}\bigg)w_z\bigg\|_{L_z^2 L_\theta^{\infty}}
	\lesssim_\beta \alpha^{-1}\|g\|_{\mathcal{E}^{2}}.
\end{align}
From this, \eqref{7g91} and \eqref{7g118}, it follows that
\begin{align}\label{7e11}
	&~~~~~\alpha(1-\eta)^{4}\bigg|\bigg\langle \bigg[D^2_{z},\frac{U(\Phi-\frac{3}{4\alpha}L^{-1}_{z,K}(F^{\ast}_\gamma)\sin(2\theta))}{\sin(2\theta)}\bigg]D_{\theta}g,D^2_{z}g(w^\lambda)^2\bigg\rangle \bigg|
	\\ \notag
	&\lesssim_\beta \alpha \left(\alpha^{-1}\|g\|_{\mathcal{E}^{2}}+\alpha^{-1} \|g\|_{\mathcal{H}^{2}}+\alpha^{\frac{1}{2}}\right)\|D^2_{z}g w^\lambda\|_{L^2}
	\left(\|D_z D_{\theta}g w^\lambda\|_{L^2}+
	\|D_{\theta}gw_\theta^\lambda\|_{L_z^{\infty}L^2_\theta}+\|D_{\theta}gw^\lambda\|_{L^2}\right)\\ \notag
	&\lesssim_\beta \left(\alpha^{-1}\|g\|_{\mathcal{E}^{2}}+\alpha^{-1}\|g\|_{\mathcal{H}^{2}}+\alpha^{\frac{1}{2}}\right)\|g\|_{\mathcal{E}^{2}}\|g\|_{\mathcal{H}^{2}}.
\end{align}
Combining \eqref{7g94} and \eqref{7g95} together, we achieve that
\begin{align}\label{7e12}
	&~~~~~\alpha(1-\eta)^{4}\bigg|\bigg\langle [D^2_{z},V\bigg(\Phi-\frac{3}{4\alpha}L^{-1}_{z,K}(F^{\ast}_\gamma)\sin(2\theta)\bigg)]\alpha D_{z} g,D^2_{z}g(w^\lambda)^2\bigg\rangle \bigg|
	\\ \notag
	&\lesssim_\beta \alpha \left(\alpha^{\frac{3}{2}}+\alpha^{-\frac{1}{2}}\|g\|_{\mathcal{H}^{2}}\right)\|D^2_{z}gw^\lambda\|_{L^2}\left(\|D^2_{z}gw^\lambda\|_{L^2}+\|D_{z}gw_\theta^\lambda\|_{L_z^{\infty}L^2_\theta}+\|D_{z}gw^\lambda\|_{L^2}\right)
	\\ \notag
	&\lesssim_\beta \left(\alpha^{\frac{3}{2}}+\alpha^{-\frac{1}{2}}\|g\|_{\mathcal{H}^{2}}\right)\|g\|^2_{\mathcal{E}^2}.
\end{align}
According to \eqref{7e11} and \eqref{7e12}, we deduce that
\begin{align}\label{7e13}
	|\bar{K}_{23}|
	&\lesssim_\beta \left(\alpha^{-1}\|g\|_{\mathcal{E}^{2}}
	+\alpha^{-1}\|g\|_{\mathcal{H}^{2}}+\alpha^{\frac{1}{2}}\right)\|g\|_{\mathcal{E}^{2}}\|g\|_{\mathcal{H}^{2}}+\alpha^{\frac{3}{2}}\|g\|^2_{\mathcal{E}^2}.
\end{align}
Substituting \eqref{7e8}, \eqref{7e9} and \eqref{7e13} into \eqref{est-bK2-1}, we finally get that
\begin{align}\label{7e14}
	|\bar{K}_2|
	&\leq \frac{1}{100}(\|g\|^2_{\mathcal{H}_{\eta}^2}+\|g\|^2_{\mathcal{E}_{\eta}^2})+C_\beta\alpha^{\frac{1}{2}}(1-\eta)^{-\frac{1}{2}}\|g\|^2_{\mathcal{E}^2} \\ 
	\notag
	&~~~+C_\beta\left(\alpha^{-1}\|g\|_{\mathcal{E}^{2}}+\alpha^{-1}\|g\|_{\mathcal{H}^{2}}+\alpha^{\frac{1}{2}}\right)\|g\|_{\mathcal{E}^{2}}\|g\|_{\mathcal{H}^{2}}.
\end{align}
We then plug \eqref{7e4} and \eqref{7e14} into \eqref{est-bK-1}, to deduce \eqref{est-norm-E2eta} directly.

\end{proof}

\subsection{Product rules and transport estimates for fundamental solutions}
In this subsection, we aim to establish the \textit{a priori} energy estimates for $\mathcal{H}^{-1}$-norm, $\mathcal{H}^2_{\eta}$-norm and $\mathcal{E}^2_{\eta}$-norm of $T_{F^\ast_{\gamma}}$.
Before we proceed, let us first present some fundamental lemmas concerning estimates related to $F^\ast_{\gamma}$.

\begin{lemm}\label{7fm2}
Assume that $f(z,\theta)|_{\partial D}=0$. 
Let $\beta \in (0,1]$.
There exists a sufficiently small constant $\alpha>0$, such that if $\alpha\ll \beta$ and $|\mu|\leq \alpha^\frac{1}{2}$, then there holds
\begin{align}\label{est-lem7.7-1}
	\|fD_\theta F^\ast_{\gamma} \|_{\mathcal{H}^{2}}\leq C\frac{\alpha^2}{\beta^4}\|f\|_{\mathcal{H}^{2,\ast}},~~~~\|fD_z F^\ast_{\gamma} \|_{\mathcal{H}^{2}}\leq C\frac{\alpha}{\beta^4}\|f\|_{\mathcal{H}^{2,\ast}},
\end{align}
and
\begin{align}\label{est-lem7.7-2}
	\|fD_\theta F^\ast_{\gamma} \|_{\mathcal{H}^{2}}\leq C\frac{\alpha^2}{\beta^4}\|f\|_{\mathcal{H}^{2}},~~~~\|fD_z F^\ast_{\gamma} \|_{\mathcal{H}^{2}}\leq C\frac{\alpha}{\beta^4}\|f\|_{\mathcal{H}^{2}}.
\end{align}
\end{lemm}
\begin{proof}
Recalling \eqref{def_Gama+c*}, it is easy to check that for any $1\leq k\leq3$,
\begin{align}
	&|D^k_z \Gamma^\ast_{\gamma}|\lesssim \frac{1}{\gamma^k}\Gamma^\ast_{\gamma},
	\label{est-Dz-Gamma-1}
	\\
	|D_\theta \Gamma_{\theta,\gamma}|\lesssim \frac{\alpha}{\gamma}\Gamma_{\theta,\gamma},~~~~
	&|D^2_\theta \Gamma_{\theta,\gamma}|+|D^3_\theta \Gamma_{\theta,\gamma}|\lesssim\frac{\alpha^2}{\gamma^2}\Gamma_{\theta,\gamma}+\frac{\alpha}{\gamma}\Gamma_{\theta,\gamma}\sin(2\theta).
	\label{est-Dtheta-Gamma-1}
\end{align}
In view of \eqref{def_Gama+c*} and \eqref{4eq13}, we have
\begin{align}\label{est-bd-Gamma}
	\Gamma_{\theta,\gamma}\leq 1,~~~~\Gamma^\ast_{\gamma} w_z\lesssim  \frac{1}{\gamma}w^\ast_z.
\end{align}
This, combined with \eqref{form_F1} and \eqref{est-Dtheta-Gamma-1}, establishes that
\begin{align}\label{7f1}
	\|fD_\theta F^\ast_{\gamma} w^\eta\|_{L^{2}}
	&\lesssim \frac{\alpha^2}{\gamma}\|f(\Gamma_{\theta,\gamma} w_\theta^\eta)(\Gamma^\ast_{\gamma} w_z)\|_{L^{2}}
	\lesssim \frac{\alpha^2}{\gamma^2}\|f w_\theta^\eta w^\ast_z\|_{L^{2}} 
	\lesssim \frac{\alpha^2}{\beta^2}\|f\|_{H^{2,\ast}}.
\end{align}
Using \eqref{form_F1}, \eqref{est-Dz-Gamma-1}-\eqref{est-bd-Gamma}, with arguments similar to \eqref{7f1}, we obtain that
\begin{align}\label{7f2}
	\|D_z(fD_\theta F^\ast_{\gamma}) w^\eta\|_{L^{2}}+\|D^2_z(fD_\theta F^\ast_{\gamma}) w^\eta\|_{L^{2}}\lesssim \frac{\alpha^2}{\beta^4}\|f\|_{H^{2,\ast}}.
\end{align}
While for terms related to $D_\theta f$, we can simply handle them by employing \eqref{form_F1}, \eqref{est-Dtheta-Gamma-1} and \eqref{est-bd-Gamma}, for example:
\begin{align}\label{7f3-1}
	\|D_\theta fD_\theta F^\ast_{\gamma} w^\lambda\|_{L^{2}}
	&\lesssim \frac{\alpha^2}{\gamma}\|D_\theta f(\Gamma_{\theta,\gamma} w_\theta^\lambda)(\Gamma^\ast_{\gamma} w_z)\|_{L^{2}}
	\lesssim \frac{\alpha^2}{\gamma^2}\|D_\theta f w_\theta^\lambda w^\ast_z\|_{L^{2}} 
	\lesssim \frac{\alpha^2}{\beta^2}\|f\|_{H^{2,\ast}}.
\end{align}
Similar argument leads us to get that
\begin{align}\label{7f3}
	\|D_\theta fD^2_\theta F^\ast_{\gamma} w^\lambda\|_{L^{2}}+\|D^2_\theta fD_\theta F^\ast_{\gamma} w^\lambda\|_{L^{2}}\lesssim \frac{\alpha^2}{\beta^2}\|f\|_{H^{2,\ast}}.
\end{align}
Again thanks to \eqref{form_F1}, \eqref{est-Dtheta-Gamma-1}-\eqref{est-bd-Gamma}, we can deduce from the same way that
\begin{align}\label{7f4}
	\|D_\theta fD_zD_\theta F^\ast_{\gamma} w^\lambda\|_{L^{2}}+\|D_zD_\theta fD_\theta F^\ast_{\gamma} w^\lambda\|_{L^{2}}\lesssim \frac{\alpha^2}{\beta^3}\|f\|_{H^{2,\ast}}.
\end{align}
In view of \eqref{form_F1}, \eqref{def_Gama+c*}, \eqref{est-Dtheta-Gamma-1} and Corollary \ref{7em2}, one obtains that
\begin{align}\label{7f5}
	\|fD^2_\theta F^\ast_{\gamma} w^\lambda\|_{L^{2}}
	&\lesssim \frac{\alpha^3}{\gamma^2}\|f(\Gamma_{\theta,\gamma} w_\theta^\lambda)(\Gamma^\ast_{\gamma} w_z)\|_{L^{2}}+\frac{\alpha^2}{\gamma}\|f(\Gamma_{\theta,\gamma} w_\theta^\eta)(\Gamma^\ast_{\gamma} w_z)\|_{L^{2}}\\ \notag
	&\lesssim \frac{\alpha^3}{\gamma^3}\|f w^\ast_z\|_{L_z^{2}L_\theta^{\infty}}\|\Gamma_{\theta,\gamma} w_\theta^\lambda\|_{L_\theta^{2}}
	+\frac{\alpha^2}{\gamma^2}\|f w_\theta^\eta w^\ast_z\|_{L^{2}} \\ \notag
	&\lesssim \frac{\alpha^2}{\beta^2}\|f\|_{H^{2,\ast}},
\end{align}
where we have used the fact that
\begin{align}\label{est-Gamma-lambda-L2}
	\|\Gamma_{\theta,\gamma} w_\theta^\lambda\|_{L_\theta^{2}}
	\lesssim \sqrt{\frac{\gamma}{\alpha}}.
\end{align}
By applying \eqref{form_F1}, \eqref{est-Dz-Gamma-1}-\eqref{est-bd-Gamma}, and following a derivation analogous to \eqref{7f5}, we establish that
\begin{align}
	&\|fD^3_\theta F^\ast_{\gamma} w^\lambda\|_{L^{2}}\lesssim \frac{\alpha^2}{\beta^2}\|f\|_{H^{2,\ast}},
	\label{7f6}
	\\
	&\|D_z fD^2_\theta F^\ast_{\gamma} w^\lambda\|_{L^{2}}+\|f D_z D^2_\theta F^\ast_{\gamma} w^\lambda\|_{L^{2}}\lesssim \frac{\alpha^2}{\beta^3}\|f\|_{H^{2,\ast}}.
	\label{7f7}
\end{align}
Combining \eqref{7f1}-\eqref{7f7} together, we finally obtain that the first inequality in \eqref{est-lem7.7-1} holds.
Similarly, the other three inequalities can be derived directly.
Thus, we complete the proof of Lemma \ref{7fm2}.
\end{proof}

In comparison with Lemma \ref{7fm2}, we consider the case where the function $f(z)$ is independent of $\theta$, yielding the following similar lemma. We omit the proof here.
\begin{lemm}\label{7fm3}
Assume that $f(z)$ is independent of $\theta$ and satisfies $f(0)=0$. 
Let $\beta\in(0,1]$.
There exist constants $\alpha>0$ sufficiently small and $\eta(\beta)$, such that if $\alpha\ll 1-\eta\ll \beta$ and $|\mu|\leq \alpha^\frac{1}{2}$, then
\begin{align*}
	&\|fD_\theta F^\ast_{\gamma} \|_{\mathcal{H}^{2}}\leq C\frac{\alpha^2}{\beta^4}(1-\eta)^{-\frac{1}{2}}\sum_{i=0}^{2}\|D^i_z fw_z^{\ast}\|_{L^2_z},~~~~
	\|fD_z F^\ast_{\gamma} \|_{\mathcal{H}^{2}}\leq C\frac{\alpha}{\beta^4}(1-\eta)^{-\frac{1}{2}}\sum_{i=0}^{2}\|D^i_z fw_z^{\ast}\|_{L^2_z},
	\\
	&\|fD_\theta F^\ast_{\gamma} \|_{\mathcal{H}^{2}}\leq C\frac{\alpha^2}{\beta^4}(1-\eta)^{-\frac{1}{2}}\sum_{i=0}^{2}\|D^i_z fw_z\|_{L^2_z},~~~~
	\|fD_z F^\ast_{\gamma} \|_{\mathcal{H}^{2}}\leq C\frac{\alpha}{\beta^4}(1-\eta)^{-\frac{1}{2}}\sum_{i=0}^{2}\|D^i_z fw_z\|_{L^2_z}.
\end{align*}
\end{lemm}

With Lemmas \ref{7fm1}, \ref{7fm2} and \ref{7fm3} at hand, we are ready to establish the \textit{a priori} energy estimates for $\mathcal{H}^{-1}$-norm, $\mathcal{H}^2_{\eta}$-norm and $\mathcal{E}^2_{\eta}$-norm of $T_{F^\ast_{\gamma}}$.
\begin{prop}\label{7Tf}
Let $L^{-1}_{z,K}(g)(0)=0$, and $\beta\in(0,1]$, there exist constants $\alpha>0$ sufficiently small and $\eta(\beta)$, such that if $\alpha\ll 1-\eta\ll \beta$ and $|\mu|\leq \alpha^\frac{1}{2}$, then
\begin{align}
	|\langle T_{F^\ast_{\gamma}}, g(w^K)^2\rangle|
	&\leq C_\beta \alpha(1-\eta)^{-\frac{1}{2}}\left(\alpha\|g\|_{\mathcal{H}^2}+\|g\|^2_{\mathcal{H}^2}\right),
	\label{est-TF-H-1}\\
	|\langle T_{F^\ast_{\gamma}}, g\rangle_{\mathcal{H}_{\eta}^2}|
	&\leq C_\beta \alpha(1-\eta)^{-\frac{5}{2}}\left(\alpha\|g\|_{\mathcal{H}^2}+\|g\|^2_{\mathcal{H}^2}\right),
	\label{est-TF-H2eta}\\
	|\langle T_{F^\ast_{\gamma}}, g\rangle_{\mathcal{E}_{\eta}^2}|
	&\leq C_\beta \left(\alpha^2\|g\|_{\mathcal{E}^2}+\alpha\|g\|_{\mathcal{H}^2}\|g\|_{\mathcal{E}^2}+\alpha^2\|g\|^2_{\mathcal{E}^2}\right).
	\label{est-TF-E2eta}
\end{align}
\end{prop}

\begin{proof}
We will divide this proof into three steps.

\medskip

\noindent
\textbf{Step 1: Estimate for $\mathcal{H}^{-1}$-norm.} 
\medskip

\noindent
Recalling the definition of $\mathcal{H}^{-1}$-norm in \eqref{def_norm_H-1}, we decompose it as follows:
\begin{align}\label{est-bL-1}
	\langle T_{F^\ast_{\gamma}}, g(w^K)^2\rangle
	=\langle T_{F^\ast_{\gamma}}^l, g(w^K)^2\rangle
	+\langle T_{F^\ast_{\gamma}}^n, g(w^K)^2\rangle
	\tri \bar{L}_1+\bar{L}_2.
\end{align}
Here $T^l_{F^\ast_{\gamma}}$ and $T^n_{F^\ast_{\gamma}}$ are defined by
\begin{align}
	T^l_{F^\ast_{\gamma}}&\tri\frac{U(\frac{3}{4\alpha}L^{-1}_{z,K}(F^{\ast}_\gamma)\sin(2\theta))}{\sin(2\theta)}D_{\theta}F^\ast_{\gamma}+V\bigg(\frac{3}{4\alpha}L^{-1}_{z,K}(F^{\ast}_\gamma)\sin(2\theta)\bigg)\alpha D_{z}F^\ast_{\gamma},
	\label{def-Tl}
	\\
	T^n_{F^\ast_{\gamma}}&\tri\frac{U(\Phi-\frac{3}{4\alpha}L^{-1}_{z,K}(F^{\ast}_\gamma)\sin(2\theta))}{\sin(2\theta)}D_{\theta}F^\ast_{\gamma}+V\bigg(\Phi-\frac{3}{4\alpha}L^{-1}_{z,K}(F^{\ast}_\gamma)\sin(2\theta)\bigg)\alpha D_{z}F^\ast_{\gamma}.
	\label{def-Tn}
\end{align}
With the help of \eqref{3f7}, \eqref{form_F1}, \eqref{cal_U} and \eqref{cal_V}, one can simplify $T^l_{F^\ast_{\gamma}}$ to:
\begin{align}
	T^l_{F^\ast_{\gamma}}
	=&\frac12\left(-3 L^{-1}_{z}(\Gamma^{\ast}_{\gamma})+\alpha\Gamma^{\ast}_{\gamma}\right)\frac{2\alpha}{3\gamma}F^\ast_{\gamma}\left(\cos(2\theta)-\sin^2(\theta)\right) 
	\label{cal-Fgamma}\\
	&+L^{-1}_{z}(\Gamma^{\ast}_{\gamma})\left(\cos(2\theta)-\sin^2(\theta)\right)
	\frac{\alpha}{\gamma} F^\ast_{\gamma}(L^{-1}_{z}(\Gamma^\ast_{\gamma})-1)
	\notag \\
	=&\frac{\alpha}{\gamma}\left(\frac{\alpha}{3}-\gamma\right)F^\ast_{\gamma}\Gamma^{\ast}_{\gamma}\left(\cos(2\theta)-\sin^2(\theta)\right).
	\notag
\end{align}
This, along with the fact that $|w^K|\leq|w^\eta|$, Lemma \ref{7fm1} and \eqref{est-bd-Gamma}, yields directly that
\begin{align}\label{7f14-1}
	|\bar{L}_1|
	\lesssim&\|T^l_{F^\ast_{\gamma}}\|_{\mathcal{H}^2}\|g\|_{\mathcal{H}^2}
	\lesssim_\beta \alpha
	\|F^\ast_{\gamma}\|_{\mathcal{H}^{2,\ast}}\|g\|_{\mathcal{H}^2}
	\lesssim_\beta \alpha^2(1-\eta)^{-\frac{1}{2}}\|g\|_{\mathcal{H}^2}.
\end{align}
As for $\bar{L}_2$, by virtue of Lemmas \ref{7fm1}, \ref{7fm2} and \ref{7fm3}, we deduce that
\begin{align}\label{7f15-1}
	&~~~~\bigg|\bigg\langle \frac{U(\Phi-\frac{3}{4\alpha}L^{-1}_{z,K}(F^{\ast}_\gamma)\sin(2\theta))}{\sin(2\theta)}D_{\theta}F^\ast_{\gamma}, g(w^K)^2\bigg\rangle\bigg| 
	\\ \notag
	&\lesssim \bigg\|\frac{U(\Phi-\frac{3}{4\alpha}L^{-1}_{z,K}(F^{\ast}_\gamma)\sin(2\theta))}{\sin(2\theta)}D_{\theta}F^\ast_{\gamma}\bigg\|_{\mathcal{H}^2}\|g\|_{\mathcal{H}^2}\\ \notag
	&\lesssim_\beta\alpha^2(1-\eta)^{-\frac{1}{2}}\|g\|_{\mathcal{H}^2}
	\sum_{i=0}^{2}\left(\|D^i_z U(\tilde{G}_g+G^\ast_g)w_z\|_{L^2_z}+\|D^i_z U(\tilde{G}_{F^{\ast}_\gamma})w_z^\ast\|_{L^2_z}\right)
	\\ \notag
	&~~~~+\alpha^2\|g\|_{\mathcal{H}^2}\left(\bigg\|\frac{U(\tilde{\Phi}_g)}{\sin(2\theta)}\bigg\|_{\mathcal{H}^2}
	+\bigg\|\frac{U(\tilde{\Phi}_{F^{\ast}_\gamma})}{\sin(2\theta)}\bigg\|_{\mathcal{H}^{2,\ast}}\right)
	\\ \notag
	&\lesssim_\beta\alpha^2(1-\eta)^{-\frac{1}{2}}\|g\|_{\mathcal{H}^2}
	\left(\alpha^{-1}\|g\|_{\mathcal{H}^2}+\|F^{\ast}_\gamma\|_{\mathcal{H}^{2,\ast}}\right)
	\\ \notag
	&\lesssim_\beta
	\alpha(1-\eta)^{-\frac{1}{2}}\|g\|^2_{\mathcal{H}^2}+\alpha^3(1-\eta)^{-1}\|g\|_{\mathcal{H}^2}.
\end{align}
Note that 
\begin{align*}
	&V(f(z)\sin(2\theta))=2f(z)(\cos(2\theta)-\sin^2(\theta)),
	\quad
	D_\theta (\cos(2\theta)-\sin^2(\theta))=-3\sin^2(2\theta).
\end{align*}
These two estimates, combined with Lemma \ref{7fm1} and Lemma \ref{7fm2}, establish that
\begin{align}\label{7f17-1}
	&~~~~\bigg|\bigg\langle V\bigg(\Phi-\frac{3}{4\alpha}L^{-1}_{z,K}(F^{\ast}_\gamma)\sin(2\theta)\bigg)\alpha D_{z}F^\ast_{\gamma}, g(w^K)^2\bigg\rangle\bigg| \\ \notag
	&\lesssim_\beta\alpha^2(1-\eta)^{-\frac{1}{2}}\|g\|_{\mathcal{H}^2}\sum_{i=0}^{2}
	\left(\|D^i_z (\tilde{G}_g+G^\ast_g)w_z\|_{L^2_z}+\|D^i_z \tilde{G}_{F^{\ast}_\gamma}w^\ast_z\|_{L^2_z}\right) \\ \notag
	&~~~~+\alpha^2\|g\|_{\mathcal{H}^2}(\|V(\tilde{\Phi}_g)\|_{\mathcal{H}^2}+\|V(\tilde{\Phi}_{F^{\ast}_\gamma})\|_{\mathcal{H}^{2,\ast}})\\ \notag
	&\lesssim_\beta \alpha(1-\eta)^{-\frac{1}{2}}\|g\|^2_{\mathcal{H}^2}+\alpha^3(1-\eta)^{-1}\|g\|_{\mathcal{H}^2}.
\end{align}
Combining \eqref{7f15-1} and \eqref{7f17-1} together, we find that 
\begin{align}\label{7f19-1}
	&|\bar{L}_2|
	\lesssim_\beta \alpha(1-\eta)^{-\frac{1}{2}}\|g\|^2_{\mathcal{H}^2}+\alpha^3(1-\eta)^{-1}\|g\|_{\mathcal{H}^2}.
\end{align}
Substituting \eqref{7f14-1} and \eqref{7f19-1} into \eqref{est-bL-1}, then using the fact that $\alpha\ll(1-\eta)$, one gets \eqref{est-TF-H-1} immediately.

\medskip

\noindent
\textbf{Step 2: Estimate for $\mathcal{H}^{2}_{\eta}$-norm.} 
\medskip

\noindent
We recall the definition of $\mathcal{H}^{2}_{\eta}$-norm in \eqref{def_norm_H2eta}, to obtain that
\begin{align}\label{est-bM-1}
	\langle T_{F^\ast_{\gamma}}, g\rangle_{\mathcal{H}_{\eta}^2}
	=\langle T^l_{F^\ast_{\gamma}}, g\rangle_{\mathcal{H}_{\eta}^2}
	+\langle T^n_{F^\ast_{\gamma}}, g\rangle_{\mathcal{H}_{\eta}^2}
	\tri \bar{M}_1+\bar{M}_2.
\end{align}
It follows from the similar way as \eqref{7f14-1} that
\begin{align}\label{7f14}
	|\bar{M}_1|
	\lesssim&\|T^l_{F^\ast_{\gamma}}\|_{\mathcal{H}_{\eta}^2}\|g\|_{\mathcal{H}_{\eta}^2}
	\lesssim (1-\eta)^{-2}\|T^l_{F^\ast_{\gamma}}\|_{\mathcal{H}^2}\|g\|_{\mathcal{H}^2}
	\lesssim_\beta \alpha^2(1-\eta)^{-\frac{5}{2}}\|g\|_{\mathcal{H}^2}.
\end{align}
By employing a similar derivation procedure as in \eqref{7f15-1}, we directly deduce that
\begin{align}\label{7f16}
	&~~~~\bigg|\bigg\langle \frac{U(\Phi-\frac{3}{4\alpha}L^{-1}_{z,K}(F^{\ast}_\gamma)\sin(2\theta))}{\sin(2\theta)}D_{\theta}F^\ast_{\gamma}, g\bigg\rangle_{\mathcal{H}^2_\eta}\bigg| 
	\\ \notag
	&\lesssim (1-\eta)^{-2}\bigg\|\frac{U(\Phi-\frac{3}{4\alpha}L^{-1}_{z,K}(F^{\ast}_\gamma)\sin(2\theta))}{\sin(2\theta)}D_{\theta}F^\ast_{\gamma}\bigg\|_{\mathcal{H}^2}\|g\|_{\mathcal{H}^2}\\ \notag
	&\lesssim_\beta \alpha(1-\eta)^{-\frac{5}{2}}\|g\|^2_{\mathcal{H}^2}+\alpha^3(1-\eta)^{-3}\|g\|_{\mathcal{H}^2}.
\end{align}
The similar derivation to \eqref{7f17-1} directly yields that
\begin{align}\label{7f18}
	&\bigg|\bigg\langle V\bigg(\Phi-\frac{3}{4\alpha}L^{-1}_{z,K}(F^{\ast}_\gamma)\sin(2\theta)\bigg)\alpha D_{z}F^\ast_{\gamma}, g\bigg\rangle_{\mathcal{H}^2_\eta}\bigg| 
	\lesssim_\beta \alpha(1-\eta)^{-\frac{5}{2}}\|g\|^2_{\mathcal{H}^2}+\alpha^3(1-\eta)^{-3}\|g\|_{\mathcal{H}^2}.
\end{align}
We collect \eqref{7f16} and \eqref{7f18} together, to achieve that
\begin{align}\label{7f19}
	&|\bar{M}_2|\lesssim_\beta\alpha(1-\eta)^{-\frac{5}{2}}\|g\|^2_{\mathcal{H}^2}+\alpha^3(1-\eta)^{-3}\|g\|_{\mathcal{H}^2}.
\end{align}
Since $\alpha\ll(1-\eta)$, we insert \eqref{7f14} and \eqref{7f19} into \eqref{est-bM-1}, to get that \eqref{est-TF-H2eta} holds.

\medskip

\noindent
\textbf{Step 3: Estimate for $\mathcal{E}^{2}_{\eta}$-norm.} 
\medskip

\noindent
We remember the definition of $\mathcal{E}^{2}_{\eta}$-norm in \eqref{def_norm_E2eta}, to find that
\begin{align}\label{est-bN-1}
	\langle T_{F^\ast_{\gamma}}, g\rangle_{\mathcal{E}_{\eta}^2}
	=\langle T_{F^\ast_{\gamma}}^l, g\rangle_{\mathcal{E}_{\eta}^2}
	+\langle T_{F^\ast_{\gamma}}^n, g\rangle_{\mathcal{E}_{\eta}^2}
	\tri \bar{N}_1+\bar{N}_2.
\end{align}
In view of \eqref{est-Dz-Gamma-1}, \eqref{est-Gamma-lambda-L2} and \eqref{cal-Fgamma}, we have
\begin{align}\label{7f14'}
	|\bar{N}_1| 
	&\lesssim\|T^l_{F^\ast_{\gamma}}\|_{\mathcal{E}_{\eta}^2}
	\|g\|_{\mathcal{E}_{\eta}^2}\lesssim\alpha\|F^\ast_{\gamma}\Gamma^\ast_{\gamma}\|_{\mathcal{E}^2}\|g\|_{\mathcal{E}^2}
	\lesssim_\beta \alpha^{\frac52}\|\Gamma_{\theta,\gamma} w^{\lambda}_{\theta}\|_{L^2_\theta}
	\|g\|_{\mathcal{E}^2}
	\lesssim_\beta \alpha^2 \|g\|_{\mathcal{E}^2}.
\end{align}
By using \eqref{est-Gamma-lambda-L2}, Lemma \ref{6lem4}, Corollary \ref{6prop6-19} and Lemma \ref{7ha1}, we see that
\begin{align}\label{7f20}
	\bigg|\bigg\langle \frac{U(\Phi_g)}{\sin(2\theta)}D_{\theta}F^\ast_{\gamma}, g\bigg\rangle_{\mathcal{E}^2_\eta}\bigg|
	&\lesssim\left(\frac{\alpha^2}{\beta^4}\bigg\|\frac{U(\tilde{\Phi}_g)}{\sin(2\theta)}\bigg\|_{\mathcal{E}^2}
	+\|U(\tilde{G}_g+G^\ast_g)D_{\theta}F^\ast_{\gamma}\|_{\mathcal{E}^2}
	\right)\|g\|_{\mathcal{E}^2} 
	\\ \notag
	&\lesssim_\beta\alpha^2\left(\|g\|_{\mathcal{E}^2}
	+\alpha^\frac{1}{2}\sum_{i=0}^2\|D^i_z U(\tilde{G}_g+G^\ast_g)w_z\|_{L_z^2}\|\Gamma_{\theta,\gamma}w_\theta^\lambda\|_{L_\theta^2}
	\right)\|g\|_{\mathcal{E}^2} 
	\\ \notag
	&\lesssim_\beta \alpha^2 \left(\|g\|_{\mathcal{E}^2}
	+\alpha^{-1} \|g\|_{\mathcal{H}^2}\right)
	\|g\|_{\mathcal{E}^2} 
	\\ \notag
	&\lesssim_\beta \alpha\|g\|_{\mathcal{H}^2}\|g\|_{\mathcal{E}^2}+\alpha^2\|g\|^2_{\mathcal{E}^2},
\end{align}
where $\Phi_g\tri \tilde{\Phi}_g+(\tilde{G}_g+G^\ast_g)\sin(2\theta)$.
Thanks to \eqref{est-Gamma-lambda-L2}, Lemmas \ref{7fm1}, \ref{6lem4} and \ref{7ha1}, one can deduce that
\begin{align}\label{7f21}
	&~~~~\bigg|\bigg\langle \frac{U(\tilde{\Phi}_{F^\ast_{\gamma}})}{\sin(2\theta)}D_{\theta}F^\ast_{\gamma}, g\bigg\rangle_{\mathcal{E}^2_\eta}\bigg|
	+\left|\langle U(\tilde{G}_{F^\ast_{\gamma}})D_{\theta}F^\ast_{\gamma}, g\rangle_{\mathcal{E}^2_\eta}\right| \\ \notag
	&\lesssim_{\beta}\alpha^{\frac{5}{2}}
	\|g\|_{\mathcal{E}^2}\|\Gamma_{\theta,\gamma}w_\theta^\lambda\|_{L^2_\theta}\sum_{i=0}^2\left(\bigg\|D^i_z\bigg(\frac{U(\tilde{\Phi}_{F^\ast_{\gamma}})}{\sin(2\theta)}\bigg)w_z^\ast\bigg\|_{L_z^2L^{\infty}_\theta}+\|D_z^iU(\tilde{G}_{F^\ast_{\gamma}})w^\ast_z
	\|_{L_z^2}\right)\\ \notag
	&\lesssim_\beta \alpha^{2}\|g\|_{\mathcal{E}^2}\left(\alpha^{-\frac12}\|F^\ast_{\gamma}\|_{H^{3,\ast}}+\|F^\ast_{\gamma}\|_{H^{2,\ast}}\right)\\ \notag
	&\lesssim_\beta
	\alpha^{\frac{5}{2}}(1-\eta)^{-\frac{1}{2}}\|g\|_{\mathcal{E}^2}.
\end{align}
The application of \eqref{est-Gamma-lambda-L2}, Lemma \ref{6lem4}, Corollary \ref{6prop6-19} and Lemma \ref{7ha1}, yields directly that
\begin{align}\label{7f22}
	|\langle V(\Phi_g)\alpha D_{z}F^\ast_{\gamma}, g\rangle_{\mathcal{E}^2_\eta}| 
	&\lesssim\alpha\|g\|_{\mathcal{E}^2}
	\left(\frac{\alpha}{\beta^4}\|V(\tilde{\Phi}_g)\|_{\mathcal{E}^2}
	+\|V((\tilde{G}_g+G^\ast_g)\sin(2\theta))D_{z}F^\ast_{\gamma}\|_{\mathcal{E}^2}\right) 
	\\ \notag
	&\lesssim_\beta\alpha^2\|g\|_{\mathcal{E}^2}
	\left(\|g\|_{\mathcal{E}^2}
	+\alpha^\frac{1}{2}\sum_{i=0}^2\|D^i_z (\tilde{G}_g +G^\ast_g )w_z\|_{L_z^2}\|\Gamma_{\theta,\gamma}w_\theta^\lambda\|_{L_\theta^2}
	\right) 
	\\ \notag
	&\lesssim_\beta\alpha^2\|g\|_{\mathcal{E}^2} (\|g\|_{\mathcal{E}^2}
	+\alpha^{-1}\|g\|_{\mathcal{H}^2})
	\\ \notag
	&\lesssim_\beta\alpha\|g\|_{\mathcal{H}^2}\|g\|_{\mathcal{E}^2}+\alpha^2\|g\|^2_{\mathcal{E}^2},
\end{align}
and
\begin{align}\label{7f23}
	&~~~~\left|\langle V(\tilde{\Phi}_{F^\ast_{\gamma}})\alpha D_{z}F^\ast_{\gamma}, g\rangle_{\mathcal{E}^2_\eta}\right|
	+\left|\langle V(\tilde{G}_{F^\ast_{\gamma}}\sin(2\theta))\alpha D_{z}F^\ast_{\gamma}, g\rangle_{\mathcal{E}^2_\eta}\right| 
	\\ \notag
	&\lesssim_{\beta}\alpha^{\frac{5}{2}}\|g\|_{\mathcal{E}^2}\|\Gamma_{\theta,\gamma}w_\theta^\lambda\|_{L^2_\theta}
	\sum_{i=0}^2\left(\|D^i_z V(\tilde{\Phi}_{F^\ast_{\gamma}})w_z^\ast\|_{L_z^2L^{\infty}_\theta}+\|D_z^i \tilde{G}_{F^\ast_{\gamma}}w^\ast_z
	\|_{L_z^2}\right)
	\\ \notag
	&\lesssim_\beta \alpha^{\frac{3}{2}}\|g\|_{\mathcal{E}^2}\|F^\ast_{\gamma}\|_{H^{3,\ast}}\\ \notag
	&\lesssim_\beta \alpha^{\frac{5}{2}}(1-\eta)^{-\frac{1}{2}}\|g\|_{\mathcal{E}^2}.
\end{align}
The combination of \eqref{7f20}-\eqref{7f23}, yields that
\begin{align}\label{7f24}
	&|\bar{N}_2|\lesssim_\beta\alpha^{\frac{5}{2}}(1-\eta)^{-\frac{1}{2}}\|g\|_{\mathcal{E}^2}+\alpha\|g\|_{\mathcal{H}^2}\|g\|_{\mathcal{E}^2}+\alpha^2\|g\|^2_{\mathcal{E}^2}.
\end{align}
Plugging \eqref{7f14'}and \eqref{7f24} into \eqref{est-bN-1}, we finally get \eqref{est-TF-E2eta} immediately.
\end{proof}

\section{Blow-up phenomenon of the Euler equations}\label{sec:blow-up}
Firstly, we rewrite $F$ as a perturbation of $F^{\ast}_{\gamma}$, namely, $F=F^{\ast}_{\gamma}+g$. 
Then we recall that $g$ satisfies
\begin{align}\label{8eq1}
\left\{\begin{array}{l} \mathcal{L}_{\Gamma}(g)=-T+R_0+R_1+R_2,\\[1ex]
	L^{\alpha}_{z}(\Phi)+L_{\theta}(\Phi)=F, \\[1ex]
	\mu= \frac{3}{4\alpha}L^{-1}_{z,K}(R_2-T)(0),~~\gamma = \frac{1+\mu}{1-\mu}\beta, \\[1ex]
	L^{-1}_{z,K}(g)(0)=0,~~g|_{\partial D}=0,~~\Phi|_{\partial D}=0, 
\end{array}\right.
\end{align}
where the operator $\mathcal{L}_{\Gamma}$ and the remaining terms $R_i$ $(i=0,1,2)$ are defined in \eqref{def-LGammag}-\eqref{def-R2-1}, respectively.
We have obtained the coercivity estimate of the operator $\mathcal{L}_{\Gamma}$ in Section \ref{sec:Coer}, the elliptic estimates of the solution $\Phi$ in Section \ref{sec:Elli}, and the estimates of the transport term $T$ in Section \ref{sec:Est-T}. 
With all these estimates at hand, we intend to prove our main theorems in this section.

\subsection{Final \textit{a priori} estimate}
As observed in the first equation of \eqref{8eq1}, all terms except $R_0$, $R_1$ and $R_2$ have been rigorously addressed in Sections \ref{sec:Coer}-\ref{sec:Est-T}.
We begin by estimating $R_2$. 
Before proceeding, we establish some fundamental lemmas.

\begin{lemm}\label{8pr1}
Let $f(z,\theta)|_{\partial D}=g(z,\theta)|_{\partial D}=0$. Then we have
\begin{align}\label{est-fg-H2}
	\|fg\|_{\mathcal{H}^{2}}\leq C\sqrt{\frac{\beta}{\alpha}}\|f\|_{\mathcal{H}^{2}}\|g\|_{\mathcal{H}^{2}}
	,~~~
	\|fg\|_{\mathcal{H}^{2}}\leq C\sqrt{\frac{\beta}{\alpha}}\|f\|_{\mathcal{H}^{3,\ast}}\|g\|_{\mathcal{H}^{2}}.
\end{align}
\end{lemm}
\begin{proof}
By using Corollary \ref{7em2}, it is easy to verify that the first inequality of \eqref{est-fg-H2} holds true.
It remains to prove the second one.
By virtue of Corollary \ref{7em2}, we infer that
\begin{align*}  \|fgw^\eta\|_{L^2}\lesssim\|f\|_{L^\infty}\|gw^\eta\|_{L^2}\lesssim \sqrt{\frac{\beta}{\alpha}}\|f\|_{\mathcal{H}^{2,\ast}}\|gw^\eta\|_{L^2},
\end{align*}
and
\begin{align*}  &\|D_z(fg)w^\eta\|_{L^2}\lesssim\|f\|_{L^\infty}\|D_zgw^\eta\|_{L^2}+\|D_zfw_\theta^\eta\|_{L_z^\infty L_\theta^2}\|gw_z\|_{L^2_z L_\theta^\infty}\lesssim \sqrt{\frac{\beta}{\alpha}}\|f\|_{\mathcal{H}^{2,\ast}}\|g\|_{\mathcal{H}^{1}},
	\\
	&\|D_\theta(fg)w^\lambda\|_{L^2}\lesssim\|f\|_{L^\infty}\|D_\theta gw^\lambda\|_{L^2}+\|D_\theta fw_\theta^\lambda\|_{L_z^\infty L_\theta^2}\|gw_z\|_{L^2_z L_\theta^\infty}\lesssim \sqrt{\frac{\beta}{\alpha}}\|f\|_{\mathcal{H}^{2,\ast}}\|g\|_{\mathcal{H}^{1}}.
\end{align*}
Along the similar way, one obtains that
\begin{align*}  
	&\|D^2_\theta(fg)w^\lambda\|_{L^2}
	+\|D^2_z(fg)w^\eta\|_{L^2}
	\\
	\lesssim&\|f\|_{L^\infty}\|D^2_\theta gw^\lambda\|_{L^2}+\|D^2_\theta fw_\theta^\lambda\|_{L_z^\infty L_\theta^2}\|gw_z\|_{L^2_z L_\theta^\infty}+\|D_\theta fw_\theta^\lambda\|_{L_z^\infty L_\theta^2}\|D_\theta 
	gw_z\|_{L^2_z L_\theta^\infty} \\
	&+\|f\|_{L^\infty}\|D^2_z gw^\eta\|_{L^2}+\|D^2_z fw_\theta^\eta\|_{L_z^\infty L_\theta^2}\|gw_z\|_{L^2_z L_\theta^\infty}+\|D_z fw_\theta^\eta\|_{L_z^\infty L_\theta^2}\|D_z 
	gw_z\|_{L^2_z L_\theta^\infty} \\
	\lesssim &\sqrt{\frac{\beta}{\alpha}}\|f\|_{\mathcal{H}^{3,\ast}}\|g\|_{\mathcal{H}^{2}},
\end{align*}
and
\begin{align*}  
	\|D_\theta D_z(fg)w^\lambda\|_{L^2}
	&\lesssim\|f\|_{L^\infty}\|D_\theta D_z gw^\lambda\|_{L^2}+\|D_\theta D_z fw_\theta^\lambda\|_{L_z^\infty L_\theta^2}\|gw_z\|_{L^2_z L_\theta^\infty} \\
	&~~~+\|D_\theta fw_\theta^\lambda\|_{L_z^\infty L_\theta^2}\|D_z 
	gw_z\|_{L^2_z L_\theta^\infty}+\|D_z f\|_{L^\infty}\|D_\theta 
	gw^\lambda\|_{L^2}\\
	&\lesssim\sqrt{\frac{\beta}{\alpha}}\|f\|_{\mathcal{H}^{3,\ast}}\|g\|_{\mathcal{H}^{2}}.
\end{align*}
We thus complete the proof of Lemma \ref{8pr1}.
\end{proof}

Applying the similar way as the proof of Lemma \ref{7fm2}, we can get the following lemma directly. And its proof is omitted here.
\begin{lemm}\label{8pr2}
Assume that $f(z,\theta)|_{\partial D}=0$. 
Let $\beta \in (0,1]$.
There exists a sufficiently small constant $\alpha>0$, such that if $\alpha\ll \beta$ and $|\mu|\leq \alpha^\frac{1}{2}$, then there holds
\begin{align*}
	\|fF^\ast_{\gamma} \|_{\mathcal{H}^{2}}\leq C\frac{\alpha}{\beta^3}\|f\|_{\mathcal{H}^{2}},~~~~\|fF^\ast_{\gamma} \|_{\mathcal{H}^{2}}\leq C\frac{\alpha}{\beta^3}\|f\|_{\mathcal{H}^{2,\ast}}.
\end{align*}
\end{lemm}

Based on Lemmas \ref{8pr1}-\ref{8pr2}, we now proceed to deal with $R_2$ by using the product rules.
\begin{prop}\label{8R1}
Let $L^{-1}_{z,K}(g)(0)=0$. There exist constants $\alpha>0$ sufficiently small and $\eta(\beta)$, such that if $\alpha\ll 1-\eta\ll\beta$ and $|\mu|\leq \alpha^\frac{1}{2}$, then 
\begin{align}
	|\langle R_2, g(w^K)^2\rangle|
	&\leq C_\beta(1-\eta)^{-1}(\alpha^{-1}\|g\|^3_{\mathcal{H}^2}
	+\alpha^{\frac{1}{2}}\|g\|^2_{\mathcal{H}^2}
	+\alpha^2\|g\|_{\mathcal{H}^2}),
	\label{est-R2-H-1}\\
	|\langle R_2, g\rangle_{\mathcal{H}_{\eta}^2}|
	&\leq C_\beta(1-\eta)^{-3}(\alpha^{-1}\|g\|^3_{\mathcal{H}^2}
	+\alpha^\frac{1}{2}\|g\|^2_{\mathcal{H}^2}
	+\alpha^2\|g\|_{\mathcal{H}^2}),
	\label{est-R2-H2eta}\\
	|\langle R_2, g\rangle_{\mathcal{E}_{\eta}^2}|
	&\leq C_\beta(\alpha^{-1}\|g\|_{\mathcal{H}^2}\|g\|^2_{\mathcal{E}^2} +\alpha^{\frac{1}{2}}\|g\|^2_{\mathcal{E}^2}+\alpha^{\frac{3}{2}}\|g\|_{\mathcal{E}^2}).
	\label{est-R2-E2eta}
\end{align}
\end{prop}
\begin{proof}
This proof will be divided into three steps.

\medskip

\noindent
\textbf{Step 1: Estimate for $\mathcal{H}^{-1}$-norm.} 

\medskip

\noindent
Recalling the definition of $R_2$ in \eqref{def-R2-1} and $F=F^\ast_\gamma+g$, we can decompose the $\mathcal{H}^{-1}$-norm as:
\begin{align}\label{est-bmI-1}
	\langle R_2, g(w^K)^2\rangle
	&=
	\frac{3}{2\alpha}\langle L^{-1}_{z,K}(g)g, g(w^K)^2\rangle
	-\frac{3}{2}\langle \sin^2(\theta)\langle F,K\rangle_{\theta}F, g(w^K)^2\rangle
	\\
	&~~~
	+\langle \mathcal{R}(\Phi_g-G^\ast_{g}\sin(2\theta))g, g(w^K)^2\rangle
	+\langle \mathcal{R}(\Phi_{F^\ast_\gamma}-G^\ast_{F^\ast_\gamma}\sin(2\theta))g, g(w^K)^2\rangle
	\notag \\
	&~~~
	+\langle \mathcal{R}(\Phi_g-G^\ast_{g}\sin(2\theta))F^\ast_\gamma, g(w^K)^2\rangle
	+\langle \mathcal{R}(\Phi_{F^\ast_\gamma}-G^\ast_{F^\ast_\gamma}\sin(2\theta))F^\ast_\gamma, g(w^K)^2\rangle
	\notag \\
	&\tri \sum_{i=1}^{6}\bar{\mathcal{I}}_i,
	\notag
\end{align}
where for $f$ being $g$ or $F^\ast_\gamma$, the notation $\Phi_f$ is defined as $\Phi_f\tri\tilde{\Phi}_f+\tilde{G}_{f}\sin(2\theta)+{G}_{f}^\ast\sin(2\theta)$.
We now estimate $\bar{\mathcal{I}}_i$ $(i=1,\cdots,6)$ term by term.
Firstly, since $L^{-1}_{z,K}(g)$ is independent of $\theta$, we then deduce from Lemma \ref{7em1} that
\begin{align}\label{8in6-1}
	|\bar{\mathcal{I}}_1|
	\lesssim \alpha^{-1}\|L^{-1}_{z,K}(g)g\|_{\mathcal{H}^2}\|g\|_{\mathcal{H}^2}
	\lesssim \alpha^{-1}\sum_{i=1}^2\|D^i_z L^{-1}_{z,K}(g)w_z\|_{L^2_z}\|g\|^2_{\mathcal{H}^2}
	\lesssim \alpha^{-1}\|g\|^3_{\mathcal{H}^2}.
\end{align}
By applying Lemmas \ref{7fm1}, \ref{8pr1} and \ref{8pr2}, we find that $\bar{\mathcal{I}}_2$ can be controlled by 
\begin{align}\label{8in5-1}
	|\bar{\mathcal{I}}_2|
	&\lesssim_{\beta} \alpha^{-\frac12}\|g\|^2_{\mathcal{H}^2}\left(\|\sin^2(\theta)\langle g,K\rangle_{\theta}\|_{\mathcal{H}^2}+\|\sin^2(\theta)\langle F^\ast_\gamma,K\rangle_{\theta}\|_{\mathcal{H}^{3,\ast}}\right)
	\\ \notag
	&~~~~+\alpha\|g\|_{\mathcal{H}^2}\left(\|\sin^2(\theta)\langle g,K\rangle_{\theta}\|_{\mathcal{H}^2}+\|\sin^2(\theta)\langle F^\ast_\gamma,K\rangle_{\theta}\|_{\mathcal{H}^{2,\ast}}\right)
	\\ \notag
	&\lesssim_{\beta} (1-\eta)^{-1}
	\left(\alpha^{-\frac12}\|g\|^3_{\mathcal{H}^2}
	+\alpha^\frac{1}{2}\|g\|^2_{\mathcal{H}^2}
	+\alpha^2\|g\|_{\mathcal{H}^2}\right),
\end{align}
where we have used the fact that
\begin{align}\label{8est-sinF}
	\|\sin^2(\theta)\langle F^\ast_\gamma,K\rangle_{\theta}\|_{\mathcal{H}^{k,\ast}}
	\lesssim_{\beta} \alpha (1-\eta)^{-1}.
\end{align}
We employ Corollary \ref{6cor1} and Lemma \ref{7ha1}, to obtain that
\begin{align}\label{est-mRg-1}
	\|\mathcal{R}(\tilde{\Phi}_g)\|_{\mathcal{H}^2}
	+\|\mathcal{R}(\tilde{G}_{g}\sin(2\theta))\|_{\mathcal{H}^2}
	\lesssim_{\beta} \|g\|_{\mathcal{H}^2}.
\end{align}
This estimate, together with Lemma \ref{8pr1}, yields directly that
\begin{align}\label{8in1-1}
	|\bar{\mathcal{I}}_3|\lesssim  \alpha^{-\frac12} (\|\mathcal{R}(\tilde{\Phi}_g)\|_{\mathcal{H}^2}+\|\mathcal{R}(\tilde{G}_{g}\sin(2\theta))\|_{\mathcal{H}^2})\|g\|^2_{\mathcal{H}^2}
	\lesssim_\beta  \alpha^{-\frac12} \|g\|^3_{\mathcal{H}^2}.
\end{align}
With the help of Lemma \ref{7fm1}, Corollary \ref{6cor2} and Lemma \ref{7ha1}, we achieve that
\begin{align}\label{est-mRF-1}
	\|\mathcal{R}(\tilde{\Phi}_{F^\ast_\gamma})\|_{\mathcal{H}^{3,\ast}}
	+\|\mathcal{R}(\tilde{G}_{F^\ast_\gamma}\sin(2\theta))\|_{\mathcal{H}^{3,\ast}}
	\lesssim_{\beta} \|F^\ast_\gamma\|_{\mathcal{H}^{3,\ast}}
	\lesssim_{\beta} \alpha (1-\eta)^{-\frac12}.
\end{align}
Combining the estimate above and Lemma \ref{8pr1} together, it is enough to prove that
\begin{align}\label{8in2-1}
	|\bar{\mathcal{I}}_4|
	&\lesssim \alpha^{-\frac12}(\|\mathcal{R}(\tilde{\Phi}_{F^\ast_\gamma})\|_{\mathcal{H}^{3,\ast}}+\|\mathcal{R}(\tilde{G}_{F^\ast_\gamma}\sin(2\theta))\|_{\mathcal{H}^{3,\ast}})\|g\|^2_{\mathcal{H}^2}
	\lesssim_\beta \alpha^{\frac{1}{2}} (1-\eta)^{-\frac{1}{2}}\|g\|^2_{\mathcal{H}^2}.
\end{align}
The application of Lemma \ref{8pr2} and \eqref{est-mRg-1}, gives that
\begin{align}\label{8in3-1}
	|\bar{\mathcal{I}}_5|
	&\lesssim \alpha (\|\mathcal{R}(\tilde{\Phi}_g)\|_{\mathcal{H}^2}+\|\mathcal{R}(\tilde{G}_{g}\sin(2\theta))\|_{\mathcal{H}^2})\|g\|_{\mathcal{H}^2}
	\lesssim_\beta \alpha\|g\|^2_{\mathcal{H}^2}.
\end{align}
By virtue of Lemma \ref{8pr2} and \eqref{est-mRF-1}, one gets that
\begin{align}\label{8in4-1}
	|\bar{\mathcal{I}}_6|
	&\lesssim \alpha (\|\mathcal{R}(\tilde{\Phi}_{F^\ast_\gamma})\|_{\mathcal{H}^{2,\ast}}+\|\mathcal{R}(\tilde{G}_{F^\ast_\gamma}\sin(2\theta))\|_{\mathcal{H}^{2,\ast}})\|g\|_{\mathcal{H}^2}
	\lesssim_\beta (1-\eta)^{-\frac{1}{2}}\alpha^{2}\|g\|_{\mathcal{H}^2}.
\end{align}
Substituting \eqref{8in6-1}-\eqref{8in4-1} into \eqref{est-bmI-1}, we finally obtain \eqref{est-R2-H-1}.

\medskip

\noindent
\textbf{Step 2: Estimate for $\mathcal{H}^{2}_{\eta}$-norm.} 
\medskip

\noindent
It is easy to check that
\begin{align*}
	\|f\|_{\mathcal{H}^2_{\eta}}\lesssim (1-\eta)^{-1} \|f\|_{\mathcal{H}^2}.
\end{align*}
We employ the estimate above and the definition of $R_2$ in \eqref{def-R2-1}, then use a similar derivation as $\bar{\mathcal{I}_i}$ $(i=1,\cdots,6)$ in \eqref{8in6-1}-\eqref{8in4-1}, finally deduce \eqref{est-R2-H2eta} immediately.

\medskip

\noindent
\textbf{Step 3: Estimate for $\mathcal{E}^{2}_{\eta}$-norm.} 
\medskip

\noindent
We remember the definition of $R_2$ in \eqref{def-R2-1} and $F=F^\ast_\gamma+g$, to decompose the $\mathcal{E}^{2}_{\eta}$-norm as follows:
\begin{align}\label{est-bmJ-1}
	\langle R_2, g(w^K)^2\rangle
	&=
	\frac{3}{2\alpha}\langle L^{-1}_{z,K}(g)g, g\rangle_{\mathcal{E}_{\eta}^2}
	-\frac{3}{2}\langle \sin^2(\theta)\langle F,K\rangle_{\theta}F, g\rangle_{\mathcal{E}_{\eta}^2}
	+\langle \mathcal{R}(\Phi-G^\ast_{F}\sin(2\theta))F, g\rangle_{\mathcal{E}_{\eta}^2}
	\\
	&\tri \sum_{i=1}^{3}\bar{\mathcal{J}}_i,
	\notag
\end{align}
We now estimate $\bar{\mathcal{J}}_i$ $(i=1,2,3)$ in turn.
At first, we apply Lemma \ref{7em1}, to discover that
\begin{align}\label{8in10}
	|\bar{\mathcal{J}}_1| 
	\lesssim  \alpha^{-1}\|L^{-1}_{z,K}(g)g\|_{\mathcal{E}^2}\|g\|_{\mathcal{E}^2}
	\lesssim \alpha^{-1}\sum_{i=1}^2\|D^i_z L^{-1}_{z,K}(g)w_z\|_{L^2} \|g\|^2_{\mathcal{E}^2}
	\lesssim \alpha^{-1}\|g\|_{\mathcal{H}^2}\|g\|^2_{\mathcal{E}^2}.
\end{align}
Again thanks to Lemma \ref{7em1}, it then follows from Lemma \ref{8pr1} and Lemma \ref{8pr2} that
\begin{align}\label{8in9}
	|\bar{\mathcal{J}}_2| 
	&\lesssim_{\beta} \alpha^{-\frac12}\|g\|^2_{\mathcal{E}^2}
	\bigg(\sum_{i=0}^2\|\langle D^i_z g,K\rangle_{\theta}w_z\|_{L^2}+\sum_{j=0}^3\|\langle D^j_z F^\ast_\gamma,K\rangle_{\theta}w^\ast_z\|_{L^2}\bigg)\\ \notag
	&~~~~+\alpha \|g\|_{\mathcal{E}^2}\sum_{i=0}^2\left(\|\langle D^i_z g,K\rangle_{\theta}w_z\|_{L^2}+\|\langle D^i_z F^\ast_\gamma,K\rangle_{\theta}w^\ast_z\|_{L^2}\right)\\ \notag
	&\lesssim_{\beta} \alpha^{-\frac12}\|g\|^2_{\mathcal{E}^2}\|g\|_{\mathcal{H}^2}+\alpha^\frac{1}{2}\|g\|^2_{\mathcal{E}^2}+\alpha^2\|g\|_{\mathcal{E}^2}.
\end{align}
With the help of \eqref{est-mRg-1}, \eqref{est-mRF-1}, Corollaries \ref{6cor2}, \ref{6cor3} and Lemmas \ref{7em1}, \ref{7ha1}, \ref{8pr1}, \ref{8pr2}, one obtains that
\begin{align}\label{8in8}
	|\bar{\mathcal{J}}_3| 
	&\lesssim_{\beta} \alpha^{-\frac12} \|\mathcal{R}(\Phi_g-G^\ast_{g}\sin(2\theta))\|_{\mathcal{H}^2}\|g\|^2_{\mathcal{E}^2}
	+\alpha^{-\frac12}\|\mathcal{R}(\Phi_g-G^\ast_{g}\sin(2\theta))\|_{\mathcal{E}^2}\|g\|_{\mathcal{H}^2}\|g\|_{\mathcal{E}^2} 
	\\ \notag
	&~~~+\alpha^{-\frac12}(1-\eta)^2\|\mathcal{R}(\Phi_{F^\ast_\gamma}-G^\ast_{F^\ast_\gamma}\sin(2\theta))\|_{\mathcal{H}^{4,\ast}}\|g\|^2_{\mathcal{E}^2}
	+\alpha \|g\|_{\mathcal{E}^2}\Big(\|\mathcal{R}(\Phi_g-G^\ast_{g}\sin(2\theta))\|_{\mathcal{E}^2}
	\\ \notag
	&~~~
	+\alpha^{\frac12}(1-\eta)^2 \sum_{i=0}^2\|D^i_z\mathcal{R}(\Phi_{F^\ast_\gamma}-G^\ast_{F^\ast_\gamma}\sin(2\theta))w^\ast_z\|_{L^{2}_z L^\infty_\theta}\|\Gamma_{\theta,\gamma}w_\theta^\lambda\|_{L^2_\theta}\Big)
	\\ \notag
	&\lesssim_\beta \alpha^{-\frac12}\|g\|_{\mathcal{H}^2}\|g\|^2_{\mathcal{E}^2}
	+\alpha^\frac{1}{2}\|g\|^2_{\mathcal{E}^2}
	+\alpha^{\frac{3}{2}}\|g\|_{\mathcal{E}^2},
\end{align}
where we have used the fact that 
\begin{align*}
	\sum_{i=0}^2\|D^i_z\mathcal{R}(\Phi_{F^\ast_\gamma}-G^\ast_{F^\ast_\gamma}\sin(2\theta))w^\ast_z\|_{L^{2}_z L^\infty_\theta}
	&\lesssim \alpha^{-\frac12}\|D^i_z\mathcal{R}(\Phi_{F^\ast_\gamma}-G^\ast_{F^\ast_\gamma}\sin(2\theta))w^\ast_z\|_{\mathcal{H}^{3,\ast}}
	\\
	&\lesssim \alpha^{-\frac12}\|F^\ast_\gamma\|_{\mathcal{H}^{3,\ast}}
	\lesssim \alpha^{\frac12}(1-\eta)^{-\frac12}.
\end{align*}
Plugging \eqref{8in10}-\eqref{8in8} into \eqref{est-bmJ-1}, we conclude that \eqref{est-R2-E2eta} holds.

\end{proof}

It remains to address $R_0$ and $R_1$ in the first equation of \eqref{8eq1}.
Noting that $\mu$ is present in both terms, we first derive the following lemma concerning the estimation of $\mu$.

\begin{prop}\label{8R2}
Let $L^{-1}_{z,K}(g)(0)=0$. There exist constants $\alpha>0$ sufficiently small and $\eta(\beta)$, such that if $\alpha\ll 1-\eta\ll \beta$ and $|\mu|\leq \alpha^\frac{1}{2}$, then
\begin{align}
	|\mu|&\leq C_\beta\alpha^{-1}(1-\eta)\|g\|_{\mathcal{H}^2_\eta}+C_\beta(1-\eta)^{-\frac{1}{2}}\left(\alpha^{-2}\|g\|^2_{\mathcal{H}^2}+\alpha \right). \label{est-mu-1}
\end{align}
\end{prop}
\begin{proof}
Recalling the definition in $\eqref{8eq1}_3$, we obtain that
\begin{align}\label{est-bk-1}
	\mu=- \frac{3}{4\alpha}L^{-1}_{z,K}(T)(0)
	+ \frac{3}{4\alpha}L^{-1}_{z,K}(R_2)(0)
	\tri \bar{\mathcal{K}}_1 +\bar{\mathcal{K}}_2.
\end{align}
In order to estimate $ \bar{\mathcal{K}}_1$, we decompose it as the following two terms:
\begin{align}\label{est-bK1-1}
	\bar{\mathcal{K}}_1 
	=- \frac{3}{4\alpha}L^{-1}_{z,K}(T^l)(0)
	- \frac{3}{4\alpha}L^{-1}_{z,K}(T^n)(0)
	\tri \bar{\mathcal{K}}_{11}+\bar{\mathcal{K}}_{12}, 
\end{align}
where $T^l$ and $T^n$ are defined by
\begin{align*}
	T^l&\tri\frac{U(\frac{3}{4\alpha}L^{-1}_{z,K}(F^{\ast}_\gamma)\sin(2\theta))}{\sin(2\theta)}D_{\theta}F+V\bigg(\frac{3}{4\alpha}L^{-1}_{z,K}(F^{\ast}_\gamma)\sin(2\theta)\bigg)\alpha D_{z}F,
	\\
	T^n&\tri\frac{U(\Phi-\frac{3}{4\alpha}L^{-1}_{z,K}(F^{\ast}_\gamma)\sin(2\theta))}{\sin(2\theta)}D_{\theta}F+V\bigg(\Phi-\frac{3}{4\alpha}L^{-1}_{z,K}(F^{\ast}_\gamma)\sin(2\theta)\bigg)\alpha D_{z}F.
\end{align*}
Applying \eqref{cal_U}, \eqref{cal_V}, \eqref{cal-Fgamma} and the fact that $F=F^\ast_\gamma+g$, one can simplify $T^l$ as follows:
\begin{align}\label{sim-Tl}
	T^l
	&=\frac12(-3L^{-1}_{z}(\Gamma^{\ast}_{\gamma})+\alpha\Gamma^{\ast}_{\gamma})D_{\theta}g+L^{-1}_{z}(\Gamma^{\ast}_{\gamma})\left(\cos(2\theta)-\sin^2(\theta)\right)\alpha D_{z}g \\
	&~~~+\frac{\alpha}{\gamma} \left(\frac{\alpha}{3}-\gamma\right) F^\ast_{\gamma} \Gamma^{\ast}_{\gamma} \left(\cos(2\theta)-\sin^2(\theta)\right).
	\notag
\end{align}
It is easy to verify that for $\xi=\eta$ or $\lambda$,
\begin{align}\label{est-Ktheta-1}
	\|K(\theta) z^{-1}(w^{\xi})^{-1}\|_{L^2}
	\lesssim 
	\|K(\theta) (w^{\xi}_{\theta})^{-1}\|_{L^2_{\theta}} \| z^{-1}w_z^{-1}\|_{L^2_z} \lesssim 1.
\end{align}
We employ \eqref{sim-Tl} and \eqref{est-Ktheta-1}, to discover that
\begin{align}\label{8in12}
	|\bar{\mathcal{K}}_{11}|
	& \lesssim \alpha^{-1}|\langle T^l,K(\theta)z^{-1}\rangle| 
	\\ \notag
	&\lesssim_{\beta} \alpha^{-1} \langle |D_\theta g|+\alpha|D_zg|, K(\theta)z^{-1}\rangle
	+
	\langle F^\ast_\gamma, K(\theta)z^{-1}\rangle 
	\\ \notag
	&\lesssim_{\beta} \alpha^{-1}  \|D_\theta g w^\lambda\|_{L^2}\|K(\theta) z^{-1}(w^{\lambda})^{-1} \|_{L^2}
	+\|D_z g w^\eta\|_{L^2}\|K(\theta) z^{-1}(w^{\eta})^{-1} \|_{L^2}
	\\ \notag
	&~~~~+\alpha \|K(\theta)\Gamma_{\theta,\gamma}\|_{L^2_\theta} \|\Gamma_{\gamma}^{\ast}z^{-1}\|_{L^2_z}
	\\ \notag
	&\lesssim_{\beta} \alpha^{-1} (1-\eta) \|g\|_{\mathcal{H}^2_\eta}+\|g\|_{\mathcal{H}^2}+\alpha.
\end{align}
Recalling \eqref{7g21} and \eqref{7g22} and \eqref{est-Ktheta-1}, it is enough to prove that
\begin{align}\label{8in13}
	|\bar{\mathcal{K}}_{12}|
	&\lesssim  \alpha^{-1}|\langle T^n,K(\theta)z^{-1}\rangle| 
	\\ \notag
	&\lesssim_\beta \alpha^{-1} (\alpha^{\frac{1}{2}}(1-\eta)^{-\frac{1}{2}}+\alpha^{-1} \|g\|_{\mathcal{H}^{2}}) \langle |D_\theta F|+\alpha|D_zF|, K(\theta)z^{-1} \rangle 
	\\ \notag
	&\lesssim_\beta \alpha^{-1} (\alpha^{\frac{1}{2}}(1-\eta)^{-\frac{1}{2}}+\alpha^{-1}\|g\|_{\mathcal{H}^{2}})(\|g\|_{\mathcal{H}^2}+\alpha^2) 
	\\ \notag
	&\lesssim_\beta \alpha^{-2} \|g\|^2_{\mathcal{H}^2}+\alpha^{-\frac{1}{2}}(1-\eta)^{-\frac{1}{2}}\|g\|_{\mathcal{H}^2}+\alpha,
\end{align}
where we have used the fact that $\alpha\ll 1-\eta\ll \beta$ and 
\begin{align*}
	\langle |D_\theta F^\ast_\gamma|+\alpha  |D_z F^\ast_\gamma|, K(\theta) z^{-1} \rangle \lesssim_{\beta} \alpha^2.
\end{align*}
Substituting \eqref{8in12} and \eqref{8in13} into \eqref{est-bK1-1}, we find that
\begin{align}\label{est-bk1-2}
	|\bar{\mathcal{K}}_{1}|
	\lesssim_{\beta} \alpha^{-1} (1-\eta) \|g\|_{\mathcal{H}^2_\eta}
	+ (1-\eta)^{-\frac12}\left(\alpha^{-2} \|g\|_{\mathcal{H}^2}^2+\alpha\right).
\end{align}
As for $\bar{\mathcal{K}}_{2}$, we deduce from the definition of $R_2$ in \eqref{def-R2-1} that
\begin{align}\label{est-bk2-1}
	\bar{\mathcal{K}}_{2} 
	&= \frac{9}{8\alpha^2} L^{-1}_{z,K}\big(L^{-1}_{z,K}(g)g\big)(0)
	-\frac{9}{8\alpha} L^{-1}_{z,K}\big((\sin\theta)^2\langle F,K\rangle_{\theta}F\big)(0)
	+\frac{3}{4\alpha} L^{-1}_{z,K}(\mathcal{R}(\Phi-G^\ast_{F}\sin(2\theta))F)(0)
	\\ \notag
	&
	\tri \bar{\mathcal{K}}_{21}+\bar{\mathcal{K}}_{22}+\bar{\mathcal{K}}_{23}.   
\end{align}
After some direct calculations, we have
\begin{align}\label{est-bk21-1}
	\bar{\mathcal{K}}_{21}
	=  -\frac{9}{8\alpha^2}\int_0^{\infty}L^{-1}_{z,K}(g)\partial_zL^{-1}_{z,K}(g)dz
	=  \frac{9}{16\alpha^2}\big(L^{-1}_{z,K}(g)(0)\big)^2
	=  0.
\end{align}
Along a similar way as \eqref{est-Ktheta-1}, we see that for $\xi=\eta$ or $\lambda$,
\begin{align*}
	\| K(\theta) (w^\xi_\theta)^{-1}\|_{L^2_\theta} \lesssim 1.
\end{align*}
This, along with Lemma \ref{7em1} and \eqref{est-Ktheta-1}, yields that
\begin{align}\label{8in15}
	|\bar{\mathcal{K}}_{22}|
	&\lesssim \alpha^{-1} \| \langle |F|, K(\theta) \rangle_{\theta} \|_{L^\infty_\theta} 
	\langle |F|,K(\theta)z^{-1} \rangle 
	\\ \notag
	&\lesssim   \alpha^{-1} \left(\|gw^\eta_{\theta}\|_{L^\infty_zL^2_\theta}\|K(\theta)(w^\eta_\theta)^{-1}\|_{L^2_\theta} +\alpha \|K(\theta) \Gamma_{\theta,\gamma}\|_{L^2_\theta} \|\Gamma^\ast_\gamma\|_{L^\infty_z} \right)
	\\ \notag
	&~~~\times \left( 
	\|gw^\eta\|_{L^2}\|K(\theta)(w^\eta)^{-1}\|_{L^2} +\alpha \|K(\theta) \Gamma_{\theta,\gamma}\|_{L^2_\theta} \|\Gamma^\ast_\gamma z^{-1}\|_{L^2_z} 
	\right)
	\\ \notag
	&\lesssim_\beta \alpha^{-1} (\frac{\alpha}{\beta}+\|g\|_{\mathcal{H}^{2}})^2
	\lesssim \alpha^{-1}  \|g\|^2_{\mathcal{H}^2}+ \|g\|_{\mathcal{H}^2}+\alpha.
\end{align}
We use the similar derivation as \eqref{7g21}, \eqref{7g22} and \eqref{8in12}, to get that
\begin{align*}
	&\|\mathcal{R}(\Phi-G^\ast_{F}\sin(2\theta))\|_{L^\infty} \lesssim_\beta
	\alpha^{-1} \|g\|_{\mathcal{H}^2}+ \alpha^{\frac12} (1-\eta)^{-\frac12} ,
	\\
	&\langle |\mathcal{R}(\Phi-G^\ast_{F}\sin(2\theta))|, K(\theta) z^{-1} \rangle \lesssim_{\beta}  \|g\|_{\mathcal{H}^2} +\alpha (1-\eta)^{-\frac12}.
\end{align*}
Applying these two estimates above and using the similar way as \eqref{8in12}, one obtains that
\begin{align}\label{8in14}
	|\bar{\mathcal{K}}_{23}|
	&\lesssim \alpha^{-1}  \|\mathcal{R}(\Phi-G^\ast_{F}\sin(2\theta))\|_{L^\infty} \langle |g|, K(\theta) z^{-1}\rangle 
	+\alpha^{-1} \|F^\ast_\gamma \|_{L^\infty} 
	\langle |\mathcal{R}(\Phi-G^\ast_{F}\sin(2\theta))|, K(\theta) z^{-1} \rangle
	\\ \notag
	&\lesssim_\beta \alpha^{-1} (\alpha^{\frac{1}{2}}(1-\eta)^{-\frac{1}{2}}+\alpha^{-1} \|g\|_{\mathcal{H}^{2}}) \|g\|_{\mathcal{H}^2}+\|g\|_{\mathcal{H}^2} +\alpha (1-\eta)^{-\frac12} 
	\\ \notag
	&\lesssim_\beta 
	\alpha^{-2}\|g\|^2_{\mathcal{H}^2}
	+\alpha^{-\frac{1}{2}}  (1-\eta)^{-\frac{1}{2}}\|g\|_{\mathcal{H}^2}
	+\alpha(1-\eta)^{-\frac{1}{2}}.
\end{align}
Plugging \eqref{est-bk21-1}, \eqref{8in15} and \eqref{8in14} into \eqref{est-bk2-1}, we achieve that
\begin{align}\label{est-bk2-2}
	|\bar{\mathcal{K}}_2| \lesssim_{\beta}   \alpha^{-2}\|g\|^2_{\mathcal{H}^2} +\alpha^{-\frac{1}{2}}  (1-\eta)^{-\frac{1}{2}}\|g\|_{\mathcal{H}^2}
	+\alpha(1-\eta)^{-\frac{1}{2}}.
\end{align}
We then insert \eqref{est-bk1-2} and \eqref{est-bk2-2} into \eqref{est-bk-1}, to conclude that \eqref{est-mu-1} holds.
\end{proof}

With the help of Proposition \ref{8R2}, we can give the following proposition concerning the estimates of $R_0$ and $R_1$.

\begin{prop}\label{8R0R1-1}
Let $L^{-1}_{z,K}(g)(0)=0$. There exist constants $\alpha>0$ sufficiently small and $\eta(\beta)$, such that if $\alpha\ll 1-\eta\ll \beta$ and $|\mu|\leq \alpha^\frac{1}{2}$, then
\begin{align}
	|\langle R_0+R_1, g(w^K)^2\rangle|
	&\leq 
	C_\beta (1-\eta) \|g\|_{\mathcal{H}^{-1}}\|g\|_{\mathcal{H}^2_{\eta}}
	\label{est-R0R1-H-1}\\
	&~~~+C_\beta(1-\eta)^{-\frac12}\left(
	\alpha^{-2}\|g\|^4_{\mathcal{H}^2}
	+\alpha^{-1}\|g\|^3_{\mathcal{H}^2}
	+\alpha\|g\|^2_{\mathcal{H}^2}
	+\alpha^2\|g\|_{\mathcal{H}^2}\right),
	\notag\\
	|\langle R_0+R_1, g\rangle_{\mathcal{H}_{\eta}^2}|
	&\leq 
	C_\beta (1-\eta)\|g\|_{\mathcal{H}^2_{\eta}}^2
	\label{est-R0R1-H2eta}\\
	&~~~+C_\beta(1-\eta)^{-\frac52}\left(
	\alpha^{-2}\|g\|^4_{\mathcal{H}^2}
	+\alpha^{-1}\|g\|^3_{\mathcal{H}^2}
	+\alpha\|g\|^2_{\mathcal{H}^2}
	+\alpha^2\|g\|_{\mathcal{H}^2}\right),
	\notag\\
	|\langle R_0+R_1, g\rangle_{\mathcal{E}_{\eta}^2}|
	&\leq C_\beta (1-\eta) \|g\|_{\mathcal{H}^2_{\eta}}\|g\|_{\mathcal{E}^2_{\eta}}
	+C_\beta (1-\eta)^{-\frac12}\left(
	\alpha^{-1}\|g\|^2_{\mathcal{H}^2}
	+\alpha^2\right)\|g\|_{\mathcal{E}^2}
	\label{est-R0R1-E2eta}\\
	&~~~+C_\beta\left(\alpha^{-1}\|g\|_{\mathcal{H}^2}+(1-\eta)^{-\frac12}(\alpha^{-2}\|g\|^2_{\mathcal{H}^2}+\alpha)\right)\|g\|^2_{\mathcal{E}^2}.
	\notag
\end{align}
\end{prop}
\begin{proof}
We apply Proposition \ref{8R2} and \eqref{4eq14}, to deduce that 
\begin{align}\label{8in16-1}
	|\langle R_0, g(w^K)^2\rangle|
	&\lesssim_{\beta} \big(
	\alpha^{-1} (1-\eta) \|g\|_{\mathcal{H}^2_\eta}
	+(1-\eta)^{-\frac12} \big( \alpha^{-2}\|g\|_{\mathcal{H}^2} +\alpha \big)
	\big) 
    \big\|F^{\ast}_{\gamma} \frac{z^{\frac{1}{\gamma}}}{(1+z^{\frac{1}{\gamma}})}w^{K}\big\|_{\mathcal{L}^2} \|g w^{K}\|_{\mathcal{L}^2}
	\\  \notag
	&\lesssim_\beta(1-\eta)\|g\|_{\mathcal{H}^2_\eta}\|g\|_{\mathcal{H}^{-1}}
	+(1-\eta)^{-\frac{1}{2}}(\alpha^{-1}\|g\|^3_{\mathcal{H}^2}+\alpha^2\|g\|_{\mathcal{H}^2}).
\end{align}
Integrating by parts and employing Proposition \ref{8R2}, we get that
\begin{align}\label{8in16-2}
	|\langle R_1, g(w^K)^2\rangle|
	&\lesssim_{\beta} \left(
	\alpha^{-1}\|g\|_{\mathcal{H}^2}
	+(1-\eta)^{-\frac12} \left( \alpha^{-2}\|g\|_{\mathcal{H}^2} +\alpha \right)
	\right) \|g w^{K}\|_{\mathcal{L}^2}^2
	\\  \notag
	&\lesssim_\beta \alpha^{-1}\|g\|_{\mathcal{H}^2}^3
	+(1-\eta)^{-\frac{1}{2}}(\alpha^{-2}\|g\|^4_{\mathcal{H}^2}+\alpha \|g\|_{\mathcal{H}^2}^2).
\end{align}
Collecting \eqref{8in16-1} and \eqref{8in16-2} together, we find that \eqref{est-R0R1-H-1} holds true.

Again thanks to Proposition \ref{8R2} and \eqref{4eq14}, we obtain that
\begin{align*}
	|\langle R_0, g\rangle_{\mathcal{H}_{\eta}^2}| 
	&\lesssim_{\beta} \left(
	\alpha^{-1} (1-\eta) \|g\|_{\mathcal{H}^2_\eta}
	+(1-\eta)^{-\frac12} \left( \alpha^{-2}\|g\|_{\mathcal{H}^2} +\alpha \right)
	\right) \|F^{\ast}_{\gamma}z^{\frac{1}{\gamma}}(1+z^{\frac{1}{\gamma}})^{-1}\|_{\mathcal{H}^2_\eta} \|g\|_{\mathcal{H}^2_\eta}
	\\  \notag
	&\lesssim_{\beta} (1-\eta)\|g\|^2_{\mathcal{H}^2_\eta}
	+(1-\eta)^{-\frac{5}{2}}(\alpha^{-1}\|g\|^3_{\mathcal{H}^2}+\alpha^2\|g\|_{\mathcal{H}^2}).
\end{align*}
Along similar way as \eqref{8in16-2}, we see that
\begin{align*}
	|\langle R_1, g\rangle_{\mathcal{H}_{\eta}^2}| 
	\lesssim_{\beta} (1-\eta)^{-2}\alpha^{-1}\|g\|^3_{\mathcal{H}^2}
	+(1-\eta)^{-\frac{5}{2}}(\alpha^{-2}\|g\|^4_{\mathcal{H}^2}+\alpha\|g\|_{\mathcal{H}^2}^2).
\end{align*}
The combination of two estimates above gives that \eqref{est-R0R1-H2eta} holds.

We then use the similar derivation, to get that
\begin{align}
	|\langle R_0, g\rangle_{\mathcal{E}_{\eta}^2}|
	\label{8in18-1}
	&\lesssim_{\beta} \left(
	\alpha^{-1} (1-\eta) \|g\|_{\mathcal{H}^2_\eta}
	+(1-\eta)^{-\frac12} \left( \alpha^{-2}\|g\|_{\mathcal{H}^2} +\alpha \right)
	\right) \|F^{\ast}_{\gamma}z^{\frac{1}{\gamma}}(1+z^{\frac{1}{\gamma}})^{-1}\|_{\mathcal{E}^2_\eta} \|g\|_{\mathcal{E}^2_\eta}
	\\  \notag
	&\lesssim_\beta ((1-\eta)\|g\|_{\mathcal{H}^2_\eta}\|g\|_{\mathcal{E}^2_\eta}+(1-\eta)^{-\frac{1}{2}}(\alpha^{-1}\|g\|^2_{\mathcal{H}^2}+\alpha^2)\|g\|_{\mathcal{E}^2}),
	\\
	|\langle R_1, g\rangle_{\mathcal{E}_{\eta}^2}|
	&
	\lesssim_\beta (\alpha^{-1}\|g\|_{\mathcal{H}^2}+(1-\eta)^{-\frac{1}{2}}(\alpha^{-2}\|g\|^2_{\mathcal{H}^2}+\alpha))\|g\|^2_{\mathcal{E}^2}.
	\label{8in18-2}
\end{align}
According to \eqref{8in18-1} and \eqref{8in18-2}, we obtain \eqref{est-R0R1-E2eta} immediately.
\end{proof}

Employing Propositions \ref{5co2}, \ref{7Tg}, \ref{7Tf}, \ref{8R1} and \ref{8R0R1-1}, we get the following proposition directly.
For simplicity, we omit its proof here.
\begin{prop}\label{8R3}
Let $L^{-1}_{z,K}(g)(0)=0$. There exist constants $\alpha>0$ sufficiently small and $\eta(\beta)$, such that if $\alpha\ll 1-\eta \ll \beta$ and $|\mu|\leq \alpha^\frac{1}{2}$, then 
\begin{align*}
	&\|g\|^2_{\mathcal{H}^{-1}}+\|g\|^2_{\mathcal{H}_{\eta}^2}
	\leq C_\beta(1-\eta)^{-3}(\alpha^{-2}\|g\|^4_{\mathcal{H}^2}+\alpha^{-1}\|g\|^3_{\mathcal{H}^2}+\alpha^\frac{1}{2}\|g\|^2_{\mathcal{H}^2}+\alpha^2\|g\|_{\mathcal{H}^2})
	+C_\beta\alpha^{-1}\|g\|^2_{\mathcal{H}^{2}}\|g\|_{\mathcal{E}^{2}},
\end{align*}
and
\begin{align*}
\|g\|^2_{\mathcal{H}^{-1}}+\|g\|^2_{\mathcal{H}_{\eta}^2}+\|g\|^2_{\mathcal{E}_{\eta}^2}
	&\leq C_\beta(1-\eta)^{-3}(\alpha^{-2}\|g\|^4_{\mathcal{H}^2}+\alpha^{-1}\|g\|^3_{\mathcal{H}^2}+\alpha^\frac{1}{2}\|g\|^2_{\mathcal{H}^2}+\alpha^2\|g\|_{\mathcal{H}^2})
	\\
	&~~~+C_\beta\alpha^{\frac{1}{2}}(1-\eta)^{-\frac{5}{2}}\|g\|^2_{\mathcal{E}^2}
	+C_\beta(1-\eta)^{-\frac{1}{2}}\alpha^{-1}(\|g\|_{\mathcal{E}^{2}}+\|g\|_{\mathcal{H}^{2}})\|g\|_{\mathcal{E}^{2}}\|g\|_{\mathcal{H}^{2}}
	\\
	&~~~+C_\beta\alpha^{\frac{3}{2}}\|g\|_{\mathcal{E}^2}
	+C_\beta(1-\eta)^{-\frac{1}{2}}\alpha^{-2}\|g\|^2_{\mathcal{H}^2}\|g\|^2_{\mathcal{E}^2}.
\end{align*}
\end{prop}

\subsection{Constructing the solutions}
With Proposition \ref{8R2} and Proposition \ref{8R3} at hand, we begin to prove Theorem  \ref{Theo1} in this subsection.
\begin{proof}[\textbf{Proof of Theorem \ref{Theo1}}]
We first  introduce a ``fake" time variable $\tau$, and then consider the evolution equation for the function $g(\tau,z,\theta)$ as follows:
\begin{align}\label{82in1}
	\partial_\tau g+\mathcal{L}_{\Gamma}(g)=-T+R_0+R_1+R_2,~~~g(0,z,\theta)=0,
\end{align}
where $T$ and $R_i$ with $i=0,1,2$ as in \eqref{8eq1}. 
We observe that $\mu$ is now a function of $\tau$.
Referring back to the definition of $\gamma$ in \eqref{4eq13}, it follows that $\gamma$ consequently depends on $\tau$ as well.
This dependency implies that $F^\ast_{\gamma}$ is ultimately parameterized by $\tau$.

By using the compactness method of evolution equation, we obtain a local solution $g(\tau,z,\theta)$ for equation \eqref{82in1}. 
Recalling \eqref{82in1} and the definition of $\mu$ in $\eqref{8eq1}_3$, we deduce that 
\begin{align*}
	\partial_\tau L^{-1}_{z,K}(g)(\tau,0)+ L^{-1}_{z,K}(\mathcal{L}_{\Gamma}(g))(\tau,0)= L^{-1}_{z,K}(-T+R_0+R_1+R_2)(0,\tau)= L^{-1}_{z,K}(R_1)(\tau,0).
\end{align*}
This implies that 
\begin{align*}
	\partial_\tau L^{-1}_{z,K}(g)(\tau,0)=-(1+\mu)L^{-1}_{z,K}(g)(\tau,0).
\end{align*}
Since $L^{-1}_{z,K}(g)(0,0)=0$, we infer 
$L^{-1}_{z,K}(g)(\tau,0)=0$ for all $\tau\geq0$. 
We recall the equation \eqref{82in1} and apply Proposition \ref{8R3}, to discover that
\begin{align}\label{82in2}
	&~~~\frac{d}{d\tau}(\|g w^K\|_{L^{2}}+\|g\|_{\mathcal{H}_{\eta}^2})+\|g\|_{\mathcal{H}^{-1}}+\|g\|_{\mathcal{H}_{\eta}^2}
	\\ \notag
	&\lesssim_{\beta} (1-\eta)^{-5}(\alpha^{-2}\|g\|^3_{\mathcal{H}^2}+\alpha^{-1}\|g\|^2_{\mathcal{H}^2}+\alpha^\frac{1}{2}\|g\|_{\mathcal{H}^2}+\alpha^2) +\alpha^{-1}\|g\|_{\mathcal{H}^{2}}\|g\|_{\mathcal{E}^{2}},
\end{align}
and
\begin{align}\label{82in3}
	&~~~\frac{d}{d\tau}(\|g w^K\|_{L^{2}}+\|g\|_{\mathcal{H}_{\eta}^2}+\|g\|_{\mathcal{E}_{\eta}^2})+
	\|g\|_{\mathcal{H}^{-1}}+\|g\|_{\mathcal{H}_{\eta}^2}+\|g\|_{\mathcal{E}_{\eta}^2}
	\\ \notag
	&\lesssim_{\beta} (1-\eta)^{-5}(\alpha^{-2}\|g\|^3_{\mathcal{H}^2}+\alpha^{-1}\|g\|^2_{\mathcal{H}^2}+\alpha^\frac{1}{2}\|g\|_{\mathcal{H}^2}+\alpha^2)+\alpha^{\frac{1}{2}}(1-\eta)^{-\frac{9}{2}}\|g\|_{\mathcal{E}^2}\\ \notag
	&+(1-\eta)^{-\frac{5}{2}}\alpha^{-1}(\|g\|_{\mathcal{E}^{2}}+\|g\|_{\mathcal{H}^{2}})\|g\|_{\mathcal{H}^{2}}+\alpha^{\frac{3}{2}}(1-\eta)^{-2}+(1-\eta)^{-\frac{5}{2}}\alpha^{-2}\|g\|^2_{\mathcal{H}^2}\|g\|_{\mathcal{E}^2}.
\end{align}
It then follows from Proposition \ref{8R2} that
\begin{align}\label{82in4}
	|\mu|&\lesssim_{\beta}\alpha^{-1}(1-\eta)\|g\|_{\mathcal{H}^2_\eta}+(1-\eta)^{-\frac{1}{2}}(\alpha^{-2}\|g\|^2_{\mathcal{H}^2}+\alpha).
\end{align}
Since $g|_{\tau=0}=0$, we deduce from \eqref{82in2}-\eqref{82in4} and the continuity argument that for all $\tau\geq 0$,
\begin{align}\label{82in5}
	\|g\|_{\mathcal{H}^{-1}}+\|g\|_{\mathcal{H}_{\eta}^2}\lesssim_{\beta,1-\eta} \alpha^2,~~~\|g\|_{\mathcal{E}_{\eta}^2}\lesssim_{\beta,1-\eta} \alpha^\frac{3}{2},~~~|\mu|\lesssim_{\beta,1-\eta} \alpha.
\end{align}
Thus, we can obtain that there exists a global solution $g(\tau,z,\theta)$ for equation \eqref{82in1}, which satisfies \eqref{82in5}. 

We apply the operator $\partial_{\tau}$ to equation \eqref{82in1} and consider the following evolution equation for the function $\partial_\tau g(\tau,z,\theta)$:
\begin{align}\label{82in6}
	\partial^2_\tau g+\mathcal{L}_{\Gamma}(\partial_\tau g)=-\partial_\tau T+\partial_\tau R_0+\partial_\tau R_1+\partial_\tau R_2+\bar{R},~~~\partial_\tau g(0,z,\theta)=\alpha^2,
\end{align}
where $\bar{R}\tri \mathcal{L}_{\Gamma}(\partial_\tau g)-\partial_\tau\mathcal{L}_{\Gamma}(g)$. 
Note that $(g,F^\ast_\gamma,\gamma,\mu)$ are all dependent of $\tau$. 
Firstly, it follows from some direct calculations that
\begin{align}
	|\partial_\tau F^\ast_\gamma| &\lesssim_\beta F^\ast_\gamma(1+|\ln z|+|\ln K(\theta)|)|\partial_\tau \gamma|,\label{82in7}
	\\
	|\partial_\tau \gamma| &\lesssim_\beta |\partial_\tau \mu|.\label{82in8}
\end{align}
We recall the definition of $\mu$ in $\eqref{8eq1}_3$, and employ the absorption method and Corollary \ref{7em2}, to obtain that
\begin{align}\label{82in9}
	|\partial_\tau \mu|&\lesssim_{\beta,1-\eta}\alpha^{-1}\|\partial_\tau g\|_{\mathcal{H}^1}.
\end{align}
Then, we define the  $\mathcal{H}_{\eta}^1([0,\infty)\times[0,\frac{\pi}{2}])$ norm
and
$\mathcal{E}^1_{\eta}([0,\infty)\times[0,\frac{\pi}{2}])$ norm as follows:
\begin{align*}
	\|f\|^2_{\mathcal{H}_{\eta}^1}&\tri (1-\eta)^2(\| fw^{\eta}\|^2_{L^2}+\|D_z fw^{\eta}\|^2_{L^2})+(1-\eta)^{-2}\|D_{\theta} fw^{\lambda}\|^2_{L^2},
	\\
	\|f\|^2_{\mathcal{E}^1_{\eta}}&\tri \alpha(1-\eta)^{2}(\| fw^{\lambda}\|^2_{L^2}+\|D_z fw^{\lambda}\|^2_{L^2}).
\end{align*}
Similar argument as the proof of \eqref{82in3}, we infer from \eqref{82in5}-\eqref{82in9} and $\alpha\ll 1-\eta \ll \beta$ that
\begin{align}\label{82in10}
	\frac{d}{d\tau}(\|\partial_\tau g  w^K\|_{L^{2}}+\|\partial_\tau g\|_{\mathcal{H}_{\eta}^1}+\|\partial_\tau g\|_{\mathcal{E}_{\eta}^1})
	+\|\partial_\tau g\|_{\mathcal{H}^{-1}}+\|\partial_\tau g\|_{\mathcal{H}_{\eta}^1}+\|\partial_\tau g\|_{\mathcal{E}_{\eta}^1}
	\leq 0.
\end{align}
Owing to $\partial_\tau g|_{\tau=0}=\alpha^2$, we deduce from \eqref{82in7}-\eqref{82in10} and the continuity argument that  for all $\tau\geq 0$,
\begin{align*}
	\|\partial_\tau g\|_{\mathcal{H}^{-1}}+\|\partial_\tau g\|_{\mathcal{H}_{\eta}^1}+\|\partial_\tau g\|_{\mathcal{E}_{\eta}^1}
	&\lesssim_{\beta,1-\eta} \alpha^2e^{-\frac{\tau}{100}},
	\\
	|\partial_\tau F^\ast_\gamma|+|\partial_\tau \gamma|+|\partial_\tau \mu|
	&\lesssim_{\beta,1-\eta} \alpha e^{-\frac{\tau}{100}}.
\end{align*}
Taking $\tau\rightarrow\infty$ in \eqref{82in1} and applying the above two inequalities, we obtain that there exists a solution $g(z,\theta)$ for equation \eqref{8eq1}. 
This completes the proof of Theorem \ref{Theo1}.
\end{proof}

\subsection{Blow-up phenomenon}
Based on Theorem \ref{Theo1}, our goal in this subsection is to prove blow-up phenomenon of the Euler equations.
\begin{prop}\label{8prop1}
	For any $\beta\in(0,1]$, there exists a positive constant $\alpha(\beta)$, the initial vorticity $\omega_0 \in C^{0,\frac{\alpha}{20\beta}}(D_0)$ with $|\omega_0|\leq c|\bm{x}|^{\frac{\alpha}{2\beta}}(1+|\bm{x}|)^{-\left({\frac{\alpha}{2\beta}+\frac{\alpha}{2}}\right)}$ for some constant $c>0$, and the positive time $T^\ast$, such that the axi-symmetric 3D Euler equations \eqref{1eq2} admit a solution 
	$\omega\in C^{0,\frac{\alpha}{2\beta}}_{t}([0,T^\ast); C^{0,\frac{\alpha}{20\beta}}_{x}( D_0))$.
	The solution develops a finite-time singularity at $t=T^\ast>0$ with scaling index $\frac{\beta}{\alpha}$.
	Moreover, there exists a parameter $\gamma\in\mathbb{R}^+$ that satisfies $|\frac{\beta}{\gamma}-1|\ll 1$ and $T^\ast= \frac{1}{2}+\frac{\beta}{2\gamma}$, such that the following blow-up results hold:
	\begin{align}\label{8est-theo-1}
		\lim\limits_{t\rightarrow T^\ast} \int_0^t \|\omega(s)\|_{L^{\infty}(D_0)} ds = +\infty, \quad 
		\lim\limits_{t\rightarrow T^\ast} \int_0^t \Big\|\frac{u_r}{r} (s)\Big\|_{L^{\infty}(D_0)} ds = +\infty.
	\end{align}
	Furthermore, for any $1\leq p<\frac{2}{2-\beta}$, the velocity components exhibit the following integrability property:
	\begin{align}\label{8est-theo-2}
		\int_0^{T^\ast} \|(u_r,u_3) (s)\|^p_{L^{\infty}_{loc}(D_0)}ds  < +\infty.
	\end{align}
\end{prop}

\begin{proof}
We divide this proof into three steps.

\medskip

\noindent \textbf{Step 1: Proof of estimates 
	\eqref{8est-theo-1} and \eqref{8est-theo-2}. }

\medskip

\noindent 
By virtue of \eqref{form_F1} and \eqref{82in5}, we find that
\begin{align*}
	\|\omega(t)\|_{L^{\infty}} &\geq \frac{1}{t_{\gamma}} \left| F^{\ast}_{\gamma}\left(1,\frac{\pi}{4}\right)\right| - \frac{1}{t_{\gamma}}\big\| g\big\|_{L^{\infty}}  \geq\frac{1}{t_{\gamma}}\frac{\alpha}{3\gamma}-\frac{1}{t_{\gamma}}\big\| g\big\|_{\mathcal{H}^{2}} \geq\frac{1}{t_{\gamma}}\frac{\alpha}{3\gamma}-\frac{1}{t_{\gamma}}~C_{\beta,1-\eta}\alpha^2 
	\geq\frac{1}{t_{\gamma}}\frac{\alpha}{6\gamma},
\end{align*}
where we have used the fact that $\alpha\ll 1-\eta\ll \beta$.
This implies that 
\begin{align*}
	\int_0^{t} \|\omega(s)\|_{L^{\infty}} ds 
	\geq \frac{\alpha}{6\gamma}\int_0^{t} \frac{1}{s_{\gamma}} ds
	=\frac{\alpha}{6\gamma}\left(\frac{1}{2}+\frac{\beta}{2\gamma}\right)\left|\ln\left(1-\frac{2\gamma}{\gamma+\beta}{t}\right)\right|.
\end{align*}
Taking $t\rightarrow T^\ast\tri \frac{1}{2}+\frac{\beta}{2\gamma}$, it is easy to check that
\begin{align*}
	\lim\limits_{t\rightarrow T^\ast} \int_0^{t} \|\omega(s)\|_{L^{\infty}} ds =+\infty.
\end{align*}
After some direct calculations, one obtains that
\begin{align}\label{est-ur-1}
	\frac{u_r}{r}(t)
	&=\mathcal{R}(\Psi)=\mathcal{R}\left(\frac{1}{t_{\gamma}}\Phi(z,\theta)\right)
	\\
	&=\frac{1}{t_{\gamma}} \left(\mathcal{R}(\Phi-G^\ast_F\sin (2\theta))
	+ \frac{3}{2\alpha} L^{-1}_{z,K}(F^\ast_\gamma)
	+ \frac{3}{2\alpha} L^{-1}_{z,K}(g)
	-\frac{3}{2}(\sin\theta)^2\langle F,K\rangle_{\theta}
	\right).
	\notag
\end{align}
The main term in the right-hand side of the equality above is
\begin{align*}
	\frac{3}{2\alpha}  \Big\|L^{-1}_{z,K}(F^\ast_\gamma)\Big\|_{L^\infty}= \|L^{-1}_z(\Gamma^\ast_\gamma)\|_{L^\infty}=\|\frac{2}{1+z^{\frac1\gamma}}\|_{L^\infty}=2.
\end{align*}
Substituting the equality above into \eqref{est-ur-1} and applying the fact that $\alpha\ll 1-\eta\ll \beta$, we get that
\begin{align*}
	\Big\|\frac{u_r}{r}(t)\Big\|_{L^\infty} 
	&\geq  \frac{2}{t_{\gamma}}
	-\frac{1}{t_{\gamma}} \left(\|\mathcal{R}(\Phi-G^\ast_F\sin (2\theta))\|_{L^\infty}
	+ \frac{3}{2\alpha} \|L^{-1}_{z,K}(g)\|_{L^\infty}
	+\frac{3}{2}\|(\sin\theta)^2\langle F,K\rangle_{\theta}\|_{L^\infty}
	\right)
	\\
	&\geq \frac{2}{t_{\gamma}}-\frac{1}{t_{\gamma}}C_{\beta,1-\eta} \alpha^{\frac{1}{2}} 
	\geq  \frac{1}{t_{\gamma}},
\end{align*}
which gives that
\begin{align*}
	\int_0^{t}\Big\|\frac{u_r}{r}(s)\Big\|_{L^\infty} ds 
	\geq \int_0^{t} \frac{1}{s_{\gamma}} ds
	=\left(\frac{1}{2}+\frac{\beta}{2\gamma}\right)\left|\ln\left(1-\frac{2\gamma}{\gamma+\beta}{t}\right)\right|.
\end{align*}
Taking $t\rightarrow T^\ast$, it follows that
\begin{align*}
	\lim\limits_{t\rightarrow T^\ast} \int_0^{t} \Big\|\frac{u_r}{r}(s)\Big\|_{L^\infty}  ds =+\infty.
\end{align*}

It remains to verify \eqref{8est-theo-2}. By using \eqref{est-ur-1}, we deduce that
\begin{align}
	\||\bm{x}|^{\frac{\alpha}{2}-1} u_r\|_{L^\infty} 
	\leq& \|R^{\frac{1}{2}} r^{-1} u_r\|_{L^\infty} 
	\leq t_{\gamma}^{-1+\frac{\beta}{2}} \|z^{\frac12}\mathcal{R}(\Phi)\|_{L^\infty}
	\label{est-ur-Lloc}\\
	\lesssim & t_{\gamma}^{-1+\frac{\beta}{2}}
	\Big(\|z^{\frac12}\mathcal{R}(\Phi_{F^\ast_\gamma}-G^\ast_{F^\ast_\gamma}\sin (2\theta))\|_{L^\infty}
	+ \|z^{\frac12}\mathcal{R}(\Phi_g-G^\ast_g\sin (2\theta))\|_{L^\infty}
	\nonumber\\
	&
	+\|z^{\frac12}\alpha^{-1} L^{-1}_{z,K}(F^\ast_\gamma)\|_{L^\infty}
	+\|z^{\frac12}(\sin\theta)^2\langle F^\ast_\gamma,K\rangle_{\theta}\|_{L^\infty}
	\nonumber\\
	&
	+\|z^{\frac12}\big( \alpha^{-1} L^{-1}_{z,K}(g)
	-(\sin\theta)^2\langle g,K\rangle_{\theta}\big)\|_{L^\infty}
	\Big).
	\nonumber
\end{align}
We now analyze each term on the right-hand side of \eqref{est-ur-Lloc} sequentially.
It is easy to observe that
\begin{align*}
	&\sup_{\theta\in(0,\frac{\pi}{2})}\|z^{\frac12}\mathcal{R}(\Phi_g-G^\ast_g\sin (2\theta))\|_{L^\infty[1,\infty)} \lesssim \alpha^{-\frac12} \|g\|_{\mathcal{H}^2} \lesssim_{1-\eta} \alpha^{\frac32};
	\\
	&\sup_{\theta\in(0,\frac{\pi}{2})}\|z^{\frac12}\alpha^{-1} L^{-1}_{z,K}(F^\ast_\gamma)\|_{L^\infty} =2 \|z^{\frac12}(1+z^{\frac1\gamma})^{-1}\|_{L^\infty} \leq 2;
	\\
	&\sup_{\theta\in(0,\frac{\pi}{2})}\|z^{\frac12}(\sin\theta)^2\langle F^\ast_\gamma,K\rangle_{\theta}\|_{L^\infty} \lesssim \alpha \|z^{\frac12+\frac1\gamma}(1+z^{\frac1\gamma})^{-2}\|_{L^\infty} \lesssim \alpha;
	\\
	&\sup_{\theta\in(0,\frac{\pi}{2})}\|z^{\frac12}\big( \alpha^{-1} L^{-1}_{z,K}(g)
	-(\sin\theta)^2\langle g,K\rangle_{\theta}\big)\|_{L^\infty[1,\infty)} \lesssim \alpha^{-1} \|g\|_{\mathcal{H}^2} \lesssim_{1-\eta} \alpha.
\end{align*}
The first term on the right-hand side of \eqref{est-ur-Lloc} still requires analysis.
To control this term, we need to choose another weight as $z$ goes to infinity.
Let us define
\begin{align*}
	w_z^{\ast\ast}(z)\tri \frac{1+z^{\frac1\beta+\frac14}}{z^{\frac1\beta+\frac14}},
\end{align*}
which tends to $1$ as $z$ goes to infinity.
By some direct calculations, one can check that
\begin{align*}
	\|\Gamma_{\gamma}^\ast w_z^{\ast\ast}\|_{L^2}^2
	=\int_0^\infty \frac{z^{\frac2\gamma}}{(1+z^{\frac1\gamma})^4} \frac{(1+z^{\frac1\beta+\frac14})^2}{z^{\frac2\beta+\frac12}}dz
	\lesssim \int_0^\infty z^{-\frac34}(1+z)^{-1}dz\lesssim 1.
\end{align*}
Then, we define the space $H^{k,\ast\ast}$ by replacing the radial weight $w_z^{\ast}$ in the definition of $H^{k,\ast}$ with $w_z^{\ast\ast}$.
Along with a similar way, this yields that
\begin{align*}
	\|z^{\frac12}\mathcal{R}(\Phi_{F^\ast_\gamma}-G^\ast_{F^\ast_\gamma}\sin (2\theta))\|_{L^\infty}
	\lesssim 
	\alpha^{-\frac12} \|F_{\gamma}^\ast \|_{H^{k,\ast\ast}}
	\lesssim_{1-\eta} \alpha^{\frac12}.
\end{align*}
Collecting all estimates above, it follows from \eqref{est-ur-Lloc} that for any $p\in[1,\frac{2}{2-\beta})$,
\begin{align}\label{est-ur}
	\int_0^{T^\ast} \||\bm{x}|^{\frac{\alpha}{2}-1} u_r(s)\|_{L^\infty}^p  ds
	\lesssim \int_0^{T^\ast} s_{\gamma}^{(-1+\frac\beta2)p}ds
	\lesssim1.
\end{align}
This implies that $u_r\in L^p([0,T^\ast];L^\infty_{loc})$.
As for $u_3$, we have
\begin{align*}
	\frac{u_3}{|\bm{x}|} 
	=-\frac{1}{t_\gamma} \big(
	(\cos\theta)^{-1} \Phi +2\cos\theta \Phi+\alpha\cos\theta D_z\Phi-\sin\theta \partial_{\theta}\Phi
	\big).
\end{align*}
By adapting the derivation method employed in \eqref{est-ur}, we derive that $u_3\in L^p([0,T^\ast];L^\infty_{loc})$.

\noindent \textbf{Step 2: The H\"older regularity. }

\medskip

\noindent 
First, we intend to show the H\"older regularity of the initial data.
According to \eqref{82in5}, we have
\begin{align*}
	\|g w_z\|_{L^\infty_{\theta}L^2_z}+\| D_z g w_z\|_{L^\infty_{\theta}L^2_z} \lesssim 
	\alpha^{-\frac{1}{2}}\|g\|_{\mathcal{H}^2} 
	\lesssim_{\beta,1-\eta} \alpha^{\frac{3}{2}}.
\end{align*}
This, together with the definition of $w_z$, yields directly that
\begin{align}\label{est-g-1}
	\sup_{\theta\in (0,\frac{\pi}{2})} \left(
	\|g\|_{L^2_z[1,\infty)}+\|D_z g\|_{L^2_z[1,\infty)}
	+\| g z^{-\frac{2}{\beta}}\|_{L^2_z[0,1]}+\|D_z gz^{-\frac{2}{\beta}}\|_{L^2_z[0,1]}
	\right)
	\lesssim_{\beta,1-\eta} \alpha^{\frac{3}{2}}.
\end{align}
We claim that
\begin{align}\label{est-g-2}
	g=O(z^{-\frac12}),\quad \text{as}~~z\rightarrow\infty;
	\quad
	g=O(z^{\frac{2}{\beta}-\frac12}),\quad
	\text{as}~~z\rightarrow 0.
\end{align}
Indeed, it follows from \eqref{est-g-1} that for any $z\in[1,\infty)$ and $\theta\in  (0,\frac{\pi}{2})$,
\begin{align*}
	|g(z,\theta)|^2
	=\left| \int_z^\infty g \partial_{z'} g dz'  \right|
	\lesssim z^{-1} \sup_{\theta\in (0,\frac{\pi}{2})} 
	\left(\|g\|_{L^2_z[1,\infty)}\|D_z g\|_{L^2_z[1,\infty)}\right)
	\lesssim_{\beta,1-\eta} \alpha^3 z^{-1}.
\end{align*}
Another estimate in \eqref{est-g-2} can be deduced in a similar way, and we thus omit its proof here.
Recalling \eqref{form_F1}, one gets, by some direct calculations, that
\begin{align}\label{est-F*-1}
	F^\ast_{\gamma}= O(z^{-\frac{1}{\gamma}}), \quad \text{as}~~z\rightarrow\infty;
	\quad
	F^\ast_{\gamma}=O(z^{\frac{1}{\gamma}}),\quad 
	\text{as}~~z\rightarrow 0.
\end{align}
Since $F=F^\ast_{\gamma}+g$, we can find that
\begin{align}\label{est-F-1}
	F= O(z^{-\frac12}),\quad \text{as}~~z\rightarrow\infty;
	\quad
	F=O(z^{\frac{1}{\gamma}}),\quad 
	\text{as}~~z\rightarrow 0.
\end{align}
It is easy to check that
\begin{align*}
	\omega_0(r,x_3)= \Omega(0,R,\theta) =F(R,\theta).
\end{align*}
This, together with $z|_{t=0}=R=|x|^{\alpha}$ and \eqref{est-F-1}, gives that
\begin{align*}
	|\omega_0(r,x_3)| 
	\lesssim \frac{|x|^{\frac{\alpha}{\gamma}}}{(1+|x|)^{\frac{\alpha}{\gamma}+\frac{\alpha}{2}}}
	\lesssim  
	\frac{|x|^{\frac{\alpha}{2\beta}}}{(1+|x|)^{\frac{\alpha}{2\beta}+\frac{\alpha}{2}}},
\end{align*}
where we have used the fact that $\gamma=\frac{1+\mu}{1-\mu}\beta$ and $\mu$ is small enough.

Next, we want to prove the H\"older regularity for the solution.
From \eqref{form_F1}, it follows that for $t=0$,
\begin{align}\label{est-F*-2}
	F^\ast_{\gamma}= O((\sin\theta)^{\frac{\alpha}{3\gamma}}(\cos\theta)^{\frac{2\alpha}{3\gamma}}), \quad \text{as}~~\sin(2\theta)\rightarrow 0;
	\quad
	F^\ast_{\gamma}=O(z^{\frac{1}{\gamma}})=O(|x|^{\frac{\alpha}{\gamma}}),\quad 
	\text{as}~~|x|\rightarrow 0.
\end{align}
In view of \eqref{82in5}, we achieve that
\begin{align*}
	\sup_{z\in[0,\infty)} \left(\|(\sin(2\theta))^{-\frac{\lambda}{2}} g\|_{L^2_{\theta}}
	+\|(\sin(2\theta))^{-\frac{\lambda}{2}} D_{\theta}g\|_{L^2_{\theta}}\right)
	\lesssim_{\beta,1-\eta} \alpha^{\frac{3}{2}}.
\end{align*}
Along a similar way as \eqref{est-g-2}, we find that for $t=0$,
\begin{align}\label{est-g-3}
	g=O((\sin(2\theta))^{\frac{\lambda-1}{2}}) =O((\sin(2\theta))^{\frac{\alpha}{20\beta}}),\quad \text{as}~~ \sin(2\theta)\rightarrow 0;
	\quad
	g=O(|x|^{\frac{2\alpha}{\beta}-\frac\alpha2}),\quad
	\text{as}~~|x|\rightarrow 0.
\end{align}
Owing to $r=|x| \cos(\theta)$ and $x_3=|x| \sin(\theta)$, we get, by applying \eqref{est-F*-2} and \eqref{est-g-3}, that
\begin{align*}
	F^\ast_{\gamma}= O( r^{\frac{2\alpha}{3\gamma}} |x_3|^{\frac{\alpha}{3\gamma}} ),
	\quad
	g=O( (r |x_3|)^{\frac{\alpha}{20\beta}} |x|^{\frac{19\alpha}{10\gamma}-\frac{\alpha}{2}} ),
	\quad\text{as}~~ r\rightarrow0~~\text{or} ~~x_3\rightarrow 0.
\end{align*}
This yields directly that for $t=0$,
$F^\ast_{\gamma} \in  C_x^{\frac{\alpha}{3\gamma}}(D_0)$ and $g\in C_x^{\frac{\alpha}{20\beta}}(D_0)$. 
We thus deduce that $\omega_0\in C_x^{\frac{\alpha}{20\beta}}(D_0)$.

For any $t\in[0,T^\ast)$, we have
\begin{align*}
	|\omega(t,r,x_3)-\omega_0(r,x_3)|
	=&|\Omega(t,R,\theta)-\Omega(0,R,\theta)|
	\leq \left|\left(\frac{1}{t_{\gamma}}-1\right) F(z,\theta)\right| + |F(z,\theta)-F(R,\theta)|
	\\
	\leq& t\|F\|_{L^\infty}+|z-R|^{\frac{\alpha}{\gamma}}
	\leq t^{\frac{\alpha}{\gamma}}.
\end{align*}
This yields that $\omega\in C_{t}^{\frac{\alpha}{\gamma}}[0,T^\ast)$. Therefore, one gets that $\omega\in C_{t}^{\frac{\alpha}{2\beta}}[0,T^\ast)$.

\medskip

\noindent \textbf{Step 3: The scaling index.}

\medskip

\noindent 
Finally, we aim to show that the scaling index is $\frac{\beta}{\alpha}$. We can deduce from some computations that
\begin{align}\label{est-w-scaling}
	\omega(t,r,x_3)
	&=\Omega (t,R,\theta)
	=\frac{1}{t_\gamma} F\Big(\frac{R}{t_\gamma^\beta},\theta\Big)
	= \frac{1}{t_\gamma} F\Big(\Big(\frac{|x|}{t_\gamma^{\frac{\beta}{\alpha}}}\Big)^{\alpha},\arctan \Big(\frac{x_3}{r}\Big)\Big)
	\\ \notag
	&=\frac{1}{t_\gamma}
	F\Big(\Big(\Big(\frac{r}{t_\gamma^{\frac{\beta}{\alpha}}}\Big)^{2}+\Big(\frac{x_3}{t_\gamma^{\frac{\beta}{\alpha}}}\Big)^{2}\Big)^{\frac{\alpha}{2}},\arctan \Big(\frac{x_3}{t_\gamma^{\frac{\beta}{\alpha}}} \Big(\frac{r}{t_\gamma^{\frac{\beta}{\alpha}}}\Big)^{-1} \Big)\Big)
	\\ \notag
	&=\frac{1}{t_\gamma} H\Big(\frac{r}{t_\gamma^{\frac{\beta}{\alpha}}},\frac{x_3}{t_\gamma^{\frac{\beta}{\alpha}}}\Big),
\end{align}
where $H$ is defined by
\begin{align*}
	H(x,y)\tri F\big((x^{2}+y^{2})^{\frac{\alpha}{2}},\arctan (y x^{-1}) \big).
\end{align*}
The equality \eqref{est-w-scaling} implies that  the scaling index is $\frac{\beta}{\alpha}$.
We thereby finish the proof of Proposition \ref{8prop1}.
\end{proof}

\section{Non-implosion mechanism of the Euler equations}
\label{sec:non-implosion}

Building on Theorem \ref{Theo1} and Proposition \ref{8prop1}, we can obtain the blow up of the vorticity and the integrability of the velocity up to the blow-up time, which leads to the proofs of \eqref{est-theo-1} and \eqref{est-theo-2} with $p\in[1,\frac{2}{2-\beta})$ in Theorem \ref{Theo2}. 
In order to complete the proof of Theorem \ref{Theo2}, it remains to establish \eqref{est-theo-2} for the range $\frac{2}{2-\beta}\leq p<+\infty$, as well as to prove \eqref{est-theo-2'}.
In the analysis carried out in the preceding sections, it is observed that the time integrability of the velocity is intimately connected to the decay behavior at infinity of both the fundamental solution $F^\ast_\gamma$ and the perturbation $g$.
This observation motivates the introduction of weighted energy estimates with enhanced weights. The admissible range of the weight may be roughly delineated in terms of the parameters $\beta$ and $\gamma$, and this allows us to establish \eqref{est-theo-2'} under the prescribed conditions.

We first assume that
\begin{align}\label{key-est-1}
\mu<0 ~~(i.e., \gamma<\beta),\quad c_1\alpha\leq|\mu|\leq c_2 \alpha.
\end{align}
where $|\mu|< c_2 \alpha$ can be obtained from \eqref{82in5}.

We then find a new weight to replace $w_z$ as:
\begin{align*}
\tilde{w}_z=\frac{(1+z^{\frac{1}{\beta}})^{2+A}}{z^{\frac2\beta}},
\end{align*}
where the exponent $A=1-\frac\beta2+c_1\alpha$ satisfies $1-\frac{\beta}{2}<A<\frac\beta\gamma-\frac\beta2$.
The key lies in choosing $A$ such that $A - (1-\frac{\beta}{2}) = c_1\alpha$, which is also the lower bound for $|\mu|$. This particular choice of $A$ enables us to handle the remaining term related to $R_0$.

The corresponding weights $\tilde{w}^\eta$ and $\tilde{w}^\lambda$ are defined by
\begin{align*}
\tilde{w}^\eta= w_{\theta}^{\eta}\cdot \tilde{w}_z,
\quad
\tilde{w}^\lambda= w_{\theta}^{\lambda}\cdot \tilde{w}_z.
\end{align*}
We also find a new weight to replace $w_z^\ast$ as:
\begin{align*}
	\tilde{w}_z^\ast=\frac{(1+z^{\frac{1}{\beta}})^{1+\frac{\beta}{4}}}{z^{\frac1\beta+\frac14}}.
\end{align*}
The corresponding weights $\tilde{w}^{\ast,\eta}$ and $\tilde{w}^{\ast,\lambda}$ are defined by
\begin{align*}
	\tilde{w}^{\ast,\eta}= w_{\theta}^{\eta}\cdot \tilde{w}_z^\ast,
	\quad
	\tilde{w}^{\ast,\lambda}= w_{\theta}^{\lambda}\cdot \tilde{w}_z^\ast.
\end{align*}

Next, we give some definitions of Sobolev spaces with new wights.
For any $k\in\mathbb{N}^+$, the $\tilde{\mathcal{H}}^{k}([0,\infty)\times[0,\frac{\pi}{2}])$ and $\tilde{\mathcal{H}}^{k,\ast}([0,\infty)\times[0,\frac{\pi}{2}])$ norms are defined as follows:
\begin{align*}
\|f\|^2_{\tilde{\mathcal{H}}^{k}}
\tri&~\sum_{i=0}^{k}\|D^i_z f\tilde{w}^{\eta}\|^2_{L^2}+\sum_{0\leq i+j\leq k,j\geq 1}\|D^i_zD^j_{\theta} f\tilde{w}^{\lambda}\|^2_{L^2},
\\
\|f\|^2_{\tilde{\mathcal{H}}^{k,\ast}}
\tri&~\sum_{i=0}^{k}\|D^i_z f\tilde{w}^{\ast,\eta}\|^2_{L^2}+\sum_{0\leq i+j\leq k,j\geq 1}\|D^i_zD^j_{\theta} f\tilde{w}^{\ast,\lambda}\|^2_{L^2}.
\notag
\end{align*}
Furthermore, the notation $\langle \cdot, \cdot \rangle_{\tilde{\mathcal{H}}^k}$ and $\langle \cdot, \cdot \rangle_{\tilde{\mathcal{H}}^{k,\ast}}$
denote the corresponding inner products.

\subsection{Coercivity}
We aim to derive the coercivity of operator $\mathcal{L}_{\Gamma}$ within the $\tilde{\mathcal{H}}^{2}$ framework in this section.

\subsubsection{Coercivity of main operator}
Analogous to Lemma \ref{5co0}, we begin by studying the core operator $\mathcal{L}^{\beta}$ defined in \eqref{def-core-oper}.
\begin{lemm}\label{9co0}
For any $\beta\in(0,1]$, there holds
\begin{align}\label{2Co-est-0}
	-\langle \mathcal{L}^{\beta}(g),g\tilde{w}_z^2 \rangle_z
	\geq
	\frac34 c_1\alpha\|g\tilde{w}_z\|_{L^2_z}
	-C\alpha^{-2} \|g w_z\|_{L^2_z},
\end{align}
where $C>0$ is a constant independent of $\alpha$.
\end{lemm}
\begin{proof}
Integration by parts yields directly that
\begin{align}\label{2Co-est-1}
	-\langle\mathcal{L}^{\beta}(g),g\tilde{w}_z^2\rangle_z
	=-\langle g + \beta D_z g - L_z^{-1}(\Gamma^{\ast}_{\beta})g, g\tilde{w}_z^2\rangle_z
	=-\big\langle 1 - L_z^{-1}(\Gamma^{\ast}_{\beta})-\frac{\beta}{2}\frac{D_z^\ast \tilde{w}_z^2}{\tilde{w}_z^2}, g^2\tilde{w}_z^2\big\rangle_z.
\end{align}
After some direct calculations, we have
\begin{align}\label{2Co-cal-Dzww}
	\frac{D_z^\ast \tilde{w}_z^2}{\tilde{w}_z^2}
	=1+2\frac{D_z\tilde{w}_z}{\tilde{w}_z}
	=1+\frac{2A}{\beta}-\frac{2(2+A)}{\beta}\frac{1}{1+z^{\frac{1}{\beta}}}.
\end{align}
Substituting \eqref{3f7} and \eqref{2Co-cal-Dzww} into \eqref{2Co-est-1}, one gets that
\begin{align}\label{2Co-est-2}
	-\langle\mathcal{L}^{\beta}(g),g\tilde{w}_z^2\rangle_z
	=c_1\alpha \|g \tilde{w}_z\|_{L^2_z}^2
	-A\big\langle \frac{1}{1+z^{\frac1\beta}}, g^2 \tilde{w}_z^2\big\rangle_z.
\end{align}
Taking $B_0=\alpha^{-2\beta}$, we have
\begin{align}\label{sec9:est-cor-1}
	\Big|A\big\langle \frac{1}{1+z^{\frac1\beta}}, g^2 \tilde{w}_z^2\big\rangle_z\Big|
	=&~A\int_0^{B_0} \frac{1}{1+z^{\frac1\beta}} g^2 \tilde{w}_z^2 \,dz
	+A\int_{B_0}^{+\infty} \frac{1}{1+z^{\frac1\beta}} g^2 \tilde{w}_z^2 \,dz \notag
	\\
	\lesssim&~ \alpha^{-2} \|g w_z\|_{L^2_z(z\leq B_0)}^2+\alpha^2 \|g \tilde{w}_z\|_{L^2_z(z\geq B_0)}^2.
\end{align}
Combined with \eqref{2Co-est-2}, this immediately yields that
\begin{align*}
	-\langle \mathcal{L}^{\beta}(g),g\tilde{w}_z^2 \rangle_z
	\geq 
	c_1\alpha\|g\tilde{w}_z\|_{L^2_z}
	-C\alpha^{-2} \|g w_z\|_{L^2_z}-C\alpha^2 \|g \tilde{w}_z\|_{L^2_z}^2.
\end{align*}
By taking $\alpha$ suitably small, we deduce \eqref{2Co-est-0} directly.
\end{proof}

With the estimate for the core operator established in Lemma \ref{9co0}, we now proceed to investigate the coercivity of the main operator $\mathcal{L}^{\beta}_{\Gamma}$, defined in \eqref{def-main-oper}. In what follows in this section, we will split the domain into $[0, B_0]$ and $[B_0, +\infty)$ and omit some of the computational details. Unless a different decomposition is used, in which case full details will be given.
	
	Before proceeding, and in analogy with Lemma \ref{7fm2}, we present some fundamental lemmas concerning estimates for $F^\ast_{\gamma}$.
	
	\begin{lemm}\label{9lemma1}
		Assume that $f(z)$ is independent of $\theta$ and satisfies $f(0)=0$. 
		Let $\beta\in(0,1]$.
		There exist constants $\alpha>0$ sufficiently small and $\eta(\beta)$, such that if $\alpha\ll 1-\eta\ll \beta$ and $|\mu|\leq c_2\alpha$, then
		\begin{align*}
			\|f  F^\ast_{\beta} \|_{\tilde{\mathcal{H}}^{2}}\leq C\frac{\alpha}{\beta^4}(1-\eta)^{-\frac{1}{2}}\sum_{i=0}^{2}\|D^i_z fw_z^{\ast}\|_{L^2_z},
			~~~~
			\|f F^\ast_{\beta} \|_{\tilde{\mathcal{H}}^{2}}\leq C\frac{\alpha}{\beta^4}(1-\eta)^{-\frac{1}{2}}\sum_{i=0}^{2}\|D^i_z fw_z\|_{L^2_z}.
		\end{align*}
	\end{lemm}
\begin{proof}
		The proof follows that of Lemma \ref{7fm2}, modifying only the estimate \eqref{est-bd-Gamma}, which now reads:
		\begin{align}\label{sec9:est-bd-Gamma}
			\Gamma^\ast_{\beta} \tilde{w}_z\lesssim  \frac{1}{\beta}w^\ast_z,
			\quad
			\Gamma^\ast_{\beta} \tilde{w}_z\lesssim  \frac{1}{\beta}w_z.
		\end{align}
		The proof is complete, with the details omitted.
	\end{proof}
Considering the product estimates of $D_\theta{ F^\ast_{\beta}}$ and $D_z{ F^\ast_{\beta}}$ in the new space $\tilde{\mathcal{H}}^{2}$, we can obtain results similar to Lemma \ref{7fm2} and Lemma \ref{7fm3}. The results are highly similar and therefore omitted. We omit these similar product estimates here and use them directly in the proof later in the paper.  
\begin{prop}\label{9co1}
There exist constants $\alpha>0$ sufficiently small and $\eta(\beta)$, such that if $\alpha\ll 1-\eta\ll \beta$ and $|\mu|\leq c_2\alpha$, then
\begin{align}
	&-\langle\mathcal{L}^{\beta}_{\Gamma}(g), g\rangle_{\tilde{\mathcal{H}}^2}
	\geq 
	\frac12c_1\alpha \|g\|_{\tilde{\mathcal{H}}^2}^2
	-C\alpha^{-2}
	\|g\|_{\mathcal{H}^2}^2.
	\label{sec9:est-H2eta}
\end{align}
\end{prop}
\begin{proof}
It follows from the definition of $\mathcal{L}^{\beta}_{\Gamma}(g)$ that
\begin{align*}
	-\langle\mathcal{L}^{\beta}_{\Gamma}(g), g\rangle_{\tilde{\mathcal{H}}^2}
	=-\langle\mathcal{L}^{\beta}(g), g\rangle_{\tilde{\mathcal{H}}^2}+\langle\frac{3}{2\alpha}L^{-1}_{z,K}(g)F^{\ast}_{\beta}, g\rangle_{\tilde{\mathcal{H}}^2}.
\end{align*}
Using inequality \eqref{2Co-est-0}, we get
\begin{align*}
	-\langle\mathcal{L}^{\beta}(g), g\rangle_{\tilde{\mathcal{H}}^2}
	\geq \frac34 c_1\alpha \|g\|_{\tilde{\mathcal{H}}^2}^2
	-C\alpha^{-2} \|g \|_{\mathcal{H}^2}^2.
\end{align*}
By Lemma \ref{9lemma1}, we obtain that
\begin{align*}
	|\langle\frac{3}{2\alpha}L^{-1}_{z,K}(g)F^{\ast}_{\beta}, g\rangle_{\tilde{\mathcal{H}}^2}|
	&\lesssim 
	\|g\|_{\tilde{\mathcal{H}}^2}\|\frac{3}{2\alpha}L^{-1}_{z,K}(g)F^{\ast}_{\beta}\|_{\tilde{\mathcal{H}}^2} \\ 
	&\lesssim 
	\frac{1}{\beta^4}(1-\eta)^{-\frac{1}{2}}\|g\|_{\tilde{\mathcal{H}}^2}\sum_{i=0}^{2}\|D^i_z L^{-1}_{z,K}(g)w_z\|_{L^2_z}\\ 
	&\lesssim 
	\frac{c_1\alpha}{4}\|g\|^2_{\tilde{\mathcal{H}}^2}+C_{\beta}\alpha^{-1}(1-\eta)^{-1}\|g\|^2_{\mathcal{H}^2}.
\end{align*}
Combining the above inequality, we complete the proof of this proposition.
\end{proof}

\subsubsection{Parameter stability of main operator}
The goal of this subsection is to establish the coercivity of the operator $\mathcal{L}_{\Gamma}$ under the new weight, relying on the proven coercivity of $\mathcal{L}^{\beta}_{\Gamma}$ from Proposition \ref{9co1}.
In Section \ref{sec:Coer}, the continuity estimate \eqref{continu-F*} was applied to obtain the smallness parameter $\delta_0$, which was crucial for establishing the parameter stability of the main operator. 
In contrast, the key simplification here is that we can omit \eqref{continu-F*} entirely. This is possible for two reasons: (i) the finite interval estimates are already provided in Section \ref{sec:Coer}, and (ii) for $z \geq B_0=\alpha^{-2\beta}$, the condition $\gamma < \beta$ implies that both $\Gamma^{\ast}_{\beta}$ and $\Gamma^{\ast}_{\gamma}$ behave like $O(\alpha^2)$ at infinity. 
As a result, the proof of coercivity for $\mathcal{L}_{\Gamma}$ under the new weight (Proposition \ref{9co2} below) becomes more straightforward than that of Proposition \ref{5co2}.

\begin{prop}\label{9co2}
There exist constants $\alpha>0$ sufficiently small and $\eta(\beta)$, such that if $\alpha\ll 1-\eta\ll \beta$ and $|\mu|\leq c_2\alpha$, then
\begin{align}\label{sec9:oper-L2-H2}
	&
	\langle\mathcal{L}_{\Gamma}(g), g\rangle_{\mathcal{H}^2}\geq \frac{1}{4}c_1 \alpha\|g\|^2_{\tilde{\mathcal{H}}^2}
	-C_{\beta}\alpha^{-4}(1-\eta)^{-\frac12}\|g\|_{\mathcal{H}^2}^2
	,
\end{align}
where $C>0$ is a constant independent of $\alpha$.
\end{prop}
\begin{proof}
According to \eqref{sec9:est-H2eta} in Proposition \ref{9co1}, we compute that
\begin{align}\label{9p1}
	\langle\mathcal{L}_{\Gamma}^{\beta}(g), g\rangle_{\tilde{\mathcal{H}}^2}\geq \frac{1}{2}c_1\alpha\|g\|^2_{\tilde{\mathcal{H}}^2}
	-C\alpha^{-2}\|g\|_{\mathcal{H}^2}^2
	.
\end{align}
Using $\gamma<\beta$, we get that
\begin{align}\label{est-bg-1}
	&\|L^{-1}_z(\Gamma^\ast_{\beta})-L^{-1}_z(\Gamma^\ast_{\gamma})\|_{L^\infty_z(z\geq B_0)}
	+\frac{3}{2\alpha}\|F^{\ast}_{\beta}-F^{\ast}_{\gamma}\|_{L^\infty_{\theta}L^\infty_z(z\geq B_0)}
	\lesssim \alpha^2 \beta^{-1}.
\end{align}
And it is easy to check that
\begin{align}\label{est-bg-2}
	&\|L^{-1}_z(\Gamma^\ast_{\beta})-L^{-1}_z(\Gamma^\ast_{\gamma})\|_{L^\infty_z(z\leq B_0)}
	+\frac{3}{2\alpha}\|F^{\ast}_{\beta}-F^{\ast}_{\gamma}\|_{L^\infty_{\theta}L^\infty_z(z\leq B_0)}
	\lesssim \beta^{-1}.
\end{align}
Making use of \eqref{def-wP}, \eqref{est-bg-1} and \eqref{est-bg-2}, we can deduce from a similar way as \eqref{5p3} that
\begin{align}\label{9p3}
	|\langle\tilde{P}(g), g(\tilde{w}^\eta)^2\rangle|
	&\lesssim 
	\alpha^{2}\beta^{-1}\|g\tilde{w}^\eta\|_{L^2}\big(\|g\tilde{w}^\eta\|_{L^2}+(1-\eta)^{-\frac12}\|g\tilde{w}^\eta\|_{L^2}\big)
	\\
	\notag
	&\quad~
	+\alpha^{-4}\beta^{-1}\|gw^\eta\|_{L^2}\big(\|gw^\eta\|_{L^2}+(1-\eta)^{-\frac12}\|gw^K\|_{L^2}\big)
	\\ \notag
	&\lesssim\alpha^{2}\beta^{-1}(1-\eta)^{-\frac12}\|g\tilde{w}^\eta\|_{L^2}^2
	+\alpha^{-4}\beta^{-1}(1-\eta)^{-\frac12}\|gw^\eta\|_{L^2}^2
	.
\end{align}

The remaining terms can be handled similarly to the proof of \eqref{9p3} and are therefore omitted. It follows that \eqref{sec9:oper-L2-H2} holds.

\end{proof}

\subsection{Estimates of transport term}
The objective of this subsection is to establish estimates for the transport term in the $\tilde{H}^2$ space endowed with the new weights $\tilde{w}^\eta$ and $\tilde{w}^\lambda$.
Following the approach of Section \ref{sec:Est-T}, we begin by presenting several lemmas that will be essential for our subsequent treatment of the transport term.
In parallel with Lemma \ref{7em1}, we state the following lemma under the new weight $\tilde{w}_z$.
\begin{lemm}\label{9em1}
	Let $f(z,\theta)|_{\partial D}=0$. Then we have
	\begin{align}\label{9est1}
		\|f(1+z)^{\frac12}\|_{L^{\infty}_{z}}\leq C\|D_{z}f w_{z}\|_{L^{2}_{z}},
		\quad 
		\|f(1+z)^{\frac12}\|_{L^{\infty}_{z}}\leq C\|D_{z}f \tilde{w}^\ast_{z}\|_{L^{2}_{z}}.
	\end{align}
\end{lemm}
\begin{proof}
	Recalling the definitions of $w_z$ and $\beta\in (0,1]$, then applying integration by parts and $f(z,\theta)|_{\partial D}=0$, one obtains that
	\begin{align*}
		f^2&
		\lesssim \left(\int_z^{\infty}|\partial_z f|dz\right)^2 
		\lesssim \int_z^{\infty}|D_z f|^2(w_z)^2 dz\int_z^{\infty}(zw_z)^{-2} dz
		\\
		&\lesssim \|D_{z}fw_{z}\|^2_{L^{2}_{z}}\int_z^{\infty}(1+z)^{-2} dz \lesssim \|D_{z}fw_{z}\|^2_{L^{2}_{z}}(1+z)^{-1},
	\end{align*}
	which yields directly that the first estimate of \eqref{9est1} holds.
	We immediately obtain from the definition of $\tilde{w}^{\ast}_{z}$ that
	\begin{align*}
		\|f(1+z)^{\frac12}\|^2_{L^{\infty}_{z}}
		&\lesssim \|D_{z}f\tilde{w}^{\ast}_{z}\|^2_{L^{2}_{z}}\int_0^{\infty}(1+z)^{-\frac{3}{2}} dz 
		\lesssim \|D_{z}f\tilde{w}^{\ast}_{z}\|^2_{L^{2}_{z}}.
	\end{align*}
	This implies that the second estimate of \eqref{9est1} holds.
\end{proof}
Remembering the definitions of some weight functions $w_{z}$ and $w^{\lambda}$ in Section \ref{sec:notation}, we can derive the following corollary. 
Following Lemmas \ref{7em1} and \ref{9em1}, we skip the proof.
\begin{coro}\label{9em2}
	Let $f(z,\theta)|_{\partial D}=0$. Then we have
	$$\|f(1+z)^{\frac12}\|_{L^{\infty}}\leq C\sqrt{\frac{\beta}{\alpha}}\|D_{\theta}D_z fw^{\lambda}\|_{L^{2}},
	~~~
	\|f(1+z)^{\frac12}\|_{L^{\infty}}\leq C\sqrt{\frac{\beta}{\alpha}}\|D_{\theta}D_z f\tilde{w}^{\ast,\lambda}\|_{L^{2}}.
	$$
\end{coro}
Similar to Lemma \ref{7ha1}, we can derive the following lemma with new weight.
\begin{lemm}\label{9ha1}
	Let $f(z,\theta)|_{\partial D}=0$ and $\tilde{G}_f(z)=\frac{3}{4\alpha}z^{-\frac{5}{\alpha}}\int_{0}^{z}\rho^{\frac{5}{\alpha}-1}\langle f, K\rangle_{\theta}d\rho$. 
	Then we have
	$$\|D^k_z\tilde{G}_f\tilde{w}^{\ast}_z\|_{L^{2}_{z}}\leq C\|f\|_{\tilde{\mathcal{H}}^{k,\ast}},
	~~~~
	\alpha\|D^{k+1}_z\tilde{G}_f\tilde{w}^{\ast}_z\|_{L^{2}_{z}}\leq C\|f\|_{\tilde{\mathcal{H}}^{k,\ast}}.$$   
\end{lemm}
We derive a more refined estimate for the fundamental solution $F^\ast_\gamma$. Then estimate can be handled similarly to the proof of Lemma \ref{7fm1} and are therefore omitted.
\begin{lemm}\label{9fm1}
	Let $\beta \in (0,1]$ and $\gamma = \frac{1+\mu}{1-\mu}\beta$.
	There exist constants $\alpha>0$ sufficiently small and $\eta(\beta)$, such that if $\alpha\ll 1-\eta\ll \beta$ and $|\mu|\leq c_2\alpha$, then there holds 
	\begin{align*}
		\| F^\ast_{\gamma} \|_{\tilde{\mathcal{H}}^{k,\ast}}\leq C_{\beta}\alpha (1-\eta)^{-\frac{1}{2}},
	\end{align*}
	where integer $k$ satisfies $0\leq k\leq 4$.
\end{lemm}

In a manner similar to Lemma \ref{7div1}, a similar null structure can be derived, which reads:
\begin{lemm}\label{9div1}
	Let $f(z,\theta)|_{\partial D}=0$. 
	For any $\xi\in[0,+\infty)$, we denote that 
	\begin{align}\label{def-twxi}
		\bar{\tilde{w}}^\xi\tri \tilde{w}^\xi z^{\frac{1}{2}-\frac{3}{2\alpha}}(\cos\theta)^{-\frac{1}{2}},
	\end{align}
	then we have
	\begin{align}
		\label{stru-Tf-2}
		\Big\langle \frac{1}{\bar{\tilde{w}}^\xi}T_{\bar{\tilde{w}}^\xi f},(\tilde{w}^\xi)^2f\Big\rangle=0.
	\end{align}
\end{lemm}
The proof of Lemma \ref{9div1} closely follows that of Lemma \ref{7div1}, so we omit it.
\begin{rema}
	In what follows, we mainly employ the cases $\xi=\eta$ or $\lambda$ from Lemma \ref{9div1}.
\end{rema}

Next, we will establish the \textit{a priori} energy estimate for the $\tilde{\mathcal{H}}^{2}$-norm of $T_g$ by combining (i) the elliptic estimates for equation \eqref{6model} from Section \ref{sec:Elli}, (ii) the auxiliary $L^\infty$ bounds in Lemma \ref{9em1} and Corollary \ref{9em2}, and (iii) the transport null structure from Lemma \ref{9div1}.
\begin{prop}\label{9Tg}
	Let $L^{-1}_{z,K}(g)(0)=0$, there exist constants $\alpha>0$ sufficiently small and $\eta(\beta)$, such that if $\alpha\ll 1-\eta\ll \beta$ and $|\mu|\leq c_2\alpha$, then 
	\begin{align}\label{sec9:est-norm-H2eta}
		\big|\langle T_g, g\rangle_{\tilde{\mathcal{H}}^2}\big|
		\leq& ~
		\frac{1}{10}c_1\alpha\|g\|^2_{\tilde{\mathcal{H}}^2}
+ C_{\alpha}.
	\end{align}
\end{prop}

\begin{proof}
	Mirroring the approach used for $\bar{J}_1$ in \eqref{est-bJ1-1}, we split the $\tilde{\mathcal{L}}^{2}$-norm as follows:
	\begin{align}\label{sec9:est-bJ1-1}
		\tilde{\bar{J}}_1
		\tri\langle \tilde{T}_g^d, g(\tilde{w}^\eta)^2\rangle
		-\langle \tilde{T}_g^e, g(\tilde{w}^\eta)^2\rangle
		\tri \tilde{\bar{J}}_{11}+\tilde{\bar{J}}_{12}.
	\end{align}
	Here $\tilde{T}^d_g$ and $\tilde{T}^e_g$ are defined by
	\begin{align*}
		\tilde{T}^d_g\tri&\frac{1}{\bar{\tilde{w}}^\eta}\bigg[\frac{U(\Phi)}{\sin(2\theta)}D_{\theta}(g\bar{\tilde{w}}^\eta)+V(\Phi)\alpha D_{z}(g\bar{\tilde{w}}^\eta)\bigg]=\frac{1}{\bar{\tilde{w}}^\eta}T_{g\bar{\tilde{w}}^\eta},
		\\
		\tilde{T}^e_g\tri&\frac{1}{\bar{\tilde{w}}^\eta}g\bigg[\frac{U(\Phi)}{\sin(2\theta)}D_{\theta}(\bar{\tilde{w}}^\eta)+V(\Phi)\alpha D_{z}(\bar{\tilde{w}}^\eta)\bigg]=\frac{1}{\bar{\tilde{w}}^\eta}gT_{\bar{\tilde{w}}^\eta}.
	\end{align*}
	Here $\bar{w}^\eta$ is defined by \eqref{def-twxi} with $\xi=\eta$.
	Applying Lemma \ref{9div1} with $\xi=\eta$ yields
	\begin{align}\label{9g16}
		\tilde{\bar{J}}_{11}=0.
	\end{align}
	Then by exactly the same procedure as that in the decomposition of \eqref{est-bI-1}, we can deduce from \eqref{cal_U} and \eqref{cal_V} that
	\begin{align}\label{sec9:est-bJ12-1}
		\tilde{\bar{J}}_{12}
		=-\langle \tilde{T}_g^{e,l}, g(\tilde{w}^\eta)^2\rangle
		-\langle \tilde{T}_g^{e,n}, g(\tilde{w}^\eta)^2\rangle
		\tri \tilde{\bar{J}}_{121}+\tilde{\bar{J}}_{122},
	\end{align}
	where $\tilde{T}^{e,l}_g$ and $\tilde{T}^{e,n}_g$ are defined by
	\begin{align*}
		\tilde{T}^{e,l}_g
		\tri&\frac{1}{\bar{\tilde{w}}^\eta}g\bigg[\frac12\left(-3L^{-1}_{z}(\Gamma^{\ast}_{\gamma})+\alpha\Gamma^{\ast}_{\gamma}\right)D_{\theta}\bar{\tilde{w}}^\eta+L^{-1}_{z}(\Gamma^{\ast}_{\gamma})(\cos(2\theta)-\sin^2(\theta))\alpha D_{z}\bar{\tilde{w}}^\eta\bigg],
		\\
		\tilde{T}^{e,n}_g
		\tri&\frac{1}{\bar{\tilde{w}}^\eta}g\bigg[\frac{U(\Phi-\frac{3}{4\alpha}L^{-1}_{z,K}(F^{\ast}_\gamma)\sin(2\theta))}{\sin(2\theta)}D_{\theta}\bar{\tilde{w}}^\eta+V\Big(\Phi-\frac{3}{4\alpha}L^{-1}_{z,K}(F^{\ast}_\gamma)\sin(2\theta)\Big)\alpha D_{z}\bar{\tilde{w}}^\eta\bigg].
	\end{align*}
	Recalling the definition of $\bar{\tilde{w}}^\eta$ in \eqref{def-twxi}, one can infer that
	\begin{align}\label{9g17}
		\bigg\|\frac{D_{\theta}\bar{\tilde{w}}^\eta}{\bar{\tilde{w}}^\eta}\bigg\|_{L^\infty}+\alpha\bigg\|\frac{D_{z}\bar{\tilde{w}}^\eta}{\bar{\tilde{w}}^\eta}\bigg\|_{L^\infty}\leq C.
	\end{align}
	We then compute that
	\begin{align}
		|L^{-1}_{z}(\Gamma^{\ast}_{\gamma})|\lesssim \alpha^2, \quad \text{for}~~z\geq B_0,
		\label{est-Ga-1}
	\end{align}
	We also decompose the estimate of $\tilde{\bar{J}}_{121}$ into two intervals: $z\in [0,B_0]$ and $z\in [B_0,+\infty]$.
	For $z \in [B_0, +\infty)$, we employ \eqref{9g17}-\eqref{est-Ga-1}; whereas for $z \in [0, B_0)$, we adapt the technique from \eqref{7g20}. This leads to
	\begin{align}\label{9g20}
		|\tilde{\bar{J}}_{121}|
		\lesssim \alpha^2\left(\frac{\alpha}{\beta}+1\right)\|g\tilde{w}^\eta\|^2_{L^2}
		+\alpha^{-4}\left(\frac{\alpha}{\beta}+1\right)\|gw^\eta\|^2_{L^2}
		\leq \frac{1}{100}c_1\alpha\| g \|^2_{\tilde{\mathcal{H}}^2}+\alpha^{-4}\| g \|^2_{\mathcal{H}^2}.
	\end{align}
	An application of Lemmas \ref{9em1}, \ref{9ha1} and \ref{9fm1} yields
	\begin{align}\label{9g21}
		\bigg\|\frac{U(\Phi-\frac{3}{4\alpha}L^{-1}_{z,K}(F^{\ast}_\gamma)\sin(2\theta))}{\sin(2\theta)}(1+z)^{\frac12}\bigg\|_{L^{\infty}}
		\lesssim 
		\alpha^{\frac{1}{2}}(1-\eta)^{-\frac{1}{2}}
		+
		\alpha^{-1}\|g\|_{\mathcal{H}^{2}},
	\end{align}
and
	\begin{align}\label{9g22}
		\bigg\|V\Big(\Phi-\frac{3}{4\alpha}L^{-1}_{z,K}(F^{\ast}_\gamma)\sin(2\theta)\Big)(1+z)^{\frac12}\bigg\|_{L^{\infty}}\lesssim_\beta \alpha^{\frac{1}{2}}(1-\eta)^{-\frac{1}{2}}+\alpha^{-1}\|g\|_{\mathcal{H}^{2}}.
	\end{align}
	With \eqref{82in5}, \eqref{9g17}, \eqref{9g21} and \eqref{9g22} at hand, we can deduce that
	\begin{align}\label{9g23}
		|\tilde{\bar{J}}_{122}|&\lesssim_\beta \big(\alpha^{\frac{1}{2}}(1-\eta)^{-\frac{1}{2}}+\alpha^{-1}\|g\|_{\mathcal{H}^{2}}\big)\|g\tilde{w}^\eta(1+z)^{-\frac14}\|^2_{L^2}
		\\
		&\lesssim_\beta\big(\alpha^{\frac{1}{2}}(1-\eta)^{-\frac{1}{2}}+\alpha^{-1}\|g\|_{\mathcal{H}^{2}}\big)
		\big(
		\alpha^{-\frac{4}{\beta}} \|gw^\eta\|^2_{L^2_\theta L^2_z(0\leq z\leq B_1)}
		+
		\alpha \|g\tilde{w}^\eta\|^2_{L^2_\theta L^2_z(z\geq B_1)}
		\big)
		\nonumber
		\\
		&\lesssim_\beta \alpha^{\frac{5}{4}}
		\|g\tilde{w}^\eta\|^2_{L^2}
		+\|g\|_{\mathcal{H}^{2}}\|g\tilde{w}^\eta\|^2_{L^2}
		+\alpha^{\frac12-\frac{4}{\beta}}(1-\eta)^{-\frac{1}{2}}
		\|g\|^2_{\mathcal{H}^2}
		+\alpha^{-\frac4\beta-1} \|g\|_{\mathcal{H}^2}^3,
		\nonumber
	\end{align}
	with $B_1=\alpha^{-2}$.
	Substituting \eqref{9g20} and \eqref{9g23} into \eqref{sec9:est-bJ12-1}, one obtains that
	\begin{align}\label{9g24}
		|\tilde{\bar{J}}_{12}|
		\leq &~\frac{1}{100}c_1\alpha\| g \|^2_{\tilde{\mathcal{H}}^2}
		+C_\beta\big(\alpha^{-4}+\alpha^{\frac12-\frac{4}{\beta}}(1-\eta)^{-\frac{1}{2}}\big)\| g \|^2_{\mathcal{H}^2}
		+C_{\beta}\alpha^{-\frac4\beta-1} \|g\|_{\mathcal{H}^2}^3.
	\end{align}
	We then insert \eqref{9g16} and \eqref{9g24} into \eqref{sec9:est-bJ1-1}, we find that
	\begin{align}\label{est-tJ-1}
		|\tilde{\bar{J}}_{1}|
		\leq &~\frac{1}{100}c_1\alpha\| g \|^2_{\tilde{\mathcal{H}}^2}
		+C_\beta\big(\alpha^{-4}+\alpha^{\frac12-\frac{4}{\beta}}(1-\eta)^{-\frac{1}{2}}\big)\| g \|^2_{\mathcal{H}^2}
		+C_{\beta}\alpha^{-\frac4\beta-1} \|g\|_{\mathcal{H}^2}^3.		
	\end{align}

The estimate of the remaining terms except for $D_zD_{\theta}g$ is similar to that of $\tilde{\bar{J}}_{1}$, so we will omit the proof here.

Then we split the inner product of $D_zD_{\theta}g$ as follows:
	\begin{align}\label{sec9:est-bJ33-1}
		\tilde{\bar{J}}_{33}
		\tri\langle T_{D_zD_{\theta}g}, D_zD_{\theta}g(\tilde{w}^\lambda)^2\rangle
		+\langle \mathcal{R}_5, D_zD_{\theta}g(\tilde{w}^\lambda)^2\rangle
		\tri \tilde{\bar{J}}_{331}+\tilde{\bar{J}}_{332},
	\end{align}
	where the commutator $\mathcal{R}_5$ is defined in \eqref{def-R5}.
	Repeating the same line as \eqref{9g24}, we deduce from Lemma \ref{9div1}, \eqref{9g21} and \eqref{9g22} that
	\begin{align}\label{9g102}
		|\tilde{\bar{J}}_{331}|
		\leq &~\frac{1}{200}c_1\alpha\| g \|^2_{\tilde{\mathcal{H}}^2}
		+C_\beta\big(\alpha^{-4}+\alpha^{-\frac{4}{\beta}}\big)\| g \|^2_{\mathcal{H}^2}
		+C_{\beta}\alpha^{-\frac{4}{\beta}-1}\| g \|^3_{\mathcal{H}^2}.
	\end{align}
	We control $\bar{J}_{332}$ through a decomposition similar to \eqref{est-bJ213-1}, specifically:
	\begin{align}\label{sec9:est-bJ332-1}
		\tilde{\bar{J}}_{332}
		=\langle \mathcal{R}_5^l, D_zD_{\theta}g(w^\lambda)^2\rangle
		+\langle \mathcal{R}_5^n, D_zD_{\theta}g(w^\lambda)^2\rangle
		\tri \tilde{\bar{J}}_{3321}+\tilde{\bar{J}}_{3322},
	\end{align}
	where $\mathcal{R}_5^l$ and $\mathcal{R}_5^n$ are defined in \eqref{def-R5l} and \eqref{def-R5n}, respectively.
	Adapting the approach developed for \eqref{9g20}, we derive
	\begin{align}\label{9g103}
		|\tilde{\bar{J}}_{3321}|
	    \leq \frac{1}{200}c_1\alpha\| g \|^2_{\tilde{\mathcal{H}}^2}+ C\alpha^{-4}\| g \|^2_{\mathcal{H}^2}.
	\end{align}
	By \eqref{7g39}, and following a derivation analogous to \eqref{9g21}, we establish that
	\begin{align}\label{9g104}
		&~~~~\bigg|\bigg\langle D_\theta\bigg(\frac{U(\Phi-\frac{3}{4\alpha}L^{-1}_{z,K}(F^{\ast}_\gamma)\sin(2\theta))}{\sin(2\theta)}\bigg)D_zD_{\theta}g,D_zD_{\theta}g(w^\lambda)^2\bigg\rangle\bigg| \\ \notag
		&~~~~+\bigg|\bigg\langle D_{\theta}\left(V(\Phi-\frac{3}{4\alpha}L^{-1}_{z,K}(F^{\ast}_\gamma)\sin(2\theta))\right)\alpha D^2_{z} g,D_zD_{\theta}g(w^\lambda)^2\bigg\rangle\bigg|
		\\ \notag
		&\lesssim_\beta  \big(\alpha^{\frac{1}{2}}(1-\eta)^{-\frac{1}{2}}+\alpha^{-1}\|g\|_{\mathcal{H}^{2}}\big)\big(\alpha^{-\frac{4}{\beta}}\| g \|^2_{\mathcal{H}^2}+ \alpha\| g \|^2_{\tilde{\mathcal{H}}^2}\big).
	\end{align}
	We deduce from Proposition \ref{6prop0}, Corollaries \ref{6cor1}, \ref{6cor2} and \ref{9em2} that
	\begin{align}
		\label{9g106}
		&\bigg\|D_zD_\theta\left(\frac{U(\tilde{\Phi}_{F^\ast_{\gamma}})}{\sin(2\theta)}\right)w_\theta^\lambda(1+z)^{\frac12}\bigg\|_{L_z^{\infty}L_\theta^2}
		\lesssim
		\|\partial_\theta U(\tilde{\Phi}_{F^\ast_{\gamma}})\|_{\mathcal{H}^{3,\ast}} \lesssim_\beta\alpha(1-\eta)^{-\frac{1}{2}},
		\\
		&\bigg\|D_zD_\theta\left(\frac{U(\tilde{\Phi}_{g})}{\sin(2\theta)}\right)\tilde{w}_z\bigg\|_{L_z^2L_\theta^{\infty}}
		\lesssim \sqrt{\frac{\beta}{\alpha}}\left(\|\partial_\theta U(\tilde{\Phi}_{g})\|_{\tilde{\mathcal{H}}^{2}}+\|\partial^2_\theta \tilde{\Phi}_{g}\|_{\tilde{\mathcal{H}}^{2}}\right) \lesssim_\beta\alpha^{-\frac12}\|g\|_{\tilde{\mathcal{H}}^{2}}.
		\label{9g107}
	\end{align}
	The combination of \eqref{est-DzDtheta-1}, \eqref{9g106}-\eqref{9g107}, Corollary \ref{7em2} and Lemma \ref{9em1} yields directly that
	\begin{align}
		\label{9g108}
		&~~~~\bigg|\bigg\langle D_zD_\theta\bigg(\frac{U(\Phi-\frac{3}{4\alpha}L^{-1}_{z,K}(F^{\ast}_\gamma)\sin(2\theta))}{\sin(2\theta)}\bigg)D_{\theta}g,D_zD_{\theta}g(\tilde{w}^\lambda)^2\bigg\rangle\bigg| 
		\\ \notag
		&\lesssim_\beta \alpha(1-\eta)^{-\frac{1}{2}}\|D_{\theta}g \tilde{w}_z\|_{L^2_z L^\infty_\theta}\|D_zD_{\theta}g \tilde{w}^\lambda(1+z)^{-\frac12}\|_{L^2}
		\\ \notag
		&~~~~
		+\alpha^{-\frac12}\|g\|_{\tilde{\mathcal{H}}^{2}}\|D_{\theta}g w_\theta^\lambda(1+z)^{\frac12}\|_{L^\infty_z L^2_\theta}\|D_zD_{\theta}g \tilde{w}^\lambda(1+z)^{-\frac12}\|_{L^2} \\ \notag
		&\lesssim_\beta \alpha^{\frac32}(1-\eta)^{-\frac{1}{2}}
		\| g \|^2_{\tilde{\mathcal{H}}^2}
		+\|g\|_{\mathcal{H}^{2}}\| g \|^2_{\tilde{\mathcal{H}}^2}
		+ \alpha^{-\frac{4}{\beta}}\| g \|^2_{\mathcal{H}^2}
		+ \alpha^{-1-\frac{4}{\beta}}\| g \|^3_{\mathcal{H}^2}.
	\end{align}
	It follows from a similar way as \eqref{7g73}-\eqref{7g78} that
	\begin{align*}
		\alpha\bigg\|D_zD_\theta V\bigg(\Phi-\frac{3}{4\alpha}L^{-1}_{z,K}(F^{\ast})\sin(2\theta)\bigg)w_\theta^\lambda(1+z)^{\frac12}\bigg\|_{L_z^{\infty}L_\theta^2}\lesssim_\beta \alpha(1-\eta)^{-\frac{1}{2}}+\|g\|_{\mathcal{H}^{2}}.
	\end{align*}
	This, combined with Corollary \ref{7em2}, gives that
	\begin{align}\label{9g115}
		&~~~~~\bigg|\bigg\langle D_zD_\theta V\bigg(\Phi-\frac{3}{4\alpha}L^{-1}_{z,K}(F^{\ast}_\gamma)\sin(2\theta)\bigg)\alpha D_z g,D_zD_{\theta}g(\tilde{w}^\lambda)^2\bigg\rangle \bigg|
		\\ \notag
		&\lesssim_\beta \left(\alpha(1-\eta)^{-\frac{1}{2}}+\|g\|_{\mathcal{H}^{2}}\right)\|D_{z}g \tilde{w}_z\|_{L^2_z L^\infty_\theta} \|D_zD_{\theta}g \tilde{w}^\lambda(1+z)^{-\frac12}\|_{L^2} 
		\\ \notag
		&\lesssim_\beta
		\alpha^{\frac{3}{2}}(1-\eta)^{-\frac{1}{2}}\| g \|^2_{\tilde{\mathcal{H}}^2}
		+\alpha^{\frac12}\|g\|_{\mathcal{H}^{2}}\| g \|^2_{\tilde{\mathcal{H}}^2}
		+\big(\alpha^{\frac{1}{2}}(1-\eta)^{-\frac{1}{2}}+\alpha^{-\frac12}\|g\|_{\mathcal{H}^{2}}\big)\alpha^{-\frac{4}{\beta}}\| g \|^2_{\mathcal{H}^2}.
	\end{align}
	With the help of \eqref{7g94}, one gets that
	\begin{align}\label{9g116}
		&\bigg|\bigg\langle D_z V\bigg(\Phi-\frac{3}{4\alpha}L^{-1}_{z,K}(F^{\ast}_\gamma)\sin(2\theta)\bigg)\alpha D_zD_\theta g,D_zD_{\theta}g(\tilde{w}^\lambda)^2\bigg\rangle\bigg| \\ \notag
		\lesssim_\beta& \big(\alpha^{\frac{3}{2}}(1-\eta)^{-\frac{1}{2}}+\alpha^{-\frac{1}{2}}\|g\|_{\mathcal{H}^{2}}\big)\| g \|^2_{\tilde{\mathcal{H}}^2}.
	\end{align}
	In Section \ref{sec:Est-T}, $\|D_z \Big(\frac{U(\tilde{\Phi}_{g})}{\sin(2\theta)}\Big)\|_{L^\infty}$ is a difficult term to deal with, we thus introduce the $\mathcal{E}^{2}$-norm.
	However, we can now simply control this term by choosing $B_2 = \alpha^{-4}$ and splitting the domain into two intervals: $[0, B_2]$ and $[B_2, +\infty)$.
	Taking advantage of Corollaries \ref{9em2}, \ref{6cor3} and Lemma \ref{6lem4}, we get that
	\begin{align}\label{9g117}
		\bigg\|D_z \bigg(\frac{U(\tilde{\Phi}_{g})}{\sin(2\theta)}\bigg)(1+z)^{\frac12}\bigg\|_{L^{\infty}}
		&\lesssim \sqrt{\frac{\beta}{\alpha}}\bigg\|D^2_zD_\theta \bigg(\frac{U(\tilde{\Phi}_{g})}{\sin(2\theta)}\bigg)w^\lambda\bigg\|_{L^{2}} 
		\\
		&\lesssim \sqrt{\frac{\beta}{\alpha}}\|\partial_\theta D^2_z U(\tilde{\Phi}_{g})w^\lambda\|_{L^{2}} 
		\lesssim_\beta \alpha^{-1}\|g\|_{\mathcal{E}^{2}}.
		\notag
	\end{align}
	In view of \eqref{9g117}, we similarly deduce that
	\begin{align}
		\label{9g118}
		\bigg\|D_z \bigg(\frac{U(\Phi-\frac{3}{4\alpha}L^{-1}_{z,K}(F^{\ast}_\gamma)\sin(2\theta))}{\sin(2\theta)}\bigg)(1+z)^{\frac12}\bigg\|_{L^{\infty}}
		\lesssim_\beta \alpha^{-1}\|g\|_{\mathcal{E}^{2}}
		+\alpha^{-1}\|g\|_{\mathcal{H}^{2}}
		+\alpha^{\frac{1}{2}}(1-\eta)^{-\frac{1}{2}},
	\end{align}
	which yields that
	\begin{align}\label{9g119}
		&~~~~\bigg|\bigg\langle D_z \bigg(\frac{U(\Phi-\frac{3}{4\alpha}L^{-1}_{z,K}(F^{\ast}_\gamma)\sin(2\theta))}{\sin(2\theta)}\bigg)D^2_\theta g,D_zD_{\theta}g(\tilde{w}^\lambda)^2\bigg\rangle\bigg| \\ \notag
		&\lesssim \bigg\|D_z \bigg(\frac{U(\Phi-\frac{3}{4\alpha}L^{-1}_{z,K}(F^{\ast}_\gamma)\sin(2\theta))}{\sin(2\theta)}\bigg)(1+z)^{\frac12}\bigg\|_{L^{\infty}} 
		\|D^2_\theta g \tilde{w}^\lambda\|_{L^2}
		\\
		\notag
		&~~~~\times \big(\|D_zD_\theta g \tilde{w}^\lambda(1+z)^{-\frac12}\|_{L^2_\theta L^2_z(0\leq z\leq B_2)}
		+\|D_zD_\theta g \tilde{w}^\lambda(1+z)^{-\frac12}\|_{L^2_\theta L^2_z(z\geq B_2)}\big)
		\\
		\notag
		&\lesssim_\beta \big(\alpha^{-1}\|g\|_{\mathcal{E}^{2}}+\alpha^{-1}\|g\|_{\mathcal{H}^{2}}+\alpha^{\frac{1}{2}}(1-\eta)^{-\frac{1}{2}}\big) \big(\alpha^2\| g \|^2_{\tilde{\mathcal{H}}^2}+\alpha^{-\frac{6}{\beta}}\| g \|^2_{\mathcal{H}^2}\big),
	\end{align}
	with $B_2=\alpha^{-4}$.
	Taking advantage of \eqref{9g104}, \eqref{9g108}-\eqref{9g116} and \eqref{9g119}, it is easy to deduce that
	\begin{align}\label{9g120}
		|\tilde{\bar{J}}_{3322}|
		\lesssim_\beta &~\alpha^{\frac{3}{2}}(1-\eta)^{-\frac{1}{2}}
		\| g \|^2_{\tilde{\mathcal{H}}^2}
		+\|g\|_{\mathcal{H}^{2}}\| g \|^2_{\tilde{\mathcal{H}}^2}
		+\alpha\|g\|_{\mathcal{E}^{2}}\| g \|^2_{\tilde{\mathcal{H}}^2}
		\\
		&~+ \alpha^{-\frac{6}{\beta}}\| g \|^2_{\mathcal{H}^2}
		+\alpha^{-1-\frac{6}{\beta}}\| g \|^3_{\mathcal{H}^2}
		+\alpha^{-1-\frac{6}{\beta}}\|g\|_{\mathcal{E}^2}\| g \|^2_{\mathcal{H}^2}
		.
		\notag
	\end{align}
	We then substitute \eqref{9g103} and \eqref{9g120} into \eqref{sec9:est-bJ332-1}, to discover that
	\begin{align}\label{sec9:est-bJ332-2}
		|\tilde{\bar{J}}_{332}|
		\leq&~ \frac{1}{200}c_1\alpha\|g\|^2_{\tilde{\mathcal{H}}^2}+ C_{\beta}\alpha^{-\frac{6}{\beta}}\| g \|^2_{\mathcal{H}^2}
		+C_{\beta}\alpha^{-1-\frac{6}{\beta}}\| g \|^3_{\mathcal{H}^2}
		+C_{\beta}\alpha^{-1-\frac{6}{\beta}}\|g\|_{\mathcal{E}^2}\| g \|^2_{\mathcal{H}^2}.
	\end{align}
	Plugging \eqref{9g102} and \eqref{sec9:est-bJ332-2} into \eqref{sec9:est-bJ33-1}, we easily deduce that
	\begin{align}\label{9g121}
		|\tilde{\bar{J}}_{33}| 
		\leq&~\frac{1}{100}c_1\alpha\|g\|^2_{\tilde{\mathcal{H}}^2}+ C_{\beta}\alpha^{-\frac{6}{\beta}}\| g \|^2_{\mathcal{H}^2}
		+C_{\beta}\alpha^{-1-\frac{6}{\beta}}\| g \|^3_{\mathcal{H}^2}
		+C_{\beta}\alpha^{-1-\frac{6}{\beta}}\|g\|_{\mathcal{E}^2}\| g \|^2_{\mathcal{H}^2}
		.
	\end{align}

	Consequently, collecting the estimates, we get \eqref{sec9:est-norm-H2eta} immediately.
\end{proof}

	Having established Lemma \ref{7fm1}, we proceed to obtain the a priori energy estimates for the $\tilde{\mathcal{H}}^2$-norm of $T_{F^\ast_{\gamma}}$.

	\begin{prop}\label{9Tf}
		Let $\beta\in(0,1]$, there exist constants $\alpha>0$ sufficiently small and $\eta(\beta)$, such that if $\alpha\ll 1-\eta\ll \beta$ and $|\mu|\leq c_2\alpha$, then
		\begin{align}
			|\langle T_{F^\ast_{\gamma}}, g\rangle_{\tilde{\mathcal{H}}^2}|
			\leq C_\beta \big(\alpha(1-\eta)^{-1}+(1-\eta)^{-\frac12}\|g\|_{\mathcal{H}^2}\big)\alpha \|g\|_{\tilde{\mathcal{H}}^2}.
			\label{sec9:est-TF-H2eta}
		\end{align}
	\end{prop}
	\begin{proof}
		We perform the decomposition as follows:
		\begin{align}\label{sec9:est-bM-1}
			\langle T_{F^\ast_{\gamma}}, g\rangle_{\tilde{\mathcal{H}}^2}
			=\langle T^l_{F^\ast_{\gamma}}, g\rangle_{\tilde{\mathcal{H}}^2}
			+\langle T^n_{F^\ast_{\gamma}}, g\rangle_{\tilde{\mathcal{H}}^2}
			\tri \tilde{\bar{M}}_1+\tilde{\bar{M}}_2,
		\end{align}
		where $T^l_{F^\ast_{\gamma}}$ and $T^n_{F^\ast_{\gamma}}$ are defined in \eqref{def-Tl} and \eqref{def-Tn}, respectively.
	   In view of Lemma \ref{9fm1}, \eqref{est-bd-Gamma} and \eqref{cal-Fgamma}, one can obtain that
	   \begin{align}\label{9f14-1}
		  |\tilde{\bar{M}}_1|
		  \lesssim \|T^l_{F^\ast_{\gamma}}\|_{\tilde{\mathcal{H}}^2}\|g\|_{\tilde{\mathcal{H}}^2}
		  \lesssim_\beta \alpha
		  \|F^\ast_{\gamma}\|_{\tilde{\mathcal{H}}^{2,\ast}}\|g\|_{\tilde{\mathcal{H}}^2}
		\lesssim_\beta \alpha^2(1-\eta)^{-\frac{1}{2}}\|g\|_{\tilde{\mathcal{H}}^2}.
	   \end{align}
	   Applying Lemma \ref{7fm1}, it is enough to prove that
	   \begin{align}\label{9f16}
	   	&~~~~\bigg|\bigg\langle \frac{U(\Phi-\frac{3}{4\alpha}L^{-1}_{z,K}(F^{\ast}_\gamma)\sin(2\theta))}{\sin(2\theta)}D_{\theta}F^\ast_{\gamma}, g\bigg\rangle_{\tilde{\mathcal{H}}^2}\bigg| 
	   	\\ \notag
	   	&\lesssim (1-\eta)^{-2}\bigg\|\frac{U(\Phi-\frac{3}{4\alpha}L^{-1}_{z,K}(F^{\ast}_\gamma)\sin(2\theta))}{\sin(2\theta)}D_{\theta}F^\ast_{\gamma}\bigg\|_{\tilde{\mathcal{H}}^2}\|g\|_{\tilde{\mathcal{H}}^2}\\ \notag
	   	&\lesssim_\beta\alpha^2(1-\eta)^{-\frac{1}{2}}\|g\|_{\tilde{\mathcal{H}}^2}
	   	\sum_{i=0}^{2}\left(\|D^i_z U(\tilde{G}_g+G^\ast_g)w_z\|_{L^2_z}+\|D^i_z U(\tilde{G}_{F^{\ast}_\gamma})w_z^\ast\|_{L^2_z}\right)
	   	\\ \notag
	   	&~~~~+\alpha^2\|g\|_{\tilde{\mathcal{H}}^2}\left(\bigg\|\frac{U(\tilde{\Phi}_g)}{\sin(2\theta)}\bigg\|_{\mathcal{H}^2}
	   	+\bigg\|\frac{U(\tilde{\Phi}_{F^{\ast}_\gamma})}{\sin(2\theta)}\bigg\|_{\mathcal{H}^{2,\ast}}\right)
	   	\\ \notag
	   	&\lesssim_\beta\alpha^2(1-\eta)^{-\frac{1}{2}}\|g\|_{\tilde{\mathcal{H}}^2}
	   	\left(\alpha^{-1}\|g\|_{\mathcal{H}^2}+\|F^{\ast}_\gamma\|_{\mathcal{H}^{2,\ast}}\right)
	   	\\ \notag
	   	&\lesssim_\beta \big(\alpha^2(1-\eta)^{-1}+(1-\eta)^{-\frac12}\|g\|_{\mathcal{H}^2}\big)\alpha \|g\|_{\tilde{\mathcal{H}}^2}.
	   \end{align}
	   Again thanks to Lemma \ref{7fm1}, we can deduce from a similar derivation as \eqref{7f18} that
	   \begin{align}\label{9f18}
	   	\bigg|\bigg\langle V\bigg(\Phi-\frac{3}{4\alpha}L^{-1}_{z,K}(F^{\ast}_\gamma)\sin(2\theta)\bigg)\alpha D_{z}F^\ast_{\gamma}, g\bigg\rangle_{\tilde{\mathcal{H}}^2}\bigg| 
	   \lesssim_\beta \big(\alpha^2(1-\eta)^{-1}+(1-\eta)^{-\frac12}\|g\|_{\mathcal{H}^2}\big)\alpha \|g\|_{\tilde{\mathcal{H}}^2}.
	   \end{align}
	   The combination of \eqref{9f16} and \eqref{9f18} yields directly that
	   \begin{align}\label{9f19}
	   	&|\tilde{\bar{M}}_2|\lesssim_\beta\big(\alpha^2(1-\eta)^{-1}+(1-\eta)^{-\frac12}\|g\|_{\mathcal{H}^2}\big)\alpha \|g\|_{\tilde{\mathcal{H}}^2}.
	   \end{align}
	   Substituting \eqref{9f14-1} and \eqref{9f19} into \eqref{sec9:est-bM-1}, we get \eqref{sec9:est-TF-H2eta} immediately.
    \end{proof}

    \subsection{Estimates of the remaining terms}
    Our goal in this section is to handle the remaining terms. 
    To this end, we first present several useful lemmas.
    \begin{lemm}\label{9em1-1}
    	Let $f(z,\theta)|_{\partial D}=0$. Then we have
    	\begin{align}\label{9est1-1}
    		\|f(1+z)^{\frac12}\tilde{w}_zw_z^{-1}\|_{L^{\infty}_{z}}\leq C\|D_{z}f \tilde{w}_{z}\|_{L^{2}_{z}}.
    	\end{align}
    \end{lemm}
    \begin{proof}
    	It is easy to verify that
    	\begin{align*}
    		\tilde{w}_zw_z^{-1}=(1+z^{\frac1\beta})^{A}.
    	\end{align*}
    	Since $\beta\in (0,1]$,  we employ integration by parts and $f(z,\theta)|_{\partial D}=0$, to get that
    	\begin{align*}
    		f^2&
    		\lesssim \left(\int_z^{\infty}|\partial_z f|dz\right)^2 
    		\lesssim \int_z^{\infty}|D_z f|^2(\tilde{w}_z)^2 dz\int_z^{\infty}(z\tilde{w}_z)^{-2} dz
    		\\
    		&\lesssim \|D_{z}fw_{z}\|^2_{L^{2}_{z}}\int_z^{\infty}(1+z)^{-2-\frac{2A}{\beta}} dz \lesssim \|D_{z}fw_{z}\|^2_{L^{2}_{z}}(1+z)^{-1}\tilde{w}_z^2w_z^{-2}.
    	\end{align*}
    	This yields \eqref{9est1-1}, thereby proving the lemma.
    \end{proof}
    The following corollary, stated without proof, follows immediately from Lemma \ref{9em1-1}.
    \begin{coro}\label{9em2-1}
    	Let $f(z,\theta)|_{\partial D}=0$. Then we have
    	$$\|f(1+z)^{\frac12}\tilde{w}_zw_z^{-1}\|_{L^{\infty}}\leq C\sqrt{\frac{\beta}{\alpha}}\|D_{\theta}D_z f\tilde{w}^{\lambda}\|_{L^{2}}.
    	$$
    \end{coro}

    We introduce another weighted Sobolev spaces, denoted by $\tilde{\mathcal{H}}^k_{w+}$ and $\tilde{\mathcal{H}}^k_{w-}$ $(k\in\mathbb{N}^+)$, defined as follows:
    \begin{align*}
    	\|f\|_{\tilde{\mathcal{H}}^k_{w+}}
    	\tri\sum_{i=0}^{k}\|D^i_z f\tilde{w}^{\eta}(1+z)^{\frac12}\|^2_{L^2}+\sum_{0\leq i+j\leq k,j\geq 1}\|D^i_zD^j_{\theta} f\tilde{w}^{\lambda}(1+z)^{\frac12}\|^2_{L^2},
    	\\
    	\|f\|_{\tilde{\mathcal{H}}^k_{w-}}
    	\tri\sum_{i=0}^{k}\|D^i_z f\tilde{w}^{\eta}(1+z)^{-\frac12}\|^2_{L^2}+\sum_{0\leq i+j\leq k,j\geq 1}\|D^i_zD^j_{\theta} f\tilde{w}^{\lambda}(1+z)^{-\frac12}\|^2_{L^2}.
    \end{align*}
    It is easy to verify that
    \begin{align}\label{sec9:est-Hw}
    	\|fg\|_{\tilde{\mathcal{H}}^k}\lesssim \|f\|_{\tilde{\mathcal{H}}^k_{w+}}	\|g\|_{\tilde{\mathcal{H}}^k_{w-}}.
    \end{align}
    Our goal is to establish the following lemmas concerning product rules, relying on Corollary \ref{7em2}, Corollary \ref{9em2} and Corollary \ref{9em2-1}.

    \begin{lemm}\label{9pr1}
    	Let $f(z,\theta)|_{\partial D}=g(z,\theta)|_{\partial D}=0$. Then we have
    	\begin{align}\label{sec9:est-fg-H2}
    		\|fg\|_{\tilde{\mathcal{H}}^2_{w+}}\leq C\sqrt{\frac{\beta}{\alpha}}\|f\|_{\mathcal{H}^{2}}\|g\|_{\tilde{\mathcal{H}}^2}
    		,~~~
    		\|fg\|_{\tilde{\mathcal{H}}^2_{w+}}\leq C\sqrt{\frac{\beta}{\alpha}}\|f\|_{\tilde{\mathcal{H}}^{3,\ast}}\|g\|_{\tilde{\mathcal{H}}^2}.
    	\end{align}
    \end{lemm}
    
    \begin{proof}
    	Applying Corollary \ref{9em2}, we see that
    	\begin{align*}  
    		\|fg\tilde{w}^\eta (1+z)^{\frac12}\|_{L^2}\lesssim\|f(1+z)^{\frac12}\|_{L^\infty}\|g\tilde{w}^\eta\|_{L^2}\lesssim \sqrt{\frac{\beta}{\alpha}}\|f\|_{\mathcal{H}^{2}}\|g\tilde{w}^\eta\|_{L^2},
    	\end{align*}
    	and
    	\begin{align*}  
    		\|D_z(fg)\tilde{w}^\eta (1+z)^{\frac12}\|_{L^2}
    		&\lesssim\|f(1+z)^{\frac12}\|_{L^\infty}\|D_zg\tilde{w}^\eta\|_{L^2}+\|D_zfw_\theta^\eta(1+z)^{\frac12}\|_{L_z^\infty L_\theta^2}\|g\tilde{w}_z\|_{L^2_z L_\theta^\infty}
    		\\
    		&\lesssim \sqrt{\frac{\beta}{\alpha}}\|f\|_{\mathcal{H}^{2}}\|g\|_{\tilde{\mathcal{H}}^{1}},
    		\\
    		\|D_\theta(fg)\tilde{w}^\lambda (1+z)^{\frac12}\|_{L^2}
    		&\lesssim\|f(1+z)^{\frac12}\|_{L^\infty}\|D_\theta g\tilde{w}^\lambda\|_{L^2}
    		+\|D_\theta fw_\theta^\lambda(1+z)^{\frac12}\|_{L_z^\infty L_\theta^2}\|g\tilde{w}_z\|_{L^2_z L_\theta^\infty}
    		\\
    		&\lesssim \sqrt{\frac{\beta}{\alpha}}\|f\|_{\mathcal{H}^{2}}\|g\|_{\tilde{\mathcal{H}}^{1}}.
    	\end{align*}
    	By using Corollary \ref{7em2}, Corollary \ref{9em2}, Lemma \ref{9em1-1} and Corollary \ref{9em2-1}, we obtain that
    	\begin{align*}  
    		&\|D^2_\theta(fg)\tilde{w}^\lambda (1+z)^{\frac12}\|_{L^2}
    		+\|D^2_z(fg)\tilde{w}^\eta (1+z)^{\frac12}\|_{L^2}
    		\\
    		\lesssim&\|f(1+z)^{\frac12}\|_{L^\infty}\|D^2_\theta g\tilde{w}^\lambda\|_{L^2}
    		+\|D^2_\theta fw^\lambda\|_{L^2}\|g(1+z)^{\frac12}\tilde{w}_zw_z^{-1}\|_{L^\infty}
    		\\
    		&
    		+\|D_\theta fw_\theta^\lambda(1+z)^{\frac12}\|_{L_z^\infty L_\theta^2}\|D_\theta 
    		g\tilde{w}_z\|_{L^2_z L_\theta^\infty} 
    		+\|f(1+z)^{\frac12}\|_{L^\infty}\|D^2_z g\tilde{w}^\eta\|_{L^2}
    		\\
    		&
    		+\|D^2_z fw^\eta\|_{L^2}\|g(1+z)^{\frac12}\tilde{w}_zw_z^{-1}\|_{L^\infty}
    		+\|D_z fw_\theta^\eta(1+z)^{\frac12}\|_{L_z^\infty L_\theta^2}\|D_z 
    		g\tilde{w}_z\|_{L^2_z L_\theta^\infty} \\
    		\lesssim &\sqrt{\frac{\beta}{\alpha}}\|f\|_{\mathcal{H}^{2}}\|g\|_{\tilde{\mathcal{H}}^{2}},
    	\end{align*}
    	and
    	\begin{align*}  
    		&\|D_\theta D_z(fg)\tilde{w}^\lambda (1+z)^{\frac12}\|_{L^2}
            \\
    		\lesssim
    		&\|f(1+z)^{\frac12}\|_{L^\infty}\|D_\theta D_z g\tilde{w}^\lambda\|_{L^2}
    		+\|D_\theta D_z fw^\lambda\|_{L^2}\|g(1+z)^{\frac12}\tilde{w}_zw_z^{-1}\|_{L^\infty} \\
    		&+\|D_\theta fw_\theta^\lambda(1+z)^{\frac12}\|_{L_z^\infty L_\theta^2}\|D_z 
    		g\tilde{w}_z\|_{L^2_z L_\theta^\infty}
    		+\|D_zf \tilde{w}_z\|_{L^2_z L_\theta^\infty}\|D_\theta 
    		gw_\theta^\lambda(1+z)^{\frac12}\|_{L_z^\infty L_\theta^2}\\
    		\lesssim
    		&\sqrt{\frac{\beta}{\alpha}}\|f\|_{\mathcal{H}^{2}}\|g\|_{\tilde{\mathcal{H}}^{2}}.
    	\end{align*}
    	Collecting these estimates above, we can obtain the first inequality of \eqref{sec9:est-fg-H2}.
    	
    	For the second one, the proof is similar, so we omit it here.
    \end{proof}

    Our next step is to deal with $R_2$ by applying the product rules established in Lemmas \ref{9pr1}.
    \begin{prop}\label{9R1}
    	Let $L^{-1}_{z,K}(g)(0)=0$. There exist constants $\alpha>0$ sufficiently small and $\eta(\beta)$, such that if $\alpha\ll 1-\eta\ll\beta$ and $|\mu|\leq c_2\alpha$, then 
    	\begin{align}\label{sec9:est-R2-H2eta}
    		|\langle R_2, g\rangle_{\tilde{\mathcal{H}}^{2}}|
    		\leq 
    		\frac{1}{10}c_1\alpha\|g\|^2_{\tilde{\mathcal{H}}^2}
+ C_{\alpha}.
    	\end{align}
    \end{prop}
    
    \begin{proof}
    	Recalling the definition of $R_2$ in \eqref{def-R2-1} and $F=F^\ast_\gamma+g$, we can decompose the $\tilde{\mathcal{H}}^{2}$-norm as:
    	\begin{align}\label{sec9:est-bmI-1}
    		\langle R_2, g\rangle_{\tilde{\mathcal{H}}^2}
    		&=
    		\frac{3}{2\alpha}\langle L^{-1}_{z,K}(g)g, g\rangle_{\tilde{\mathcal{H}}^2}
    		-\frac{3}{2}\langle \sin^2(\theta)\langle F,K\rangle_{\theta}F, g\rangle_{\tilde{\mathcal{H}}^2}
    		\\
    		&~~~
    		+\langle \mathcal{R}(\Phi_g-G^\ast_{g}\sin(2\theta))g, g\rangle_{\tilde{\mathcal{H}}^2}
    		+\langle \mathcal{R}(\Phi_{F^\ast_\gamma}-G^\ast_{F^\ast_\gamma}\sin(2\theta))g, g\rangle_{\tilde{\mathcal{H}}^2}
    		\notag \\
    		&~~~
    		+\langle \mathcal{R}(\Phi_g-G^\ast_{g}\sin(2\theta))F^\ast_\gamma, g\rangle_{\tilde{\mathcal{H}}^2}
    		+\langle \mathcal{R}(\Phi_{F^\ast_\gamma}-G^\ast_{F^\ast_\gamma}\sin(2\theta))F^\ast_\gamma, g\rangle_{\tilde{\mathcal{H}}^2}
    		\notag \\
    		&\tri \tilde{\bar{\mathcal{I}}}_1 + \tilde{\bar{\mathcal{I}}}_2 +\tilde{\bar{\mathcal{I}}}_3+\tilde{\bar{\mathcal{I}}}_4+\tilde{\bar{\mathcal{I}}}_5+\tilde{\bar{\mathcal{I}}}_6,
    		\notag
    	\end{align}
    	where for $f$ being $g$ or $F^\ast_\gamma$, the notation $\Phi_f$ is defined as $\Phi_f\tri\tilde{\Phi}_f+\tilde{G}_{f}\sin(2\theta)+{G}_{f}^\ast\sin(2\theta)$.
    	We now estimate $\tilde{\bar{\mathcal{I}}}_i$ $(i=1,\cdots,6)$ term by term.
    	We first give a useful estimate as:
    	\begin{align}\label{sec9:est-Hw-}
    		\|g\|_{\tilde{\mathcal{H}}^{2}_{w-}}
    		\lesssim
    		\alpha \|g\|_{\tilde{\mathcal{H}}^{2}}+\alpha^{-\frac{2}{\beta}}\|g\|_{{\mathcal{H}}^{2}}.
    	\end{align}
    	Owing to the $\theta$-independence of $L^{-1}_{z,K}(g)$, an application of \eqref{sec9:est-Hw}, \eqref{sec9:est-Hw-}, and Lemma \ref{9em1} yields that
    	\begin{align}\label{9in6-1}
    		|\tilde{\bar{\mathcal{I}}}_1|
    		&\lesssim \alpha^{-1}\|L^{-1}_{z,K}(g)g\|_{\tilde{\mathcal{H}}^{2}_{w+}}\|g\|_{\tilde{\mathcal{H}}^{2}_{w-}}
    		\lesssim 
    		\alpha^{-1}\sum_{i=1}^2\|D^i_z L^{-1}_{z,K}(g)w_z\|_{L^2_z}\|g\|_{\tilde{\mathcal{H}}^{2}}
    		\big(\alpha \|g\|_{\tilde{\mathcal{H}}^{2}}+\alpha^{-\frac{2}{\beta}}\|g\|_{{\mathcal{H}}^{2}}\big)
    		\\
    		&\lesssim 
    		\|g\|_{\mathcal{H}^2}\|g\|_{\tilde{\mathcal{H}}^{2}}^2+\alpha^{-\frac4\beta}\|g\|_{\mathcal{H}^2}^3.
    		\notag
    	\end{align}
    	We apply \eqref{8est-sinF}, \eqref{sec9:est-Hw}, \eqref{sec9:est-Hw-}, Lemma \ref{7fm1} and Lemma \ref{9pr1} to control $\bar{\mathcal{I}}_2$ by
    	\begin{align}\label{9in5-1}
    		|\tilde{\bar{\mathcal{I}}}_2|
    		&\lesssim \|g\|_{\tilde{\mathcal{H}}^{2}_{w-}}\|\sin^2(\theta)\langle F,K\rangle_{\theta}g\|_{\tilde{\mathcal{H}}^{2}_{w+}}
    		+ \|g\|_{\tilde{\mathcal{H}}^{2}}\|\sin^2(\theta)\langle F,K\rangle_{\theta}F^\ast_\gamma\|_{\tilde{\mathcal{H}}^{2}}
    		\\ \notag
    		&\lesssim_{\beta} \alpha^{-\frac12}\big(\alpha \|g\|_{\tilde{\mathcal{H}}^{2}}+\alpha^{-\frac{2}{\beta}}\|g\|_{{\mathcal{H}}^{2}}\big)\|g\|_{\tilde{\mathcal{H}}^{2}}\left(\|\sin^2(\theta)\langle g,K\rangle_{\theta}\|_{\mathcal{H}^2}+\|\sin^2(\theta)\langle F^\ast_\gamma,K\rangle_{\theta}\|_{\tilde{\mathcal{H}}^{3,\ast}}\right)
    		\\ \notag
    		&~~~~+\alpha\|g\|_{\tilde{\mathcal{H}}^{2}}
    		\left(\|\sin^2(\theta)\langle g,K\rangle_{\theta}\|_{\mathcal{H}^2}+\|\sin^2(\theta)\langle F^\ast_\gamma,K\rangle_{\theta}\|_{\tilde{\mathcal{H}}^{2,\ast}}\right)
    		\\ \notag
    		&\lesssim_{\beta} (1-\eta)^{-1}
    		\left(\alpha^2\|g\|_{\tilde{\mathcal{H}}^{2}}
    		+\alpha^{\frac32}\|g\|_{\tilde{\mathcal{H}}^{2}}^2
    		+\alpha^{\frac12}\|g\|_{\mathcal{H}^2}\|g\|_{\tilde{\mathcal{H}}^{2}}^2
    		+\alpha^{-1-\frac4\beta}\|g\|^2_{\mathcal{H}^2}
    		+\alpha^{-\frac32-\frac4\beta}\|g\|^3_{\mathcal{H}^2}
    		\right).
    	\end{align}
    	In view of \eqref{est-mRg-1}, \eqref{sec9:est-Hw}, \eqref{sec9:est-Hw-}, Lemma \ref{9pr1}, we get that
    	\begin{align}\label{9in1-1}
    		|\tilde{\bar{\mathcal{I}}}_3|
    		&\lesssim \|\mathcal{R}(\Phi_g-G^\ast_{g}\sin(2\theta))g\|_{\tilde{\mathcal{H}}^{2}_{w+}}\|g\|_{\tilde{\mathcal{H}}^{2}_{w-}}
    		\\ \notag
    		&\lesssim  \alpha^{-\frac12} (\|\mathcal{R}(\tilde{\Phi}_g)\|_{\mathcal{H}^2}+\|\mathcal{R}(\tilde{G}_{g}\sin(2\theta))\|_{\mathcal{H}^2})\|g\|_{\tilde{\mathcal{H}}^{2}}
    		\big(\alpha \|g\|_{\tilde{\mathcal{H}}^{2}}+\alpha^{-\frac{2}{\beta}}\|g\|_{{\mathcal{H}}^{2}}\big)
    		\\
    		\notag
    		&\lesssim_{\beta}
    		\alpha^{-\frac12} \|g\|_{\mathcal{H}^2}\|g\|_{\tilde{\mathcal{H}}^{2}}
    		\big(\alpha \|g\|_{\tilde{\mathcal{H}}^{2}}+\alpha^{-\frac{2}{\beta}}\|g\|_{{\mathcal{H}}^{2}}\big)
    		\\
    		\notag
    		&\lesssim_\beta 
    		\alpha^{\frac12}\|g\|_{\mathcal{H}^2}\|g\|_{\tilde{\mathcal{H}}^{2}}^2
    		+\alpha^{\frac12-\frac4\beta}\|g\|^3_{\mathcal{H}^2}.
    	\end{align}
    	Combining \eqref{est-mRF-1}, \eqref{sec9:est-Hw}, \eqref{sec9:est-Hw-} and Lemma \ref{9pr1}, it suffices to prove that
    	\begin{align}\label{9in2-1}
    		|\tilde{\bar{\mathcal{I}}}_4|
    		\lesssim \|\mathcal{R}(\Phi_{F^\ast_\gamma}-G^\ast_{F^\ast_\gamma}\sin(2\theta))g\|_{\tilde{\mathcal{H}}^{2}_{w+}}\|g\|_{\tilde{\mathcal{H}}^{2}_{w-}}
    		\lesssim_\beta (1-\eta)^{-\frac12}
    		\big(
    		\alpha^{\frac32}\|g\|_{\tilde{\mathcal{H}}^{2}}^2
    		+\alpha^{-\frac12-\frac4\beta}\|g\|^2_{\mathcal{H}^2}\big).
    	\end{align}
    	Then \eqref{est-mRg-1} implies that
    	\begin{align}\label{9in3-1}
    		|\tilde{\bar{\mathcal{I}}}_5|
    		\lesssim_\beta \alpha^{\frac32}\|g\|^2_{\tilde{\mathcal{H}}^{2}}
    		+\alpha^{\frac12}\|g\|^2_{\mathcal{H}^2}.
    	\end{align}
    	With the help of \eqref{est-mRF-1}, we find that
    	\begin{align}\label{9in4-1}
    		|\tilde{\bar{\mathcal{I}}}_6|
    		\lesssim_\beta (1-\eta)^{-\frac{1}{2}}\alpha^{2}\|g\|_{\tilde{\mathcal{H}}^{2}}.
    	\end{align}
    	Substituting \eqref{9in6-1}-\eqref{9in4-1} into \eqref{sec9:est-bmI-1}, we finally obtain \eqref{sec9:est-R2-H2eta}.

    \end{proof}
    
    We now state the following proposition regarding the estimates of $R_0$ and $R_1$.
    
    \begin{prop}\label{9R0R1-1}
    	Let $L^{-1}_{z,K}(g)(0)=0$. There exist constants $\alpha>0$ sufficiently small and $\eta(\beta)$, such that if $\alpha\ll 1-\eta\ll \beta$, $\mu<0$ and $c_1\alpha<|\mu|< c_2 \alpha$, then
    	\begin{align}\label{sec9:est-R0R1-H2eta}
    		|\langle R_0+R_1, g\rangle_{\tilde{\mathcal{H}}^{2}}|
    		\leq C_{\beta,1-\eta}  \alpha^{\frac32}
    		\|g\|_{\tilde{\mathcal{H}}^{2}}
    		+C_{\beta,1-\eta} \big(c_1\alpha +\alpha^2\big)\alpha\|g \|_{\tilde{\mathcal{H}}^{2}}^2
    		+C_{\beta,1-\eta} \alpha^{-1} \|g \|_{\mathcal{H}^{2}}^2.
    	\end{align}
    \end{prop}
    \begin{proof}
    	In order to control $R_0$, we first claim that
    	\begin{align}\label{est-Fgamma}
    		\|F^\ast_\gamma z^{\frac1\gamma}(1+z^{\frac1\gamma})^{-1}\|_{\tilde{\mathcal{H}}^2}
    		\lesssim_{\beta,1-\eta}  \alpha^{\frac12} .
    	\end{align}
    	We then apply \eqref{4eq14} and \eqref{est-Fgamma}, to deduce that 
    	\begin{align}\label{9in16-1}
    		|\langle R_0, g\rangle_{\tilde{\mathcal{H}}^{2}}|
    		\lesssim \|R_0\|_{\tilde{\mathcal{H}}^{2}}\|g \|_{\tilde{\mathcal{H}}^{2}}
    		\lesssim_{\beta} \alpha \|F^{\ast}_{\gamma} z^{\frac{1}{\gamma}}(1+z^{\frac{1}{\gamma}})^{-1}\|_{\tilde{\mathcal{H}}^{2}} \|g\|_{\tilde{\mathcal{H}}^{2}}
    		\lesssim_{\beta,1-\eta}   \alpha^{\frac32}
    		\|g\|_{\tilde{\mathcal{H}}^{2}}.
    	\end{align}
    	The treatment of $R_1$, which cannot be controlled directly because $|\mu| \lesssim_{\beta,1-\eta}  \alpha$, necessitates a more precise estimate. 
    	Before proceeding, we deduce a useful estimate.
    	It follows from integration by parts and \eqref{2Co-cal-Dzww} that
    	\begin{align*}
    		\langle g + \beta D_z g, g\tilde{w}_z^2\rangle_z
    		=\big\langle 1 -\frac{\beta}{2}\frac{D_z^\ast \tilde{w}_z^2}{\tilde{w}_z^2}, g^2\tilde{w}_z^2\big\rangle_z
    		=c_1\alpha \|g \tilde{w}_z\|_{L^2_z}^2
    		-(2+A)\big\langle \frac{1}{1+z^{\frac1\beta}}, g^2\tilde{w}_z^2\big\rangle_z.
    	\end{align*}
    	This, along with \eqref{sec9:est-cor-1}, yields directly that
    	\begin{align}\label{sec9:est-g}
    		\big|\langle g + \beta D_z g, g\tilde{w}_z^2\rangle_z\big|
    		\lesssim 
    		\big(c_1\alpha +\alpha^2\big)\|g \tilde{w}_z\|_{L^2_z}^2
    		+\alpha^{-2} \|g w_z\|_{L^2_z}^2.
    	\end{align}
    	We then employ \eqref{sec9:est-g} and $|\mu|\leq \alpha$, to obtain that
    	\begin{align*}
    		\big|\langle R_1, g\rangle_{\tilde{\mathcal{H}}^{2}}\big|
    		\lesssim |\mu| \big|\langle g+\beta D_z g, g(\tilde{w}^\eta)^2\rangle_{\tilde{\mathcal{H}}^{2}}\big|
    		\lesssim_{\beta,1-\eta} 
    		\big(c_1\alpha +\alpha^2\big)\alpha\|g \|_{\tilde{\mathcal{H}}^{2}}^2
    		+\alpha^{-1} \|g \|_{\mathcal{H}^{2}}^2.
    	\end{align*}
    	This, combined with \eqref{9in16-1}, gives \eqref{sec9:est-R0R1-H2eta}.
    	
    	To complete the proof, it suffices to prove \eqref{est-Fgamma}.
    	We note that
    	\begin{align*}
    		c_1\alpha<|\mu|<c_2\alpha,\quad \mu<0.
    	\end{align*}
    	It then follows from a similar way as \eqref{est-L2} that
    	\begin{align}\label{sec9:est-L2}
    		\|F^\ast_{\gamma}z^{\frac1\gamma}(1+z^{\frac1\gamma})^{-1}\tilde{w}^\eta\|^2_{L^{2}}
    		&\lesssim \left(\frac{\alpha}{\gamma}\right)^2(1-\eta)^{-1}\int_0^{\infty} \frac{z^{\frac{4}{\gamma}}}{(1+z^{\frac{1}{\gamma}})^6}\frac{(1+z^{\frac{1}{\beta}})^{4+2A}}{z^{\frac{4}{\beta}}}dz \\ 
    		&\lesssim \left(\frac{\alpha}{\beta}\right)^2(1-\eta)^{-1}
    		\Big(\int_0^{1} z^{\frac4\gamma-\frac4\beta}dz 
    		+\int_{1}^{+\infty} z^{\frac2\beta-\frac2\gamma-1+\frac{2c_1\alpha}{\beta}}dz 
    		\Big)
    		\notag
    		\\
    		&\lesssim 
    		\alpha^{2}\beta^{-2} (1-\eta)^{-1}
    		 \big(1+2\alpha^{-1}\beta^{-1}\big)
    		\lesssim_{\beta,1-\eta} \alpha,
    		\notag
    	\end{align}
        where we have used the fact that
        \begin{align*}
        	\beta\gamma^{-1}-1-c_1\alpha =-\frac{2\mu}{1+\mu}-c_1\alpha>0,\quad
        	\beta\gamma^{-1}-1-c_1\alpha  \lesssim_{\beta,1-\eta}  \alpha.
        \end{align*}
        By an argument analogous to that used for Lemma \ref{7fm1} and combined with \eqref{sec9:est-L2}, we deduce \eqref{est-Fgamma}.
        The proof of this lemma is complete.
    \end{proof}

    A direct application of Propositions \ref{9co2}, \ref{9Tg}, \ref{9Tf}, \ref{9R1}, and \ref{9R0R1-1} yields the following proposition.
    \begin{prop}\label{9R3}
    	Let $L^{-1}_{z,K}(g)(0)=0$, $\mu<0$ and $c_1\alpha<|\mu|< c_2 \alpha$. There exist constants $\alpha>0$ sufficiently small and $\eta(\beta)$, such that if $\alpha\ll 1-\eta\ll \beta$, then
    	\begin{align*}
    		\|g\|_{\tilde{\mathcal{H}}^2}\lesssim C_\alpha.
    	\end{align*}
    \end{prop}
    \begin{proof}
    	Relying on Propositions \ref{9co2}, \ref{9Tg}, \ref{9Tf}, \ref{9R1}, and \ref{9R0R1-1}, together with the fact that $\alpha \ll 1-\eta \ll \beta$, we complete the proof of this proposition.
    \end{proof}
\begin{rema}
   We point out that in proving the boundedness of $\|g\|_{\tilde{\mathcal{H}}^2}$, condition $c_1 \alpha\leq|\mu|$ is not necessary. It suffices to use $|\mu|>0$, and by choosing the new weight and the constants $B_0-B_2$ that depends on $|\mu|$, we can still prove the boundedness of $\|g\|_{\tilde{\mathcal{H}}^2}$.
\end{rema}

\subsection{Non-implosion mechanism}

With Propositions \ref{8prop1}, \ref{9R3} at hand, we are ready to prove Theorem \ref{Theo2}.

\begin{proof}[\textbf{Proof of Theorem \ref{Theo2}}]
The vorticity blow-up of Euler  equations has been proven in Proposition \ref{8prop1}. In order to clarify the main idea, we next focus on the the proof of the space-time integrability result $(u_r,u_3)\in L^\infty([0,T^\ast];L^{\infty}_{loc}(D_0))$. Other result $(u_r,u_3)\in L^p([0,T^\ast];L^{\infty}_{loc}(D_0))$ for any $p\in [1,\infty)$  is simpler by changing $A\in(1-\frac 1 p-\frac \beta 2,1-\frac \beta 2)$ in the definition of the new space $\tilde{\mathcal{H}^2}$, so the proof is omitted here. Then, it remains to verify \eqref{est-theo-3} with $\mu<0$. 

By using \eqref{est-ur-1}, we deduce that
	\begin{align}
		\||\bm{x}|^{(1+c_1\alpha)\alpha\beta^{-1}-1} u_r\|_{L^\infty} 
		\leq&~ \|R^{(1+c_1\alpha)\beta^{-1}} r^{-1} u_r\|_{L^\infty} 
		\leq t_{\gamma}^{c_1\alpha} \|z^{\frac12+\frac{A}{\beta}}\mathcal{R}(\Phi)\|_{L^\infty}
		\label{sec9:est-ur-Lloc}\\
		\lesssim &~ t_{\gamma}^{c_1\alpha}
		\Big(\|z^{\frac12+\frac{A}{\beta}}\mathcal{R}(\Phi_{F^\ast_\gamma}-G^\ast_{F^\ast_\gamma}\sin (2\theta))\|_{L^\infty}
		+ \|z^{\frac12+\frac{A}{\beta}}\mathcal{R}(\Phi_g-G^\ast_g\sin (2\theta))\|_{L^\infty}
		\nonumber\\
		&
		+\|z^{\frac12+\frac{A}{\beta}}\alpha^{-1} L^{-1}_{z,K}(F^\ast_\gamma)\|_{L^\infty}
		+\|z^{\frac12+\frac{A}{\beta}}(\sin\theta)^2\langle F^\ast_\gamma,K\rangle_{\theta}\|_{L^\infty}
		\nonumber\\
		&
		+\|z^{\frac12+\frac{A}{\beta}}\big( \alpha^{-1} L^{-1}_{z,K}(g)
		-(\sin\theta)^2\langle g,K\rangle_{\theta}\big)\|_{L^\infty}
		\Big).
		\nonumber
	\end{align}
	We proceed to analyze the terms on the right-hand side of \eqref{est-ur-Lloc} in sequence. We first note that
	\begin{align*}
		&\sup_{\theta\in(0,\frac{\pi}{2})}\|z^{\frac12+\frac{A}{\beta}}\mathcal{R}(\Phi_g-G^\ast_g\sin (2\theta))\|_{L^\infty[1,\infty)} \lesssim \alpha^{-\frac12} \|g\|_{\tilde{\mathcal{H}^2}} \leq C_\alpha;
		\\
		&\sup_{\theta\in(0,\frac{\pi}{2})}\|z^{\frac12+\frac{A}{\beta}}\alpha^{-1} L^{-1}_{z,K}(F^\ast_\gamma)\|_{L^\infty} =2 \|z^{\frac12+\frac{A}{\beta}}(1+z^{\frac1\gamma})^{-1}\|_{L^\infty} \leq 2;
		\\
		&\sup_{\theta\in(0,\frac{\pi}{2})}\|z^{\frac12+\frac{A}{\beta}}(\sin\theta)^2\langle F^\ast_\gamma,K\rangle_{\theta}\|_{L^\infty} \lesssim \alpha \|z^{\frac12+\frac{A}{\beta}+\frac1\gamma}(1+z^{\frac1\gamma})^{-2}\|_{L^\infty} \lesssim \alpha;
		\\
		&\sup_{\theta\in(0,\frac{\pi}{2})}\|z^{\frac12+\frac{A}{\beta}}\big( \alpha^{-1} L^{-1}_{z,K}(g)
		-(\sin\theta)^2\langle g,K\rangle_{\theta}\big)\|_{L^\infty[1,\infty)} \lesssim \alpha^{-1} \|g\|_{\tilde{\mathcal{H}}^2} \leq C_\alpha.
	\end{align*}
	Turning to the first term on the right-hand side of \eqref{sec9:est-ur-Lloc}, we control it using Corollary \ref{9em2}, which yield
	\begin{align*}
		\|z^{\frac12+\frac{A}{\beta}}\mathcal{R}(\Phi_{F^\ast_\gamma}-G^\ast_{F^\ast_\gamma}\sin (2\theta))\|_{L^\infty}
		\lesssim &~
		\alpha^{-\frac12} \|z^{\frac{A}{\beta}}F_{\gamma}^\ast \|_{\tilde{\mathcal{H}}^{2,\ast}}
		\lesssim_{\beta,1-\eta} 1.
		\nonumber
	\end{align*}
	The above estimates, together with \eqref{est-ur-Lloc}, imply that
	\begin{align}\label{sec9:est-ur}
		\sup_{s\in[0,T^\ast]} \||\bm{x}|^{(1+c_1\alpha)\alpha\beta^{-1}-1} u_r(s)\|_{L^\infty}  
		\leq C_\alpha \sup_{s\in[0,T^\ast]} s_{\gamma}^{c_1\alpha}
		\leq C_\alpha,
	\end{align}
	which yields that $u_r\in L^{\infty}([0,T^\ast];L^\infty_{loc})$.
	As for $u_3$, following the same technical line as in the derivation of \eqref{sec9:est-ur}, we establish that $u_3\in L^{\infty}([0,T^\ast];L^\infty_{loc})$.
	
\end{proof}

\section*{Acknowledgement}
W. Deng is partially supported by Postdoctoral Fellowship Program of CPSF(No. GZB20230940), and China Postdoctoral Science Foundation (No. 2024M764279). 
M. Li is partially supported by Postdoctoral Fellowship Program of CPSF (No. GZB20240024), and China Postdoctoral Science Foundation (No. 2024M760057 and No. 2025T180840).
\bibliographystyle{abbrv}

\end{document}